\documentclass[spanningrule]{cambridge7A}

\pdfoutput=1 
 \usepackage[numbers]{natbib}

  \usepackage[figuresright]{rotating}
  \usepackage{floatpag}
  \rotfloatpagestyle{empty}

\usepackage{amsmath}% if you are using this package,
  \usepackage{amsthm}
  \usepackage{graphicx}

 \usepackage{txfonts}

\usepackage{mathrsfs}

 \usepackage{makeidx}
 \makeindex

  \usepackage[number=none]{glossary}
\makeglossary

  \theoremstyle{plain}% default
  \newtheorem{theorem}{Theorem}[chapter]
  \newtheorem{lemma}[theorem]{Lemma}
  \newtheorem{proposition}[theorem]{Proposition}
  \newtheorem{corollary}[theorem]{Corollary}

  \newtheorem*{theorem*}{Theorem}
  \newtheorem*{lemma*}{Lemma}
  \newtheorem*{proposition*}{Proposition}
  \newtheorem*{corollary*}{Corollary}
  \newtheorem*{conjecture*}{Conjecture}

  \theoremstyle{definition}
  \newtheorem{definition}[theorem]{Definition}
  \newtheorem{example}[theorem]{Example}
  
  \newtheorem{remark}[theorem]{Remark}

  \newtheorem*{definition*}{Definition}
  \newtheorem*{example*}{Example}
  \newtheorem*{prob*}{Problem}
  \newtheorem*{remark*}{Remark}
  \newtheorem*{notation*}{Notation}
  \newtheorem*{exer*}{Exercise}

  \def\makeRRlabeldot#1{\hss\llap{#1}}
  \renewcommand\theenumi{{\rm (\roman{enumi})}}
  \renewcommand\theenumii{{\rm (\alph{enumii})}}
  \renewcommand\theenumiii{{\rm (\arabic{enumiii})}}
  \renewcommand\theenumiv{{\rm (\Alph{enumiv})}}

\usepackage{scalerel}

\newcommand\reallywidehat[1]{\arraycolsep=0pt\relax%
\begin{array}{c}
\stretchto{
  \scaleto{
    \scalerel*[\widthof{\ensuremath{#1}}]{\kern-.5pt\bigwedge\kern-.5pt}
    {\rule[-\textheight/2]{1ex}{\textheight}} %WIDTH-LIMITED BIG WEDGE
  }{\textheight} % 
}{0.5ex}\\           % THIS SQUEEZES THE WEDGE TO 0.5ex HEIGHT
#1\\                 % THIS STACKS THE WEDGE ATOP THE ARGUMENT
\rule{-1ex}{0ex}
\end{array}}

\usepackage{amssymb,amscd, amsxtra}
\usepackage{latexsym}
\usepackage[all,ps]{xy}
\usepackage{comment}
\usepackage{pdfsync}

\usepackage{graphicx}
\usepackage{mathtools}
\usepackage{epigraph}   %% CITATION AT THE BEGINNING
\usepackage[symbol]{footmisc}
\usepackage{tikz-cd}
\usepackage{tikz}
\usetikzlibrary{decorations.pathmorphing,decorations.markings,arrows,calc,shapes.geometric,arrows.meta,positioning}
\tikzset{math3d/.style=
    {x= {(-0.353cm,-0.353cm)}, z={(0cm,1cm)},y={(1cm,0cm)}}}

\usepackage{color}
\definecolor{Chocolat}{rgb}{0.36, 0.2, 0.09}
\definecolor{BleuTresFonce}{rgb}{0.215, 0.215, 0.36}

\usepackage[colorlinks,final,hyperindex, ]{hyperref}
\hypersetup{citecolor=BleuTresFonce, linkcolor=Chocolat, urlcolor  = BleuTresFonce}
\usepackage[noabbrev,capitalize]{cleveref}
 
\DeclareMathAlphabet{\mathbbold}{U}{bbold}{m}{n}

\newcommand{\BCH}{\mathrm{BCH}}
\def\R{\mathfrak{R}}
\def\m{\mathfrak{m}}
\def\k{\mathbbold{k}}
\def\dd{\mathrm{d}}
\def\Ai{\calA_\infty}
\def\Li{\calL_\infty}
\def\sAi{\calS\calA_\infty}
\def\sLi{\calS\calL_\infty}
\def\dsLi{\calS^{-1}\calL_\infty}
\def\Tw{\mathrm{Tw}}
\def\at{\alpha}
\def\F{\mathrm{F}}
\def\G{\mathrm{G}}
\def\NN{\mathbb{N}}
\def\RR{\mathbb{R}}
\def\wA{\widehat{A}}
\def\wo{\widehat{\otimes}}
\newcommand{\cc}{\circledcirc}

\DeclareMathAlphabet{\pazocal}{OMS}{zplm}{m}{n}

\def\a{\mathfrak{a}}
\newcommand{\Sy}{\mathbb{S}}

\def\calA{\pazocal{A}}
\def\calB{\pazocal{B}}
\def\calC{\pazocal{C}}

\def\calF{\pazocal{F}}
\def\calG{\pazocal{G}}

\def\calL{\pazocal{L}}

\def\calP{\pazocal{P}}
\def\calQ{\pazocal{Q}}

\def\calS{\pazocal{S}}
\def\S1{\pazocal{S}^{-1}}
\def\calT{\pazocal{T}}
\def\calU{\pazocal{U}}

\DeclareMathOperator{\Hom}{Hom}
\DeclareMathOperator{\eend}{end}
\DeclareMathOperator{\End}{End}
\DeclareMathOperator{\id}{id}

\def\colim{\mathop{\mathrm{colim}}}
\def\wcolim{\mathop{\widehat{\mathrm{colim}}}}

\DeclareMathOperator{\ad}{ad}
\DeclareMathOperator{\Def}{Def}

\DeclareMathOperator{\Der}{Der}

\DeclareMathOperator{\As}{As}
\DeclareMathOperator{\Ass}{Ass}
\DeclareMathOperator{\uAs}{uAs}

\DeclareMathOperator{\BV}{BV}
\DeclareMathOperator{\ncBV}{ncBV}

\DeclareMathOperator{\Gerst}{Gerst}
\DeclareMathOperator{\ncGerst}{ncGerst}

\DeclareMathOperator{\Gra}{Gra}

\DeclareMathOperator{\HyperCom}{HyperCom}
\DeclareMathOperator{\PreLie}{PreLie}
\DeclareMathOperator{\uPreLie}{uPreLie}
\DeclareMathOperator{\Perm}{Perm}
\DeclareMathOperator{\Dias}{Dias}
\DeclareMathOperator{\Dend}{Dend}

\DeclareMathOperator{\Lie}{Lie}
\DeclareMathOperator{\Com}{Com}

\DeclareMathOperator{\D}{D}

\def\Im{\mathop{\mathrm{Im}}}

\def\C{\pazocal{C}}
\def\I{\mathrm{I}}

\def\g{\mathfrak{g}}
\newcommand{\MC}{\mathrm{MC}}

\newcommand{\ac}{\scriptstyle \text{\rm !`}}

\bookkeywords{Operads, deformation theory, Maurer--Cartan elements, twisting procedure, homotopy associative algebras,  homotopy Lie algebras, graph complexes.}

\begin{document}

\title{Maurer--Cartan methods in deformation theory: the twisting procedure}\footnotetext{\hrule\smallskip\noindent This material will be published by Cambridge University Press \& Assessment as ‘Maurer-Cartan Methods in Deformation Theory: the twisting procedure’ by Vladimir Dotsenko, Sergey Shadrin and Bruno Vallette. This version is free to view and download for personal use only. Not for re-distribution, re-sale or use in derivative works. \copyright Cambridge University Press \& Assessment}

\author{Vladimir Dotsenko, Sergey Shadrin, and Bruno Vallette}

\bookabstract{
This monograph provides a conceptual study of the twisting procedure, which creates functorially new differential graded Lie algebras, associative algebras or operads (as well as their homotopy versions) from a  Mau\-rer--Cartan element. It relies on an explicit description of the gauge symmetries of Maurer--Cartan elements via the integration theory of complete pre-Lie algebras. We give a criterion on quadratic operads for the existence of a meaningful twisting procedure of their associated categories of algebras. We also give a concise conceptual treatment of the twisting procedure for operads \`a la Willwacher and a sufficient source of motivating examples related to graph homology, both recovering known graph complexes and introducing some new ones. Finally, we outline some of the most striking applications of the twisting procedure outside deformation theory and operad theory \emph{per se}, including a user-friendly survey of ``higher Lie theory'' of D. Robert--Nicoud and the third author and its applications in rational homotopy theory.}

\maketitle

\newpage 

\epigraph{Il me semble qu'on pourrait tirer de ce travail une confirmation des points de vue suivants : d'abord l'int\'er\^et que pr\'esente l'\'etude de groupes d\'efinis \`a partir d'autres structures poss\'edant plusieurs op\'erations (par exemple des alg\`ebres de Lie). En effet, la simplicit\'e apparente des axiomes des groupes ne fait souvent que masquer une extr\^eme complexit\'e, et d'autres structures, plus riches par le nombre de leurs axiomes, se laissent plus facilement \'etudier. Il conviendrait donc de rechercher si d'autres structures alg\'ebriques pourraient permettre la construction de nouvelles cat\'egories de groupes.}{Michel Lazard, Ph.D. Thesis \cite{Lazard50}}

\setcounter{tocdepth}{1}

\tableofcontents

\chapter*{Introduction}\footnotetext{\hrule\smallskip\noindent This material will be published by Cambridge University Press \& Assessment as ‘Maurer-Cartan Methods in Deformation Theory: the twisting procedure’ by Vladimir Dotsenko, Sergey Shadrin and Bruno Vallette. This version is free to view and download for personal use only. Not for re-distribution, re-sale or use in derivative works. \copyright Cambridge University Press \& Assessment}

The seminal works of L. Maurer \cite{Maurer88} and E. Cartan \cite{Cartan04} investigating the integrability of Lie algebras to  Lie groups effectively  introduced what differential geometers now call the Maurer--Cartan $1$-form. 
In this language, the \emph{Maurer--Cartan equation}
\[d\omega+{\tfrac12}[\omega,\omega]=0\] becomes the flatness condition for the connection defined by that form. 
In general, a flat connection in a vector bundle $E\to M$ allows one to define a twisted de Rham differential on the sheaf of $E$-valued differential forms. 
In the case of principal bundles, one actually deals with differential forms with values in the structure Lie algebra, which form a \emph{differential graded Lie algebra}, that is a Lie algebra structure in the category of chain complexes. This is the conceptual framework for the Maurer--Cartan equation. 
Its solutions are called the \emph{Maurer--Cartan elements}, in general, and they coincide with flat connections in the case of principal bundles. 
Each of them produce a \emph{twisted differential} under the formula 
\[d^\omega(\theta)=d\theta+[\omega,\theta]\ .\] 
The gauge group, consisting of the gauge transformations of a principal bundle, is the group of symmetries that acts on flat connections. 
The Maurer--Cartan equation, the \emph{twisting procedure}, and the gauge group action 
constitue the \emph{Maurer--Cartan methods}, which lie at the core of gauge theory. 
In fact, they make sense for differential graded Lie algebras of more abstract nature. \\

Around 1960, the Maurer--Cartan equation started to be understood as the structural equation in deformation theory: in the study of the 
deformations of complex manifold structures by A. Fr\"olicher and A. Nijenhuis~\cite{MR83801}, K. Kodaira, L. Nirenberg and D.C. Spencer~\cite{KodairaSpencerI, KodairaSpencerII, MR112157}, and M. Kuranishi \cite{MR141139}, and in the study of deformations of associative algebra structures by M. Gerstenhaber \cite{Gerstenhaber64}. 
A few years later, A. Nijenhuis and R.W. Richardson~\cite{MR195995} noticed the omnipresence of differential graded Lie algebras in deformation theory. 
Their work was one of inspirations behind that of V.P. Palamodov \cite{MR0508121} where deformation theories of complex structures and of commutative algebras are brought together, following the definition of the tangent complex by G. Tyurina (unpublished). 
These examples and the unifying role played by the conceptual notion of a differential Lie algebra eventually led P. Deligne~\cite{DeligneMillson} and V. Drinfeld \cite{MR3285856} to formulate the general principle of deformation theory claiming that over a field of characteristic $0$, any deformation problem can be encoded by a differential graded Lie algebra. 
To be more precise, given an underlying ``space'' (manifold, chain complex, etc.) and a type of structures, there should exist a differential graded Lie algebra such that the structures of that type on that space are in one-to-one correspondence with the Maurer--Cartan elements of that differential graded Lie algebra. 
Then, Maurer--Cartan elements in the twisted differential graded Lie algebra should correspond to deformations of the original structure.
Finally, the Maurer--Cartan elements lying in the same orbit under the gauge group action should  correspond to equivalent structures. 
This guiding principle received recently a precise statement, including a definition of what is a ``deformation problem'', and a proof by J. Lurie \cite{Lurie10} and J. Pridham \cite{Pridham10}, see also the Bourbaki seminar given by B. Toen \cite{Toen17} on the subject; this is the 
\emph{fundamental theorem of deformation theory}.  \\

In the 1970's, the Maurer--Cartan methods in differential graded Lie algebras were one of the building blocks used in the development of the rational homotopy theory by D. Quillen \cite{Quillen69} and D. Sullivan \cite{Sullivan77}, see also the works of 
K.-T. Chen \cite{Chen73, Chen77} and 
 of M. Schlessinger and J. Stasheff \cite{SchlessingerStasheff}. 
In this context, the Maurer--Cartan elements in the Lie models \cite{Neisendorfer78, Hain84, BFMT20} of a topological space correspond to its points. 
 In this case, the twisting procedure creates a Lie model of the same topological space but pointed at the given Maurer--Cartan element. 
Finally, the gauge group orbits of the Maurer--Cartan elements correspond to the connected components of the topological space. \\

Since then, deformation theory and rational homotopy theory grew up together and never ceased to interact with each other. 
This range of ideas led to groundbreaking new results. To name but a few, let us emphasize here the following two beautiful and influential achievements: the deformation theory of representations of the fundamental groups of compact K\"ahler manifolds by W.M. Goldman and J.J. Millson \cite{GoldmanMillson88} and the deformation quantisation of Poisson manifolds by M. Kontsevich \cite{Kontsevich03}.
We refer the reader to the recent exhaustive book by M. Manetti \cite{Manetti22} on Lie methods in deformation theory for a more thorough historical recollection. \\
 
Nowadays, one can notice that these conceptual ideas were successfully applied in many different research areas, which demonstrates the universality of the Maurer--Cartan equation, the twisting procedure, and the gauge group action. Let us offer a brief recollection of some of those situations, which appeared recently. 
 The construction of the Floer cohomology of Lagrangian submanifolds  in symplectic geometry described in \cite{FOOO09I} by K. Fukaya, Y.-G. Oh, H. Ohta, and K.  Ono,  is given by first  considering a curved homotopy associative algebra and then twisting it with a Maurer--Cartan element, when one exists, in order to produce a meaningful differential. 
 The higher Lie theory \cite{Hinich01, Getzler09, Robert-NicoudVallette20} which amounts to integrating  homotopically coherent generalisations of Lie algebras ($\Li$-algebras) into $\infty$-groupoids relies on sets, or rather, moduli spaces of Maurer--Cartan elements. 
  In higher algebra, twisted homotopy Lie algebras together with some notion of $\infty$-morphisms, provide us with a suitable higher categorical enrichment for the categories of homotopy algebras over an operad \cite{DolgushevRogers17, DolgushevHoffnungRogers14}.  
 Finally, in quantum algebra, the twisting procedure for operads themselves allowed T. Willwacher to reinterpret the graph complex introduced earlier by M. Kontsevich in \cite{Kontsevich97} and to prove that its degree $0$ cohomology group are given by the Grothendieck--Teichm\"uller Lie algebra \cite{Willwacher15}. 
 Going even further with this interpretation, T. Willwacher was able to prove the following conjecture of M. Kontsevich \cite{Kontsevich99}: the group of homotopy automorphisms of the little disks operad is isomorphic to the (pro-unipotent) Grothendieck--Teichm\"uller group, see also B. Fresse \cite{Fresse17II} for a more stable approach. \\

This monograph guides the reader through various versions of the twisting procedure, aiming to settle an extensive toolbox, including new properties, for various applications and an elaborate survey of the said applications. The guiding principle that we rely on is that Maurer--Cartan elements should be studied through their symmetries. Needless to say, this suggestion of P. Deligne~\cite{DeligneMillson}, first advertised in print in work of W.M. Goldman and J.J. Millson~\cite{GoldmanMillson88}, fits perfectly into the general philosophy like that of F. Klein in his Erlangen program \cite{MR1510801}: symmetries play a crucial role in studying mathematical objects. The kind of symmetries that arise in the context of the twisting procedure are called \emph{gauge symmetries}, following the original context of gauge theory. Over the recent years, we have pursued a research programme where we tried to describe ``all'' functorial procedures producing new homotopy algebra structures, from a given one, using suitable gauge symmetries: this way, it is possible to recover 
the homotopy transfer theorem \cite{DotsenkoShadrinVallette16}, 
the Koszul hierarchy \cite{Markl15, DotsenkoShadrinVallette16}, and the Givental action~\cite{DotsenkoShadrinVallette15bis, DSVV20}. 
This monograph completes that programme by including the twisting procedure into this picture for the first time. 
This agrees with the ideas of M. Lazard quoted in the epigraph: arguments that seem to rely on long and complicated calculations get a simple conceptual explanation using group theory. \\

The action of gauge symmetries is defined via flows of certain vector fields, thus, one needs to solve ordinary differential equations or, in more algebraic terms, to integrate infinitesimal Lie algebra actions to group actions. 
This is achieved by the universal Baker--Campbell--Hausdorff (BCH) formula \cite{MR1576644,MR1576434,MR1575931,zbMATH02646605,MR0021940, BF12}, which is an infinite series in the free Lie algebra on two variables. 
Moreover, the homotopically coherent version of the Maurer--Cartan equation in homotopy Lie algebras is also an infinite series itself. 
So, one finds themselves on the borderline of algebra and analysis, needing to make sense of sums of infinite series. The way to handle it, going back to the generalisation of the Lie theory to filtered Lie algebras and groups due to M. Lazard in his Ph.D. thesis \cite{Lazard50}, is to consider filtered chain complexes whose topology,  defined by a basis of open neighbourhoods of the origin consisting of decreasing sub-modules, is required to be complete. Nowadays this type of topology has become omnipresent in commutative algebra, algebraic geometry, deformation theory, rational homotopy theory, and microlocal analysis. Our approach, informed by the operad theory, makes use of the general symmetric monoidal properties of filtered and complete differential graded modules. \\

An important feature of differential graded Lie algebras arising in the deformation theory of algebras encoded by operads is that they come from pre-Lie algebras of convolution type, see \cite[Chapter~10]{LodayVallette12}. 
The notion of a pre-Lie algebra sits between those of an associative algebra and of a Lie algebra: any associative algebra is a pre-Lie algebra and the skew-symmetrisation of the pre-Lie binary product renders a Lie bracket. 
In \cite{DotsenkoShadrinVallette16}, we showed, under a strong weight graded assumption, that the integration of a Lie algebra coming from a pre-Lie algebra can be done by closed 
combinatorial formulas that are more managable than the generic BCH formula. 
The symmetric monoidal properties mentioned above ensure that most of that integration theory for pre-Lie algebras adapts without much change to the complete setting. 
This integration theory of complete pre-Lie algebras is perhaps the part of our work that has the strongest connection to the seminal work of M. Lazard \cite{Lazard50}. 
The upshot of our approach is the largest \emph{deformation gauge group} of algebraic structures modelled by operads that we are aware of. \\

The intrinsic meaning of the twisting procedure for differential graded Lie algebras admits the following bright interpretation. First, one moves from differential graded Lie algebras to homotopically coherent and more general $\Li$-algebras. The structure of an $\Li$-algebra on a chain complex $V$ is, under certain finiteness assumptions, encoded by a differential on the exterior algebra of the linear dual chain complex $V^*$. Such a differential, interpreted geometrically, is a vector field, satisfying a relation called \emph{homological}, on the supermanifold given by~$V$ with the homological degrees shifted by one. 
If that homological vector field vanishes at the origin, one obtains the usual notion of an $\Li$-algebra, otherwise, one arrives at the definition of a \emph{curved} $\Li$-algebra. The twisting procedure may be viewed as a change of coordinates moving the origin to the point where the homological vector field vanishes; this vanishing condition is encoded by the Maurer--Cartan equation. \\

We provide this geometric intuition with a precise and conceptual algebraic counterpart using the large deformation gauge group mentioned above: the action of the simplest gauge symmetries of the convolution algebra controlling curved $\Li$-algebra structures produce the twisting procedure. This allows us to reprove in a straightforward and short way its various properties, notably the crucial ones related to complete (curved) $\Li$-algebras used in deformation theory, like in \cite{Getzler09, DolgushevRogers15}. 
A similar twisting procedure is available in the case of (curved) $\Ai$-algebras, that is  algebras associative up to a infinite system of coherent homotopies. In fact, we choose to present the results for (curved) $\Ai$-algebras in more detail
since  that formalism seems less known, 
since its presentation is simpler, and 
since 
the corresponding results for (curved) $\Li$-algebras are then obtained \emph{mutatis mutandis}. 
Finally, these two examples of twisting procedures lead us to ask the following natural question: why can one twist $\Li$-algebras and $\Ai$-algebras and what about the other types of algebras? The conceptual understanding of the twisting procedure mentioned above also allows us to give a criterion explaining when a given type of algebras  admits a meaningful twisting procedure. Heuristically speaking, a category of homotopy algebras over a quadratic operad can be twisted if and only if the Koszul dual category of algebras admits a coherent notion of a unit. \\

The topic which originally motivated us to  understand the twisting procedure in detail is that of the \emph{operadic twisting} introduced by T. Willwacher in his seminal work on the Grothendieck--Teichm\"uller Lie algebra and M. Kontsevich's graph complexes \cite{Willwacher15}. It was later studied in detail by V. Dolgushev and T. Willwacher in \cite{DolgushevWillwacher15}, see also the survey by V. Dolgushev and C. Rogers~\cite{DolgushevRogers12}. This theory relies on twisting an operad by one of its Maurer--Cartan elements, and it turns out that the language developed in this monograph allows us to encode the operadic twisting procedure in a very direct and straightforward new way. We note that this approach to the operadic twisting was first suggested by J.~Chuang and A.~Lazarev in \cite{MR3004818}. 
A salient point of the operadic twisting procedure lies in the fact that it naturally gives rise to meaningful graph complexes, recovering some of the seminal ones originally introduced by M. Kontsevich in \cite{Kontsevich93, Kontsevich97}. \\

We hope that our monograph gives exactly the kind of a gentle introduction which is needed to make the theory of  operadic twisting accessible to a much wider audience: 
it offers a concise and conceptual way of thinking about the twisted operad and a sufficient source of motivating examples related to graph homology. We recover known computations of graph homology but with more direct methods and we  introduce some new ones related to the noncommutative analogues of Gerstenhaber and Batalin--Vilkovisky algebras introduced in our recent work~\cite{DotsenkoShadrinVallette15}. 
We also survey the key role played by the operadic twisting procedure in the understanding of the Grothendieck--Teichm\"uller Lie algebra, the Deligne conjecture, and its Lie version. \\

We conclude this book 
by offering the reader with short surveys of some recent applications 
 of the twisting procedure and more generally, the Maurer--Cartan methods. It is however important to emphasise that the twisting procedure appears in too many mathematical contexts to hope for an exhaustive survey of all possible applications. 
 Topics like the deformation theory of algebras over properads \cite{MerkulovVallette09I, MerkulovVallette09II}, 
the twisting morphisms (or  twisting cochains) used to produce models for the homology of fibered spaces \cite{Brown59, Cartan55general, HusemollerMooreStasheff74}, 
 the twisting procedure for homology of hairy graph complexes and its applications \cite{MerkulovWillwacher15, TW17, KW19, FTW20, CGP21, Willwacher22}, 
and many others remained outside our scope. 

\subsection*{\sc Organisation of the monograph}
Chapter~\ref{sec:MCIntro} begins with a survey on gauge theory, which is the differential-geometric origin of the Maurer--Cartan equation, and continues with basic but extensive recollections on the Maurer--Cartan elements and their symmetries in differential graded Lie algebras. 
In Chapter~\ref{sec:OptheoyFilMod}, we establish the various symmetric monoidal properties of filtered and complete differential graded modules in order to develop their operadic theory. 
We settle the integration theory for complete differential graded left-unital pre-Lie algebras in Chapter~\ref{sec:TopoDefTh}; this gives rise to a large gauge group which is shown to govern the deformation theory of homotopy algebras over an operad. Chapter \ref{sec:GaugeTwist} contains the first immediate application of the previous chapter: the action of the arity $0$ elements of the deformation gauge group is shown to give the twisting procedure for (curved homotopy) associative algebras and (curved homotopy) Lie algebras.
From this conceptual interpretation, we easily derive ``all'' the properties of the twisting procedure. In Section~\ref{subsec:TwistableAlg}, we give a criterion on a quadratic operad that ensures that the corresponding category of homotopy algebras admits a meaningful twisting procedure. In Chapter~\ref{sec:TwNsOp}, we present a concise conceptual treatment of T. Willwacher's operadic twisting procedure. Chapter~\ref{sec:Computations} discusses some examples of the operadic twisting, especially those relevant for graph homology, and their applications to the Grothendieck--Teichm\"uller Lie algebra, the Deligne conjecture, and a Lie version of this latter one. The last Chapter~\ref{sec:Applications} provides the literature with short surveys on some of the seminal applications of the twisting procedure in a wider mathematical context: fundamental theorem of deformation theory, higher Lie theory, rational homotopy theory, simplicial theory of homotopy algebras, and Floer cohomology of Lagrangian submanifolds. 

\subsection*{\sc Conventions}
Objects studied in this book are $\k$-modules; for simplicity, we work over a field $\k$ of characteristic zero though 
many results still hold over an arbitrary ring. 
The rule of thumb to use when deciphering necessary assumptions is as follows: $\k$ is a ring when working with algebras over nonsymmetric operads (for example, with $\Ai$-algebras), and $\k$ is a field of characteristic zero when working with algebras over symmetric operads (for example, with $\Li$-algebras).\\

We work in the underlying category of chain complexes, so that all differentials have degree $-1$. 
To accommodate that, we grade cohomology groups in negative degrees, when working with the cohomology of a topological space or a manifold (with an exception to the rule for the very first section~\ref{sec:MCDiffGeo} where we stick to the classical conventions). 
The linear dual of a chain complex is understood component-wise $\left(V^\vee\right)_n
\coloneqq\left(V_{-n}\right)^\vee$.
Homological degrees of elements and operations create signs when evaluating operations on their arguments, according to the usual Koszul sign rule and convention, see \cite[Section~1.5.3]{LodayVallette12}. \\

For operad theory, we use the same conventions and notations than the book \cite{LodayVallette12}. 
For instance, we use the abbreviation ``dg'' to mean ``differential graded'' and ``ns'' to mean ``nonsymmetric''. Somewhat abusing terminology, we use the term  ``cooperad'' for the algebraic structure defined by partial/infinitesimal decomposition maps $\Delta_{(1)} : \calC \to \calC\, {\circ}_{(1)}\, \calC$, see \cite[Section~6.1.4]{LodayVallette12}.
In the nonsymmetric case, the upshot of such maps is a linear combination of elements of the form $\mu \circ_i \nu$, where this convention stands for the 2-vertex tree with one internal edge linking the root of the corolla $\nu$ to the $i^{\textrm{th}}$-leaf of the corolla $\mu$. \\

Since many differentials of various types appear throughout the text, we chose the following notations to clarify the situation: 
the underlying differentials are denoted by $d$, the differentials of operads are denote by $\dd$, and the differentials of mapping spaces are denoted by $\partial$. More precisely, for two chain complexes $(A,d_A)$ and $(B, d_B)$, the differential on $\Hom(A,B)$ is given by 
\[\partial f \coloneqq d_B\circ f - (-1)^{|f|}f\circ d_A~, \]
for homogenous maps $f : A\to B$.

\subsection*{\sc Acknowledgements}
It is with great pleasure that we thank Joan Bellier-Mill\`es, Ricardo Campos, Joana Cirici, Gabriel Drummond-Cole,  
Richard Hain, 
Daniel Robert-Nicoud, Agata Smoktunowicz, Jim Sta\-sh\-eff, 
Victor Turchin 
and Tho\-mas Willwacher  for interesting discussions. 
We would like to address our appreciations to the editorial staff of the Cambridge University press, with particular emphasis to Roger Astley, Tom Harris and Anna Scriven for their patience and wise advice. 
We are grateful to 
the University of Amsterdam, 
Trinity College Dublin,
the University Nice Sophia Antipolis, the University Sorbonne Paris Nord, and the 
 Max Planck Institute for Mathematics in Bonn for the excellent working conditions during the preparation of this book. 
V. D. and B.V. were supported by the Institut Universitaire de France. 
S.S. was supported by the Netherlands Organisation for Scientific Research. 
V.D. was supported by a Fellowship of the University of Strasbourg Institute for Advanced Study through the French national program ``Investment for the future'' (IdEx-Unistra, grant USIAS-2021-061)
Preparation of the final version of the manuscript was supported by the project HighAGT ANR-20-CE40-0016.

\chapter{Maurer--Cartan methods}\footnotetext{\hrule\smallskip\noindent This material will be published by Cambridge University Press \& Assessment as ‘Maurer-Cartan Methods in Deformation Theory: the twisting procedure’ by Vladimir Dotsenko, Sergey Shadrin and Bruno Vallette. This version is free to view and download for personal use only. Not for re-distribution, re-sale or use in derivative works. \copyright Cambridge University Press \& Assessment}

%[Title of chapter]{Title of chapter\footnote{Text of footnote}}
\label{sec:MCIntro}

The purpose of this chapter is to give a rather exhaustive survey on the Maurer--Cartan equation and its related methods, which lie at the core of the present monograph. 
We first give a recollection of the Maurer--Cartan equation and its gauge symmetries in differential geometry.
This chapter is viewed as a motivation for the rest of the book, which consists of higher algebraic generalisations of the key notions of gauge theory. Reading it is not  mandatory to understand what follows but this might help the reader to get some concrete pictures in mind before passing to a more abstract treatment. 
Then, we establish the general theory of the Maurer--Cartan equation in differential graded Lie algebras. 
With that in hand, we discuss the philosophy of deformation theory suggesting that studying Maurer--Cartan elements of differential graded Lie algebras, as well as the symmetries of those elements, is the central question of \emph{any} deformation theory problem in characteristic $0$~. 
In the last chapter~\ref{sec:Applications}, we shall discuss more recent developments making that philosophy precise by means of higher category theory.  \\

Throughout this chapter, various infinite series arise. For simplicity, we work with the strong assumption that the various differential graded Lie algebras are nilpotent, so that all these series are actually finite once evaluated on elements. We refer to the treatment of complete algebras given in the next chapter~\ref{sec:OptheoyFilMod} 
for the correct setup in which convergence is understood in the rest of the text.

\section{Maurer--Cartan equation in differential geometry}\label{sec:MCDiffGeo}
In this section, we give a short outline of the differential geometric notions of which the subject of this book is a far reaching algebraic generalisation.   
We review the fundamental objects and the classical results  of gauge theory: vector and principal bundles, connections, and curvatures. 
Nowadays, these notions play a key role in analysis, geometry, and topology \cite{MilnorStasheff74, DK97, NaberInter, Audin04}; they also provide physicists with the suitable conceptual language to express modern theories \cite{Manin97, NaberFound, Hamilton17}. 

Throughout this section, we work over the field of real numbers and we denote by $M$ a smooth manifold. We assume the reader familiar with the basic notions of differential manifolds, as treated in \cite{Warner83}, for instance. 
For more details on this section, we refer the reader to the textbooks  \cite{KobayashiNomizu63, Tu17}. 
\\

Given a smooth vector bundle $E\to M$, one considers the space 
\[\Omega^\bullet(M,E)\coloneqq \Gamma\left(\Lambda^\bullet T^*M\otimes E\right)\cong 
\Omega^\bullet(M)\otimes_{\Omega^0 (M)} \Gamma(E)\]
of \emph{differential forms with values in $E$}.\index{differential forms}
In order to extend the de Rham differential map to $E$-valued differential forms, one is led to the following notion.

\begin{definition}[Connection]\index{connection!vector bundle}\label{def:VectBunConnection}
A \emph{connection} of a smooth vector bundle $E\to M$ 
 is an $\mathbb{R}$-linear map  
 \[\nabla\colon \Omega^0(M,E)\cong \Gamma(E) \to\Omega^1(M,E)\cong\Gamma(T^*M\otimes E)\] 
 satisfying the \emph{Leibniz rule}\index{Leibniz rule}
 \[
\nabla(fs)=df\otimes s + f\nabla s~, 
 \]
for all $f\in\Omega^0(M)$ and all $s\in\Gamma(E)$~.
\end{definition}

\begin{lemma}
For any connection $\nabla$, there is a unique $\RR$-linear operator 
\[d^\nabla \colon \Omega^\bullet(M,E) \to \Omega^{\bullet +1}(M,E)\] satisfying 
\begin{enumerate}
\item $d^\nabla=\nabla$~, for $\bullet=0$, 

\item and the generalised Leibniz rule
\begin{equation}\label{Eq:GeneLeibnizRule}\index{Leibniz rule!generalised}
d^\nabla(\alpha \wedge \omega)=d\alpha\wedge \omega + (-1)^k \alpha \wedge d^\nabla \omega ~, 
\end{equation}
for any $\alpha\in\Omega^k(M)$ and any $\omega \in\Omega^l(M,E)$~.
\end{enumerate}
\end{lemma}

\begin{proof}
It is given by the following definition
 \[
d^\nabla(\alpha \otimes s)\coloneqq d\alpha\otimes s + (-1)^k \alpha \wedge \nabla s~, 
 \]
for all $\alpha\in\Omega^k(M)$ and all $s\in\Gamma(E)$~.
\end{proof}

Let us now look for a condition implying that $d^\nabla$ is a differential, that is squares to zero. In this direction, we first consider the composite 
\[d^\nabla \circ \nabla \colon \Omega^0(M,E) \to \Omega^{2}(M,E)~, \]
which is actually $\Omega^0(M)$-linear. 

\begin{definition}[Curvature]\index{curvature!vector bundle}
The \emph{curvature} of a connection $\nabla$ is the $\End(E)$-valued 2-form 
\[\theta \in \Omega^2(M,\End(E))~\]
obtained from $d^\nabla \circ \nabla$ under the isomorphism 
\[
\Hom_{\Omega^0(M)}\left(\Omega^0(M,E), \Omega^2(M,E)\right)
\cong 
\Omega^2(M,\End(E))~.
\]
\end{definition}

\begin{definition}[Flat connection]\index{connection!flat}
A connection $\nabla$ is called \emph{flat} when its curvature is trivial, i.e.
$\theta=0$~.
\end{definition}

This condition is necessary to get a differential; it is actually enough. 

\begin{lemma}\label{lem:flatconSq}
The composite $d^\nabla \circ d^\nabla \colon \Omega^\bullet(M,E) \to \Omega^{\bullet +2}(M,E)$ is equal to 
\[ \left(d^\nabla\circ d^\nabla\right)(\alpha\otimes s)= \alpha \wedge \left(d^\nabla \circ \nabla\right) (s)~,\]
for $\alpha\in \Omega^\bullet(M)$ and $s\in \Gamma(E)$~.
\end{lemma}

\begin{proof}
The generalised Leibniz rule \eqref{Eq:GeneLeibnizRule} gives 
\begin{align*}
\left(d^\nabla\circ d^\nabla\right)(\alpha\otimes s)=d^\nabla\left( d\alpha \otimes s + (-1)^k \alpha \wedge \nabla s\right)= 
\alpha \wedge \left(d^\nabla\circ \nabla\right)(s)~,
\end{align*}
where $k$ stands for the degree of $\alpha$.
\end{proof}

\begin{proposition}\label{cor:FlatDiff}
For any flat connection $\nabla$~, the operator $d^\nabla$ squares to zero. 
\end{proposition}

\begin{proof}
This is a direct corollary of Lemma~\ref{lem:flatconSq}. 
\end{proof}
 
 \begin{definition}[Twisted de Rham differential/complex]\index{twisted!de Rham differential}\index{twisted!de Rham  complex}
 The differential $d^\nabla$ on the space of $E$-valued differential forms 
  associated to a flat connection $\nabla$ is called the 
  \emph{twisted de Rham differential}. The cochain complex 
$\left(\Omega^\bullet(M,E), d^\nabla \right)$
  is called the \emph{twisted de Rham complex}. 
  \end{definition}
  
\begin{remark}
  One recovers the classical de Rham differential on $\Omega^\bullet(M)$ by considering the trivial line bundle.
\end{remark}

Let us now  look at the local situation. Using the local trivialisation of the vector bundle $\pi \,\colon E\to M$ above an open subset $U\subset M$, any basis $(v_1, \ldots, v_n)$ of the typical (finite dimensional) fiber $V$ induces a collection $e=(e_1, \ldots, e_n)$  of sections, with 
$e_i\in \Gamma(E|_U)$~, such that $\left(e_1(x), \ldots, e_n(x)\right)$ is a basis of the fiber $E_x\coloneqq \pi^{-1}(x)$~. Such a collection is called a \emph{local frame over $U$}.\index{local frame} In  such a local frame, the data of the connection $\nabla$ is equivalent to a collection of local 1-forms $\omega_{ij}\in \Omega^1(U)$ such that 
\[
\nabla e_j=\sum_{i=1}^n  \omega_{ij}\otimes e_i~. 
\]

\begin{definition}[Local connection form]\index{connection!local form}
The \emph{local connection form} with respect to the frame $e=(e_1, \ldots, e_n)$ 
is the matrix $\omega_e\coloneqq \left(\omega_{ij}\right)_{i,j=1,\ldots, n}\in \mathrm{gl}_n\big(\Omega^1(U)\big)$.
\end{definition}

\begin{proposition}\label{lemma:curvatureform}
The curvature  is given locally by 
\[\theta_e\cong d \omega_e + \omega_e^2 = d\omega_e +\tfrac12 [\omega_e, \omega_e]\]
 in $\mathrm{gl}_n\big(\Omega^2(U)\big)$~.
\end{proposition}

\begin{proof}
It is obtained by the following straightforward computation: 
\begin{align*}
d^\nabla\circ \nabla(e_i)&=\sum_{i=1}^n d^\nabla(\omega_{ij}\otimes e_i)
=\sum_{i=1}^n d\omega_{ij}\otimes e_i - \sum_{i=1}^n \omega_{ij}\wedge \nabla e_i\\
&=
\sum_{i=1}^n \left(d\omega_{ij}  - \sum_{k=1}^n    \omega_{kj} \wedge \omega_{ik}   \right)\otimes e_j
=\sum_{i=1}^n \left(d\omega_{ij}  + \sum_{k=1}^n    \omega_{ik} \wedge \omega_{kj}   \right)\otimes e_i~.
\end{align*}
\end{proof}

In other words, the twisted map $d^\nabla$ is a differential if and only 
the local connection forms satisfy 
the following first kind of \emph{``Maurer--Cartan equation''} \index{Maurer--Cartan equation!differential geometry}
\[d \omega_e + \omega_e\cdot \omega_e = d\omega_e +\tfrac12 [\omega_e, \omega_e]=0~.\]

Such an equation does not depend of the choice of local frames as the property for a connection to be flat is global. 
Explicitly, a \emph{change of local frame} over $U$ from 
$e=\left(e_1, \ldots, e_n\right)$ to $e'=\big(e'_1, \ldots, e'_n\big)$ is an invertible matrix $A\in \mathrm{GL}_n\big(\Omega^0(U)\big)$ such that 
$e'=eA$~. 

\begin{proposition}\label{prop:LocalGaugeAction}
The local connection form with respect to the frame $e'$ 
is given by the matrix 
\[\omega_{e'}=A^{-1}d A + A^{-1} \omega_e A\]
 in $\mathrm{gl}_n\big(\Omega^1(U)\big)$~. 
\end{proposition}

\begin{proof}
This follows from  the  straightforward computation: 
\begin{align*}
\nabla\big(e'_j\big)&= \sum_{k=1}^n  \nabla\left(A_{kj}\, e_k\right)
= \sum_{k=1}^n  \left(dA_{kj}\otimes e_k+A_{kj} \nabla e_k\right)\\
&=\sum_{i=1}^n \left(\sum_{k=1}^n  A^{-1}_{ik}\left(dA_{kj}\right) 
+ \sum_{k,l=1}^n A^{-1}_{il}\, \omega_{lk}\, A_{kj} \right)e'_i~. 
\end{align*}
\end{proof}

This is the first instance of \emph{``gauge group action''} on solutions to the Maurer--Cartan equation. \\

Any vector bundle $E$ induces a linear dual bundle $E^*$ and an endomorphism bundle $\End(E)$~. 
In turn, any connection $\nabla$ on $E$ gives rise to canonical connections on $E^*$ and $\End(E)$ as follows. 
Let us first recall the non-degenerate pairing 
\[
(\; , \, ) \ \colon \Omega^i(M,E)\otimes \Omega^j(M,E^*)
\xrightarrow{\wedge}
\Omega^{i+j}(M)\otimes_{\Omega^0 (M)} \Gamma(E\otimes E^*)
\xrightarrow{\langle\,, \, \rangle} \Omega^{i+j}(M)~, 
\]
where $\langle\,,  \rangle$ stands for the linear paring, that is 
\[(\alpha\otimes s, \beta\otimes t)\coloneqq (\alpha \wedge \beta)\, \otimes \langle s,t \rangle~, \]
for $\alpha \in \Omega^i(M)$, $\beta \in \Omega^j(M)$, $s\in \Gamma(E)$, and $t\in \Gamma(E^*)$~. 
To any  connection $\nabla$ on $E$, one associates a connection $\nabla^*$ on $E^*$ characterised  by 
\[
\left(s, \nabla^* t \right)=d(s,t)-\left(\nabla s, t\right)
\]
and then a connection  $\widehat{\nabla}$ on $\End(E)\cong E \otimes E^*$ given by 
\[ 
\widehat{\nabla}(s\otimes t)\coloneqq \nabla s \otimes t +s\otimes \nabla^*t~. 
\]

\begin{proposition}
The twisted de Rham differential on the endomorphism bundle $\End(E)$ is locally  given by 
\[
d^{\widehat{\nabla}}f=df +\omega_e f - (-1)^k f\omega_e=df + [\omega_e, f]\]
 in $\mathrm{gl}_n\big(\Omega^{k+1}(U)\big)$~, for any $f\in \mathrm{gl}_n\big(\Omega^{k}(U)\big)$~~. 
\end{proposition}

\begin{proof}
Using the local frame over $U$, one can write 
\[
f=\sum_{i,j=1}^n f_{ij}\otimes \delta_{ij}~, 
\]
where $\delta_{ij}=e_i\otimes e_j^*\in \Gamma(\End(E)|_U)$ sends $e_j$ to $e_i$ and where $f_{ij}\in \Omega^k(U)$~. 
By definition, we have 
\begin{align*}
\widehat{\nabla}(\delta_{ij})&=\widehat{\nabla}\big(e_i\otimes e_j^*\big)=
\nabla e_i\otimes e_j^*+e_i\otimes \nabla^* e_j^*\\&=
\left(\sum_{l=1}^n \omega_{li}\otimes e_l\right)\otimes e^*_j -
e_i\, \otimes \left(\sum_{l=1}^n \omega_{jl}\otimes e_l^*\right)
= 
\sum_{l=1}^n \omega_{li}\otimes \delta_{lj}-
\sum_{l=1}^n \omega_{jl}\otimes \delta_{il}
~.
\end{align*}

The generalised Leibniz rule \eqref{Eq:GeneLeibnizRule} gives 
\begin{align*}
d^{\widehat{\nabla}}f&= d^{\widehat{\nabla}}\left(\sum_{i,j=1}^n f_{ij}\otimes \delta_{ij}\right)=
\sum_{i,j=1}^n df_{ij}\otimes \delta_{ij}+(-1)^k \sum_{i,j=1}^n f_{ij}\otimes \widehat{\nabla}(\delta_{ij})\\
&=\sum_{i,j=1}^n df_{ij}\otimes \delta_{ij}+
(-1)^k
\sum_{i,j,l=1}^n f_{ij}\wedge \omega_{li} \otimes \delta_{lj}
-(-1)^k\sum_{i,j,l=1}^n 
f_{ij}\wedge \omega_{jl}\otimes \delta_{il}
\\&=\sum_{i,j=1}^n 
\left(
df_{ij}
+\sum_{l=1}^n \omega_{il}\wedge f_{lj}
-(-1)^k
\sum_{l=1}^n f_{il}\wedge \omega_{lj}
\right)
\otimes \delta_{ij}
~.
\end{align*}

\end{proof}

For the first time, we encounter  a \emph{``differential twisted by a solution to the Maurer--Cartan equation''}.
\\

Let us now pass from the local picture to the global one. 
To this extend, one needs an extra action of a Lie group of the fiber bundle, leading to the notion of a principal bundle.  
Developing the notions of connection and curvature at this level will make even more noticeable the role played by the methods from Lie theory. 
The two theories of vector and principal bundles are essential equivalent: any vector bundle induces a canonical principal bundle, called the frame bundle, and one can associate vector bundles, like the adjoint bundle, to any principal bundle. 
Let $G$ be a real Lie group with Lie algebra $\g$.

\begin{definition}[Principal bundle]\index{principal bundle}
A \emph{principal $G$-bundle} is a fiber bundle $P\to M$ equipped with a smooth (right) action of $G$ which is free, transitive, and fiber preserving. 
\end{definition}

The definition implies  the  identifications $P_x\cong G$, for the fibers,  and $P/G\cong M$, for the orbit space.

\begin{example}
The toy model of principal bundle is the \emph{frame bundle}\index{frame bundle} $\mathrm{Fr}(E)\to M$ associated to any vector bundle $E\to M$: 
elements of its fibers are ordered bases of the fibers of $E$. In this case, the structure Lie group $G=\mathrm{GL}_n$ is the general linear group, where $n$ is the dimension of the fibers of $E$~. 
\end{example}

In differential geometry, a \emph{distribution} is a subbundle of the tangent bundle. \index{distribution}

\begin{definition}[Vertical distribution]\index{distribution!vertical}
The \emph{vertical distribution} $T^vP\subset TP$ of a principal bundle $P$ is defined by 
\[T^v_p P\coloneqq \ker \mathrm{D}_p \, \pi~,\quad  \text{for any} \ p\in P~,\]
where $\mathrm{D}_p\, \pi  \colon T_p P \to T_{\pi(p)}M$ stands for the derivative of the structural projection $\pi \colon P\to M$~. 
\end{definition}

One can see that each fiber of the vertical distribution is isomorphic to the tangent Lie algebra $T^v_p P\cong \g$~.

\begin{definition}[Horizontal distribution]\index{distribution!horizontal}
A \emph{horizontal distribution} $T^hP\subset TP$ of a principal bundle $P$ is a distribution complementary to the vertical distribution:
\[T^vP\oplus T^hP = TP~.\]
\end{definition}

Notice the discrepancy between these two notions: the vertical distribution is uniquely and canonically defined while there exists possibly many horizontal distributions. 
Any horizontal distribution gives rise to a $\g$-valued 1-form $\omega$ on $P$ defined by
\[ 
\omega_p \colon T_pP=T^v_pP \oplus T^h_pP \xrightarrow{\mathrm{proj}} T^v_pP \cong \g
~,\]
 for any $p\in P$~, where the first map is the projection on $T^v_pP$  along $T^h_pP$~. 
 In order to make explicit its properties, we need first  to recall the following notions. 
 \\
 
 Let us denote by $\mathrm{R}\colon P\times G \to P$ the right action of $G$ on the principal bundle $P$~. (We will  use the simpler notation $\mathrm{R}_g(-)\coloneqq \mathrm{R}(-, g)$ for the right action of an element $g\in G$.)
 The \emph{fundamental vector field $X^\xi$ associated to $\xi \in \g$}\index{fundamental vector field} is defined by 
 \[
 X^\xi_p\coloneqq \mathrm{D}_{(p,e)}\, \mathrm{R}(0, \xi)\in T^v_pP~, 
 \]
 for any $p\in P$~. 
 The \emph{adjoint representation}\index{adjoint representation} 
 \[\mathrm{Ad}_g \coloneqq \mathrm{D}_e \mathrm{C}_g \colon \g \to \g\]
 is given by the derivation at the unit $e$ of the Lie group $G$ of the conjugation map 
 $\mathrm{C}_g(x)\coloneqq gxg^{-1}$, for any $g,x\in G$~.

\begin{definition}[Connection on a principal bundle]\index{connection!principal bundle}
A \emph{connection} on a principal bundle $P$ is a $\g$-valued 1-form $\omega \in \Omega^1(P,\g)=\Omega^1(P)\otimes \g$ satisfying 
the following properties:\begin{description}
\item[\sc vertical vector field:] $\omega_p\left( X^\xi_p\right)=\xi$~, for any $p\in P$ and $\xi \in \g$~,

\item[\sc equivariance:] $\mathrm{Ad}_g \left(\mathrm{R}_g^*\, \omega\right) =\omega$~, for any $g\in G$~.
\end{description}
\end{definition}

The first conditions says that a connection restricts to the identity map of $\g$ under the identification $T^v_p P\cong \g$~.

\begin{proposition}\label{prop:Horiz=Connection}
The data of a horizontal distribution on a principal bundle is equivalent to the data of a connection. 
\end{proposition}

\begin{proof}
The left-to-right assignment is defined above. In the other way round, given a connection $\omega$, one defines a horizontal distribution as its kernel $T^h_pP\coloneqq \ker \omega_p$~.  We refer the reader to \cite[Section~28]{Tu17} for complete details about this proof. 
\end{proof}

\begin{example}\leavevmode
\begin{enumerate}
\item The \emph{Maurer--Cartan connection}\index{connection!Maurer--Cartan } $\omega_G \in \Omega^1(G, \g)$ on the trivial principal bundle $G\to *$ 
 is defined by
 \[\omega_G\coloneqq \mathrm{D}_g \mathrm{L}_{g^{-1}}\colon T_g G \to T_eG \cong \g~,\]
 where $\mathrm{L}_{g^{-1}} \colon G\to G$ is the left multplication by $g^{-1}$~, for $g\in G$~.  

\item Given a connection $\nabla$ on a vector bundle $E\to M$, there is a canonical way \cite[Section~29]{Tu17} to define a connection $\omega$ on the associated frame bundle $\mathrm{Fr}(E)\to M$ such that the pullback along a local frame $e \colon U \to \mathrm{Fr}(E)|_U$ gives back the local connection form: $e^* \omega=\omega_e$~.
\end{enumerate}
\end{example}

The graded vector space $\Omega^\bullet(P, \g)\coloneqq \Omega^\bullet(P)\otimes \g$ of $\g$-valued differential forms on $P$ acquires a canonical differential graded Lie algebra structure (Definition~\ref{Def:DGLieAlg}) as the tensor product of a differential graded commutative algebra with a Lie algebra. This is the relevant algebraic context to express the properties of connections on principal bundle. 

\begin{definition}[Curvature]\index{curvature!principal bundle}
The \emph{curvature} of a connection $\omega$ on a principal bundle is the $\g$-valued 2-form 
defined by 
\[
\Omega\coloneqq d \omega +\tfrac12[\omega, \omega]
~.\]
\end{definition}

\begin{proposition}\label{lemm:PullbackCurvature}
For any connection $\nabla$ on a vector bundle $E\to M$, the pullback along a local frame $e \colon U \to \mathrm{Fr}(E)|_U$ of the induced curvature of the frame bundle $\mathrm{Fr}(E)\to M$  is equal to the local curvature form: 
\[e^* \Omega=\theta_e~.\]
\end{proposition}

\begin{proof}
It follows from the direct computation
\begin{align*}
e^*\Omega&=e^*\left(d\omega\right)+\tfrac12 e^*[\omega, \omega]=d\left(e^*\omega\right)+\tfrac12 [e^*\omega, e^*\omega] = d \omega_e +\tfrac12[\omega_e, \omega_e]=\theta_e~,
\end{align*}
since the pullback preserves the differential and the Lie bracket (second equality), since $e^* \omega=\omega_e$ (third equality), and by Proposition~\ref{lemma:curvatureform} (forth equality). 
\end{proof}

Any connection on a principal bundle induces a decomposition $TP=T^vP\oplus T^hP$ 
of the tangent bundle 
by Proposition~\ref{prop:Horiz=Connection}. The associated vertical and horizontal components of a vector field 
$X\in \Gamma(TP)$ are respectively denoted by $X=X^v+X^h$~.

\begin{theorem}\label{prop:PropertiesofCurv}
The curvature associated to any connection on a principal bundle satisfies the following properties:
\begin{description}
\item[\sc horizontalilty:] $\Omega(X,Y) =d\omega\left(X^h, Y^h\right)$~, for any $X,Y\in \Gamma(TP)$~,

\item[\sc equivariance:] $\mathrm{Ad}_g \left(\mathrm{R}_g^*\, \Omega\right) =\Omega$~, for any $g\in G$~,

\item[\sc bianchi identity:] $d\Omega=\left[\Omega, \omega\right]$~.\index{Bianchi identity}
\end{description}
\end{theorem}

\begin{proof}
The proof of the first point is a direct consequence of the definition of the curvature and the fact that the connection $\omega$ vanishes on horizontal vectors. 
The second point is showed by the same type of computation as in the proof of the above proposition \ref{lemm:PullbackCurvature}. 
The third point is obtained by the following ``differential graded Lie'' type computation 
\begin{align*}
d\Omega&=
\tfrac12d[\omega, \omega]=[d\omega, \omega]=\left[\Omega-\tfrac12 [\omega, \omega], \omega\right]=[\Omega, \omega]~, 
\end{align*}
since the differential is a derivation (second equality) and by the Jacobi identity (last equality).
\end{proof}

\begin{definition}[Flat connection]\index{connection!flat}
A connection $\omega$  on a principal bundle is called \emph{flat} when its curvature   is trivial: 
\[\Omega=d \omega +\tfrac12[\omega, \omega]=0~.\]
\end{definition}

This is the global form of the Maurer--Cartan equation mentioned above. Geometrically, this property is equivalent to the integrability of the horizontal vector fields. 

\begin{theorem}
A connection on a principal bundle is flat if and only if the horizontal vector fields are preserved by the Lie bracket, i.e.
\[
\left[X^h, Y^h\right]=[X,Y]^h~, \quad \text{for any} \ X,Y\in \Gamma(TP)~.
\]
\end{theorem}

\begin{proof}
This is a straightforward consequence of the formula 
\[\Omega(X,Y)=-\omega\left([X,Y]\right)~,\]
for horizontal vector fields $X, Y\in \Gamma\left(T^hP\right)$~. We refer the reader to \cite[Section~3.1]{Morgan98} for complete details. 
\end{proof}

The following construction allows us to come back  to vector bundles.

\begin{definition}[Associated vector bundle]\index{associated vector bundle}
The \emph{vector bundle associated to a principal bundle $P$ and to a (finite-dimensional) representation 
$\rho \colon G \to \mathrm{GL}(V)$} is defined by the coequalizer $P\times_\rho V$, that is the quotient of $P\times V$ under the equivalence relation $(p\cdot g,v)\sim (p, \rho_{g}(v))$, for any $p\in P$, $v\in V$, and $g\in G$~.
\end{definition}

It is straightforward to check that the associated bundle $P\times_\rho V$ defines a vector bundle over $M$ with fibers isomorphic to $V$~. 

\begin{example}\leavevmode
\begin{enumerate}
\item The vector bundle $P\times_\mathrm{Ad} \g$ associated to the adjoint representation $\mathrm{Ad}\colon \G \to \mathrm{GL}(\g)$ is called the \emph{adjoint bundle}.\index{adjoint bundle}

\item The vector bundle associated to the frame bundle $\mathrm{Fr}(E)$ and  the identity representation 
is isomorphic to the original vector bundle $E$~. This shows that any vector bundle is a vector bundle associated to a principal bundle. 
\end{enumerate}
\end{example}

The $V$-valued differential forms on $P$, that is $\Omega^\bullet(P, V)\coloneqq \Omega^\bullet(P)\otimes  V$~,  coincide with the differential forms on $P$ with values in the trivial vector bundle $P\times V$~, that is $\Omega^\bullet(P, P\times V)$ under the notation introduced at the beginning of this section. Let us now make explicit the differential forms on $M$ with values in the associated bundle $E\coloneqq P\times_\rho V$ in terms of the $V$-valued differential forms on $P$~. 

\begin{definition}[Tensorial differential forms]\index{differential forms!tensorial}
A $V$-valued differential form $\alpha\in \Omega^k (P,V)$ on a principal bundle $P$ is called \emph{tensorial} when it is: 
\begin{description}
\item[\sc horizontal:] for any $p\in P$, we have $\alpha_p(u_1, \ldots, u_k)=0$, when at least one tangent vector is vertical, i.e. 
$u_i\in T^v_pP$, for some $1\leqslant i\leqslant k$~,
\item[\sc equivariant:] $\rho_g \left(\mathrm{R}_g^*\, \alpha\right) =\alpha$~, for any $g\in G$~.
\end{description}
We denote the graded vector space of tensorial differential forms by $\Omega^\bullet_{\rho} (P,V)$~.
\end{definition}

\begin{example}\leavevmode
\begin{enumerate}
\item
Theorem~\ref{prop:PropertiesofCurv} shows that the curvature of a connection on a principal bundle is tensorial with respect to the adjoint representation, i.e. $\Omega\in \Omega^2_{\mathrm{Ad}}(P, \allowbreak \g)$~. 
\item The set of connections on a principal bundle forms an affine space with $\Omega^1_\mathrm{Ad}(P, \g)$ as associated vector space.
\end{enumerate}
\end{example}

\begin{proposition}\label{prop:IsoTensValuedForms}
The graded vector space of tensorial $V$-valued differential forms on $P$ 
is isomorphic to the graded vector space of $E$-valued differential forms on $M$:
\[\Omega^\bullet_{\rho} (P,V)\cong \Omega^\bullet (M,E)~,\]
where $E=P\times_\rho V$ is the associated vector bundle. 
\end{proposition}

\begin{proof}
The isomorphism from left to right is given as follows. Let $\alpha\in \Omega^k_{\rho} (P, \allowbreak V)$~, $x\in M$~, and $v_1, \ldots, v_k\in T_x M$~. 
We choose a point $p\in P_x$ in the fiber above $x$ and we choose lifts $u_1, \ldots, u_k\in T_pP$ for the vectors $v_1, \ldots, v_k$~, that is 
$\mathrm{D}_p\pi (u_i)=v_i$~, for $1\leqslant i \leqslant k$~. We denote by $f_p \colon V \to E_x$ the linear isomorphisms defined by 
$v \mapsto \overline{(p, v)}$~. The image of the tensorial differential form $\alpha$ is given by 
\[
(v_1, \ldots, v_k) \mapsto f_p\left(\alpha_p(u_1, \ldots, u_k)\right)~, 
\]
which lives in $\Omega^k(M, E)$~. Since $\alpha$ is equivariant, this definition does not depend on the choice of the point $p\in P_x$ and since $\alpha$ is horizontal, this definition does not depend on the choices of the lifts $u_1, \ldots, u_k\in T_pP$~. \\

In this other way round, let $\beta\in \Omega^k(M, E)$~, $p\in P$~,  and $u_1, \ldots, u_k\in T_pP$~. The image of the $E$-valued differential form $\beta$ on $M$ under the reverse isomorphism is given by 
\[
f_p^{-1}\left(\beta_{\pi(p)}\left(
\mathrm{D}_p\pi (u_1), \ldots, \mathrm{D}_p\pi (u_k)
\right)
\right)~,
\]
which is clearly a  tensorial $V$-valued differential form on $P$~.
\end{proof}

Since the differential $d$ fails to preserve horizontal  $V$-valued differential forms on $P$, one is led to  the following definition. 

\begin{definition}[Covariant derivative]\index{covariant derivative}
The \emph{covariant derivative} associated to a connection $\omega$ on a principal bundle is defined by 
\[
d^\omega(\alpha)(X_1, \ldots, X_{k+1})\coloneqq (d \alpha) \left(X_1^h, \ldots, X_{k+1}^h\right)
~, 
\]
for $\alpha \in \Omega^k(P,V)$ and $X_1, \ldots, X_{k+1}\in \Gamma{(TP)}$~. 
\end{definition}

\begin{example}
In terms of the covariant derivative, the first point of Theorem~\ref{prop:PropertiesofCurv} asserts that $\Omega=d^\omega(\omega)$ ~.
\end{example}

\begin{lemma}
The covariant derivative restricts to tensorial $V$-valued differential forms on $P$: 
\[d^\omega\colon \Omega^\bullet_{\rho} (P,V) \to \Omega^{\bullet+1}_{\rho} (P,V)~.\]
\end{lemma}

\begin{proof}
It is enough to check that the differential $d$ preserves equivariant $V$-valued differential forms on $P$~. 
\end{proof}

Proposition~\ref{prop:IsoTensValuedForms} implies that the covariant derivative induces a degree $1$ linear operator on $\Omega^\bullet (M,E)$~. 
From the explicit isomorphisms given in the above proof, one can see that the image of $d^\omega \colon \Omega^0_{\rho} (P,V) \to \Omega^1_{\rho} (P,V)$ gives a map $\nabla_\rho\colon \Omega^0 (M,E) \to \Omega^1 (M,E)$ which  satisfies the Leibniz rule (Definition~\ref{def:VectBunConnection}); so it defines a connection on the associated vector bundle. 
Furthermore, one can see that the covariant derivative $d^\omega$ corresponds to the twisted de Rham differential $d^{\nabla_\rho}$~. These two differentials can be expressed in Lie theoretical terms as follows. 
\\

The infinitesimal version of the group representation $\rho\colon G \to \mathrm{GL}(V)$ produces a Lie algebra representation 
$\mathrm{D}_e \,\rho \colon \g \to \mathrm{gl}(V)$~. This latter one defines an action of the differential graded Lie algebra 
$\Omega^\bullet(P, \g)$ on the graded vector space $\Omega^\bullet(P, V)$ under the formula: 
\begin{multline*}
\left(\tau\cdot \alpha \right)_p\left(v_1, \ldots, v_{k+l}\right)\coloneqq \\
\frac{1}{k!l!}\sum_{\sigma\in \mathbb{S}_{k+l}}\mathrm{sgn}(\sigma)\,
\mathrm{D}_e \,\rho\left(\tau_p\left(v_{\sigma(1)}, \ldots, v_{\sigma(k)}\right)\right)
\left(
\alpha_p\left(v_{\sigma(k+1)}, \ldots, v_{\sigma(k+l)}\right)
\right)
~,
\end{multline*}
for $\tau\in \Omega^k(P, \g)$~, $\alpha\in \Omega^l(P, V)$~, $p\in P$~, and $v_1, \ldots, v_{k+l}\in T_pP$~.

\begin{proposition}
On tensorial $V$-valued differential forms $\alpha \in \Omega^\bullet_{\rho} (P,V)$ on $P$, the covariant derivative is equal to 
\[d^\omega(\alpha)=d \alpha + \omega \cdot \alpha ~. \]
\end{proposition}

\begin{proof}
We refer the reader to the proof of \cite[Theorem~31.19]{Tu17}. 
\end{proof}

In the case of the adjoint bundle, the infinitesimal Lie action is given by 
\[ 
\begin{array}{llll}
\mathrm{ad}\coloneqq \mathrm{D}_e \mathrm{Ad}\colon &\g &\to & \mathrm{gl(\g)}\\
&x &\mapsto &[x,-]~.
\end{array}
\]
So we get the formula 
\[d^\omega(\alpha)=d \alpha +[\omega, \alpha]~, \]
which is the global form of the example of a differential twisted by a Maurer--Cartan element. 

\begin{example}
Using this property, the Bianchi identity of Theorem~\ref{prop:PropertiesofCurv} amounts to $d^\omega(\Omega)=0$~. 
\end{example}

\begin{corollary}\label{lem:PrincipleCurv0Diff}
When the connection $\omega$ is flat, the covariant derivative $d^\omega$ squares to zero. 
\end{corollary}

\begin{proof}
This follows from the following formula:
\begin{align*}
\left(d^\omega \circ d^\omega\right)(\alpha)&= 
d^\omega\left(d \alpha + \omega \cdot \alpha\right)=d(\omega\cdot \alpha) + \omega\cdot (d\alpha +\omega \cdot \alpha)\\&
= d \omega \cdot \alpha - \omega \cdot d \alpha + \omega\cdot d\alpha + \tfrac12 [\omega, \omega] \cdot \alpha
\\&=
(d \omega) \cdot \alpha + \tfrac12 [\omega, \omega] \cdot \alpha  \\
&=
\Omega \cdot \alpha
~,
\end{align*}
since the action $\cdot$ is by a differential graded Lie algebra (second line). 
\end{proof}

Corollary~\ref{lem:PrincipleCurv0Diff} is the exact analogue for the covariant derivative of principal bundles of Proposition~\ref{cor:FlatDiff} for the twisted de Rham differential of vector bundles. 
It implies the following result: when a  connection $\omega$ of a principal bundle is flat, then so is the  connection $\nabla_\rho$ on any associated vector bundle.\\ 

Let us now study the group of symmetries of a principal bundle and its action on the above-mentioned notions. 
Its name is motivated by its applications in physics. 

\begin{definition}[Gauge group]\index{gauge!group}
The \emph{gauge group} $\mathscr{G}(P)$ of a principal bundle $\pi \colon P \to M$ is the group consisting of all fiber-preserving and equivariant diffeomorphisms of $P$, called \emph{gauge transformations}\index{gauge!transformation}: 
\[
\mathscr{G}(P)\coloneqq \left\{
f\colon P \xrightarrow{\cong} P \ | \ f \circ \pi =\pi \ ; \ f\left(\mathrm{R}_g(p)\right)=\mathrm{R}_g(f(p))~, \forall p\in P~, \forall g\in G
\right\}~. 
\]
\end{definition}

\begin{remark}
In the physics literature, the structure group $G$ is called the \emph{gauge group}. In the mathematical literature, the above mentioned group $\mathscr{G}(P)$ is also called the \emph{group of gauge transformations}. We chose the present terminology in order to match with the general definition of the gauge group given in the next section~\ref{sec:MCdgLie}.
\end{remark}

\begin{proposition}
The gauge group is isomorphic to the group of equivariant $G$-valued functions on $P$: 
\[\mathscr{G}(P)\cong \mathcal{C}^\infty(P, G)^G~.\]
\end{proposition}

\begin{proof}
We consider here the conjugation action on $G$, so an equivariant $G$-valued function on $P$ is a smooth map 
$\sigma\colon P \to G$ such that 
\[\sigma(\mathrm{R}_g(p))=\mathrm{C}_{g^{-1}}(\sigma(p))=g^{-1}\sigma(p)g~,\]
for any $p\in P$ and $g\in G$~. 
Given a gauge transformation $f \colon P \to P$, one considers the equivariant $G$-valued function $\sigma_f \colon P \to G$ defined by 
\[F(p)=\mathrm{R}_{\sigma_F(p)}(p)~.\]
\end{proof}

The gauge group is in general infinite dimensional and it acts on $\g$-valued differential forms $\alpha \in \Omega^\bullet(P, \g)$ on the left by pullback: 
\[f.\, \alpha\coloneqq \left(f^{-1}\right)^*\alpha~.\]
This action restricts naturally to (flat) connections. Let us give it a more explicit description. 

\begin{theorem}
The action of a gauge transformation $f \colon P \to P$ on a (flat) connection $\omega$  is given by 
\[f.\,\omega=\mathrm{Ad}_{\sigma_f}\circ \omega + \left(\sigma_f^{-1}\right)^*\omega_G
 ~.\]
\end{theorem}

\begin{proof}
We refer the reader to \cite[Chapter~II]{KobayashiNomizu63} for details.
\end{proof}

This is the global form of the ``gauge action'' formula given in Proposition~\ref{prop:LocalGaugeAction}.\\

Let us denote the affine space of connections on the principal bundle $P$ by $\mathrm{C}(P)$ and the set of flat connections by 
$\mathrm{MC}(P)$~. 
The associated moduli spaces 
\[\mathscr{C}(P)\coloneqq \mathrm{C}(P)/\mathscr{G}(P)\quad \text{and} \quad \mathscr{MC}(P)\coloneqq \mathrm{MC}(P)/\mathscr{G}(P)\]\index{connection!moduli spaces}
of (flat) connections under the gauge group are of fundamental importance in mathematics \cite{AB83, Donaldson83, DK97, Audin04} and in physics \cite{Manin97, Hamilton17}.
In mathematics, the moduli space of flat connections is isomorphic to the character variety 
$\mathrm{Hom}(\pi_1(M), \allowbreak G)/G$~, where $\pi_1(M)$ stands for the fundamental group of $M$~.
In physics terminology, connections are called \emph{gauge fields} and their moduli spaces represent the configuration spaces of quantum field theories on which Feynman path integrals are ``defined'' and ``evaluated''. 

\section{Maurer--Cartan equation in differential graded Lie algebras}\label{sec:MCdgLie}

The expression $d\omega+\tfrac12[\omega,\omega]$ for the curvature of the $\mathfrak{g}$-valued $1$-form defining a connection $\omega$  in a principal bundle $P$ makes sense since the space $\Omega^\bullet(P, \g)$ of all $\mathfrak{g}$-valued forms has a richer structure than merely a chain complex: it is a differential graded Lie algebra. In this section, we recall the general formalism for studying the Maurer--Cartan equation and the symmetries of its solutions in differential graded Lie algebras. 
From now on, we switch from the cohomological degree convention to the homological degree convention, and work over a ground field $\k$ of characteristic $0$. 

\begin{definition}[Differential graded Lie algebra]\index{differential graded Lie algebra}\label{Def:DGLieAlg}
A \emph{differential graded (dg) Lie algebra}  is the data $\g=(A, {d}, [\, , ])$ of a chain complex, that is a collection of vector spaces $A=\{A_n\}_{n\in \mathbb{Z}}$ with linear maps $d \colon A_n \to A_{n-1}$ of degree $-1$ satisfying $d^2=0$, equipped with a degree-preserving linear map 
$[\, ,] : A^{\otimes 2} \to A$, called the \emph{Lie bracket}, satisfying the following properties 
\begin{description}
\item[\sc derivation:] \qquad
$\rule{0pt}{12pt}d([x,y])=[dx, y] + (-1)^{|x|}[x,dy]$~,
\item[\sc skew-symmetry:]
$\rule{0pt}{12pt}[x,y]=-(-1)^{|x||y|}[y,x]$~,
\item[\sc Jacobi identity:]\index{Jacobi identity}
\ $\rule{0pt}{12pt}
[[x,y],z]+
(-1)^{|x|(|y|+|z|)}[[y,z],x]+
(-1)^{|z|(|x|+|y|)}[[z,x],y]
=0~,$
\end{description}
where the notation $|x|$ stands for the degree of homogeneous elements $x\in A_{|x|}$~.
\end{definition}

\begin{definition}[Maurer--Cartan element/equation]\index{Maurer--Cartan equation}\index{Maurer--Cartan element}\label{def:MCEl}
Let $\g=(A, {d}, [\, , ])$ be a dg Lie algebra. 
A \emph{Maurer--Cartan element} $\omega$ if an element of $A_{-1}$ that is a solution to the \emph{Maurer--Cartan equation}
\begin{equation}\label{eq:MCeqLie}
d\omega+\tfrac12[\omega,\omega]=0~.
\end{equation}
The set of  Maurer--Cartan elements in $\g$ is denoted by $\mathrm{MC}(\g)$.
\end{definition}

The intuition behind the formulas that we are about to write comes from the situation where the dg Lie algebra $\g$ is finite-dimensional, so that the Maurer--Cartan equation is actually a finite collection of actual polynomial equations in a finite-dimensional vector space, and the Maurer--Cartan set is a variety (intersection of quadrics). 

\begin{lemma}\label{lemm:TangentMC}
In a (finite dimensional) dg Lie algebra $\g$, the tangent space 
$T_\omega(\mathrm{MC}(\g))$ 
of the Maurer--Cartan variety $\mathrm{MC}(\g)$ at a point $\omega$ consists of elements $\eta\in A_{-1}$ satisfying
\begin{equation}\label{eq:tangentspace}
d^\omega(\eta)\coloneqq d\eta+[\omega,\eta]=0~.
\end{equation}
\end{lemma}

\begin{proof}
The Maurer--Cartan variety is the zero locus of the curvature function 
\[\Omega \colon \omega \in A_{-1} \mapsto d\omega+\tfrac12[\omega,\omega]\in A_{-2}~.\]
Its derivative at $\omega$ is equal to $\mathrm{D}_\omega\,  \Omega\, (\eta)=d\eta+[\omega,\eta]$~.
Finally, the tangent space of $\mathrm{MC}(\g)$ at $\omega$ is the zero locus of this derivative. 
\end{proof}

Considering the \emph{adjoint operator}\index{adjoint operator}
\[\ad_\omega\coloneqq [\omega,-]~,\]
the abovementioned map is equal to the sum $d^\omega=d +\ad_\omega$~.

\begin{proposition}\label{prop:dtwSq0}
For any Maurer--Cartan element $\omega \in \MC(\g)$ of a dg Lie algebra $\g$, the map $d^\omega$ is a derivation satisfying 
\[d^\omega \circ d^\omega=0~.\]
\end{proposition}

\begin{proof}
The Jacobi identity is equivalent to the fact that the adjoint operator  
is a derivation:
\begin{align*}
\ad_\omega([x,y])=[\omega, [x,y]]&=(-1)^{1+|y|(|x|+1)}[y,[\omega, x] + (-1)^{1+|x|+|y|}[x, [y,\omega]]\\
&=[\ad_\omega x, y]+(-1)^{|x|} [x, \ad_\omega y]~. 
\end{align*}
The linear map $d^\omega=d + \ad_\omega$ is a derivation as a sum of derivations. 
Then, the image of  any element $\eta$ of $A$ under the composite $d^\omega \circ d^\omega$ is given by 
\[d^\omega \circ d^\omega(\eta)=\underbrace{d^2(\eta)}_{=0} + \underbrace{d([\omega, \eta]) + [\omega, d\eta]}_{=[d\omega, \eta]}+\underbrace{[\omega, [\omega, \eta]]}_{=\tfrac12 [[\omega, \omega], \eta]} 
 = \Big[ \underbrace{d\omega+\tfrac12 [\omega, \omega]}_{=0}, \eta\Big]=0~,\]
 using the properties of a dg Lie algebra (differential, derivation and Jacobi identity) and the Maurer--Cartan equation. 
\end{proof}

\begin{definition}[Twisted differential and twisted dg Lie algebra]\index{twisted!differential}\index{twisted!dg Lie algebra}
For any dg Lie algebra $\g$ and any Maurer--Cartan element $\omega \in \MC(\g)$, the differential
\[d^\omega=d +\ad_\omega=d+[\omega,  -]\]
 is called the \emph{twisted differential}. The associated dg Lie algebra 
 \[
\g^\omega\coloneqq \left(A, d^\omega, [\, ,]\right)~. 
 \]
is called the \emph{twisted} dg Lie algebra.
\end{definition}

\begin{remark}
The terminology was chosen by analogy with the twisted de Rham differential mentioned in Section \ref{sec:MCDiffGeo}. 
\end{remark}

Lemma~\ref{lemm:TangentMC} and Proposition~\ref{prop:dtwSq0} show that, for each $\omega\in\mathrm{MC}(\g)$ and each $\lambda\in A_0$, the element $d^\omega (\lambda)=d\lambda+[\omega,\lambda]\in T_\omega(\mathrm{MC}(\g))$ lives in the tangent space at $\omega$. In other words, any element $\lambda \in A_0$ induces to a vector field 
 \[
\Upsilon_\lambda\in \Gamma(T(\mathrm{MC}(\g)))
 \] given by 
 \[
\Upsilon_\lambda(\omega) := d\lambda + [\omega,\lambda] ~.
 \]

\begin{definition}[Gauge symmetries of Maurer--Cartan elements]\index{gauge!symmetry}\index{gauge!equivalence}
The flows associated to  vector fields $\Upsilon_\lambda$ for $\lambda\in A_0$ are called the \emph{gauge symmetries} of Maurer--Cartan elements. Two Maurer--Cartan elements $\alpha$ and $\beta$ are said to be \emph{gauge equivalent} if there exists an element $\lambda\in A_0$ for which the flow of the vector field $\Upsilon_\lambda$ relates $\alpha$ to $\beta$ in finite time.
\end{definition}

We shall now introduce a useful method for simplifying calculations in dg Lie algebras, like the explicit expression for the integration of the gauge flow. The differential $d$ of a dg Lie algebra $\g=(A, {d}, [\, , ])$ is said to be \emph{internal} if there exists an element $\delta$ in $A_{-1}$ such that $d=\ad_\delta$~.\index{internal differential}

\begin{definition}[Differential trick]\index{differential trick}
Let $\g=(A, {d}, [\, , ])$ be a dg Lie algebra. The \emph{differential trick}  amounts to considering the one-dimensional extension 
 \[
\g^+:=(A\oplus \k \delta, d,  [\, , ])~,
 \]
 \noindent
\[\text{where}\quad  |\delta|\coloneqq -1~,\quad  d(\delta)\coloneqq 0~, \quad  [\delta, x]\coloneqq d x~, \quad \text{and} \quad [\delta, \delta]\coloneqq 0~.\]
\end{definition}

The dg Lie algebra $\g^+$ is the universal extension of $\g$ which makes its differential an inner derivation. The canonical map 
\[\xymatrix@C=20pt@R=-8pt@M=6pt{
\iota \colon \g \ar@{^{(}->}[r] & \g^+\\
\ \ \ \, x \ar@{|->}[r] &
\left\{
{\begin{array}{ll}
\delta+x\ , & \text{for}\  \ |x|= -1~, \\
x & \text{for} \ \ |x|\neq -1~,
\end{array}}
\right.
}\]
embeds $\g$ inside $\g^+$. The fact that this map is not a morphism of dg Lie algebras is surprisingly useful: as we shall now see, as a consequence, it alters and simplifies various equations and formulas. The main simplification is that an element $\omega\in A_{-1}$ is a Maurer--Cartan element in $\g$ if and only if $\delta +\omega$ is a \emph{square-zero element}\index{square-zero element} 
 \[
[\delta+\omega,\delta+\omega]=0 
 \]
 in the extension $\g^+$. Let us denote by 
 $\mathrm{Sq}(\g^+)$ the set of degree $-1$ square-zero elements in $\g^+$ of the form $\delta+\omega$.
 The above embedding provides us with a canonical identification of the sets (or varieties in the finite dimensional case)
 \[\MC(\g)\cong \mathrm{Sq}(\g^+)\]
 and thus of tangent spaces. 
% \[T_\omega \MC(\g)\cong T_{\delta+\omega} \mathrm{Sq}(\g^+)\]
However the formulas in the latter case are much simpler: 
the arguments given above show that $\ad_{\delta+\omega}=[\delta+\omega, -]$ is a square-zero derivation of $\g^+$, for any $\delta+\omega\in \mathrm{Sq}(\g^+)$, and that the tangent space at this point is given by 
\[T_\omega \MC(\g)\cong T_{\delta+\omega} \mathrm{Sq}(\g^+)=\left\{\eta \in A_{-1}\ | \ [\delta+\omega, \eta]=0\right\}~.\]
In the extension $\g^+$, the formula for the vector field induced by any $\lambda \in A_0$ is 
\[\Upsilon^+_\lambda(\delta+\omega) :=  [\delta+\omega,\lambda]=\ad_{-\lambda}(\delta+\omega)~.\]

\begin{definition}[Nilpotent dg Lie algebra]\index{nilpotent Lie algebra}
A dg Lie algebra $\g=(A, d, [\, , ])$ is called \emph{nilpotent} when there exists an integer $n\in \NN$ such that 
 \[
[x_1,[x_2, \ldots ,[x_{n-1},x_n]\ldots]]=0
 \]
for any $x_1, \ldots, x_n\in A$~.
\end{definition}

This nilpotency condition is enough to ensure that all the infinite series of brackets we will now consider make sense. Such a condition is too strong to cover all the examples that we have in mind: in the next chapter \ref{sec:OptheoyFilMod}, we will settle the more general framework of complete dg Lie algebras inside which fit all our examples and for which all the results below also hold. 

\begin{proposition}
Let $\g=(A, d, [\, ,])$ be a nilpotent dg Lie algebra. The integration of the flow associated to the 
 vector fields $\Upsilon_{-\lambda}$~, for $\lambda\in A_0$, starting at $\alpha\in \MC(\g)$ gives at time $t$: 
 \[\frac{\id-\exp(t\ad_{\lambda})}{t\ad_\lambda}(t\,d\lambda)+\exp(t\ad_{\lambda})(\alpha)~.\]
\end{proposition}

\begin{proof}
We use the differential trick and work in the extension $\g^+$. The differential equation 
\[\frac{d(\delta+\gamma(t))}{dt}=\Upsilon^+_{-\lambda}(\delta+\gamma(t))=
\ad_{\lambda}(\delta+\gamma(t))\]
associated to the flow $\Upsilon^+_{-\lambda}$ is then easy to solve since there is now no more constant term. The adjoint operator being a linear map, the solution to this differential equation is given by the following exponential: 
\begin{align*}
\exp(t\ad_{\lambda})(\delta+ \alpha) &=
\exp(t\ad_{\lambda})(\delta)+\exp(t\ad_{\lambda})(\alpha)\\
&=\delta+(\exp(t\ad_{\lambda})-\id)(\delta) +\exp(t\ad_{\lambda})(\alpha)\\
& =\delta + 
\frac{\id-\exp(t\ad_{\lambda})}{t\ad_{\lambda}}(t\, d\lambda)+\exp(t\ad_{\lambda})(\alpha)~.
\end{align*}

\end{proof}

\begin{remark}
This formula is the algebraic counterpart of the formula for the action of the gauge transformations on connections of principal bundles given in the previous section~\ref{sec:MCDiffGeo}. 
\end{remark}

The universal formula underlying the integration of finite dimensional real Lie algebras into Lie groups is the following one. 

\begin{definition}[Baker--Campbell--Hausdorff formula]\index{Baker--Campbell--Hausdorff formula}
The \emph{Baker--Campbell--Hausdorff (BCH) formula} is the element in the associative algebra of formal power series on two variables $x$ and $y$ given by 
\[\BCH(x,y)\coloneqq \ln\left(e^xe^y\right)~.\]
\end{definition}
It is straightforward to notice that the BCH formula is associative and unital:
\[\BCH(\BCH(x,y),z)=\BCH(x,\BCH(y,z)) \quad \text{and} \quad
\BCH(x,0)=x=\BCH(0,x)~.
\]
The celebrated theorem of Baker, Campbell, and Hausdorff, see \cite{BF12}, states that this formula can be written using only the commutators $[a,b]=a\otimes b - b\otimes a$:
\[
\BCH(x,y)=x+y+\tfrac12[x,y]+\tfrac{1}{12}[x,[x,y]]+\tfrac{1}{12}[y,[x,y]]+\cdots~.\\
\]
It can thus be applied to any nilpotent Lie algebra.

\begin{definition}[Gauge group]\index{gauge!group}
The \emph{gauge group} associated to a nilpotent dg Lie algebra $\g=(A, {d}, [\, , ])$ is the group obtained from $A_0$ via the Baker--Campbell--Hausdorff formula:
 \[
\Gamma:=\left(
A_0, x\cdot y := \BCH(x,y), 0 .
\right)\ ,
 \] 
\end{definition}

The name ``gauge group'' is  justified by the following proposition.

\begin{theorem}\label{prop:GaugeActionMC}
Let $\g=(A, {d}, [\, , ])$ be a nilpotent dg Lie algebra. The formula 
 \[
\lambda.\alpha \coloneqq \frac{\id-\exp(\ad_{\lambda})}{\ad_{\lambda}}(d\lambda)+\exp(\ad_{\lambda})(\alpha)
 \]
for the gauge action defines a left action of the gauge group $\Gamma$ on $\mathrm{MC}(\g)$~.
\end{theorem}

\begin{proof}
We use the differential trick to simplify the calculations: we have to show that the assignment 
\[\lambda.(\delta+\alpha)\coloneqq \exp(\ad_{\lambda})(\delta+\alpha)\]
defines an action of the gauge group $\Gamma$ on $\mathrm{Sq}(\g^+)$~. Since $\ad_{\lambda}$ is a derivation, then $\exp(\ad_{\lambda})$ is a morphism of graded Lie algebras. This implies that $\lambda.(\delta+\alpha)$ is again a square-zero element:
\begin{align*}
[\lambda.(\delta+\alpha),\lambda.(\delta+\alpha)]&=
\left[ \exp(\ad_{\lambda})(\delta+\alpha), \exp(\ad_{\lambda})(\delta+\alpha)\right]\\&=
\exp(\ad_{\lambda})([\delta+\alpha,\delta+\alpha])=0~.
\end{align*}
It is straightforward to  check that the action of $0$ is trivial 
 \[
0.(\delta+\alpha) =  \exp(\ad_0)(\delta+\alpha)=\id(\delta+\alpha)=\delta+\alpha\ .
 \]
The BCH formula satisfies the relation 
\[
\exp(\ad_{\BCH(x, y)})=\exp(\ad_{x})\circ \exp(\ad_{y})~.
\]
(The BCH formula is actually characterised by this relation, see \cite[Proposition~5.14]{Robert-NicoudVallette20}.)
To prove it, let us work in the associative algebra of formal power series in three variables $x$, $y$, $z$.
Notice that for any element $a$, the adjoint operator $\ad_a=l_a-r_a$ is equal to the difference of the left multiplication by $a$ with the right multiplication by $a$. Since the underlying algebra is associative, these two linear maps commute and thus 
\[\exp(\ad_a)(b)=\exp(l_a)\circ\exp(r_{-a})(b)= e^{a} \, b \, e^{-a}~.\]
Using this, we get 
\begin{align*}
\exp(\ad_{\BCH(x, y)})(z)&=e^{\BCH(x, y)} \,z\,  e^{-\BCH(x, y)}\\
&=(e^xe^y)\, z\, (e^xe^y)^{-1}=
e^{x}(e^{y}\, z\, e^{-y})e^{-x}\\
&=\exp(\ad_{x})\circ \exp(\ad_{y})(z)~. 
\end{align*}
Finally, with this relation implies 
\begin{align*} 
\lambda.(\mu.(\delta+\alpha))&= \exp(\ad_{\lambda})(\exp(\ad_{\mu})(\delta+\alpha)) 
\\ & = (\exp(\ad_{\lambda}) \circ \exp(\ad_{\mu}))(\delta+\alpha)
\\ & =
\exp(\ad_{\BCH(\lambda, \mu)})(\delta+\alpha)
\\ & =(\lambda\cdot \mu).(\delta+\alpha)~.
\end{align*}
\end{proof}

\begin{definition}[Moduli space of Maurer--Cartan elements]\index{Maurer--Cartan element!moduli space}\label{def:ModuliSpacedgLie}
The \emph{moduli space of Maurer--Cartan elements} is the set of equivalence classes of Maurer--Cartan elements under the gauge group action:
\[ \mathscr{MC}(\g)\coloneqq \mathrm{MC}(\g)/\Gamma~. \]
\end{definition}

This moduli space loses the data provided by the gauge symmetries themselves. One may instead consider the following main protagonist of deformation theory.

\begin{definition}[Deligne groupoid]\index{Deligne!groupoid} \label{def:DeligneGroupoid}
Let $\g=(A, {d}, [\, , ])$ be a nilpotent dg Lie algebra. The \emph{Deligne groupoid} associated to $\g$ has the Maurer--Cartan set $\mathrm{MC}(\g)$ as its set of objects, and the gauge symmetries $\lambda$ such that $\lambda.\alpha=\beta$ as the set of (iso)morphisms from $\alpha$ to $\beta$. 
\end{definition}

\section{Deformation theory with differential graded Lie algebras}\index{deformation theory}

The purpose of this section is to explain how one can study deformation theory with dg Lie algebras using Maurer--Cartan elements and their gauge symmetries. This will provide us with a transition from this chapter to the next chapter as infinitesimal deformations are controlled by nilpotent objects while formal deformations are controlled by complete objects. \\

As mentioned in the previous section, one can twist a dg Lie algebra $\g= (A, d, [\, ,])$ with any Maurer--Cartan element
 \[
\varphi\in \MC(\g)\coloneqq 
\left\{
\varphi\in A_{-1}\ |\ d\varphi +\tfrac12 [\varphi, \varphi]=0 
\right\} 
 \]
to produce a twisted dg Lie algebra
 \[
\g^\varphi\coloneqq \left(A, d^\varphi\coloneqq d+[\varphi, \, ], [\, ,]\right)~. 
 \]
The relationship between dg Lie algebras and deformation theory relies ultimately on the following key property, which says that deformations of a Maurer--Cartan element coincide with Maurer--Cartan elements of the twisted dg Lie algebra. 

\begin{lemma}\label{lem:DefTwist}
Let $\g$ be a dg Lie algebra and let  $\varphi\in \MC(\g)$ be a Maurer--Cartan element. The following equivalence holds 
\[
\alpha \in \MC\big(\g^\varphi\big) \iff \varphi+\alpha \in \MC(\g)~.
\]
\end{lemma}

\begin{proof}
While we shall see a conceptual explanation in Lemma~\ref{lem:subGaugeGroup},  see also Corollary~\ref{cor:twa+b}, it is  easy to prove this result directly by showing that both conditions are equivalent to 
\[
d\alpha+[\varphi, \alpha]+\tfrac12 [\alpha, \alpha]=0~.
\]
\end{proof}

The fundamental theorem of deformation theory recently proved by J.~P.~Pridham \cite{Pridham10} and J.~Lurie \cite{Lurie10} claims that any deformation problem over a field $\k$ of characteristic zero can be  encoded by a dg Lie algebra. We shall give its precise statement in Section~\ref{sec:Lurie}, but let us now explain what it means heuristically.  Given an underlying ``space'' $V$, a type of structure $\calP$ that it can support, and an equivalence relation on $\calP$-structures on $V$, there should exist a dg Lie algebra $\g=(A, d , [\, , ])$ such that its  set of Maurer--Cartan elements is in one-to-one correspondence with the set of $\calP$-structures on $V$ and such that the action of the gauge group $\Gamma$ on $\MC(\g)$ coincides with the equivalence relation considered on $\calP$-structures. 

\begin{example}\label{ex:HochschildCochainCx}
When $V$ is a finite dimensional vector space and when $\calP$ stands for associative algebra structures up to isomorphisms, the deformation dg Lie algebra is given by the Hochschild cochain complex 
\[
C^\bullet(V,V)\coloneqq\left(
\prod_{n\ge 1} \Hom\big(V^{\otimes n}, V\big),  0, [\, , ]
\right)~,\]
where the Lie bracket is the one defined by Gerstenhaber~\cite{Gerstenhaber64} (see Section~\ref{sec:CurvedAiAlg}) 
 and  where the homological degree of the factor $\Hom\big(V^{\otimes n}, V\big)$ is equal to $1-n$~. In this case, a Maurer--Cartan element is precisely a binary associative product, and gauge equivalence is an isomorphism.
\end{example}

There are two possible viewpoints one can adopt here: one can either ensure convergence of deformations by working \emph{locally} with complete ``spaces'' $V$ and complete dg Lie algebras~$\g$ (as will be done in  Chapters~\ref{sec:TopoDefTh} and \ref{sec:GaugeTwist}) or work \emph{globally}, inspired by the functors of points in algebraic geometry, as follows. Let $\R$ be a local (commutative) ring\index{local ring} with the maximal ideal $\m$ and with the residue field $\k$, that is $\R\cong \k \oplus \m$~. 
Given any dg Lie algebra $\g=(A, d , [\, , ])$, one can consider its $\R$-extension defined by 
\[\g\otimes \R \coloneqq \left(
A\otimes \R, d, [\, , ]
\right)~,\]
where $d(\zeta\otimes x)\coloneqq d \zeta\otimes x$ and 
where $[\zeta \otimes x, \xi \otimes y]\coloneqq [\zeta, \xi]\otimes xy$. The other way round, one recovers the original dg Lie algebra $\g$ from its $\R$-extension $\g \otimes \R\cong \g \oplus \g\otimes \m$ by reducing modulo $\m$~. 

\begin{definition}[$\R$-deformation]\index{$\R$-deformation}
An \emph{$\R$-deformation} of a Maurer--Cartan element $\varphi \in \MC(\g)$ is a Maurer--Cartan element $\Phi \in \MC(\g \otimes \R)$ of the $\R$-extension of $\g$ such that its reduction modulo $\m$ is equal to $\varphi$~. The set of such deformations is denoted by $\Def_\varphi(\R)$. 
\end{definition}

\begin{proposition}
The set of $\R$-deformations of a Maurer--Cartan $\varphi\in \MC(\g)$ is in natural bijection with the set of Maurer--Cartan elements of the twisted dg Lie algebra $\g^\varphi \otimes m$~:
\[ \mathrm{Def}_\varphi(\R)\cong \MC(\g^\varphi\otimes \m)~.\]
\end{proposition}

\begin{proof}
Notice first that, for any Maurer--Cartan element $\varphi \in \MC(\g)$, the $\R$-extension of $\g$ twisted by $\varphi$ is isomorphic to 
\[(\g\otimes \R)^\varphi \cong \g^\varphi \otimes \R\cong \g^\varphi \oplus \g^\varphi\otimes \m~.\]
The result is then a direct application of Lemma~\ref{lem:DefTwist}: any element $$\Phi=\varphi+\bar{\Phi}\in \g \oplus \g \otimes \m$$ is an $\R$-deformation of $\varphi$ if and only if $\bar{\Phi}$ is Maurer--Cartan element of the twisted dg Lie algebra $\g^\varphi\otimes \m$~. 
\end{proof}

\begin{example}
An $\R$-deformation of an associative algebra structure $\varphi$ on $V$ is an $\R$-linear associative algebra structure $\Phi$ on $V\otimes \R$ whose reduction modulo $\m$ is equal to~$\varphi$. 
This comes from the formula 
\[\Hom_\k(V^{\otimes n}, V)\otimes \R\cong \Hom_\R\big((V\otimes \R)^{\otimes n}, V\otimes \R\big)~.\]
\end{example}

For the case when the ring $\R$ is Artinian\index{local Artinian ring}, that is when there exists $K\in \NN$ such that $\m^K=0$, all series we ever consider will converge automatically since the dg Lie algebra $\g \otimes \m$ is nilpotent. 
The case of complete algebras considered in Chapter~\ref{sec:TopoDefTh} corresponds to the case when $\R$ is \emph{complete} with respect to its $\m$-adic topology\index{local ring!complete}; it is then a limit of local Artinian rings. In both cases, the BCH formula produces a convergent series, and the gauge group $\bar{\Gamma}\coloneqq (A_0 \otimes \m, \BCH, 0)$ is well defined. Its action 
\[\lambda. \Phi\coloneqq\frac{\id-\exp(\ad_\lambda)}{\ad_\lambda}(d\lambda)+\exp(\ad_\lambda)(\Phi)\]
on $\R$-deformations $\Phi=\varphi+\bar{\Phi}$ is also well defined. 

\begin{definition}[Moduli space of $\R$-deformations]\index{$\R$-deformation!moduli space}
The \emph{moduli space of $\R$-deformations} is the set of classes
\[\mathscr{D}ef_\varphi(\R)\coloneqq {\Def}_\varphi(\R)/\bar{\Gamma}\]
of $\R$-deformations modulo the action of the gauge group $\bar{\Gamma}$~.
\end{definition}

The first seminal example is given by the \emph{algebra of dual numbers}\index{algebra of dual numbers}
\[\R=\k[t]/\big(t^2\big)\]
which is an Artinian local ring. 

\begin{definition}[Infinitesimal deformation]\index{deformation!infinitesimal}
An \emph{infinitesimal deformation} of a Maurer--Cartan element $\varphi \in \MC(\g)$ 
is a Maurer--Cartan element of the form 
\[\Phi=\varphi+\bar{\Phi}t \in \MC\big(\g \otimes \k[t]/\big(t^2\big)\big)~.\]
\end{definition}

Infinitesimal deformations are related to the homology group of degree $-1$ of the twisted Lie algebra as follows. 

\begin{theorem}\label{thm:InfDefGeneral}
There are canonical bijections 
\[
\Def_\varphi\big(\k[t]/\big(t^2\big)\big)\cong Z_{-1}\big(\g^\varphi\big) \quad \text{and} \quad 
\mathscr{D}ef_\varphi\big(\k[t]/\big(t^2\big)\big)\cong H_{-1}\big(\g^\varphi\big)~.
\]
\end{theorem}

\begin{proof}
Any degree $-1$ element $\Phi=\varphi+\bar{\Phi}t \in \g \otimes \k[t]/\big(t^2\big)$ is a Maurer--Cartan element if and only if is satisfies the equation 
\begin{align*}
\underbrace{d\varphi  +\tfrac12 [\varphi, \varphi]}_{=0} + \Big(\underbrace{d\big(\bar{\Phi}\big)+ \big[\varphi,  \bar{\Phi}\big]}_{=d^\varphi(\bar{\Phi})}   \Big)t=0~. 
\end{align*}
So infinitesimal deformations coincide with  cycles of degree $-1$ in the twisted dg Lie algebra. 

Two infinitesimal deformations $\Phi_1=\varphi+\bar{\Phi}_1t$ and $\Phi_2=\varphi+\bar{\Phi}_2t$ are equivalent if there exists an element $\lambda\in A_0$ such that 
\begin{align*}
\lambda t . \left(\varphi+\bar{\Phi}_1t\right)=
\varphi+\left(d\lambda+[\varphi, \lambda]+\bar{\Phi}_1\right)t= \varphi+\bar{\Phi}_2t~. 
\end{align*}
This latter equation is equivalent to 
$\bar{\Phi}_2 -\bar{\Phi}_1=d^\varphi(\lambda)$~. This proves that two infinitesimal deformations are equivalent if and only if they are homologous in the twisted dg Lie algebra.
\end{proof}

The other seminal example is given by the \emph{algebra of formal power series}\index{algebra of formal power series}
\[\R=\k[[t]]\]
which is a complete local ring. 

\begin{definition}[Formal deformation]\index{deformation!formal}
A \emph{formal deformation} of a Maurer--Cartan element $\varphi \in \MC(\g)$ 
is a Maurer--Cartan element of the form 
\[\Phi=\varphi+{\Phi}_1t+{\Phi}_2t^2+\cdots \in \MC\big(\g \otimes \k[[t]]\big)~.\]
\end{definition}

The obstructions to formal deformations are related to the homology group of degree $-2$ of the twisted dg Lie algebra as follows. 

\begin{theorem}\label{thm:Obstruction}
If $H_{-2}\big(\g^\varphi\big)=0$, then any cycle of degree $-1$ of $g^\varphi$ extends to a formal deformation of $\varphi$~. 
\end{theorem}

\begin{proof}
In the present case, the Maurer--Cartan equation $d\Phi+\tfrac12 [\Phi, \Phi]=0$ splits with respect to the power of $t$ as 
\begin{equation}\tag{$\ast$}\label{eq:MCformal}
d\Phi_n+[\varphi, \Phi_n]+\tfrac12 \sum_{k=1}^{n-1} [\Phi_k, \Phi_{n-k}]=0
\end{equation}
for any $n\geqslant 1$~. For $n=1$, the equation~\eqref{eq:MCformal} gives 
\[d\Phi_1+[\varphi, \Phi_1]=d^\varphi(\Phi_1)=0~,\] 
so the first term of a formal deformation coincides with a cycle of degree $-1$ of the twisted dg Lie algebra. 

Let us now consider such a degree $-1$ cycle $\Phi_1$ and let us assume that we have $H_{-2}\big(\g^\varphi\big)=0$~. We show, by induction on $n\geqslant 1$, that there exist elements $\Phi_1, \ldots, \Phi_{n}\in A_{-1}$ satisfying the equations~\eqref{eq:MCformal} up to $n$~. The case $n=1$ obviously holds true. Assume that this statement holds true up to $n-1$. The first two terms of Equation~\eqref{eq:MCformal} at $n$ are equal to $d^\varphi(\Phi_n)$; let us show that the third term is degree $-2$ cycle with respect to the twisted differential:
\begin{align*}
d^\varphi\left(\sum_{k=1}^{n-1} [\Phi_k, \Phi_{n-k}]\right)
& =
\sum_{k=1}^{n-1} \left(\left[d^\varphi(\Phi_k), \Phi_{n-k}\right] - \left[\Phi_k, d^\varphi(\Phi_{n-k})\right]\right)\\
& =2\sum_{k=1}^{n-1} \left[d^\varphi(\Phi_k), \Phi_{n-k}\right]
\\
& =-\sum_{k=1}^{n-1}\sum_{l=1}^{k-1} \big[[\Phi_l, \Phi_{k-l}], \Phi_{n-k}\big]\\
& =-\sum_{\substack{a+b+c=n\\ a,b,c\geqslant 1}}\big[[\Phi_a, \Phi_b], \Phi_c\big]
=0~,
\end{align*}
by the Jacobi identity. 
Since $H_{-2}\big(\g^\varphi\big)=0$, there exist $\Phi_n\in A_{-1}$ satisfying Equation~\eqref{eq:MCformal} at $n$, which concludes the proof. 
\end{proof}

The  homology group of degree $-1$ of the twisted dg Lie algebra detects the Maurer--Cartan elements that are rigid, that is the ones that cannot be deformed nontrivially. 

\begin{theorem}
If $H_{-1}\big(\g^\varphi\big)=0$, then any formal deformation of $\varphi$ is equivalent to the trivial one. 
\end{theorem}

\begin{proof}
We once again use the differential trick and work in the extension 
\[\big(\g \otimes \k[[t]]\big)^+= 
\g \otimes \k[[t]] \oplus  \k \delta~.\]
Given a  formal deformation $\Phi=\varphi+\sum_{n\geqslant 1}{\Phi}_ n t^n$ of $\varphi$,  let us try to find, by induction, an element 
\[\lambda=\lambda_1 t +\lambda_2  t^2+\cdots \in A_0\otimes \k[[t]]\]
satisfying 
\begin{equation}\tag{$\ast\ast$}\label{eq:GaugeFormal}
\exp(\ad_{\lambda})(\delta+\varphi)=\delta+\Phi
\end{equation}
in $\big(\g \otimes \k[[t]]\big)^+$~. For $n=1$, the relation satisfied by the coefficients of $t$ in Equation~\eqref{eq:GaugeFormal} is 
\[
\ad_{\lambda_1}(\delta+\varphi)=-d^\varphi(\lambda_1) =\Phi_1~.
\]
Recall that Equation~\eqref{eq:MCformal} in the proof of \cref{thm:Obstruction} shows that $\Phi_1$ is cycle of degree $-1$ for the twisted differential; 
since $H_{-1}\big(\g^\varphi\big)=0$, it is also a boundary, so we may find 
such an element $\lambda_1\in A_0$~. Suppose now that there exist elements $\lambda_1, \ldots, \lambda_{n-1}\in A_0$ satisfying Equation~\eqref{eq:GaugeFormal} modulo $t^n$ and let us look for an element $\lambda_n\in A_0$ satisfying it modulo $t^{n+1}$. Under the notations
\[\widetilde{\lambda}\coloneqq \lambda_1 t + \cdots+ \lambda_{n-1}t^{n-1}
\quad \text{and} \quad 
\exp(\ad_{\widetilde{\lambda}})(\delta+\varphi)=\sum_{n=0}^{\infty} \left(\exp(\ad_{\widetilde{\lambda}})(\delta+\varphi)\right)_n t^n~, \]
the relation satisfied by the coefficients of $t^n$ in Equation~\eqref{eq:GaugeFormal} is 
\[
\ad_{\lambda_n}(\delta+\varphi)+
\left(\exp(\ad_{\widetilde{\lambda}})(\delta+\varphi)\right)_n=
-d^\varphi(\lambda_n)+
\left(\exp(\ad_{\widetilde{\lambda}})(\delta+\varphi)\right)_n=\Phi_n~.
\]
The assumption $H_{-1}\big(\g^\varphi\big)=0$ ensures that such an element $\lambda_n$ exists provided that 
$\Phi_n$ and $\left(\exp(\ad_{\widetilde{\lambda}})(\delta+\varphi)\right)_n$ have the same image under the twisted differential $d^\varphi$. We consider the element $\widetilde{\Phi}\coloneqq {\Phi}_1t+\cdots+{\Phi}_{n-1} t^{n-1}$~.
Equation~\eqref{eq:MCformal} guarantees that 
\[
d^\varphi(\Phi_n)=-\tfrac12 \left(\left[\widetilde{\Phi},\widetilde{\Phi} \right]\right)_n~.
\]
On the other hand, Equation~\eqref{eq:GaugeFormal} implies 
\begin{align*}
 \left[\widetilde{\Phi},\widetilde{\Phi} \right]
&\equiv 
\left[\exp(\ad_{\widetilde{\lambda}})(\delta+\varphi) - (\delta+\varphi), \exp(\ad_{\widetilde{\lambda}})(\delta+\varphi) - (\delta+\varphi)\right]
 \pmod{t^{n+1}}\\
 &\equiv -2
\left[(\delta+\varphi), \exp(\ad_{\widetilde{\lambda}})(\delta+\varphi) \right]
 \pmod{t^{n+1}}\\
  &\equiv -2\,
d^\varphi\left(\exp(\ad_{\widetilde{\lambda}})(\delta+\varphi) \right)
 \pmod{t^{n+1}}~,
\end{align*}
since $\exp(\ad_{\widetilde{\lambda}})$ is a morphism of Lie algebras. 
In the end, we get 
\begin{align*}
d^\varphi(\Phi_n)=-\tfrac12 \left(\left[\widetilde{\Phi},\widetilde{\Phi} \right]\right)_n=
\left(d^\varphi\left(\exp(\ad_{\widetilde{\lambda}})(\delta+\varphi) \right)\right)_n=
d^\varphi\left(\left(\exp(\ad_{\widetilde{\lambda}})(\delta+\varphi) \right)_n\right)~, 
\end{align*}
which concludes the proof
\end{proof}

The present proof shows how the differential trick works heuristically: it states that ``if a property holds true in a graded Lie algebra with trivial differential, then it holds true in any dg Lie algebras".

\chapter{Operad theory for filtered and complete modules}\footnotetext{\hrule\smallskip\noindent This material will be published by Cambridge University Press \& Assessment as ‘Maurer-Cartan Methods in Deformation Theory: the twisting procedure’ by Vladimir Dotsenko, Sergey Shadrin and Bruno Vallette. This version is free to view and download for personal use only. Not for re-distribution, re-sale or use in derivative works. \copyright Cambridge University Press \& Assessment}
\label{sec:OptheoyFilMod}

In algebra, one  has to consider infinite series on many occasions. In order to make sense, these formulas require an extra topological assumption on the underlying  module. 
In this chapter, we first recall the notion of filtered and then complete modules which provides us with such a complete topology. 
This type of topology considered in many research areas such as algebraic geometry, deformation theory, rational homotopy theory and microlocal analysis,  finds its source in the generalisation of the Lie theory to filtered Lie algebras and groups due to M. Lazard in his Ph.D. thesis \cite{Lazard50}. It became later omnipresent in commutative algebra, see N. Bourbaki \cite{Bourbaki61}.\\

In this chapter, we establish the various properties for the monoidal structures on the categories of filtered modules and complete modules and for their associated monoidal functors. The main goal is to develop the  theory of  operads and operadic algebras in this context. 
First, this allows us to compare the various categories of filtered and complete algebras and recover conceptually the various  known definitions of filtered complete algebras that one can find in the literature. Then, as we will see in the next chapters, this allows us to get for free operadic results on the complete setting since the previously performed operadic calculus still hold in this generalised context, as it is in any monoidal category satisfying the monoidal properties mentioned above. \\

This present exposition shares common points with that of 
P. Deligne \cite[Section~1]{Deligne71}, and that of M. Markl \cite[Chapter~$1$]{Markl12}; it  is close to the treatment of B. Fresse \cite[Section~$7.3$]{Fresse17I}, though one does not  find there all the results  needed in the present treatment. 
For simplicity, we work over a ground field $\k$ though many results hold in a more general context. 

\section{Filtered algebras}

\begin{definition}[Filtered module]\index{filtered!module}
A \emph{filtered  module} $(A,\F)$ is a $\k$-module $A$ equipped with a filtration of $\k$-submodules
$$A=\F_0 A \supset \F_1 A \supset \F_2 A \supset \cdots \supset \F_k A \supset F_{k+1}A \supset \cdots$$
\end{definition}

This condition implies that the subsets $\{x+\F_k A\, | \,  x\in A, k\in \NN\}$
form  a neighbourhood basis of a first-countable topology on $A$, which is thus a Fr\'echet--Urysohn space and so a sequential space. Since this topology is induced by submodules, any filtered module is trivially a topological module, that is  the scalar multiplication and the sum of elements are continuous maps, when one considers the discrete topology on the ground field $\k$.

\begin{example}
Let $I$ be an ideal of a $\k$-algebra $A$~. The submodules $\F_k A\coloneqq I^k A$, for $k\ge 0$, form a filtration of $A$ and the associated topology is called the \emph{$I$-adic topology} of $A$. 
\end{example}

\begin{lemma}\label{lem:Fkclosed}
The subsets $\F_k A$, for $k\in \NN$, are closed with respect to this topology. 
\end{lemma}

\begin{proof}
Let $\left\{x_n\in \F_k A\right\}_{n\in \NN}$ be a sequence converging to $x\in A$. There exists $N\in \NN$ such that, for all $n\ge N$, we have $x_n-x\in \F_k A$. 
Since the element $x_N$ lives in $\F_k A$, which is a submodule of $A$, this implies that $x$ lives in $\F_k A$ too. 
\end{proof}

Let $(A, \F)$ and $(B, \G)$ be two filtered modules. We consider the following induced filtration on the mapping space $\Hom(A,B)$: 
$$
 \calF_k \Hom(A, B)\coloneqq\left\{
f :A \to B\ | \ f(\F_n A)\subset \G_{n+k} B\ , \ \forall n\in \NN
\right\}\ .
$$
This filtration endows $\hom(A,B)\coloneqq\calF_0 \Hom(A, B)$ with a filtered module structure. We invite the reader to check that while any map in $\calF_k \Hom(A, B)$ is continuous with respect to the associated topologies, not every continuous map is of this form. 

\begin{definition}[Filtered map]\index{filtered!map}
A \emph{filtered map} $f : (A, \F) \to (B, \G)$ between two filtered modules is an element $f\in \hom(A,B)$, that is a linear  map  preserving the respective filtrations: $f(\F_n A)\subset \G_{n} B$, for all $n\in\NN$. 
\end{definition}

The induced filtration on the tensor product of two filtered modules is given by the formula
\[\calF_k (A\otimes B)\coloneqq\sum_{n+m=k} \mathrm{Im} \big(\F_n A \otimes \mathrm{G}_m B\to A\otimes B\big) \ .\]

\begin{lemma}\label{lem:FilModSymCat}
The category of filtered modules with filtered maps, equipped with the internal hom and the filtration of tensor products, forms a bicomplete closed symmetric  monoidal  category, whose monoidal product preserves colimits. 
\end{lemma}

\begin{proof}
Given a collection $\left(A^i, \F^i\right)_{i\in \mathcal{I}}$ of filtered modules, their coproduct is given by $A\coloneqq\bigoplus_{i\in \mathcal{I}} A^i$  with filtration 
\[ \F_k A\coloneqq\left\{a_{i_1}+\cdots+a_{i_n}\ | \ a_{i_j}\in \F^{i_j}_k A^{i_j}, \ \forall j\in \{1, \ldots, n\}\right\} 
\]
and their product is given by $B\coloneqq\prod_{i\in \mathcal{I}} A^i$ with filtration 
\[
\G_k B\coloneqq\left\{(a_{i})_{i\in \mathcal{I}}\ | \ a_{i}\in \F^i_k A^{i}, \ \forall i\in \mathcal{I}\right\} \ .
\]
Cokernels of filtered maps $f : (A,\F) \to (B, \G)$ are given by $p : B \twoheadrightarrow B/\Im f$ equipped with the  filtration $p(\G_k)\cong\G_k B/(\Im f \cap \G_k B)$.
Since this category is (pre)additive, all coequalizers of pairs $(f,g)$ are given by cokernels of differences $f-g$. So this category admits all colimits. 
In the same way, kernels of filtered maps $f : (A,\F) \to (B, \G)$ are given by $f^{-1}(0)$ with filtration $f^{-1}(0) \cap \F_k B$.
Since this category is (pre)additive, all equalizers of pairs $(f,g)$ are given by kernels of differences $f-g$. So this category admits all limits. 
Notice that this category, though additive, fails to be abelian: in general, maps do not have categorical images. 

Finally, it is straightforward to check that the various structure maps of the (strong) symmetric closed monoidal  category of modules are filtered. From the above description of coproducts and cokernels, it is easily seen that the monoidal product preserves them, so it preserves all colimits.
\end{proof}

The properties of Lemma~\ref{lem:FilModSymCat} ensures that one can develop the operadic calculus in the symmetric mo\-no\-id\-al category of filtered modules, see \cite[Chapter~$5$]{LodayVallette12}. An operad in this context will be referred to as a  \emph{filtered operad}. 

\begin{example}
For a filtered module $(A, \F)$, the associated \emph{filtered endormorphism operad}\index{filtered!endormorphism operad} is the part of the endomorphism operad of $A$ consisting of filtered maps, that is 
$$\mathrm{end}_A\coloneqq\left\{\hom(A^{\otimes n}, A)\right\}_{n\in \NN}\ .$$
An element of the filtered endomorphism operad is thus a linear map $$f : A^{\otimes n} \to A$$ satisfying 
$$      
f\left(\F_{k_1} A \otimes \cdots \otimes \F_{k_n} A\right) \subset \F_{k_1+\cdots+k_n} A
\ . $$
The full induced filtration on its underlying collection is given by
$$f\in \calF_k \eend_A(n)\ \   \text{when} \ \      
f\left(\F_{k_1} A \otimes \cdots \otimes \F_{k_n} A\right) \subset \F_{k_1+\cdots+k_n+k} A
\ .$$
\end{example}

\begin{definition}[Filtered $\calP$-algebra]\label{def:FilAlgebras}\index{filtered!$\calP$-algebra}
Let $\calP$ be a filtered operad and let $(A, \F)$ be a filtered module. A \emph{filtered $\calP$-algebra structure} on $(A,\F)$ amounts to the data of a filtered morphism of operads 
$$\calP \to \eend_A\ .$$
\end{definition}

We denote the category of modules by $\mathsf{Mod}$ and the category of filtered modules by~$\mathsf{FilMod}$. 

\begin{proposition}\label{prop:1Adj}
The  functor $\sqcup \, : \,  \mathsf{FilMod}\to \mathsf{Mod}$, which forgets the filtration, admits a left adjoint full and faithful functor 
$$\vcenter{\hbox{
\begin{tikzcd}[column sep=1.2cm]
\mathrm{Dis}  \colon 
\mathsf{Mod} 
\arrow[r, harpoon, shift left=1ex, "\perp"']
&
\arrow[l, harpoon,  shift left=1ex]
\mathsf{FilMod} 
\ : \ \sqcup 
\end{tikzcd}
}}$$
given by the trivial filtration: 
$$\mathrm{Dis}(A)\coloneqq(A, A=\F^{\mathrm{tr}}_0 A\supset 0=\F_1 ^{\mathrm{tr}}A=\F_2^{\mathrm{tr}} A= \cdots) \ ,$$
which induces the discrete topology.
These two functors are strictly symmetric monoidal. 
\end{proposition}

\begin{proof}
It is straightforward to check the various properties.
\end{proof}

The monoidal part of this proposition ensures that the underlying collection, without the filtration, of any filtered operad is an operad. For instance, the filtered endomorphism operad $\eend_A$ is a strict suboperad of the endomorphism operad $\End_A$. In the other way round, any  operad can be seen as a filtered operad equipped with the trivial filtration, that is with discrete topology. 

\begin{example}
Definition~\ref{def:FilAlgebras}, applied to the ns operad $\As$ of associative algebras, produces the classical notion of filtered associative algebra, see \cite[Section~I.3]{Lazard50}, \cite[Chapter~$3$]{Bourbaki61}, or \cite[Appendix~A.1]{Quillen69}. In the case of the operad $\Lie$, we recover the notion of filtered Lie algebras of M. Lazard~\cite{Lazard50}. 
All the operadic constructions therefore hold in this setting. For instance, the morphism of symmetric operads $\Lie \to \Ass$, viewed as a morphism of filtered operads produces automatically the universal enveloping Lie algebra in the filtered world by \cite[Section~$5.2.12$]{LodayVallette12}.
\end{example}

\begin{definition}[Filtered differential graded module]\index{filtered!differential graded module}
A \emph{filtered differential graded module} is a differential graded module in the category of filtered modules. Such a data amounts to a collection $\{A_n\}_{n\in \mathbb{Z}}$ of filtered modules equipped with a square-zero degree $-1$ filtered map $d$.
\end{definition}

All the results above hold \textit{mutatis mutandis} for filtered dg modules. For instance, this operadic definition allows us to recover naturally the following notions of filtered homotopy algebras structures present in the literature. 

\begin{example}
A \emph{filtered curved $\Ai$-algebra}
\index{$\Ai$-algebra!filtered curved} structure on a filtered graded module $(A, \F)$ amounts to the data of curved $\Ai$-algebra structure $$(m_0, m_1, m_2, \allowbreak \cdots)$$ on $A$ according to Definition~\ref{def:CurvedAinfty} below such that the various structure maps satisfy 
$$m_n(\F_{k_1} A, \ldots, \F_{k_n} A)\subset \F_{k_1+\cdots+k_n} A\ . $$
This definition corresponds to  the one given by Fukaya--Oh--Ohta--Ono   in \cite{FOOO09I}, with the only difference that  these authors consider modules with the \emph{energy filtration}, indexed by non-negative real numbers $\RR^+$, see Section~\ref{sec:CurvedAiAlg} for further operadic details.

The present operadic definition of a filtered (shifted) curved $\Li$-algebra given in Definition~\ref{def:csLiAlg} recovers the usual one, which is used for instance by Dolgushev--Rogers in \cite{DolgushevRogers15, DolgushevRogers17} (with the further constraint $\F_0 A=\F_1 A$).
\end{example}

\section{Complete algebras}

Any decreasing filtration $A=\F_0 A \supset \F_1 A \supset \cdots$ 
induces a sequence of surjective maps, 
$$ \xymatrix{
0=A/\F_0 A & \ar@{->>}[l]_(0.4){p_0}
A/\F_1 A& \ar@{->>}[l]_(0.45){p_1}
A/\F_2 A & \ar@{->>}[l]_(0.45){p_2} 
A/\F_3 A&  \ar@{->>}[l]\cdots} \ ,$$
where $p_k$ is the reduction modulo $\F_k A$. Its limit, denoted by 
$$\wA\coloneqq\lim_{k\in \NN} A/\F_k A \ ,$$
is defined by the formula
$$ \wA=\big\{(x_0,  x_1, x_2, \ldots)\ | \ x_k\in A/\F_k, \ p_k(x_{k+1})=x_k\big\}\ .$$
If we consider the structure maps 
 \[
q_k : \wA \twoheadrightarrow A/\F_k A, (x_0,  x_1, x_2, \ldots)\mapsto x_k,
 \]
 then the limit module $\wA$ is endowed with the canonical filtration 
 \[\widehat{\F}_k \wA \coloneqq\ker q_k=\{
 (0, \ldots, 0, x_{k+1}, x_{k+2}, \ldots)\}.\] 
 With the associated topology, it forms a complete Hausdorff space. 
 
Let us denote by $\pi_k : A \twoheadrightarrow A/\F_k A$ the canonical projections. The canonical map $\pi : A \to \wA$, $x\mapsto (\pi_0(x), \pi_1(x), \pi_2(x),\ldots)$, associated to them, is filtered and thus continuous. 

\begin{definition}[Complete module]\index{complete!module}
A \emph{complete module} is a filtered module $(A, \F)$
 such that the canonial morphism
$$\pi\colon A \stackrel{\cong}{\longrightarrow} \wA=\lim_{k\in \NN} A/\F_k A $$
is an isomorphism. 
\end{definition}

The canonical map $\pi$ is an epimorphism if and only if the  associated topological space is complete, which explains the terminology chosen here.  When the canonical map is an isomorphism, it is an homeomorphism since $\pi^{-1}$ is filtered and thus continuous. We note that the kernel $\pi$ is equal to the intersection of all the sub-modules $\F_k A$. Therefore, it is a monomorphism if and only if 
$$\bigcap_{k\in\NN} \F_k A=\{0\}\ ;$$ 
this condition is equivalent  for the associated topology on $A$ to be Hausdorff. In that case, the topology on $A$ is in fact metrisable: one may define the \emph{valuation} of an element $x\in A$ by putting 
$$
\nu(x)\coloneqq 
\begin{cases}
+\infty\quad \text{if} \quad x=0\\ 
\quad k \quad \text{if} \quad x\in  \F_k A\backslash \F_{k+1} A\ ,
\end{cases}
$$
and the distance  
$$d(x,y)=\frac{1}{\nu(y-x)+1}\ .$$
As a consequence, filtered maps are uniformly continuous and the canonical map $\pi : A \hookrightarrow \wA$ makes $\wA$ into  the completion of $A$: the module $\wA$ is complete, contains $A$ as a dense subset, and is unique up to isometry for this property.

\begin{example}
The toy model here is the algebra of polynomials $\k[X]$ with its $X$-adic filtration $\F_k\, \k[X]\coloneqq X^k \, \k[X]$. Its topology is Hausdorff but not complete. Its completion is the algebra of formal power series $\widehat{\k[X]}\cong\k[[X]]$.
\end{example}

In any complete module, convergent series have the following simple form. 

\begin{lemma}\label{lem:Conv}
Let $(A, \F)$ be a complete module. The series $\sum_{n\in \NN} x_n$ associated to a sequence of elements $\{x_n\}_{n\in \NN}$ is convergent if and only if the sequence $x_n$ converges to $0$. 
\end{lemma}

\begin{proof}
We classically consider the sequence $X_n\coloneqq\sum_{k=0}^n x_k$, for $n\in \NN$. If the sequence $\{X_n\}$ converges, then it is a Cauchy sequence and so $x_n=X_n-X_{n-1}$ tends to $0$. In the other way round, if the sequence $\{x_n\}_{n\in \NN}$ tends to $0$, this means 
$$\forall k\in \NN,\,  \exists N\in \NN, \, \forall n\ge N, \ x_n\in \F_k A\ . $$
Since $\F_k A$ is a submodule of $A$, we have 
$$\forall k\in \NN,\, \exists N\in \NN,\, \forall m\ge n\ge N, \ X_m-X_n=x_m+ \cdots+x_{n+1}\in \F_k A\ , $$
that is the sequence $\{X_n\}$ is Cauchy is thus convergent. 
\end{proof}

We can consider the following first definition of a complete algebra over an operad. We will see later on at Theorem~\ref{thm:CompleteAlg} that it actually coincides with the conceptual one. 

\begin{definition}[Complete $\calP$-algebra]\label{Def:CompleteAlg}\index{complete!$\calP$-algebra}
A \emph{complete algebra} over a filtered operad~$\calP$ is a complete module endowed with a filtered $\calP$-algebra structure. 
\end{definition}

\begin{example}
The above example of formal power series $\k[[X]]$ is a complete algebra over the operad $\Ass$ (respectively, $\Com$), that is a complete associative (respectively, commutative associative) algebra.
\end{example}

We consider the full subcategory of filtered modules whose objects are the complete modules, and we denote it by $\mathsf{CompMod}$.

\begin{proposition}\label{prop:2Adj}
The completion of a filtered module defines a functor which is left adjoint to the forgetful functor: 
$$\vcenter{\hbox{
\begin{tikzcd}[column sep=1.2cm]
\widehat{} \ \ \ \ \colon \ 
\mathsf{FilMod} 
\arrow[r, harpoon, shift left=1ex, "\perp"']
&
\arrow[l, harpoon,  shift left=1ex]
\mathsf{CompMod} 
\ : \ \sqcup \ .
\end{tikzcd}
}}$$
\end{proposition}

\begin{proof}
This statement amounts basically to the following property: 
any filtered map $f : A \to B$, 
with target a complete module, 
factors uniquely through the canonical map 
$$\vcenter{\hbox{
\begin{tikzcd}[column sep=1cm, row sep=1cm]
A 
\arrow[r,"f"]
\arrow[d,"\pi"']
&
B \\ 
\wA 
\arrow[ur,"\exists ! \bar{f}"']
& ,
\end{tikzcd}
}}$$
which  is nothing but the universal property of the limit $\widehat{A}$. 
\end{proof}

In order to endow the category $\mathsf{CompMod}$ of complete modules with a mo\-no\-idal structure, one could first think of the underlying tensor product of filtered modules. But this one fails to preserve complete modules, as the following example shows 
$$\xymatrix{
\k[[X]]\otimes \k[[Y]] \ \ \ar@{^{(}->}[r]^(0.41){\neq} &\ \  \widehat{\k[X]\otimes \k[Y]}\cong \k[[X,Y]]} \ .$$

\begin{definition}[Complete tensor product]\index{complete!tensor product}
The \emph{complete tensor product} of two complete modules $(A, \F)$ and $(B, \G)$ is defined by the completion of their filtered tensor product:
$$ A\wo B\coloneqq\widehat{A\otimes B}\ . $$
\end{definition}

We leave it as an exercise to the reader to check that in the case when two filtered modules are not necessarily complete, the completion of their tensor product is equal to 
$\widehat{A\otimes B}\, \cong\, \wA \; \wo\;  \widehat{B}$ \ .

\begin{lemma}
The category $\left(\mathsf{CompMod}, \wo\right)$ of complete modules equipped the complete tensor product is a bicomplete closed symmetric  monoidal  category wh\-ose monoidal product preserves colimits. The same is true for complete differential graded modules.
\end{lemma}

\begin{proof}
Let  $\left(A^i, \F^i\right)_{i\in \mathcal{D}}$ be a diagram of complete modules. 
Since the completion functor is a left adjoint functor, it  would preserve colimits if these latter ones exist. Therefore we define colimits in the category of complete modules by the formula 
\[
\displaystyle \wcolim_{i\in\mathcal{D}} A^i \coloneqq \reallywidehat{\displaystyle \colim_{i\in\mathcal{D}} A^i} \ .
\]
It is straightforward to check that they satisfy the universal property of colimits from the property of the completion functor. 
For instance, coproducts of complete modules are given by 
\[\widehat{\bigoplus_{i\in \mathcal{I}}} A^i\cong\widehat{\bigoplus_{i\in \mathcal{I}} A^i}\ , \]
and finite coproducts of complete modules are simply given by the finite direct sums of their underlying filtered module structure, since this latter one is already complete. 
In the other way round, since the forgetful functor from complete modules to filtered modules  is a right adjoint functor, it  would preserve limits if these latter ones exist. One can actually see that the formulas in the category of filtered modules (Lemma~\ref{lem:FilModSymCat}) for products and kernels once applied to complete modules render complete modules. Therefore the category of complete modules admits limits since it is a (pre)additive category. 
Like the category of filtered modules and for the exact same reasons, the category of complete modules is additive but fails to be abelian. 

The various axioms of a strong monoidal category are straightforward to check. 
The preservation of the colimits by the complete tensor product is automatic from its definition, the above characterisation of colimits and Lemma~\ref{lem:FilModSymCat}: 
\begin{multline*}
\left(\wcolim_{i\in\mathcal{D}} A^i\right) \wo B \cong 
\left(\reallywidehat{\displaystyle \colim_{i\in\mathcal{D}} A^i}\right)\wo B \cong 
\reallywidehat{\left(\displaystyle \colim_{i\in\mathcal{D}} A^i\right)\otimes B} \\ \cong 
\reallywidehat{{\displaystyle \colim_{i\in\mathcal{D}}} \left(A^i\otimes B\right)} \cong 
\wcolim_{i\in\mathcal{D}} \left(A^i\wo B\right)\ .
\end{multline*}
It remains to prove that this symmetric monoidal category is closed. To this end, it is enough to prove that the internal filtered hom of complete modules $A,B$ is complete. One first notices that 
$$\bigcap_{k\in\NN} \calF_k \hom(A,B)=  \hom(A, \cap_{k\in\NN} \G_k B)=\{0\}\ .$$
Now let $\left\{f^n : A \to B\right\}_{n\in \NN}$ be a Cauchy sequence of filtered maps. This means that 
$$\forall k\in \NN, \exists N\in \NN, \forall m,n\ge N, \ f^m-f^n\in  \calF_k \hom(A,B)\ . $$
Therefore, the sequence $\left\{f^n(a)\right\}_{n\in \NN}$ in $B$  is Cauchy for any $a\in A$ and thus converges since $B$ is complete. We denote by $f(a)$ its limit. Considering the discrete topology on the ground field $\k$, the scalar multiplication and the sum of elements are continuous, so this assignment defines a linear map $f : A \to B$. When $a\in \F_l A$, the Cauchy sequence $\left\{f^n(a)\right\}_{n\in \NN}$ lives in $\G_l B$, which is closed by Lemma~\ref{lem:Fkclosed}. Hence, we have $f(a)\in G_l B$ and the whole map $f$ is filtered, that is $f\in \hom(A,B)$. Using again the argument that the $\G_{l+k} B$ are closed, one can see, after a passage to the limit, that 
$$\forall k\in \NN, \exists N\in \NN, \forall n\ge N, \ f^n-f\in  \calF_k \hom(A,B)\ , $$
since this means that $\left(f^n-f\right)(\F_l A)\subset \G_{l+k} B$.

The same proof works \textit{mutatis mutandis} for  complete differential graded modules. 
\end{proof}

So one can develop the operadic calculus in this setting. This produces automatically a notion of  \emph{complete dg  operads}\index{complete!dg  operad} together with their categories of complete dg algebras. The next proposition shows that there is nothing to change from the filtered case for the endomorphism operad. 

\begin{proposition}\label{prop:EndComplete}\index{complete!endomorphism operad}
The complete endomorphism operad of a complete dg module $A$ is canonically isomorphic to the filtered endormorphism operad:
\[\left\{\hom\big(A^{\widehat{\otimes} n}, A\big)\right\}_{n\in \NN}\cong \left\{\hom\big(A^{{\otimes} n}, A\big)\right\}_{n\in \NN}
=\eend_A \ .\]
\end{proposition}

\begin{proof}
The proof relies entirely on the universal property of the completion functor as described in the proof of Proposition~\ref{prop:2Adj}. Under the same notations, one can check that the bijection 
\[
\begin{array}{clc}
\hom\big(\wA, B\big) &\to&  \hom(A, B) \\
\bar{f} & \mapsto & \bar{f} \circ \pi
\end{array}
\]
is  a bijection of filtered modules, whenever the filtered module $B$ is complete. In the present case, this induces 
\[ \hom\big(A^{\widehat{\otimes} n}, A\big)\cong \hom\big(\widehat{A^{{\otimes} n}}, A\big)\cong \hom\big(A^{{\otimes} n}, A\big)\ ,\]
for any $n\in \NN$. 
\end{proof}

\begin{proposition}\label{prop:strict-laxmonoidal}
The completion functor 
\[\ {\widehat{ }} \ \, : \, (\mathsf{FilMod}, \otimes) \to  (\mathsf{CompMod}, \wo)\] 
is strong symmetric mo\-no\-idal and the forgetful functor 
\[\sqcup \, : \,  (\mathsf{CompMod}, \wo) \allowbreak \to \allowbreak  (\mathsf{FilMod}, \otimes)\] 
is lax symmetric monoidal.
\end{proposition}

\begin{proof}
The structure map for the monoidal structure of the completion functor is the isomorphism 
$\wA \; \wo\;  \widehat{B} \xrightarrow{\cong} \widehat{A\otimes B}$. 
The structure map for the monoidal structure of the forgetful  functor is the canonical map 
$A\otimes B \to A \wo B$. 
\end{proof}

One can iterate the above functors $\widehat{\mathrm{Dis}}  : \mathsf{Mod} \to \mathsf{FilMod} \to \mathsf{CompMod}$. Since the discrete topology is already complete, this composition of functors does not change the underlying module; it just  provides it with the trivial filtration.   

\begin{corollary}\label{coro:12Adj}
The following pair of functors are adjoint 
$$\vcenter{\hbox{
\begin{tikzcd}[column sep=1.2cm]
\widehat{\mathrm{Dis}} \colon 
\mathsf{Mod} 
\arrow[r, harpoon, shift left=1ex, "\perp"']
&
\arrow[l, harpoon,  shift left=1ex]
\mathsf{CompMod} 
\ : \ \sqcup \ ,
\end{tikzcd}
}}$$
where $\widehat{\mathrm{Dis}}(A)\coloneqq(A, \F^{\mathrm{tr}})$.
The functor $\widehat{\mathrm{Dis}}$ is a strict monoidal functor and the functor $\sqcup$ is a lax monoidal functor. 
\end{corollary}

\begin{proof}
This adjunction is actually obtained as the composite of the two adjunctions of Proposition~\ref{prop:1Adj} and 
Proposition~\ref{prop:2Adj}. The monoidal structures on the two functors are obtained as composite of two monoidal structures from Proposition~\ref{prop:1Adj} and Proposition~\ref{prop:strict-laxmonoidal}.
\end{proof}

The main point for  these six functors to be monoidal is that each of them sends an operad in the source category to an operad in the target category and similarly for their associated notion of algebras. 

\begin{proposition}
Any dg  operad $\calP$ is a filtered (respectively complete) dg operad $\calP$ when equipped with the trivial filtration. Any $\calP$-algebra $A$ is a filtered (respectively complete) $\calP$-algebra when equipped with the trivial filtration. The category of discrete (respectively filtered) $\calP$-algebra is a full subcategory of the category of complete $\calP$-algebras. 
\end{proposition}

\begin{proof}
This is a direct consequence of the monoidal structure of the functor 
${\mathrm{Dis}}$ from Proposition~\ref{prop:1Adj} (respectively 
$\widehat{\mathrm{Dis}}$ from Corollary~\ref{coro:12Adj}),  since this latter one  does not modify the underlying module.
\end{proof}

As a consequence, we will now work in the larger category of complete $\calP$-algebras and extend the various operadic results to that level. The key property that the completion functor is left adjoint to the forgetful functor allows us to get the following  simple descriptions for the notions of complete operads and complete $\calP$-algebras. 

\begin{theorem}\label{thm:CompleteAlg}\leavevmode
\begin{enumerate}
\item The structure of a complete dg operad $\calP$ is equivalent to the structure of a filtered dg operad on a complete dg $\Sy$-module $\calP$. 

\item
Let $\calP$ be  a complete dg operad. The data of a dg $\calP$-algebra structure in the monoidal category of complete dg modules is equivalent to a complete dg $\calP$-algebra structure, as defined above,  that is a filtered dg $\calP$-algebra structure on an underlying  complete dg module.

\item Let $\calP$ be a filtered dg operad. The data of a complete dg $\calP$-algebra structure, as defined above, is equivalent to the data of a complete dg $\widehat{\calP}$-algebra structure.

\end{enumerate}
\end{theorem}

\begin{proof}\leavevmode
\begin{enumerate}
\item 
Since the forgetful functor $\sqcup : \mathsf{CompMod} \to \mathsf{FilMod}$ is lax symmetric monoidal according to Proposition~\ref{prop:strict-laxmonoidal}, it sends any complete operad structure on $\calP$ to a filtered operad structure on the underlying complete dg $\Sy$-module $\calP$.  In details, recall that a filtered operad structure amounts to a collection of filtered maps $\gamma_n$ from
\[\calP\circ \calP(n)\coloneqq\bigoplus_{k\in \NN}
\calP(k)\otimes_{\Sy_k}
\left(
\bigoplus_{i_1+\cdots+i_k=n}
\mathrm{Ind}_{\Sy_{i_1}\times \cdots \times \Sy_{i_k}}^{\Sy_n}\big(
\calP(i_1)\otimes \cdots \otimes \calP(i_k)
\big)
\right)\]
to $\calP(n)$ that satisfy certain relations, see \cite[Section~$5.2.1$]{LodayVallette12}.
Similarly, a complete dg operad structure amounts to a collection of filtered maps $\widehat{\gamma}_n$ from
\[\calP\, \widehat{\circ}\, \calP(n)\coloneqq\widehat{\bigoplus}_{k\in \NN}
\calP(k)\wo_{\Sy_k}
\left(
\widehat{\bigoplus_{i_1+\cdots+i_k=n}}
\mathrm{Ind}_{\Sy_{i_1}\times \cdots \times \Sy_{i_k}}^{\Sy_n}\big(
\calP(i_1)\wo \cdots \wo \calP(i_k)
\big)
\right)\]
to $\calP(n)$ that satisfy the same type of relations. 
Since  the completion functor is left adjoint,  it preserves colimits and thus coproducts, which implies that 
$\widehat{\calP\circ \calP}$ is isomorphic to $\calP \, \widehat{\circ}\,  \calP$. By pulling back along the canonical completion map 
$$\pi : \calP\circ \calP \to \calP\, \widehat{\circ}\,  \calP,$$  any 
dg operad structure $\widehat{\gamma}$
 in complete dg modules induces a filtered dg operad structure  ${\gamma}=\pi\, \widehat{\gamma}$.
 In the other way round, any filtered dg operad structure  $\gamma$ factors through $\widehat{\calP\circ \calP}\cong \calP \, \widehat{\circ}\,  \calP$, that is through an operad structure  $\widehat{\gamma}$ in complete dg modules. 

\item For the second point, the arguments are similar. The lax symmetric mo\-no\-id\-al  functor $\sqcup : \mathsf{CompMod} \to \mathsf{FilMod}$ sends  
 any  $\calP$-algebra structure in the monoidal category of complete dg modules to a 
$\calP$-algebra structure in the monoidal category of filtered dg modules. This latter structure amounts to a morphism of filtered dg operads 
$\rho\colon \calP
\to \eend_A$,
under point $(1)$. 
Since the filtered dg operad $\calP$ is complete and since the endomorphism operad associated to a complete dg module $A$ is the same in the filtered and the complete case, after Proposition~\ref{prop:EndComplete}, a complete dg $\calP$-algebra structure on $A$ amounts to a morphism of complete dg operads 
$\widehat{\rho}\colon \calP
\to \eend_A$, 
which thus coincides with the above type of maps since the morphisms in the category of complete modules are that of the category of filtered modules. 

\item The arguments are again the same: by the universal property of the completion, any morphism of filtered dg operads $\rho \colon \calP\to \eend_A$ is equivalent to a morphism of complete dg operads $\widehat{\rho}\colon  \widehat{\calP}\to \eend_A$, when $A$ is complete.
 \end{enumerate}
 \end{proof}

This result shows that the terminology ``complete algebra'' chosen in Definition~\ref{Def:CompleteAlg} does not bring any ambiguity since the two possible notions are actually equivalent. Notice that this theorem applies to discrete dg operads and discrete dg $\calP$-algebras, when equipped with the trivial filtration. 

\begin{example}
The free complete $\calP$-algebra\index{free complete $\calP$-algebra} of a complete module $V$ over a complete operad $\calP$ is given by 
$$ \calP\, \widehat{\circ}\,  V=\widehat{\bigoplus_{n\in \NN}}
\calP(n)\wo_{\Sy_n} V ^{\wo n}\ .$$
When the operad $\calP$ and the  module $V$ are discrete ones endowed with the trivial filtration, we recover  the free $\calP$-algebra $\calP\circ V$. 
But, when the filtration arises from the weight grading for which $V$ is concentrated in weight $1$, so that $\F_0 V= \F_1 V=V$ and $\F_n V=0$, for $n\ge 2$, 
the free complete $\calP$-algebra on $V$ is equal to 
$$ \calP\, \widehat{\circ}\,  V\cong \prod_{n\in \NN}
\calP(n)\otimes_{\Sy_n} V ^{\otimes n} \ ,$$
since its  underlying filtration is given by 
\[
\calF_k\left(\bigoplus_{n\in \NN}
\calP(n)\otimes_{\Sy_n} V^{\otimes n}\right)=\bigoplus_{n\ge k}
\calP(n)\otimes_{\Sy_n} V^{\otimes n}\ .
\]
In this way, we recover the notions of free complete associative algebra present in 
\cite[Section~I.4]{Lazard50} or free complete Lie algebras present in  \cite[Section~II.1]{Lazard50} and \cite{LawrenceSullivan14}. This allows us to get automatically, that is operadically, the universal enveloping algebra in the complete case. 
\end{example}

\begin{proposition}
Let $\calP$ be a dg operad. The forgetful functor embeds the category of complete dg $\calP$-algebras as a full subcategory of filtered dg $\calP$-algebras. The completion functors sends a filtered dg $\calP$-algebra to a complete dg $\calP$-algebra. 
These two functors again form a pair of adjoint functors, where the completion functor is left adjoint.  
\end{proposition}

\begin{proof}
This is a direct corollary of Proposition~\ref{prop:strict-laxmonoidal} and Theorem~\ref{thm:CompleteAlg}.
\end{proof}

One immediate upshot of the theory  developed here is that one can now deal with the question of convergence mentioned in the previous chapter. Namely, to make the results of that chapter precise, one simply has to work in a complete dg Lie algebra $\g=(A, \F, {d}, [\, , ])$, and to define the gauge action only for elements $\lambda\in\F_1A_0$. The gauge group in this situation is the group 
  \[
\Gamma=\left(\F_1 A_0, \BCH, 0 \right)~.
  \]

\chapter{Pre-Lie algebras and the gauge group}\footnotetext{\hrule\smallskip\noindent This material will be published by Cambridge University Press \& Assessment as ‘Maurer-Cartan Methods in Deformation Theory: the twisting procedure’ by Vladimir Dotsenko, Sergey Shadrin and Bruno Vallette. This version is free to view and download for personal use only. Not for re-distribution, re-sale or use in derivative works. \copyright Cambridge University Press \& Assessment}\label{sec:TopoDefTh}

In this chapter, we recall the necessary definitions and results concerning pre-Lie algebras and the symmetries of their Maurer--Cartan elements, and we generalise M. Lazard's treatment of Lie theory to develop the integration theory of complete pre-Lie algebras. The reason to care about pre-Lie algebras in the context of deformation theory is two-fold. First, a great number of examples relevant for deformation theory, including in particular the deformation complexes of maps from Koszul operads, fit into this formalism. Second, the Baker--Campbell--Hausdorff formula integrating the degree zero component to the corresponding gauge group simplifies significantly when we work with a dg Lie algebra coming from a pre-Lie algebra; so using pre-Lie algebras comes from some obvious benefits. The class of pre-Lie algebras to which our results are applied are operadic convolution algebras, leading to tools that allow us to work efficiently with the main notion of algebraic deformation theory: \emph{the deformation gauge group}. \\

This part of the book uses extensively the groundwork from our previous paper~\cite{DotsenkoShadrinVallette16}; in fact, results of that paper extend to complete pre-Lie algebras \emph{without any changes} thanks to the general formalism of complete algebras presented in the previous chapter~\ref{sec:OptheoyFilMod}. To make the exposition self-contained, we however include some of the key relevant results of \cite{DotsenkoShadrinVallette16}. An important consequence for homotopical algebra is that the degree of generality in which our results hold allows us to define a suitable notion of $\infty$-morphism of homotopy algebras encoded by non-necessarily coaugmented cooperads, like curved $\Ai$-algebras or curved $\Li$-algebras. 

\section{Gauge symmetries in complete pre-Lie algebras}\label{subsec:CompConvAlg}

In \cite{DotsenkoShadrinVallette16}, we developed the integration theory of pre-Lie algebras under a strong weight grading assumption. Unfortunately, the deformation theory of many algebraic structures, like the ones that we will study in Chapter~\ref{sec:GaugeTwist}, requires to use curved Koszul dual cooperads which do not satisfy that assumption, and one needs to work with complete algebras instead. In this section, we explain how to extend the integration theory of weight graded left-unital dg pre-Lie algebras to the complete setting. This section can also be seen as the generalisation of the integration theory of complete Lie algebras of M. Lazard's Ph.D. thesis \cite{Lazard50} to left-unital complete pre-Lie algebras.

\begin{definition}[Complete left-unital differential graded pre-Lie algebra]\index{pre-Lie algebra!complete}\index{pre-Lie algebra!left-unital}\index{pre-Lie algebra}
A \emph{complete dg pre-Lie algebra} is a datum  
$\a=\left(
A,  \F, d, \star 
\right)$ of a dg complete module equipped with a 
filtration preserving the bilinear product whose associator is right-symmetric 
$$(a\star b)\star c- a\star(b\star c)  = (-1)^{|b||c|}\big( (a\star c)\star b-a\star(c\star b)\big)\ ,$$
such that the differential $d$ is a derivation 
$$d(a\star b)=d(a) \star b + (-1)^{|a|} a \star d(b) . $$
Such an algebra is called \emph{left-unital} if it is equipped with an element $1\in\F_0 A_0$ which is a closed element $d(1)=0$ and a left unit 
$$ 1 \star a = a\ .$$ 
\end{definition}

\begin{remark}
There is an equivalent left-symmetric version of that definition which naturally arises in the study of homogeneous spaces, going back to pioneering work of J.-L. Koszul \cite{MR145559} and E.B. Vinberg \cite{MR0158414}. The natural definition of a complete (right-symmetric) pre-Lie algebra arising in that domain would be slightly more general: one requires the map $\id+l_\lambda\colon A\to A$ to be invertible for any $\lambda$, see \cite{MR1170527}. Throughout this book, we only use algebras that are complete in the sense of Chapter~\ref{sec:OptheoyFilMod}.    
\end{remark}

The following classical result crucial for this book holds for complete algebras as well. 

\begin{proposition}\label{prop:PreLieToLie}
Let $\a=\left(
A,  \F, d, \star 
\right)$ be a complete  dg pre-Lie algebra. Then $\g=\left(
A,  \F, d, [\, , ] 
\right)$ with 
 $$[a,b]\coloneqq a\star b - (-1)^{|a||b|} b\star a$$
is a complete dg Lie algebra.
\end{proposition}

\begin{proof}
This follows from the fact that there is a morphism of operads from $\Lie$ to $\PreLie$, and it does not matter on which symmetric monoidal category we act, vector spaces or complete $\k$-modules. 
\end{proof}

Using the complete dg Lie algebra structure on a complete dg pre-Lie algebra, we can talk about Maurer--Cartan elements. The following definition is simply a translation of Definition~\ref{def:MCEl} to dg Lie algebras arising from dg pre-Lie algebras.

\begin{definition}[Maurer--Cartan equation]
\index{Maurer--Cartan equation!pre-Lie algebra}
\index{Maurer--Cartan element!pre-Lie algebra}
Let $\a=\left(
A,  \F, d, \star 
\right)$ be a complete  dg pre-Lie algebra. The condition
\begin{equation}\label{eq:MCeq}
d\alpha+\alpha\star\alpha=0
\end{equation}
is called the \emph{Maurer--Cartan equation} of $\a$, and any solution $\alpha\in A_{-1}$ is called a \emph{Maurer--Cartan element}. 
\end{definition}

Complete left-unital dg pre-Lie algebras that are relevant for our purposes arises via the convolution product construction. Let us recall how it works. 

\begin{definition}[Pre-Lie algebra associated to an operad]\index{pre-Lie algebra!associated to an operad}
Let $(\calP,\F,d,\{\circ_i\},\I)$ be a complete dg operad. The \emph{complete dg pre-Lie algebra associated to $\calP$} is
 \[
\left(\prod_{n\in\mathbb{N}}\calP(n),\G,d,\star, 1\right)~, 
 \]
where the pre-Lie product $\star$ is given by the sum of the partial composition maps $\circ_i$ and the left unit is $\id\in\I$.
\end{definition}

Recall that in any closed symmetric monoidal category the mapping space $\hom\left(\calC, \calP\right)$ from a cooperad $\calC$ to an operad $\calP$ forms a \emph{convolution operad}\index{convolution operad}, see \cite[Section~$6.4.1$]{LodayVallette12}. 
When $\calC$ is a filtered dg cooperad and when $\calP$ is a complete dg operad, the above construction associates a complete dg pre-Lie algebra structure to the complete convolution dg operad $\hom\left(\calC, \calP\right)$, with the  pre-Lie product  equal to 
\[\xymatrix@C=30pt{f\star g =\calC \ar[r]^(0.55){\Delta_{(1)}} & \calC\; \widehat{\circ}_{(1)}\, \calC 
\ar[r]^(0.46){f\; \widehat{\circ}_{(1)}\; g} & 
\calP\; \widehat{\circ}_{(1)}\, \calP  \ar[r]^(0,6){\gamma_{(1)}} & \calP \ ,}\]
where the various infinitesimal notions \cite[Section~$6.1$]{LodayVallette12} are now considered in complete setting. 

\begin{proposition}
The space of equivariant maps from a filtered dg cooperad $\calC$ to a complete dg operad~$\calP$
$$\hom_\Sy\big(\calC, \calP\big)\coloneqq{\prod}_{n\in \NN} \hom_{\Sy_n}\big(\calC(n), \calP(n)\big)$$
forms a complete left-unital dg pre-Lie algebra 
$$\left(\hom_\Sy\big(\calC, \calP\big), \partial, \star, 1 \right)$$
and thus a complete dg Lie algebra 
$$\left(\hom_\Sy\big(\calC, \calP\big), \partial, [\;,\,]\right)\ .$$
\end{proposition}

\begin{proof}
As in the classical case, one can see that the space of equivariant maps is stable under the pre-Lie product. 
\end{proof}

\begin{definition}[Complete convolution pre-Lie algebra]\index{pre-Lie algebra!convolution}
The complete dg pre-Lie algebra $\hom_\Sy\big(\calC, \calP\big)$ is called the \emph{complete convolution pre-Lie algebra}. 
\end{definition}

The central object in the study of the deformation theory of algebraic structures \cite{LodayVallette12, DotsenkoShadrinVallette16} is the complete convolution algebra associated to a filtered  cooperad $\calC$ and the complete endomorphism operad of a complete module $A$: 
\[
\a_{\, \calC, A}\coloneqq
\big(
\hom_\Sy\big(\calC, {\eend}_A\big), \partial, \star, 1 
\big)\ .
\]
The present version of this definition in the complete setting is the most general that we are aware of; it will be used in Chapter~\ref{sec:GaugeTwist} where we give a gauge group interpretation to the twisting procedure of algebraic structures. However, already the discrete case is of interest to deformation theory: if the cooperad $\calC$ is the Koszul dual $\calP^{\ac}$ of a Koszul operad $\calP$, the set of $\calP_\infty$-algebra structures on $A$ is in one-to-one correspondence with the Maurer--Cartan set of the convolution algebra $\mathrm{MC}(\a_{\, \calP^{\ac}, A})$, see~\cite[Chapter~10]{LodayVallette12}. \\

To discuss symmetries of Maurer--Cartan elements in complete pre-Lie algebras, we shall recall the notions of a filtered group and a complete group \cite[Section~I.2]{Lazard50} defined in a  way similar way to that of a filtered/complete algebra. 

\begin{definition}[Filtered/complete group]\index{filtered!group} \index{complete!group}
The datum of a \emph{filtered group} is a quadruple $(G, \F, \cdot, e)$ where $(G,  \cdot, e)$ is a group and $\F$ is a filtration of $G$ by subgroups
$$G=\F_1 G \supset \F_2 G \supset  \cdots \supset \F_k G  \supset \cdots \ ,$$
such that the commutator satisfies $xyx^{-1}y^{-1}\in \F_{i+j} G$ for $x\in \F_i G$ and $y\in \F_j G$. 
A filtered group is called \emph{complete} when the underlying topology is Hausdorff and complete. 
\end{definition}

Since by Proposition \ref{prop:PreLieToLie}, a complete pre-Lie algebra $\a=\left(
A,  \F, d, \star 
\right)$ has the associated dg Lie algebra $\g=\left(
A,  \F, d, [\, , ] \right)$, one may define the associated gauge group in the following way alluded to in the end of the previous chapter. 

\begin{definition}[Gauge group]\index{gauge!group}
The \emph{gauge group} associated to a complete left-unital dg pre-Lie algebra $\a$ is defined by 
$$\Gamma\coloneqq\big(\F_1 A_0,  \BCH(\; ,\,), 0\big)\ .$$
\end{definition}

The gauge group is well defined since the BCH product converges for elements in $\F_1 A_0$: it is a sum of iterated commutators  
$$\BCH(x,y)=\underbrace{x+y}_{\in \F_1 A_0}+\underbrace{\frac{1}{2}[x,y]}_{\in \F_2 A_0} + \underbrace{\frac {1}{12}\big([[x,y], y]+[[y,x], x]\big)}_{\in \F_3 A_0}+\cdots \ ,$$
which is thus  convergent by Lemma~\ref{lem:Conv}.

\begin{proposition}
The gauge group is a complete group, with respect to the following filtration:  
$$\F_1 \Gamma\coloneqq\F_1 A_0 \supset \F_2\Gamma\coloneqq \F_2 A_0 \supset  \cdots \supset \F_k\Gamma\coloneqq\F_k A_0  \supset \cdots \ .$$
\end{proposition}

\begin{proof}
The form of the BCH product mentioned above shows that each $\F_k\Gamma$ is a subgroup of the gauge group. It is known (and very easy to check) that the leading term in group commutator is the Lie algebra commutator, that is, 
\[
\BCH(\BCH(\BCH(x,y),-x), -y)=[x,y]+\cdots\ , 
\]
where the higher terms are iteration of brackets of at least one $x$ and one $y$ each time. Therefore, the commutators satisfy $\BCH(\BCH(\BCH(x,y),-x), -y)\in \F_{i+j} \Gamma$ for $x\in \F_i \Gamma$ and $y\in \F_j \Gamma$.
\end{proof}

Of course, the results on the gauge group action remain valid in the complete case.

\begin{proposition}\label{prop:GaugeGroupAction}
Let $\a=\left(A,  \F, d, \star, 1 \right)$ be a complete left-unital dg pre-Lie algebra. The gauge group $\Gamma$ acts on the set of Maurer--Cartan elements by the formula
 \[
\lambda.\alpha = \frac{\id-\exp(\ad_\lambda)}{\ad_\lambda}(d\lambda)+\exp(\ad_\lambda)(\alpha).
 \]
This left action is continuous, that is  the map 
\[
\begin{array}{clc}
\Gamma\times \mathrm{MC}(\a)& \to&\mathrm{MC}(\a)\\
(\lambda,  {\alpha}) & \mapsto &\lambda.{\alpha} 
\end{array}\]
is continuous in both variables. 
\end{proposition}

\begin{proof}
The first part is essentially Proposition \ref{prop:GaugeActionMC}. 
To show the continuity, we shall once again use the differential trick and work with the action 
 \[
\lambda. (\delta+\alpha)\coloneqq e^{\ad_\lambda}({\delta+\alpha})
 \] 
on solutions to the square-zero equation. Let us fix a Maurer--Cartan element $\alpha \in \mathrm{MC}(\a)$ and let us consider a  sequence $\{\lambda_n\}_{n\in \NN}$ of elements of the gauge group $\Gamma$ which tends to an element $\lambda$, that is 
\[\forall k\in \NN,\,  \exists N\in \NN, \, \forall n\ge N,\  \lambda_n-\lambda\in \F_k A_0\ . \]
So, we have $\forall k\in \NN,\,  \exists N\in \NN, \, \forall n\ge N, $
\[\begin{aligned}
e^{\ad_{\lambda_n}}({\alpha})- e^{\ad_\lambda}({\alpha})
&=\sum_{m=1}^\infty \frac{1}{m!}\left(\ad_{\lambda_n}^m -\ad_{\lambda}^m\right)(\alpha)\\
&=\sum_{m=1}^\infty \underbrace{\frac{1}{m!}
\left(
\sum_{l=1}^m \ad_{\lambda}^{l-1} \circ \ad_{\lambda_n-\lambda}\circ \ad_{\lambda_n}^{m-l} \right)(\alpha)}_{\in \F_{k+m-1} A_{-1}}\in \F_{k} A_{-1} \ .
\end{aligned}
\]
The continuity in the second variable is proved similarly.
\end{proof}

\section{Circle product formula for the gauge action}\label{subsec:CompleteLieInt}

The pre-Lie identity is weaker than the associativity identity but stronger than the Jacobi identity. We shall now recall formulas of \cite{DotsenkoShadrinVallette16} which unravel the combinatorics that enters formulas for the gauge group associated to a pre-Lie algebra. As a toy model, let us first consider a more classical situation where the complete dg Lie algebra $\g=(A, d, [\, , ])$ comes from a complete unital dg associative algebra $\mathfrak{a}=(A, d, \cdot, 1)$. For $\lambda\in \F_1 A_0$, the exponential series 
 \[
e^\lambda:=1 +\lambda + \frac{\lambda^{\star 2}}{2!} + \frac{\lambda^{\star 3}}{3!} +\cdots \ 
 \]
converges, and we have a group   
 \[
\mathfrak{G}:=\left(\left\{ e^\lambda, \lambda \in\F_1A_0\right\},
\cdot, 1
\right)\ .  
 \]
The adjoint action $\ad_\lambda$ can be written as a difference $l_\lambda-r_\lambda$ of the left and the right multiplication by $\lambda$. In an associative algebra, these two endomorphisms commute, so  
 \[
\exp(\ad_\lambda) (\delta+\alpha)= \exp(l_{\lambda})\circ \exp(-r_{\lambda})(\delta+\alpha)=e^{\lambda}(\delta+\alpha) e^{-\lambda}  \ .
 \]
Of course, the group law on elements like that is still given by the BCH formula. However, the exponential map is immediately seen to be a bijection onto the set $1+\F_1A_0$, and the group structure on that set is much simpler, being just the associative product $(1+x)\cdot(1+y)$ that we already have. As we shall see now, in the case of complete pre-Lie algebras there is also an elegant way to express the group law. All proofs of the statements below are exactly the same as proofs of the quoted results we proved in~\cite{DotsenkoShadrinVallette16}. In \emph{op. cit.}, the results were proved for weight graded algebras; we shall state them for complete algebras instead, and note that the general formalism developed in Chapter~\ref{sec:OptheoyFilMod} ensures that all proofs from the weight graded case are valid \emph{mutatis mutandis}. \\

For any element $\lambda\in \F_1 A$, we consider the following right iteration of the pre-Lie product 
$$\lambda^{\star n}\coloneqq\underbrace{(\cdots((\lambda \star \lambda) \star \lambda)\cdots )\star \lambda}_{n\  \text{times}}\in \F_n A\ . $$

\begin{definition}[Pre-Lie exponential]\index{pre-Lie exponential}
The \emph{pre-Lie exponential} of an element $\lambda \in \F_1 A$ is 
defined  by the following convergent  series 
$$e^\lambda\coloneqq1 +\lambda + \frac{\lambda^{\star 2}}{2!} + \frac{\lambda^{\star 3}}{3!} +\cdots \ . $$
\end{definition}

The choice of right-normed products here conceptually corresponds to the fact that the pre-Lie identity can be rewritten in the form
 \[
(a\star b)\star c-(-1)^{|b||c|} (a\star c)\star b= a\star(b\star c) - (-1)^{|b||c|} a\star(c\star b)\big) ,
 \]
or 
 \[
[r_c, r_b]=r_{[b,c]} ,
 \]
meaning that operators of right multiplication $r_{-\lambda}$ give a representation of the corresponding Lie algebra. The exponential we wrote appears in many different situations, including the formula for solutions to the flow differential equation for vector fields on a manifold with a flat and torsion free connection. Of course, that is not a mere coincidence: the Lie bracket of vector fields in this context arises as anti-symmetrisation of a pre-Lie product, see \cite{AgrachevGamkrelidze80}. 

\begin{proposition}\label{prop:preLieExpLog}
The exponential map is an isomorphism onto the set of \emph{group-like elements} defined by 
\[G\coloneqq1+\F_1 A_0=\left\{1+x, \ x\in  \F_1 A_0\right\}\ ,\]
with the basis of open sets at $1$ defined by $\{1+\F_k A_0\}_{k \ge 1}$.  $1+\F_1A_0$.
\end{proposition}

\begin{proof}
This is essentially proved in \cite[Lemma~1]{DotsenkoShadrinVallette16}.
\end{proof}

The inverse map, which one should think of as a pre-Lie version of logarithm, sends the element $1+a\in1+\F_1A_0$ to a series $\Omega(a)$ called the ``pre-Lie Magnus expansion'', which begins by 
$$a - \frac12 a\star a + \frac14 a\star (a \star a)
+ \frac{1}{12}(a\star a)\star a+\cdots \ , $$ see \cite{AgrachevGamkrelidze80, Manchon11} for more details.

To write a combinatorial formula for the gauge group law, we recall the following definition.

\begin{definition}[Symmetric braces]\index{pre-Lie algebra!symmetric brace}\label{def:SymBrace}
The \emph{symmetric braces} \[\{ a; b_1,\ldots, b_n\}\colon A\otimes A^{\otimes n}\to A, \quad n\ge 0 ,\] are defined recursively by setting $\{a; \}=a$ and 
\begin{multline*}
\{ a; b_1,\ldots, b_n\}:=  \{ \{a; b_1,\ldots, b_{n-1}\}; b_n\} \\ 
-  \sum_{i=1}^{n-1} (-1)^{|b_n|(|b_{i+1}|+\cdots+|b_{n-1}|)}\{a; b_1, \ldots, b_{i-1}, \{b_i; b_n\}, b_{i+1}, \ldots, b_{n-1}\}   .
\end{multline*}
\end{definition}

It is possible to show that these operations are {(graded)} symmetric with respect to $b_1,\ldots, b_n$. They satisfy certain relations which we are not going to use explicitly; the interested reader is referred to \cite[Section~$13.11.4$]{LodayVallette12} for details. 

\begin{definition}[Circle product]\index{pre-Lie algebra!circle product}
Let $\a=\left(
A,  \F, d, \star, 1 
\right)$ be a complete left unital dg pre-Lie algebra, and let $a\in A$, $b\in \F_1A_0$. The \emph{circle product} $a\circledcirc(1+b)$ is defined by the convergent series 
 \[
a \circledcirc (1+b) := \sum_{n\ge 0} \frac{1}{n!}  \{a; \underbrace{b, \ldots, b}_{n}\}\ .
 \]
\end{definition}

We are now ready to state a key result allowing one to do effective computations with gauge symmetries in pre-Lie algebras. 

\begin{theorem}\label{thm:GaugeGrpBis}
The datum 
 \[
\mathfrak{G}:=( 1+\F_1 A_0, \circledcirc, 1)  
 \]
is a complete group. The pre-Lie exponential map is a filtered isomorphism, and therefore a homeomorphism, between the gauge group 
$$
\Gamma=\left(\F_1 A_0, \BCH, 0 \right)
$$
and $\mathfrak{G}$~.
\end{theorem}

\begin{proof}
The arguments of \cite[Theorem~2]{DotsenkoShadrinVallette16} adapt to the complete case without changes. Proposition~\ref{prop:preLieExpLog} implies that $\mathfrak{G}$ is a complete group, which is filtered isomorphic, thus homeomorphic, to the gauge group.
\end{proof}

Using the circle product, it is also possible to write down an elegant formula for the gauge action.

\begin{theorem}\label{thm:GaugeActionBis}
Let $\a=\left(
A,  \F, d, \star, 1 
\right)$ be a complete left unital dg pre-Lie algebra. The gauge group $\mathfrak{G}$ acts on the set of Maurer--Cartan elements by the formula: 
$$e^\lambda\cdot\at\coloneqq\lambda.\alpha= \left(e^\lambda \star \at\right) \circledcirc e^{-\lambda}\ .$$
\end{theorem}

\begin{proof}
This is a direct corollary of Theorem~\ref{thm:GaugeGrpBis}, Proposition~\ref{prop:GaugeGroupAction} and \cite[Proposition~$5$]{DotsenkoShadrinVallette16}.
\end{proof}

\section{Operadic deformation theory in the complete setting}\label{sec:ComOpDefTh}
We aim to apply the results of the previous section  to the complete convolution algebra associated to a filtered dg cooperad $\calC$ and the complete endomorphism operad of a complete dg module $A$: 
$$\a_{\, \calC, A}\coloneqq
\big(
\hom_\Sy\big(\calC, {\eend}_A\big), \partial, \star, 1 
\big)\ .
$$
Notice that, in this case, the internal differential element is given by 
$$\delta\colon \calC \to \I\to \k \partial_A,$$ 
where the map $\calC \to \I$ is the counit of the cooperad $\calC$.
The interest in this example of application lies in the following interpretation of the Maurer--Cartan elements of this convolution algebra.
 
\begin{proposition}\label{prop:MCOmegaC}
Let $\calC$ be a filtered  dg cooperad and let $A$ be a complete dg module. The set of Maurer--Cartan elements of the complete convolution algebra $\a_{\, \calC, A}$ is in natural one-to-one correspondence with the complete $\Omega \calC$-algebra structures on $A$. 
\end{proposition}

\begin{proof}
Let us start with the case when the filtered dg cooperad $\calC$ is coaugmented. 
In this case,  the result holds true in the classical discrete setting \cite[Theorem~$6.5.7$]{LodayVallette12} and we use the same arguments and computations here. By Theorem~\ref{thm:CompleteAlg}, a complete $\Omega \calC$-algebra structures on $A$ is a morphism \[\Omega \calC \to {\eend}_A\] of filtered dg operads. The underlying operad of the cobar construction of $\calC$ is a free operad generated by the desuspension of $\overline{\calC}$. Since $\calC$ is filtered, this free operad is equal to the filtered free operad on the same space of generators; this can be seen for instance through the formula \cite[Proposition~$5.6.3$]{LodayVallette12}. Therefore we get:
\[
\Hom_{\textsf{CompOp}}\left(\widehat{\Omega \calC}, {\eend}_A\right)\cong
\Hom_{\textsf{FilOp}}\left(\Omega \calC, {\eend}_A\right)\cong \hom_{\Sy}\left(\overline{\calC}, {\eend}_A\right)_{-1}\ , 
\]
where the compatibility with the differentials on the left-hand side corresponds to the Maurer--Cartan equation on the right-hand side, by the same computation as in the classical case. 

When the cooperad $\calC$ is not coaugmented, like in the examples of Section~\ref{sec:GaugeTwist}, we consider the cobar construction $\Omega \calC \coloneqq \calT(s^{-1}\calC)$ with similar differential induced by the internal differential of $\calC$ and its partial decomposition maps. In this case,  Maurer--Cartan elements in the convolution algebra associated to a complete graded module $A$ correspond to complete $\Omega \calC$-algebra structures. 
\end{proof}

\begin{remark}
This proposition applies to the Koszul dual cooperad $\calC\coloneqq\calP^{\ac}$ of a Koszul operad $\calP$. In this case, Maurer--Cartan elements of the convolution algebra $\a_{\, \calP^{\ac},\, A}$ are nothing but complete $\calP_\infty$-algebra structures on $A$. This way, we can develop the deformation theory of complete $\Ai$-algebras or complete $\Li$-algebras, see the next section. 
\end{remark}

The gauge group obtained by integrating the underlying complete pre-Lie algebra indeed acts of the set of Maurer--Cartan elements, but the remarkable feature brought by the operadic nature of this construction is that we can include the usual gauge group in a bigger group. Namely, let us consider the subset 
\begin{align*}
U_0:=&\calF_1 \hom\big(\calC(0), {\eend}_A(0)\big)_0\times 
\calF_1 \hom\big(\calC(1), {\eend}_A(1)\big)_0\times\\ & {\displaystyle{\prod}_{n\ge 2}} \hom_{\Sy_n}\big(\calC(n), {\eend}_A(n)\big)_0
\end{align*}
of the complete convolution algebra $\a_{\, \calC, A}$, and the set of group-like elements 
\begin{align*}
1+U_0:=&\calF_1 \hom\big(\calC(0), {\eend}_A(0)\big)_0\times 
\left(\id+\calF_1 \hom\big(\calC(1), {\eend}_A(1)\big)_0\right)\times\\ &{\displaystyle{\prod}_{n\ge 2}} \hom_{\Sy_n}\big(\calC(n), {\eend}_A(n)\big)_0 .
\end{align*}

\begin{lemma}\label{lem:DefGaugeConv}\leavevmode
\begin{itemize}
\item[(i)] The BCH series converges for any two elements $x,y\in U_0$~.
\item[(ii)] The series defining the product $\circledcirc$ converges for any two elements $x,y\in 1+U_0$~.
\item[(iii)] The pre-Lie exponential is an isomorphism between $U_0$ and $1+U_0$; this isomorphism identifies the group law on $U_0$ given by the BCH series with the group law on $1+U_0$ given by the product $\circledcirc$~.
\end{itemize}
\end{lemma}

\begin{proof}
The only difference from the ``standard'' proofs is that we do not have a literal filtration ensuring convergence, and hence an extra argument is needed. Let us note that the elements of $U_0$ that do not belong to the ``usual'' gauge group $\calF_1\hom_\Sy\big(\calC, {\eend}_A\big)_0$ come from the part of the complete convolution algebra supported on arity at least two, that is
 \[
{\displaystyle{\prod}_{n\ge 2}} \hom_{\Sy_n}\big(\calC(n), {\eend}_A(n)\big)_0 .
 \]
The special feature of the complete convolution algebra is that, if one assigns to each element $x\in \hom_{\Sy_n}\big(\calC(n), {\eend}_A(n)\big)_0$ extra weight grading $w(x)=n-1$, the algebra is graded: $w(x\star y)=w(x)+w(y)$ for all homogeneous elements $x,y$. We note that for elements supported on arity at least two, their weight grading is positive, and therefore for such elements one has convergence of the BCH series and the series for $\circledcirc$ since one has a finite sum in each weight graded component. As a whole, the complete pre-Lie subalgebra $U_0$ is an extension: it has an ideal $\calF_1\hom_\Sy\big(\calC, {\eend}_A\big)_0$, and the quotient by this ideal consists of (some) elements of positive weight. As a consequence, one may combine the filtration argument for $\calF_1\hom_\Sy\big(\calC, {\eend}_A\big)_0$ with the weight grading argument for elements of positive weight to conclude convergence of the two series we consider. The last statement is established analogously.
\end{proof}

The result we proved lays the groundwork for the following definition.

\begin{definition}[Deformation gauge group]\index{deformation gauge group}
The \emph{deformation gauge group} associated 
to a filtered dg cooperad $\calC$ and a complete dg module $A$ is defined by 
\begin{align*}
\widetilde{\Gamma}\coloneqq&\left(
\calF_1 \hom\big(\calC(0), {\eend}_A(0)\big)_0\times 
\calF_1 \hom\big(\calC(1), {\eend}_A(1)\big)_0\times \right.\\
&\left.{\displaystyle{\prod}_{n\ge 2}} \hom_{\Sy_n}\big(\calC(n), {\eend}_A(n)\big)_0
,  \BCH(\; ,\,), 0
\right)\ .
\end{align*}
\end{definition}

The group $\widetilde{\Gamma}$ does not appear to have a complete group structure, yet it contains two important complete groups (which have slightly incompatible filtrations). The first of them is the ``obvious'' gauge group $\calF_1\hom_\Sy\big(\calC, {\eend}_A\big)_0$ with the hom-set filtration on it. The second one, also useful in some applications, is the group on the set
\begin{align*}
\Gamma^\circ:= &
\calF_2 \hom\big(\calC(0), {\eend}_A(0)\big)_0\times \calF_1 \hom\big(\calC(1), {\eend}_A(1)\big)_0\times\\
& {\displaystyle{\prod}_{n\ge 2}} \hom_{\Sy_n}\big(\calC(n), {\eend}_A(n)\big)_0
\end{align*}
equipped with the filtration $\F_\bullet$ with
\begin{align*}
\F_k\Gamma^\circ:= &
\calF_{k+1}\hom\big(\calC(0), {\eend}_A(0)\big)_0\times
\calF_{k} \hom\big(\calC(1), {\eend}_A(1)\big)_0\times\\ 
&{\displaystyle{\prod}_{n\ge 2}} \calF_{k+1-n}\hom_{\Sy_n}\big(\calC(n), {\eend}_A(n)\big)_0 \ ,
\end{align*}
where we formally put $\calF_k B=B$ for all filtered modules $B$ and all $k\le 0$~.

\medskip

Even though the group $\widetilde{\Gamma}$ is not complete, arguments identical to those of Lemma \ref{lem:DefGaugeConv} prove the following result.

\begin{proposition}\label{prop:Extension}
All the results of Sections \ref{subsec:CompConvAlg} and~\ref{subsec:CompleteLieInt} hold true for the deformation gauge group $\widetilde{\Gamma}$ and its avatar $\widetilde{\mathfrak{G}}$.
\end{proposition}

We shall now give a definition of an $\infty$-morphism of two $\Omega \calC$-algebras. To do that, we shall encode $\Omega \calC$-algebra structures by Maurer--Cartan elements. Without loss of generality, we may consider two Maurer--Cartan elements $\alpha$ and $\beta$ of the complete convolution algebra $\a_{\, \calC, A}$ corresponding to two $\Omega \calC$-algebra structures on the same complete module $A$. Indeed, the data of an $\Omega \calC$-algebra structure on a complete dg module $(A,\F)$ and of another $\Omega \calC$-algebra structure~$\beta$ on a complete dg module $(B,\G)$ may be encoded by two Maurer--Cartan elements $\alpha$ and~$\beta$ of the convolution algebra $\hom_\Sy\big(\calC, {\eend}_{A\oplus B}\big)$ taking values, respectively, on $\eend_A$ and $\eend_B$. 
 
\begin{definition}[$\infty$-morphism]\label{def:InftyMorph}\index{$\infty$-morphism}\index{$\infty$-morphism!composition}
An  \emph{$\infty$-morphism} $\alpha \rightsquigarrow \beta$ between $\alpha$ and $\beta$ is a  degree $0$ element 
$f : \calC \to {\eend}_A$, such that $f_{0}\in \calF_1 \hom\big(\calC(0), {\eend}_A(0)\big)_0$ and satisfying the  equation
\begin{eqnarray}\label{eqn=InftyMor}
f \star \at = \beta \cc f\ . 
\end{eqnarray}
The \emph{composition} of two infinity-morphisms $f\colon\alpha \rightsquigarrow \beta$ and $g\colon\beta\rightsquigarrow \gamma$ is defined by the circle product $g \cc f$. 
\end{definition}

The assumption $f_{0}\in \calF_1 \hom\big(\calC(0), {\eend}_A(0)\big)_0$ is needed for the right-hand side $\beta \cc f$ to be well-defined. In this book, we only consider cooperads defined by the partial definition $\Delta_{(1)} : \calC \to \calC\, \widehat{\circ}_{(1)}\, \calC$, which splits operations into two. When the cooperad $\calC$ has trivial arity $0$ part, i.e. $\calC(0)=0$, it admits a ``full'' coassociative decomposition map  $$\Delta : \calC \to \calC\, \widehat{\circ} \, \calC$$ which splits operations into two levels, see \cite[Section~$5.8.2$]{LodayVallette12}. 
In this case, the associative product $\beta \cc f$ has the following simple description: 
\[
\beta\cc f : \calC
 \xrightarrow{\Delta} \calC\, \widehat{\circ} \, \calC
  \xrightarrow{\beta \, \widehat{\circ} f}  \eend_A \widehat{\circ} \eend_A
  \xrightarrow{\gamma}  \eend_A \ .
  \]
However, the full decomposition map for the cooperad $\calC$ fails to be  well-defined in the general case, but thanks to the assumption $$f_{0}\in \calF_1 \hom\big(\calC(0), {\eend}_A(0)\big)_0,$$ the product $\cc$ is well-defined, by a convergent series. In our setup, the cooperads have elements of arity zero, and in order to define a suitable notion of an $\infty$-morphism for the homotopy algebras like curved $\Ai$-algebras and curved $\Li$-algebras, the formalism of this section is unavoidable. In fact, the theory of complete curved $\Ai$-algebras over a complete local ring developed by L.~Positselski in \cite{Positselski12} has its own notion of $\infty$-morphism which we recover via Definition~\ref{def:InftyMorph}.

\begin{definition}[$\infty$-isotopy]\index{$\infty$-isotopy}
An \emph{$\infty$-isotopy} between the Maurer--Cartan elements $\alpha$ and $\beta$ is an $\infty$-morphism $f : \calC \to {\eend}_A$ satisfying  
\[f_1\in 1+\calF_1 \hom\big(\calC(1), {\eend}_A(1)\big)_0\ .\]
\end{definition} 

The following result is the generalisation of the main result of \cite[Section~$5$]{DotsenkoShadrinVallette16} to the complete case; the results of \textit{loc. cit. } concern the case of a weight graded dg cooperad $\calC$ and a dg module $A$, both viewed as complete objects with respectively the weight filtration and the trivial filtration.

\begin{theorem}\label{thm:DeligneGroupoidII}
For a filtered dg cooperad $\calC$ and for any complete dg module $A$, the set of  $\infty$-isotopies 
forms a subgroup of the group of  invertible $\infty$-morphisms of $\Omega \calC$-algebra structures on $A$, which is isomorphic to the deformation gauge group under the pre-Lie exponential map
$$ \widetilde{\Gamma} \cong \widetilde{\mathfrak{G}}= (\infty\textsf{-}\mathsf{iso}, \cc, \id_A)  \ . $$ 
The Deligne groupoid associated to the deformation gauge group  is isomorphic to the groupoid whose objects are $\Omega \calC$-algebras and whose 
morphisms are  $\infty$-isotopies
$$\mathsf{Deligne}\left(\a_{\calC,A}\right)\coloneqq\left(\mathrm{MC}(\a_{\calC,A}), \widetilde{\Gamma}
\right)\cong
\left(\Omega\calC\textsf{-}\mathsf{Alg}, \infty\textsf{-}\mathsf{iso}\right)
\ .$$
\end{theorem}

\begin{proof}
This is a direct corollary of Proposition~\ref{prop:MCOmegaC}, Theorem~\ref{thm:GaugeGrpBis} and Theorem~\ref{thm:GaugeActionBis}, via their generalisations given at Proposition~\ref{prop:Extension}. 
\end{proof}

\chapter{The gauge origin of the twisting procedure}\footnotetext{\hrule\smallskip\noindent This material will be published by Cambridge University Press \& Assessment as ‘Maurer-Cartan Methods in Deformation Theory: the twisting procedure’ by Vladimir Dotsenko, Sergey Shadrin and Bruno Vallette. This version is free to view and download for personal use only. Not for re-distribution, re-sale or use in derivative works. \copyright Cambridge University Press \& Assessment}\label{sec:GaugeTwist}

In this chapter, we provide a first application of the complete operadic deformation theory developed in the previous section. Namely, we deal with the easiest example of gauge action: that when the gauge element is just an element of the underlying dg module, that is, concentrated in arity $0$. In this way, we get a conceptual interpretation of  the  twisting procedure of  a complete $\Ai$-algebra (or a complete $\Li$-algebra) by a Maurer--Cartan element. \\

The procedure of twisting a complete $\Ai$-algebra with a Maurer--Cartan element is a non-commutative analogue of the twisting procedure  for complete $\Li$-algebras. The former one plays a seminal role in the construction of the Floer cohomology of Lagrangian submanifolds  \cite{FOOO09I} and the latter one is used in crucial ways in deformation theory, rational homotopy theory, higher algebra, and quantum algebra; we shall discuss these applications in detail in the last chapter~\ref{sec:Applications}. \\

We start with the simpler $\Ai$-case, encoding it conceptually in the operadic deformation theory language, and then move to the analogous case of $\Li$-algebras. From the conceptual gauge action interpretation of the twisting procedure, we derive automatically ``all'' its known useful properties. We conclude this chapter with a criterion on quadratic operads which allows one to determine if the associated category of homotopy algebras admits a meaningful twisting procedure.

\section{Curved \texorpdfstring{$\Ai$}{Ai}-algebras}\label{sec:CurvedAiAlg}

In this section, we recall the classical definition of an $\Ai$-algebra and its variations (with a possible shift and possibly nonzero curvature), and explain how to use the Maurer--Cartan calculus in complete pre-Lie algebras to recover those notions. We state the results in the discrete case; to get the same notions in the filtered or complete setting, one just has to replace everywhere the endomorphism operad $\End_A$ by the complete endomorphism suboperad $\eend_A$ consisting of filtered maps.\\

Let us start with one of the most simple ns operads: the ns operad $\uAs$, which encodes unital associative algebras. It is one-dimensional in each arity $\uAs(n)=\k \upsilon_n$, for $n\ge 0$, and concentrated in degree $0$; its operadic structure is given by the formulas $\upsilon_k\circ_i\upsilon_l=\upsilon_{k+l-1}$\ . Let us now consider the linear dual ns cooperad $\uAs^*$, where we denote the dual basis by $\nu_n\coloneqq\upsilon_n^*$~. The infinitesimal decomposition coproduct, which is the dual of the partial composition products,
 is equal to  
$$\Delta_{(1)}(\nu_n)=\sum_{\substack{p+q+r=n\\ p,q,r \ge 0}} \nu_{p+1+r} \circ_{p+1} \nu_q\ .$$

We now consider the endomorphism ns operad $\End_A\coloneqq\left\{\Hom(A^{\otimes n}, \allowbreak A)\right\}_{n\ge 0}$ associated to any graded module $A$. Let us begin with noting that the convolution ns operad $\Hom(\uAs^*, \End_A)$  is canonically isomorphic to the operad $\End_A$ itself. Therefore the induced left unital pre-Lie algebra 
 \[\mathfrak{a}\coloneqq\left(\prod_{n\ge 0} \Hom(A^{\otimes n}, A),  \star, 1\right)\] is given by 
$$1=\id_A \quad \text{and} \quad f\star g = \sum_{i=1}^n f\circ_i g  \ ,$$
for $f\in \Hom(A^{\otimes n}, A)$. The Maurer-Cartan elements of this algebra are elements $\alpha$ of degree $-1$ satisfying the Maurer--Cartan equation 
$$\alpha\star \alpha=0\ . $$ 
Let us unfold this definition for an element $\alpha=(m_0, m_1, \ldots, m_n,\ldots)$. 
To ensure that we recover the ``classical'' notion of a shifted curved $\Ai$-algebra from \cite{CD01}, we shall denote $\theta\coloneqq m_0(1)\in A_{-1}$, $d\coloneqq m_1$, and 
\[\partial f \coloneqq d\circ_1 f - (-1)^{|f|} \sum_{i=1}^{n} f\circ_i d~,\]
for any degree-wise homogenous element $f\in  \Hom(A^{\otimes n}, A)$. 
With all that in mind, the datum of an element $\alpha$ is  a sequence of degree $-1$ operations $m_n\colon A^{\otimes n} \to A$, $n\ge 0$, and the Maurer--Cartan equation for $\alpha$ amounts to the conditions
\begin{eqnarray*}
&\text{arity}\ 0:& d \theta=0\ , \\
&\text{arity}\  1:& d^2=- m_2(\theta, -)-m_2(-, \theta)\ , \\
&\text{arity}\  2:& \partial m_2= 
-m_3(\theta, -,-)-m_3(-,\theta, -) -m_3(-,-,\theta)\ . \\
&\text{arity}\  3:& \partial m_3=-m_2(m_2(-,-), -)-m_2(-,m_2(-,-))-m_4(\theta, -, -,-)\\&&\qquad\qquad\ 
-m_4(-,\theta,  -,-)
-m_4(-,-,\theta, -)
-m_4(-,-,-,\theta)\ ,
\\
&\text{arity}\ n:& \partial m_n= -\sum_{\substack{p+q+r=n\\ 2\leq q \leq n-1}}m_{p+1+r}\circ_{p+1} m_q
-\sum_{i=1}^{n+1} m_{n+1}\circ_i \theta\ .
\end{eqnarray*}

\begin{definition}[Shifted curved $\Ai$-algebra]\index{$\Ai$-algebra!shifted curved }
The data  $(A, \theta, d, m_2, m_3, \ldots)$ of a graded module $A$ equip\-ped with a degree $-1$ element $\theta$ and degree $-1$ maps $d, m_2, m_3,\ldots$ satisfying the equations above is called a \emph{shifted curved $\Ai$-algebra}. The element $\theta$ is called the \emph{curvature}. \index{curvature}
\end{definition}

When the curvature $\theta$ vanishes, the operator $d$ squares to zero and thus gives rise to a differential, which is a derivation with respect to the binary operation $m_2$. The higher operations $m_n$, for $n\ge 3$, are then homotopies for the relation $$-\sum_{\substack{p+q+r=n\\ 2\leq q \leq n-1}}m_{p+1+r}\circ_{p+1} m_q=0.$$
This corresponds to replacing the operad $\uAs$ by the operad $\As$, which encodes (non-necessarily unital) associative algebras; that operad is isomorphic to $\overline{\End}_\k$, the suboperad of $\End_\k$ with trivial arity $0$ component: $\overline{\End}_\k(0)=0$~.

\begin{definition}[Shifted $\Ai$-algebra]\index{$\Ai$-algebra!}
A chain complex $(A, d, m_2, m_3, \ldots)$ eq\-ui\-pp\-ed with  degree $-1$ maps $ m_2, m_3,\ldots$ satisfying the equations above with $\theta=0$ is called a \emph{shifted $\Ai$-algebra}. 
\end{definition}

Given a shifted curved $\Ai$-algebra structure $(A, \theta', d', m'_2, m'_3, \ldots)$, one can choose to work on the desuspension $s^{-1}A$ of the underlying graded module. The degree of the operations and the signs involved in their relations will thus be modified. The curvature $\theta\coloneqq s^{-1}\theta'$ now has degree $-2$, the unary operator $d : s^{-1}A \to s^{-1}A$ still has degree $-1$, and the operations $m_n : (s^{-1}A)^{\otimes n} \to s^{-1}A$ now have degree $n-2$. They satisfy the following signed relations: 
\begin{eqnarray*}
&\text{arity}\ 0:& d\theta=0\ , \\
&\text{arity}\  1:& d^2=m_2(\theta, -)-m_2(-, \theta)\ , \\
&\text{arity}\  2:& \partial m_2=
-m_3(\theta, -,-)+m_3(-,\theta, -)-m_3(-,-,\theta)\ . \\
&\text{arity}\  3:& \partial m_3=m_2(m_2(-,-), -)-m_2(-,m_2(-,-))+m_4(\theta, -, -,-) \\&&\qquad\quad\ 
-m_4(-,\theta,  -,-)
+m_4(-,-,\theta, -)
-m_4(-,-,-,\theta)\ ,
\\
&\text{arity} \ n:& \partial m_n= \sum_{\substack{p+q+r=n\\ 2\leq q \leq n-1}}(-1)^{pq+r+1} m_{p+1+r}\circ_{p+1} m_q+\sum_{i=1}^{n+1} (-1)^{n-i} m_{n+1}\circ_i\theta\ .
\end{eqnarray*}
This corresponds to replacing the operad $\uAs\cong\End_\k$ by the ns endomorphism operad $\End_{\k s}$ associated to a degree $1$ dimension $1$ module $\k s$.

\begin{definition}[Curved $\Ai$-algebra]\label{def:CurvedAinfty}\index{$\Ai$-algebra!curved}
The data of a graded module $(A, \theta, \allowbreak d, \allowbreak m_2,\allowbreak m_3, \ldots)$ equipped with a degree $-2$ element $\theta$, a degree $-1$ map $d: A\to A$  and degree $n-2$ maps $m_n : A^{\otimes n}\to A$ satisfying the equations above is called a \emph{curved $\Ai$-algebra}. 
\end{definition}

If in a curved $\Ai$-algebra, the curvature vanishes, then the operator $d$ becomes a differential and the higher operations can be interpreted as homotopies for the signed relations, and we arrive at a classical notion of an $\Ai$-algebra \cite{Stasheff63}. This corresponds to replacing the operad $\As\cong\overline{\End}_\k$ by the ns suboperad $\overline{\End}_{\k s}\subset \End_{\k s}$~.

\begin{definition}[$\Ai$-algebra]\index{$\Ai$-algebra}
A chain complex $(A, d, m_2, m_3, \ldots)$ equipped with maps $m_n : A^{\otimes n}\to A$ degree $n-2$, for $n\ge 2$, satisfying the   equations above with $\theta=0$ is called an \emph{$\Ai$-algebra}. 
\end{definition}

To summarise, the Maurer--Cartan calculus allows us to encode all the four existing notions of $\Ai$-algebras using Maurer--Cartan elements of the convolution algebras associated the following linear dual cooperads: 
\[
\begin{array}{|l|c|c|}
\hline
& \text{curved} & \text{uncurved} \\
\hline
\rule{0pt}{11pt} \text{shifted} & {\End}_{\k} &\overline{\End}_{\k}\\
\hline
\rule{0pt}{11pt} \text{classical} &   {\End}_{\k s}& \overline{\End}_{\k s}\\
\hline
\end{array}
\]
The convolution pre-Lie algebra  $\Hom(\As^*, \End_A)\cong \prod_{n \ge 1}\lbrace\Hom(A^{\otimes n}, A)\rbrace$ encoding shifted $\Ai$-algebras is a pre-Lie subalgebra of the algebra \[\Hom(\uAs^*, \End_A)\cong \allowbreak \prod_{n \ge 0}\lbrace\Hom(A^{\otimes n}, A)\rbrace \] encoding shifted curved $\Ai$-algebras. But, as we will show in Section~\ref{subsec:TwistGrp}, it is crucial to encode the former notion in the latter bigger pre-Lie algebra, since there, the gauge group of symmetries is big enough to host the twisting procedure. We also remark that all these definitions, whether curved or uncurved, acquire a conceptual explanation via the Koszul duality theory~\cite{Positselski11, HirshMilles12}. 

\section{The twisting procedure as gauge group action}\label{subsec:TwistGrp}
In this section, we pass to the complete setting and we  apply the general theory developed in Chapter~\ref{sec:TopoDefTh} to the discrete (and therefore complete) ns cooperad  $\C=\End^c_{\k s^{-1}}\coloneqq\End_{\k s}^*$. This will allow us to  treat the 
deformation theory of curved $\Ai$-algebras. We chose this particular case, since the sign issue is a complicated problem in operad theory; the reader interested in the shifted case  has just to remove ``all'' the signs. 
The  cooperad $\calC$  is spanned by  one element 
$$\nu_n : \left(s^{-1}\right)^n\mapsto (-1)^{\frac{n(n-1)}{2}} s^{-1}$$ 
of degree $n-1$, in each arity $n\ge 0$. Its infinitesimal decomposition coproduct is given by 
$$\Delta_{(1)}\left(\nu_n\right)
=\sum_{\substack{p+q+r=n\\  p,q,r \ge 0}}(-1)^{p(q+1)} \nu_{p+1+r}\circ_{p+1} \nu_q
\ . $$

\begin{remark}
With the sign convention $\nu_n : \left(s^{-1}\right)^n\mapsto  s^{-1}$, we get the same signs as the ones of  \cite[Chapter~9]{LodayVallette12}. With the present convention, we actually get the signs which are more common in the existing literature. 
\end{remark}

To any complete graded module $A$, we  associate the complete graded left-unital convolution pre-Lie algebra 
$$\a_{\, \calC, A}\coloneqq
\big(
\hom_{\mathbb{N}}\big(\calC, {\eend}_A\big), \star, 1 
\big)\ ,$$
whose underling complete graded module is isomorphic to the product $$\prod_{n\in \NN}s^{ 1-n} \hom \big(A^{{\otimes} n}, A\big).$$
As a consequence of Section~\ref{sec:CurvedAiAlg} and Proposition~\ref{prop:MCOmegaC}, its Maurer--Cartan elements, that is degree $-1$ maps $\alpha : \calC \to {\eend}_A$ satisfying~$\alpha\star \alpha =0$,
are in one-to-one correspondence with 
complete curved $\Ai$-algebra structures on $A$, under the assignment $m_n\coloneqq\alpha(\mu_n)$, for $n\ge 0$.\\

The deformation gauge group associated to $\a$ is equal to 
\[\widetilde{\Gamma}\cong\left(
\F_1 A_{-1} \times 
\calF_1 \hom(A,A)_0\times 
{\displaystyle{\prod}_{n\ge 2}} \hom \big(A^{{\otimes} n}, A\big)_{n-1}
,  \BCH(\; ,\,), 0
\right)
\]
and is filtered isomorphic to 
\[
\widetilde{\mathfrak{G}}\cong 
\left(
\F_1 A_{-1} \times 
\left(1+ \calF_1 \hom(A,A)_0\right)
\times 
{\displaystyle{\prod}_{n\ge 2}} \hom \big(A^{{\otimes} n}, A\big)_{n-1}
,  \circledcirc, 1
\right)\]
under the pre-Lie exponential and pre-Lie logarithm maps, 
by Section~\ref{sec:ComOpDefTh}. In the complete left-unital  pre-Lie algebra $\a$, this deformation  gauge group acts on Maurer--Cartan elements via the following formula of  Theorem~\ref{thm:GaugeActionBis}: 
$$ e^\lambda\cdot\at=\left(e^\lambda \star \at\right) \circledcirc e^{-\lambda}  \ ,$$
as long as $\lambda_0$ and $\lambda_1$ live in the first layer of the filtration, that is $\lambda_0\in \F_1 A_{-1}$ and $\lambda_1\in \calF_1 \hom(A,A)_0$~.

\begin{remark}
The comprehensive deformation theory treatment in the monograph \cite{Markl12} by M.~Markl  treats  the case of free complete modules over a complete local ring.  A gauge group in this context corresponds to the group structure on $1+ \calF_1 \hom(A,A)_0$, see Chapter~$4$ of \textit{loc. cit.}. This is enough to describe the moduli spaces  of associative algebra structures up to isomorphism, in the complete setting.  The  deformation gauge group we use describes faithfully the moduli spaces of complete curved $\Ai$-algebras up to their $\infty$-isotopies.
\end{remark}

Let us now study the first and easiest example of a gauge action on Maurer--Cartan elements: 
we consider elements of the deformation gauge group $\widetilde{\Gamma}$ supported on arity $0$. Let $a\in \F_1 A_{-1}$, for brevity, we still denote by $a$ the element 
$(a, 0, \ldots)$
of $\widetilde{\Gamma}$ and by $1+a=e^a=(a, 1, 0, \ldots)$ the element of 
$\widetilde{\mathfrak{G}}$. 
By the general theory developed above, the action $e^a \cdot \alpha$ on the complete curved $\calA_\infty$-algebra structure encoded by the Maurer--Cartan element $\alpha$  gives us automatically a new complete curved $\calA_\infty$-algebra structure. 
We note that our strategy exhibits an important change of paradigm: first, we get a new complete curved $\Ai$-algebra structure by a conceptual argument (gauge group action) and then we make it explicit. Usually, in the literature like in \cite[Chapter~3]{FOOO09I}, the explicit form of the twisted operations  is given first and then proved (by direct computations) to satisfy the relations of a complete curved $\Ai$-algebra. 

\begin{proposition}\label{prop:TwCurvedGauge}
The formula for the generating operations $m_n^a$ of the complete curved $\calA_\infty$-algebra $e^a\cdot\alpha$ is
$$m_n^a=\sum_{r_0, \ldots, r_n\ge 0} (-1)^{\sum_{k=0}^n kr_k} m_{n+r_0+\cdots+r_n}\big(a^{r_0}, -, a^{r_1}, -, \ldots,  - , a^{r_{n-1}}, -,a^{r_n}  \big)\ , $$
for $n\ge 0$, where the notation $a^r$ stands  for $a^{\otimes r}$~. 
\end{proposition}

\begin{proof}
In the present case, the inverse of the pre-Lie exponential $e^a$ is equal to~$e^{-a}=1-a$. Therefore, the formula of the gauge action given in Theorem~\ref{thm:GaugeActionBis} is 
\begin{eqnarray*}e^{-a}\cdot\at=\left(e^{-a} \star \at\right) \circledcirc e^{a} =
\big((1-a) \star \alpha \big) \circledcirc (1+a)=\alpha\circledcirc (1+a)\ ,
\end{eqnarray*}
since $\lambda\star\rho=0$ for any $\rho\in \a$ and since $\star$ is linear on the left-hand side.
The image of the element $\nu_n$ under the $\alpha\circledcirc (1+a)$ is equal to 
$\big(\alpha\circ (1+a)\big)\big(\Delta(\nu_n)\big)$. One can easily see that the part of the image of the element $\nu_n$ under the decomposition map $\Delta$ of the cooperad $\calC$ with only $\nu_0$ and $\nu_1$ on the right-hand side is equal to 
$$
\sum_{r_0, \ldots, r_n\ge 0} (-1)^{\sum_{k=0}^n kr_k} \nu_{n+r_0+\cdots+r_n}\circ \big(\nu_0^{r_0}, \nu_1, \nu_0^{r_1}, \nu_1, \ldots,  \nu_1 , \nu_0^{r_{n-1}}, \nu_1,\nu_0^{r_n}  \big)
$$
Finally, the sign appearing in the formula for the $m_n^a$ is the same since the element $1+\lambda$ has degree~$0$. 
\end{proof}

Explicitly, the first of these twisted operations are:
\begin{eqnarray*}
&\text{arity} \ 0\ :& \theta^a\coloneqq m_0^a=\theta+d(a)+m_2(a,a)+m_3(a,a,a)+\cdots\ , \\
&\text{arity} \ 1\ :& d^a\coloneqq m_1^a=d(-)+m_2(a,-)-m_2(-,a)+m_3(a,a,-)-m_3(a,-,a)\\ && \quad \quad\ \   +m_3(-,a,a,)+\cdots\ , \\
&\text{arity} \ 2\ :& m_2^a=m_2(-,-)+m_3(a,-,-)-m_3(-,a,-)+m_3(-,-,a)+\cdots \ .
\end{eqnarray*}

We immediately recover the following result going back to \cite{FOOO09I}.

\begin{theorem}\label{thm:TwProcGp}
\index{$\Ai$-algebra!twisted complete curved}
Under the formula of Proposition~\ref{prop:TwCurvedGauge}, any element $a\in \F_1 A_{-1}$ of 
a complete curved $\calA_\infty$-algebra $(A, \theta, d, m_2, m_3, \ldots)$ induces a (twisted) complete
curved $\calA_\infty$-algebra 
 \[
(A, \theta^a, d^a, m_2^a, m_3^a,\ldots)
 \]
This twisted complete curved $\calA_\infty$-algebra has a trivial curvature $\theta^a=0$, that is produces an  $\calA_\infty$-algebra, if and only if the element $a$ satisfies the Maurer--Cartan equation: 
\begin{eqnarray}\label{eqn:MC}
\theta+da +\sum_{n\ge 2}m_n(a, \ldots, a)=0\ .
\end{eqnarray}\index{Maurer--Cartan element!complete curved $\calA_\infty$-algebra}
\end{theorem}

This conceptually explains why the twisting procedure on associative algebras or complete $\calA_\infty$-algebras requires the twisting element to satisfy the Maurer--Cartan equation. Without this condition, one would   \emph{a priori}  get a complete curved $\calA_\infty$-algebra. In other words, the (left-hand side of the) Maurer--Cartan equation of an element $a$ is the curvature of the complete curved $\calA_\infty$-algebra twisted by the element $a$. 

\begin{proposition}\label{lem:subGaugeGroup}
The following assignment defines a monomorphism of gr\-ou\-ps 
\begin{eqnarray*}
\left(\F_1 A_{-1}, +, 0\right) &\rightarrowtail& \widetilde{\Gamma}=\big(
\F_1 A_{-1} \times 
\calF_1 \hom(A,A)_0\times 
{\displaystyle{\prod}_{n\ge 2}} \hom \big(A^{{\otimes} n}, A\big)_{n-1}
,
\\ && \qquad\  \mathrm{BCH}(-,-), 0
\big)\\
a &\mapsto& (a, 0, 0, \ldots)\ .
\end{eqnarray*}

\end{proposition}

\begin{proof}
It is enough to check the compatibility  with respect to the group products, that is 
$$\mathrm{BCH}(a, b)=a+b\ ,$$
which holds true since $a$ and $b$ are supported on arity $0$: their brackets appearing in the Baker--Campbell--Hausdorff formula vanish. 
\end{proof}

\begin{corollary}\label{cor:twa+b}
Twisting a complete curved $\calA_\infty$-algebra $(A, \theta, d, m_2, m_3, \ldots)$ first by an element $a$ and then by an element $b$ amounts to twisting it by $a+b$~. 
\end{corollary}

\begin{proof}
This is a direct corollary of Proposition~\ref{lem:subGaugeGroup}. 
\end{proof}

The presence of the curvature element plays another peculiar role in the deformation theory of curved $\Ai$-algebras, in a sense trivialising their homotopy properties, as we shall now see. This behaviour is sometimes called the \emph{Kontsevich--Positselski vanishing phenomenon}\index{Kontsevich--Positselski vanishing phenomenon} in the literature, similar results hold when computing the Hochschild (co)homology in the curved case, see~\cite{MR3084707}. Note that this trivialisation always occurs over a field, but is typically unavailable in many interesting cases like the curved $\Ai$-algebras used when working with matrix factorisations, as argued in~\cite{ArmstrongClarke15}. 

\begin{proposition}[{\cite[Section~7.3]{Positselski11}\cite[Theorem~5.4]{ArmstrongClarke15}}]\label{prop:KPphenom}
Any  curved $\calA_\infty$-al\-ge\-bra $(A, \theta, d, m_2, m_3, \ldots)$, for which there exists a linear map 
$A \to \k$ sending $\theta$ to $1$,
is gauge equivalent to its  truncated  curved $\calA_\infty$-algebra $(A, \theta, 0, 0, \ldots, )$~.
\end{proposition}

\begin{proof}
For the sake of completeness, we provide here a proof in the language of the pre-Lie deformation theory. Notice that the computations of the following obstruction argument are very close to the arguments of \cite[Section~$7.3$]{Positselski11} and \cite[Theorem~5.4]{ArmstrongClarke15}. As above, we denote the Maurer--Cartan element encoding the curved $\calA_\infty$-algebra $(A, \theta, d, m_2, m_3, \ldots)$ in the convolution pre-Lie algebra $\a_{\, \calC, A}$ by $\alpha$, that is $m_n=\alpha(\nu_n)$, for $n\ge 0$~. 
Let us denote by $\alpha_{n}$ the element $$(0, \ldots, 0, \nu_{n}\mapsto m_n, 0, \ldots)_{n\ge 0}~,$$ so that $\alpha=\sum_{n=0}^\infty \alpha_{n}$~. 
Let us show that there exists an element $\lambda\in \widetilde{\Gamma}$,  such that $\lambda_0=0$ and $\lambda_1=0$, and which satisfies $(1+\lambda).\alpha=\alpha_0$. 
Since 
 \[
(1+\lambda).\alpha=\left((1+\lambda) \star \alpha \right) \circledcirc (1+\lambda)^{-1}=\alpha_0~,
 \] 
we have to demonstrate that $$(1+\lambda) \star \alpha =\alpha_0\circledcirc (1+\lambda)=\alpha_0~,$$ that is 
$(1+\lambda_2+\lambda_3+\cdots)\star (\alpha_0+\alpha_1+\cdots)=\alpha_0$~. This gives in arity $0$: $\alpha_0=\alpha_0$ and in arity $n\ge 1$
\begin{eqnarray}\label{Eqn:OBs}
\alpha_n+\lambda_2\star \alpha_{n-1}+\cdots +\lambda_n\star \alpha_1+\lambda_{n+1}\star \alpha_0=0\ .      
\end{eqnarray}
Let us denote by $\delta_n$ the element $\alpha_n+\lambda_2\star \alpha_{n-1}+\cdots +\lambda_n\star \alpha_1$; our goal is to prove that $\delta_n$ belongs to the image of the operator $-\star\alpha_0$~. To do that, we equip the convolution algebra $\a_{\, \calC, A}$ with the map $\delta$ defined by $$\delta(f)\coloneqq f\star\alpha_0~.$$ Note that $\delta^2(f)=(f\star\alpha_0)\star \alpha_0$, and since $\at_0$ is a degree $-1$ element, the pre-Lie identity implies that 
 \[
(f\star\alpha_0)\star \alpha_0-f\star(\alpha_0\star \alpha_0)=0,
 \] 
 which in turn implies $(f\star\alpha_0)\star \alpha_0=0$. Thus, $\delta^2(f)=0$, and the convolution algebra becomes a chain complex.

We shall first show by induction that the element $\delta_n\in C(n)$ is a cycle. For that, we write\begin{eqnarray*}
\delta_n\star \alpha_0&=&
\at_n \star \at_0+\sum_{k=2}^n(\lambda_k\star \at_{n-k+1})\star\at_0\\
&=&
\at_n \star \at_0-\sum_{k=2}^n(\lambda_k\star \at_0)\star\at_{n-k+1}+\sum_{k=2}^n\lambda_k\star (\at_{n-k+1}\star\at_0)
\end{eqnarray*}
which is equal to the following quantity using the induction hypothesis:
\begin{eqnarray*}
\delta_n\star \alpha_0&=&
\at_n \star \at_0
+\sum_{k=2}^n\at_{k-1}\star\at_{n-k+1}
+\sum_{k=3}^n\sum_{l=2}^{k-1}(\lambda_l\star \at_{k-l})\star\at_{n-k+1}
\\&&+\sum_{k=2}^n\lambda_k\star (\at_{n-k+1}\star\at_0)\\
&=&\sum_{k=1}^{n}\at_{k}\star\at_{n-k}+\sum_{l=2}^n
\sum_{k=l+1}^{n+1}\lambda_l\star(\at_{k-l}\star\at_{n-k+1})=0\ ,
\end{eqnarray*}
by the Maurer--Cartan equation satisfied by $\at$.

To complete to proof, it is now enough to show that our chain complex is acyclic. Let us denote by $\theta^*\colon A \to \k$ the $\k$-linear map which sends  $\theta$ to $1$. With this map at hand, we construct the following contracting homotopy $h$. Let $\varphi_n\in \Hom\left(\calC(n),
\Hom(A^{\otimes n}, A)\right)$ be an element of the convolution algebra~$\a_{\, \calC, A}$; it is completely defined by its value $\phi_n(\nu_n)=: f_n$. The image $h(\varphi_n)$ under the map $h$ of the element $\varphi_n$ is an element $\psi_{n+1}\in \Hom\left(\calC(n),
\Hom(A^{\otimes n}, A)\right)$ given by 
$$\psi_{n+1}(\nu_{n+1})=\theta^*\otimes f_n\in\Hom(A,\k)\otimes\Hom(A^{\otimes n}, A)\cong \Hom(A^{\otimes (n+1)}, A)~.$$  It is then straightforward to check that $hd+dh=\id_{\a_{\, \calC, A}}$.

It follows that $\delta_n$ belongs to the image of the differential, thus there exists $\lambda_{n+1}\in C(n+1)$ such that $-\lambda_{n+1}\star \alpha_0=\delta_n$, which is Equation~(\ref{Eqn:OBs}). 
\end{proof}

The proof we just presented is, in a sense, ``dual'' to the twisting procedure. Indeed, according to Theorem \ref{thm:TwProcGp}, the action of the arity $0$ elements of the gauge group gives us the twisting procedure; this was used in \cite{FOOO09I} to get a new structure with trivial curvature and thus an underlying chain complex. On the other hand, Proposition \ref{prop:KPphenom} shows that one can act with the elements of the gauge group supported in arity $\ge 2$ in order to trivialise all the operations except for the curvature. Interestingly enough, in this latter case the arity condition ensures convergence, so we do not need the completeness assumption, and thus the Kontsevich--Positselski vanishing phenomenon always occurs over a field. 

\section{Curved \texorpdfstring{$\Li$}{Li}-algebras}
In the case of symmetric operads, we can start with the same kind of simple operad, that is one-dimensional in each arity with trivial symmetric group action. This operad $\mathrm{uCom}=\End_{\k}$ encodes the category of unital commutative (associative) algebras. Its linear dual $\mathrm{uCom}^*$ produces the following complete left unital convolution pre-Lie algebra 
\[\a=\hom_\Sy( \mathrm{uCom}^*, \End_A)\cong \left(\prod_{n\ge 0} \hom\big(A^{\odot n}, A\big),  \star, 1\right) \ ,\]
where $A^{\odot n}\coloneqq A^{\otimes n}/\Sy_n$ stands for the space of coinvariants with respect to the symmetric group action; in order words, we are led to consider symmetric maps from $A^{\otimes n}$ to $A$. 
This way, we obtain the definition of a complete \emph{shifted} (possibly curved) $\Li$-algebra, notion which appeared first in \cite{ZWIEBACH199333}. In the same way as for $\Ai$-algebras, the truncation and suspension versions of the endormorphism cooperad on a one-dimensional module give rise to possibly non-shifted and possibly non-curved $\Li$-algebra structures. 
 As strange as it may seem, the notion of \emph{shifted} Lie algebra structure seems to appear ``more naturally'' than its classical notion: the shifted Schouten--Nijenhuis bracket on polyvector fields \cite{MR3324,MR0063750,MR0074879}, the shifted Lie algebra \cite{MR68218,MR0091473} corresponding to the Whitehead product on homotopy groups \cite{MR4123,MR0016672}, the shifted Gerstenhaber's Lie bracket on the Hochschild cochain complex \cite{Gerstenhaber63}, the Nijenhuis--Richardson bracket arising in deformation theory of Lie algebras \cite{MR178041,MR195995}, the shifted $\Li$-algebra formed by the Koszul hierarchy \cite{DotsenkoShadrinVallette11},  the ``antibracket'' of the Batalin--Vilkovisky formalism~\cite{MR616572} etc. However, if one uses homological grading, most ``natural'' Lie brackets are of degree $1$, while the below definition suggests to study the bracket of degree $-1$. This choice is motivated by aesthetic reasons and does not affect any sign in the arising formulas; we invite the reader intended to apply the formalism of this book to concrete problems to consult the introduction of Section~\ref{sec:SymTw} for another discussion of this matter.

\begin{definition}[Complete shifted curved $\Li$-algebra]\label{def:csLiAlg}
\index{$\Li$-algebra!complete shifted curved}
A \emph{complete shift\-ed curved $\Li$-algebra structure} on a complete graded module $(A, \F)$ is a Maurer--Cartan element 
$\alpha=(\ell_0, \ell_1, \ldots, \ell_n,\ldots)$ in the complete left unital pre-Lie convolution algebra $\a\cong \left(\prod_{n\ge 0} \hom\big(A^{\odot n}, A\big),  \star, 1\right)$. Such a data amounts to a collection of filtered maps 
$l_n\colon A^{\odot n} \to A$, 
of degree $-1$, for $n\ge 0$, satisfying the following relations, under the usual convention $\theta\coloneqq\ell_0(1)$ and $d\coloneqq\ell_1$:
\begin{eqnarray*}
&\text{arity}\ 0:& d\theta=0\ , \\
&\text{arity}\  1:& d^2=- \ell_2(\theta, -)\ , \\
&\text{arity}\  2:& \partial \ell_2=
-\ell_3(\theta, -,-)\ . \\
&\text{arity}\  3:& \partial \ell_3=-\ell_2(\ell_2(-,-), -)-\ell_2(\ell_2(-,-), -)^{(23)}-\ell_2(\ell_2(-,-), -))^{(132)}
\\&&
\qquad \quad -\ell_4(\theta, -, -,-)\ ,
\\
&\text{arity}\ n:& \partial\ell_n= -\sum_{\substack{p+q=n+1\\ 2\leq p,q \leq n-1}}
\sum_{\sigma\in \mathrm{Sh}_{p,q}^{-1}}
 (\ell_{p+1}\circ_{1} \ell_q)^{\sigma}
-\ell_{n+1}(\theta, -, \cdots, -)\ ,
\end{eqnarray*}
where $ \mathrm{Sh}_{p,q}^{-1}$ denotes the set of the inverses of $(p,q)$-shuffles.
\end{definition}

Different versions of the ns operad of $\Ai$-algebras admit symmetric versions where one considers the symmetric operad $\mathrm{uAss}$ whose components are the regular representations of the symmetric groups, that is $\mathrm{uAss}(n)=\upsilon_n\k[\Sy_n]$. Then the associated complete convolution pre-Lie algebra is isomorphic to the one in the ns case since 
$\hom_\Sy( \mathrm{uAss}^*, \End_A)\cong \hom( \mathrm{uAs}^*, \End_A)$. 
Let us denote by $\upsilon_n'$ and by $\nu_n'$ respectively the basis elements of $\mathrm{uCom}(n)$ and $\mathrm{uCom}^*(n)$. 
The morphism of operads 
$$ \mathrm{uAss}\to \mathrm{uCom} \ ,\qquad \upsilon_n^\sigma\mapsto \upsilon'_n$$
induces the following morphism of cooperads 
$$ \varsigma\colon \mathrm{uCom}^*\to \mathrm{uAss}^* \ ,\qquad \nu_n' \mapsto \sum_{\sigma \in \Sy_n}\nu_n^\sigma\ ,$$
which,  by pulling back, induces the following morphism of pre-Lie algebras 
$$ \varsigma^*\colon \hom_\Sy( \mathrm{uAss}^*, \End_A) \to \hom_\Sy( \mathrm{uCom}^*, \End_A) \ , \quad \alpha \mapsto \alpha\circ\varsigma\ . 
$$
In the case of the endormophism operad on a suspended module, which gives rise to the non-shifted versions, the morphism of operads is defined similarly, but its linear dual produces the signs $\mathrm{sgn}(\sigma)$ . \\

Since a morphism of pre-Lie algebras preserves Maurer--Cartan elements, we obtain a conceptual explanation of the following known result. 

\begin{proposition}[\cite{LadaMarkl95, FOOO09I}]\label{prop:SymmAiLi}
The symmetrisation $$\ell_n\coloneqq\sum_{\sigma\in \Sy_n} \, m_n^\sigma $$
of a complete shifted (curved) $\Ai$-algebra produces a complete shifted (curved) $\Li$-algebra. 

The antisymmetrisation $$\ell_n\coloneqq\sum_{\sigma\in \Sy_n}\mathrm{sgn}(\sigma) \, m_n^\sigma $$
of a complete (curved) $\Ai$-algebra produces a complete (curved) $\Li$-algebra. 
\end{proposition}

\section{Gauge action and twisting of curved \texorpdfstring{$\Li$}{Li}-algebras}\label{subsec:TwLii}
The above notions of complete curved $\Li$-algebra, both shifted and not, play seminal roles in deformation theory, rational homotopy theory, and higher algebra. For instance, this leads to a suitable source category for the Deligne--Hinich $\infty$-groupoid and the Getzler $\infty$-groupoid \cite{Getzler09, Henriques08, DolgushevRogers15}, rational models for spaces of maps \cite{BuijsMurillo11, Lazarev13, Berglund15}, and provides us with a suitable higher categorical enrichment to the categories of homotopy algebras \cite{DolgushevRogers17, DolgushevHoffnungRogers14}.  In each case, the twisting procedure, together with its various properties, constitutes the main toolbox. In this section, we show that these properties are actually straightforward consequences of the  above gauge group interpretation. This also allows us to get the most general  version of all these results. \\

Let $(A, \F)$ be a complete graded module and let $\calC$ be the coooperad $\coloneqq \mathrm{uCom}^*$. The deformation gauge group associated to $\a$ is equal to 
\[\widetilde{\Gamma}\cong\left(
\F_1 A_{0} \times 
\calF_1 \hom(A,A)_0\times 
{\displaystyle{\prod}_{n\ge 2}} \hom \big(A^{{\odot} n}, A\big)_0
,  \BCH(\; ,\,), 0
\right)
\]
and is isomorphic to 
\[
\widetilde{\mathfrak{G}}\cong 
\left(
\F_1 A_{0} \times 
\left(1+ \calF_1 \hom(A,A)_0\right)
\times 
{\displaystyle{\prod}_{n\ge 2}} \hom \big(A^{{\odot} n}, A\big)_0
,  \circledcirc, 1
\right)\]
under the pre-Lie exponential and pre-Lie logarithm maps.

\begin{proposition}\label{prop:TwCurvedGaugeLie}

The gauge action of an element $a\in \F_1 A_0$ on a complete shifted curved $\Li$-algebra $(A,\F, \theta=\ell_0, d=\ell_1, \ell_2, \ell_3, \ldots)$ produces the following \emph{twisted} shifted curved $\Li$-algebra structure on $A$:
$$\ell_n^a=\sum_{k\ge 0} {\textstyle \frac{1}{k!}} \ell_{k+n}\big(a^k, -,  \ldots,  -   \big)\ , \ \  \text{for}\  n\ge 0\ . $$
\end{proposition}

\begin{proof}
The proof is easy and identical to the one of Proposition~\ref{prop:TwCurvedGauge}: the new Maurer--Cartan element is equal to $e^{-a}\cdot \alpha=\alpha \circledcirc (1+a)$.
\end{proof}

For instance, the formula for the twisted curvature and the twisted (pre)dif\-fe\-ren\-tial are respectively 
\[\theta^a=\sum_{k\ge 0} {\textstyle \frac{1}{k!}} \ell_{k}\big(a^k \big) \qquad \text{and} \qquad
d^a=\sum_{k\ge 0} {\textstyle \frac{1}{k!}} \ell_{k+1}\big(a^k, - \big)\ .
 \]
An element $a\in \F_1 A_{0}$ is called a \emph{Maurer--Cartan} element 
\index{Maurer--Cartan element!complete shifted curved $\Li$-algebra}
in the complete shifted curved $\Li$-algebra $(A, \F, \theta, d, \ell_2, \ell_3, \cdots)$ when $\ell^a_0=0$. The same results as in the above case of complete (shifted) curved $\Ai$-algebras hold here \textit{mutatis mutandis}. Let us mention the following ones, which are heavily used in \textit{loc. cit.}

\begin{corollary}[\cite{DolgushevRogers15, DolgushevRogers17, Getzler09}]
Let $(A,\F, \theta, d, \ell_2, \ell_3, \ldots)$  be a complete shifted curved $\Li$-algebra and let 
$a, b\in \F_1 A_0$. The following formulas hold: 
\begin{enumerate}\itemsep3pt
\item $d \theta^a+\sum_{k\ge 1}{\textstyle \frac{1}{k!}} \ell_{k+1}\big(a^k, \theta^a \big)=0$\ , 
\item $d^a \circ d^a=-\ell_2^a(\theta^a,-)$\ ,
\item $\theta^{a+b}=\theta^a+d^a(b)+ \sum_{k\ge 2} {\textstyle \frac{1}{k!}} \ell_{k}^a\big(b^k)$ , 
\end{enumerate}
\end{corollary}

\begin{proof}\leavevmode
\begin{enumerate}
\item The left-hand side is nothing but $d^a\left(\theta^a\right)$, which vanishes because of the arity zero relation of the twisted shifted curved $\Li$-algebra.
\item This formula is the  arity $1$ relation of the twisted shifted curved $\Li$-algebra.
\item The right-hand side is nothing but $\left(\theta^a\right)^b$; so the formula is the arity $0$ part of  Corollary~\ref{cor:twa+b}, in the $\Li$-algebra case.
\end{enumerate}
\end{proof}

Let $(A,\F)$ and $(B, \G)$ be a two complete dg modules. Let $\alpha, \beta \in \MC(\a_{\calC, A\oplus B})$ be two Maurer--Cartan elements  corresponding to two complete shifted curved $\Li$-algebra structures on $(A, \F)$ and $(B, \G)$ respectively. Let $f=(f_0, f_1, \cdots)$ be an $\infty$-morphism from $(A,\alpha)$ to $(B, \beta)$, that is $f\star \alpha=\beta \circledcirc f$. Let $a\in \F_1A_0$ and let us denote by $f^a$ the element $f\circledcirc(1+a)$, that is the sequence
 \[ 
\left(\sum_{k\ge 0} {\textstyle \frac{1}{k!}} f_k(a^k), 
\sum_{k\ge 0} {\textstyle \frac{1}{k!}} f_{k+1}\big(a^k, - \big), \cdots , \sum_{k\ge 0} {\textstyle \frac{1}{k!}} f_{k+n}\big(a^k, -,  \ldots,  - \big),   
 \cdots\big)\right)\ . \]
Notice that $f(a):=\sum_{k\ge 0} {\textstyle \frac{1}{k!}} f_k(a^k)\in \F_1 B_0$. We  consider the  following two twisted complete shifted curved $\Li$-algebras on $A$ and $B$ respectively:
\[
\alpha^a\coloneqq(1-a)\cdot \alpha=\alpha\cc(1+a) \quad \text{and}\quad 
\beta^{f(a)}\coloneqq\big(1-{f(a)}\big)\cdot \beta=\beta\cc\big(1+{f(a)}\big)\ .
\]

The following proposition is usually formulated in the $\Li$-case \cite{Getzler09, DolgushevRogers15, DolgushevRogers17}; our very short proof shows that is holds in the curved case as well. 

\begin{proposition}\label{prop:PropertyCurvInfMorph}\leavevmode
\begin{enumerate}
\item The element 
\[ \big(1-{f(a)}\big) \cc f \cc (1+{a})=f^a-{f(a)}\]
is an $\infty$-morphism from $\alpha^a$ to $\beta^{f(a)}$. 

\item The curvatures of the two twisted complete shifted curved $\Li$-algebra structures are related by the following formula:
\[
\beta^{f(a)}_0=f_1^a\big(\alpha^a_0\big)=\sum_{k\ge 0} {\textstyle \frac{1}{k!}} f_{k+1}\big(a^k, \alpha^a_0 \big)
 \ .\] 

\item If the element $a$ is a Maurer--Cartan element in the complete shifted curved $\Li$-algebra $\alpha$, then so is its ``image'' $f(a)$ in the complete shifted curved $\Li$-algebra $\beta$. In this case, $f^a-f(a)=\big(0, f_1^a, f_2^a, \ldots\big)$ is a (non-curved) $\infty$-morphism between the two complete shifted $\Li$-algebras $\alpha^a$ and $\beta^{f(a)}$ respectively. 
\end{enumerate}
\end{proposition}

\begin{proof}\leavevmode
\begin{enumerate}
\item The first assertion amounts to proving 
\begin{eqnarray*}
\left(\big(1-{f(a)}\big) \cc f \cc (1+{a})\right)\star \alpha^a=\beta^{f(a)} \cc \big(1-{f(a)}\big) \cc f \cc (1+{a})\ . 
\end{eqnarray*}
The left-hand term is equal to 
\begin{multline*}
\left(\big(1-{f(a)}\big) \cc f \cc (1+{a})\right)\star \alpha^a\\ 
= f(a)+ \left(f \cc (1+{a})\right)\star \big( \alpha\cc(1+a)\big)=  f(a)+ (f\star \alpha)\cc  (1+{a})\ .
\end{multline*}
The right-hand term is equal to 
\begin{multline*}
\beta^{f(a)} \cc \big(1-{f(a)}\big) \cc f \cc (1+{a})\\
= \big(1-f(a)\big)\cc\beta \cc \big(1+{f(a)}\big) \cc \big(1-{f(a)}\big) \cc f \cc (1+{a})\\
=f(a)+\beta  \cc f \cc (1+{a})=f(a)+ (f\star \alpha) \cc (1+a)
\ . 
\end{multline*}

\item The second assertion is the part of the above relation for the $\infty$-morphism in arity $0$. This latter one is equal to $f^a-f(a)=\big(0, f^a_1, f^a_2, \ldots\big)$ and so it has no constant term. The part of arity $0$ of the equation $\beta^{f(a)}\cc \big(f^a-f(a)\big)=\big(f^a-f(a)\big)\star \alpha^a$ is 
\[
\beta^{f(a)}_0=f_1^a\big(\alpha^a_0\big)=\sum_{k\ge 0} {\textstyle \frac{1}{k!}} f_{k+1}\big(a^k, \alpha^a_0 \big)
 \ .\] 

\item The last assertion is a direct corollary of the previous one: if $\alpha^a_0=0$, then so is $\beta^{f(a)}=0$\ . 
\end{enumerate}
\end{proof}

Let us continue with $\infty$-morphisms between two complete shifted curved $\Li$-algebras $(A,\F, \alpha)$ and $(B,\G,  \beta)$. Such a map is a collection $(f_0, f_1, f_2, \ldots)$, where $b:=f_0\in \G_1 B_0$. Let us denote it by $b+f$, where $f\coloneqq(0, f_1, f_2, \ldots)$.

\begin{lemma}\label{lem:EquivMorph}
Under this convention, with the constant term split apart, a data $b+f$ is an $\infty$-morphism from $\alpha$ to $\beta$ if and only if the data $f$ is an $\infty$-morphism from $\alpha$ to the twisted structure $\beta^b$:
\[ 
b+f\colon \alpha \rightsquigarrow \beta \quad \Longleftrightarrow \quad f\colon  \alpha \rightsquigarrow \beta^b\ .
\]
\end{lemma}

\begin{proof}
The data $b+f$ in an $\infty$-morphism from $\alpha$ to $\beta$ if and only if it satisfies 
\[(b+f)\star \alpha= f\star \alpha= \beta\cc (b+f) \ .\]
The data $f$ in an $\infty$-morphism from $\alpha$ to $\beta^b$ if and only if it satisfies 
\[ 
f\star \alpha= \beta^b\cc f=\beta \cc (1+b)\cc f=\beta \cc (b+f)
\ ,\]
which concludes the proof.
\end{proof}

V.~Dolgushev and C.~Rogers introduced in \cite{DolgushevRogers15} a category whose objects are complete shifted $\Li$-algebras and whose morphisms from $(A,\F, \alpha)$ to $(B,\G, \beta)$  amount to the data of a Maurer--Cartan element $b\in \G_1 B_0$ and an $\infty$-morphism $f : \alpha \rightsquigarrow \beta^b$ without constant term. Let $g : \beta \rightsquigarrow \gamma^c$ be another such morphism; 
they define the composite of morphisms by the formula:
\[
\big(g^b-g(b)\big)\cc f \colon \alpha \rightsquigarrow \gamma^{c+g(b)}
\ . \]
This  category is denoted  by $\mathfrak{S}\mathsf{Lie}_\infty^{\text{MC}}$ in \emph{loc. cit.}

\begin{proposition}
The category $\mathfrak{S}\mathsf{Lie}_\infty^{\mathrm{MC}}$ is the sub-category of the category of complete shifted curved $\Li$-algebras with $\infty$-morphisms whose objects are 
complete shifted $\Li$-algebras and whose morphisms are $\infty$-morphisms such that the constant term is a Maurer--Cartan element in the target algebra.
\end{proposition}

\begin{proof}
Lemma~\ref{lem:EquivMorph} establishes the equivalence between the two notions of morphisms. 
Under the above convention, the formula for the composition of two $\infty$-morphisms $c+g$ and $b+f$ is given by 
 \[
 (c+g)\cc (b+f)=\big(c+g(b)\big) + \big(g^b- g(b)\big) \cc f \ ,
\]
which coincides with the Dolgushev--Rogers definition.
\end{proof}

Besides giving a conceptual explanation for the Dolgushev--Rogers category, this result also allows us to prove easily the various properties of the composite of morphisms, like the associativity for instance. Notice that this category was used in a crucial way in \cite{DolgushevHoffnungRogers14} to provide an $\infty$-categorical enrichment of the category of homotopy algebras which encodes faithfully their higher homotopy theory. 

\section{Twistable types of algebras}\label{subsec:TwistableAlg}
The purpose of this section is to describe on the level of the encoding operads which categories of algebras admit a meaningful twisting procedure. To be more precise, our aim is to characterise the (quadratic) operads $\calP$ for which any $\calP_\infty$-algebra can be twisted by any element satisfying a Maurer--Cartan type equation.This explains conceptually the particular form of the Maurer--Cartan equation. \\

Let $\calU\coloneqq(\k u, 0, \ldots)$ be the operad generated by an element $u$ of degree~$0$ and arity~$0$; it encodes the data of a degree $0$ elements in  graded modules.  Let $(E,R)$ be an operadic quadratic-linear data, that is $R\subset E\oplus \calT(E)^{(2)}$, and let $\chi : E(2) \to \k$ be an $\Sy_2$-equivariant linear map of degree $0$, where $\k$ receives the trivial $\Sy_2$-action. We consider the space of relations $R_\chi\subset \calT(E\oplus \k u)$ generated by
\[\mu \circ_1 u-\chi(\mu)\id\quad \text{and}\quad \mu \circ_2 u-\chi(\mu)\id\ ,\]
with $\mu\in E(2)$, and all the other composites of elements of $E(n)$ with at least one $u$, for $n\neq 2$.

\begin{definition}[Unital extension]\index{unital extension}
The \emph{unital extension} of $\calP\coloneqq\calP(E,R)$ by~$\chi$ is the following  operad 
\[u_\chi\calP\coloneqq
\calP(E\oplus \k u, R\oplus R_\chi)=
\frac{\calP(E,R) \vee \calU}{\left(
R_\chi
\right)}\ ,
\]
where $\vee$ stands for the coproduct of operads. 
\end{definition}

The category of $u_\chi\calP$-algebras is the category of $\calP$-algebras with a distinguished degree $0$ element which acts as a unit (with coefficients) for the generating operations of degree $0$ and arity $2$ and which vanishes once composed with any other generating operation. 
Notice that in the trivial case $\chi=0$, the unital extension amounts to  
$u_0 \calP\cong \k u \oplus \calP$
and that, in the general case,  
the underlying graded $\Sy$-module of $u_\chi \calP$ is a quotient of $\k u \oplus \calP$. The ``maximal'' case is covered by the following definition. 

\begin{definition}[Extendable quadratic-linear operad]\index{extendable quadratic-linear operad}
A quadratic-linear presentation of an operad $\calP=\calP(E,R)$ is called \emph{extendable} when there exists a non-trivial map $\chi : E(2) \to \k$ such that the canonical map 
$
\calP \hookrightarrow u_\chi\calP
$ 
is a monomorphism. 
\end{definition}

This happens if and only if the underlying graded $\Sy$-module of $u_\chi \calP$ is isomorphic to $\k u \oplus \calP$. In other words, an operad is extendable when it admits a (non-trivial) ``unitary extension'' in the terminology of \cite[Section~2.2]{Fresse17I}. As a consequence, an extendable  operad $\calP$ carries a richer structure of a $\Lambda$-operad,  crucial notion in the recognition of iterated loop spaces \cite{May72}, and its underlying $\Sy$-module carries an $\mathrm{FI}$-module structure, where $\mathrm{FI}$ stands for the category of finite sets and injections; this  notion plays a seminal role in representation theory \cite{CEF15}. Our definition may be viewed as a way to produce concrete unitary extensions of operads in the quadratic case. 

\medskip 

We shall now discuss extendability of some classical operads.

\begin{proposition}\label{prop:Extendable}\leavevmode
\begin{enumerate}
\item The quadratic operads $\Com$, $\Gerst$, $\HyperCom$, $\PreLie$, and the  qua\-dra\-tic-linear operad 
 $\BV$ are extendable. 
\item The quadratic ns operad $\As$ is extendable. 
\item The quadratic operads $\Lie$ and $\Perm$ are not extendable. 
\item The quadratic ns operads $\Dias$ and $\Dend$ are not extendable. 
\end{enumerate}
\end{proposition}

\begin{proof}\leavevmode
\begin{enumerate}
\item In the case of commutative associative algebras, the classical definition of the unit $a\cdot 1=1\cdot a =a$, that is $\chi(\cdot)=1$, works to create such an extension. 

In the cases of operads of Gerstenhaber algebras, Batalin--Vilkovisky algebras, and hypercommutative algebras, we note that those are homology operads of topological operads (of little $2$-disks, framed little $2$-disks, and Deligne--Mumford compactifications of genus zero curves with marked points respectively) that admit units topologically (action of the unit corresponds to forgetting one of the little disks, or one of the marked points), hence the unitality remains on the algebraic level. The action of the unit on the generators is forced by degree reasons: all generators except for the binary generator of the commutative suboperad must be annihilated by the unit.

Let us consider the operad of pre-Lie algebras. We shall now show that the assignment $\chi(\star)=1$, that is $1\star a=a\star 1=a$, leads to a unital extension of the maximal possible size. It suffices to show that the insertion of $u$ into any element of the operadic ideal generated by the pre-Lie identity and applying the defining relations of the unital extension that we are considering produces an element of the same ideal. Let us note that insertion of $u$ in any slot of the pre-Lie identity gives zero: 
\begin{gather*}
(1\star b)\star c-1\star(b\star c)-(1\star c)\star b+1\star(c\star b)=0\ ,\\
(a\star1)\star c-a\star(1\star c)-(a\star c)\star 1+a\star(c\star 1)=0\ ,\\
(a\star b)\star 1-a\star(b\star 1)-(a\star 1)\star b+a\star(1\star b)=0\ .
\end{gather*}
Every element of the operadic ideal generated by the pre-Lie identity is a combination of elements which are obtained from the pre-Lie identity by pre- and post-compositions. Consider one such element $\nu$, and look at $\nu\circ_i u$. If the argument $i$ of $\nu$ is one of the arguments of the pre-Lie identity, then the above computation shows that $\nu\circ_i u=0$. Otherwise, the element $\alpha\circ_i u$ is still obtained from the pre-Lie identity by pre- and post-compositions, proving our claim.

\item The proof is the same as in the case  of the operad $\Com$. 

\item In the case of Lie algebras, there is no nontrivial equivariant map $\chi$ from the space of binary operations to the ground field regarded as the trivial module.

In the case of permutative algebras, we can substitute $x_1=u$ in the structural identity 
\[x_1\cdot (x_2\cdot x_3)=x_1\cdot (x_3\cdot x_2)~,\] 
and note that it becomes 
\[\chi(\cdot) x_2\cdot x_3=\chi(\cdot) x_3\cdot x_2~,\] 
so if the canonical map is a monomorphism, we must have $\chi(\cdot)=0$, and the extension is trivial.

\item Recall the defining relations of the ns operad of dendriform algebras:
\begin{gather*}
(x_1\prec x_2)\prec x_3 = x_1\prec (x_2\prec x_3+x_2\succ x_3)\ ,\label{eq:Dend1}\\
(x_1\succ x_2)\prec x_3 = x_1\succ (x_2\prec x_3)\ , \label{eq:Dend2}\\
(x_1\succ x_2+x_1\prec x_2)\succ x_3 = x_1\succ (x_2\succ x_3)\ .\label{eq:Dend3}
\end{gather*}
Suppose that we consider the unital extension corresponding to the linear function $\chi$. Substituting $x_1=u$ in the first dendriform axiom and then setting $x_3=u$ in the third one, we get
\begin{gather*}
\chi(\prec) x_2\prec x_3 = \chi(\prec) (x_2\prec x_3+x_2\succ x_3)~,\\
\chi(\succ) (x_1\succ x_2+x_1\prec x_2)=\chi(\succ) (x_1\succ x_2)~,
\end{gather*}
so if the canonical map is a monomorphism, we must have $\chi(\prec)=\chi(\succ)=0$, and the extension is trivial. 

Recall the defining relations of the ns operad of diassociative algebras: 
\begin{gather*}
(x_1\vdash x_2) \vdash x_3 = x_1\vdash (x_2\vdash x_3)\ ,\label{eq:Dias1}\\
(x_1\vdash x_2) \vdash x_3 = (x_1\dashv x_2)\vdash x_3\ ,\label{eq:Dias2}\\
(x_1\vdash x_2) \dashv x_3 = x_1\vdash (x_2\dashv x_3)\ ,\label{eq:Dias3}\\
x_1\dashv (x_2 \vdash x_3) = x_1\dashv (x_2\dashv x_3)\ ,\label{eq:Dias4}\\
(x_1\dashv x_2) \dashv x_3 = x_1\dashv (x_2\dashv x_3)\ .\label{eq:Dias5}
\end{gather*}
Suppose that we consider the unital extension corresponding to the linear function $\chi$. Substituting $x_3=u$ in the second axiom of diassociative algebras and then setting $x_1=u$ in the fourth one, we get
\begin{gather*}
\chi(\vdash) x_1\vdash x_2 = \chi(\vdash) x_1\dashv x_2~,\\
\chi(\dashv) x_2\vdash x_3 = \chi(\dashv) x_2\dashv x_3~,
\end{gather*}
so if the canonical map is a monomorphism, we must have $\chi(\vdash)=\chi(\dashv)=0$, and the extension is trivial.
\end{enumerate}
\end{proof}

Let us now work over a field $\k$ of characteristic $0$ and suppose that the component $E(n)$ is finite dimensional for any $n\ge 1$ and that $E(0)=0$. Recall that the Koszul dual operad of a quadratic-linear opead $\calP=\calP(E,R)$ admits the following quadratic presentation $\calP^!=\calP\big(E^\vee, (qR)^\perp\big)$, with $E^\vee\coloneqq s^{-1}{\End}_{\k s^{-1}}\otimes_{\mathrm{H}} E^*$, see \cite[Section~$7.2.3$]{LodayVallette12}.

\begin{lemma}\label{lem:ExtConvAlg}
Under the abovementioned assumptions, when the Koszul dual operad $\calP^!$ is extendable,  
the (complete)  convolution pre-Lie algebra associated to $\left(u_\chi \calP^!\right)^*$ is isomorphic to 
\[ \hom_\Sy\left(\left(u_\chi \calP^!\right)^*, \eend_A\right)\cong  
A \times  \hom_\Sy\left( \calP^{\ac} , \eend_A\right)\ ,
\]
where the pre-Lie product $\bigstar$ on the right-hand side is given by 
\[
(a,f){\bigstar} (b,g) = 
\left(f(\id)(b), f \star g + f\ast b
\right)\ ,
\]
with $\star$ being the pre-Lie product on the convolution algebra \[\a_{\calP^{\ac}, A}=\left(\hom_\Sy\left( \calP^{\ac} , \eend_A\right), \star\right)\] and $f\ast b$ denoting an element depending only on $f$ and $b$.
\end{lemma}

\begin{proof}
When the Koszul dual operad $\calP^!$ is extendable, we use the underlying isomorphism 
$u_\chi \calP^! \cong\k u \oplus \calP^!$ to get 
\[ \hom_\Sy\left(\left(u_\chi \calP^!\right)^*, \eend_A\right)\cong  \hom_\Sy\left(u^* \oplus \calP^{\ac} , \eend_A\right)\cong 
A \times  \hom_\Sy\left( \calP^{\ac} , \eend_A\right). 
\]
The arity $0$ part the pre-Lie product $\bigstar$ comes from the  unique way to obtain the element $u$ in the operad $u_\chi \calP^!$ as a partial composition of two elements, that is $u=\id \circ_1 u$. Similarly for its other part, given any element $\mu\in \calP^!$, there are two ways to get it as a partial composition of two elements: either with two elements coming from $\calP^!$ or with one (bottom) coming from $\calP^!$ and one (top) which is $u$. Notice that in this latter case, we get a finite sum in each arity since $E(0)=0$ and $E(n)$ is finite dimensional for any $n\ge 1$.
\end{proof}

\begin{theorem}\label{thm:TwCatAlg}
Let $\alpha \in \MC\left(\a_{\calP^{\ac}, A}\right)$ be a complete $\calP_\infty$-algebra structure on $A$ and let $a\in \F_1A_0$. The Maurer--Cartan element $a.\alpha$ in the convolution algebra $\hom_\Sy\left(\left(u_\chi \calP^!\right)^*, \eend_A\right)$ is a $\calP_\infty$-algebra structure if and only if its arity $0$ part $(a.\alpha)(u^*)=0$ vanishes. 
\end{theorem}

\begin{proof}
This is a direct corollary of Lemma \ref{lem:ExtConvAlg} which implies that Maurer--Cartan elements in the convolution algebra $\a_{\calP^{\ac}, A}= \hom_\Sy\left( \calP^{\ac} , \eend_A\right)$ are in one-to-one correspondence with Maurer--Cartan elements in the extended convolution algebra $\hom_\Sy\left(\left(u_\chi \calP^!\right)^*, \eend_A\right)$ whose arity $0$ part vanishes. 
\end{proof}

This result prompts the following definition.

\begin{definition}[Twistable homotopy algebras]\label{def:TwHoAlg}\index{twistable homotopy algebras}
Let $\calP=\calP(E,R)$ be an arity-wise finitely generated quadratic-linear Koszul operad. We say that the category of homotopy $\calP$-algebras, that is algebras over an operad $\calP_\infty=\Omega \calP^{\ac}$, is \emph{twistable} when the Koszul dual operad $\calP^!$ is extendable. The corresponding equation $(a.\alpha)(u^*)=0$ is then naturally dubbed the \emph{Maurer--Cartan equation}. \index{Maurer--Cartan equation!twistable homotopy algebras}
\end{definition}

In plain words, when a category of $\calP_\infty$-algebras is twistable, this means that the ``dual'' category of $\calP^!$-algebras admits a meaningful extension of \emph{unital} $\calP^!$-algebras and thus the category of $\calP_\infty$-algebras admits a meaningful extension of \emph{curved} $\calP_\infty$-algebras. 
In this case, the twisting procedure works as in the case of homotopy associative or homotopy Lie algebras: any $\calP_\infty$-algebra can be twisted by an element to produce a curved $\calP_\infty$-algebra, which turns out to be  a $\calP_\infty$-algebra if and only if its twisted curvature vanishes, that is satisfises the Maurer--Cartan equation. 

\begin{proposition}\label{prop:TwistableAlg}\leavevmode
\begin{enumerate}
\item The categories of homotopy Lie algebras, homotopy Gerstenhaber algebras, homotopy gravity algebras, and homotopy permutative algebras are twistable.
\item The category of homotopy associative algebras is twistable. 
\item The categories of homotopy commutative algebras and homotopy pre-Lie algebras are not twistable. 
\item The categories of homotopy dendriform algebras and homotopy diassociative algebras are not twistable. 
\end{enumerate}
\end{proposition}

\begin{proof}
This is a direct corollary of  Proposition~\ref{prop:Extendable} and Theorem~\ref{thm:TwCatAlg}. 
\end{proof}

One can check that the maps $\chi$ introduced in Proposition~\ref{prop:Extendable} for the unital extensions of the 
ns operad $\As$  and for the 
operad $\Com$
produce the exact same twisting formulas than the ones given in Section~\ref{subsec:TwistGrp} for (curved)  $\Ai$-algebras and in Section~\ref{subsec:TwLii} for (curved)  $\Li$-algebras. The other cases are new. Below we make the case of (curved)  homotopy permutative algebras explicit and we leave the other cases to the interested reader. \\

We already recalled the definition of permutative algebras\index{permutative algebras} in Proposition~\ref{prop:Extendable} and its proof. The operad $\Perm$ is known to be Koszul. Its Koszul dual operad is the operad $\PreLie$ encoding pre-Lie algebras \cite[Section~13.4.3]{LodayVallette12}. This latter operad is isomorphic to eh of rooted trees operad $\mathrm{RT}$:  \index{operad!$\mathrm{RT}$}\index{rooted trees operad}
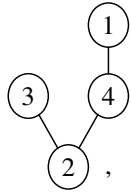
$$\vcenter{\hbox{
\xymatrix@M=5pt@R=10pt@C=0pt{
 & & *+[o][F-]{1}\ar@{-}[d]\\
*+[o][F-]{3}\ar@{-}[dr] &&  *+[o][F-]{4}\ar@{-}[dl] \\
& *+[o][F-]{2} & ,
}}}$$
where the lowest vertex is the root, see \cite[Theorem~1.9]{ChapotonLivernet01}. The partial composition products $\tau\circ_i \upsilon$ is given by the insertion of the tree $\upsilon$ at the $i^{\mathrm{th}}$ vertex of the tree $\tau$. The sub-trees attached above the $i^{\mathrm{th}}$ vertex of the tree $\tau$ are then  grafted in all possible onto the vertices of the tree $\upsilon$; for example,  the composition
$$
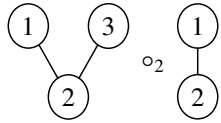
\vcenter{
\xymatrix@M=5pt@R=10pt@C=0pt{
*+[o][F-]{1}\ar@{-}[dr] &&  *+[o][F-]{3}\ar@{-}[dl] \\
& *+[o][F]{2} & 
}}
\ \circ_2 \ 
\vcenter{
\xymatrix@M=5pt@R=10pt@C=0pt{
*+[o][F-]{1}\ar@{-}[d]  \\
*+[o][F]{2} 
}}$$ 
is given by the formula
 \[
\vcenter
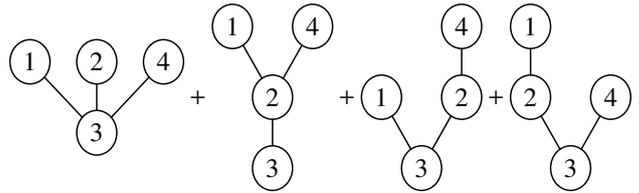
{
\xymatrix@M=5pt@R=10pt@C=0pt{
*+[o][F-]{1}\ar@{-}[drr] &&*+[o][F-]{2}\ar@{-}[d] && *+[o][F-]{4}\ar@{-}[dll]\\
&& *+[o][F]{3} & &
}}+
\vcenter{
\xymatrix@M=5pt@R=10pt@C=0pt{
*+[o][F-]{1}\ar@{-}[dr] & & *+[o][F-]{4}\ar@{-}[dl]\\
 &  *+[o][F-]{2}\ar@{-}[d]& \\
& *+[o][F]{3} & 
}}
+
\vcenter{
\xymatrix@M=5pt@R=10pt@C=0pt{
& & *+[o][F-]{4}\ar@{-}[d]\\
*+[o][F-]{1}\ar@{-}[dr]  && *+[o][F-]{2}\ar@{-}[dl] \\
& *+[o][F]{3} & 
}}+
\vcenter{
\xymatrix@M=5pt@R=10pt@C=0pt{
*+[o][F-]{1}\ar@{-}[d] & &  \\
  *+[o][F-]{2}\ar@{-}[dr]&&  *+[o][F-]{4}\ar@{-}[dl]\\
& *+[o][F]{3} & 
}}\ .
 \]

 As explained in Proposition~\ref{prop:MCOmegaC}, the data of a shifted $\Perm_\infty$-algebra on a (complete) dg module $A$ amounts to a Maurer--Cartan element in the convolution algebra $\a_{\PreLie^*, A}$. Given a rooted tree $\tau$ and a sub-tree $\upsilon\subset \tau$ (forgetting the labels), we denote by $(\tau/\upsilon, i, \sigma)$ respectively the rooted tree $\tau/\upsilon$ obtained by contracting $\upsilon$ in $\tau$ and by relabelling the vertices,  the label $i$ of the corresponding new vertex, and the overall permutation $\sigma\in\Sy_{|\tau|}$ which produces the labelling of the tree $\tau$ after  the composite $\left(\tau/\upsilon \circ_i \upsilon\right)^\sigma=\tau$. Such a decomposition is not unique, we choose the one for which the associated planar tree 
 \[\vcenter
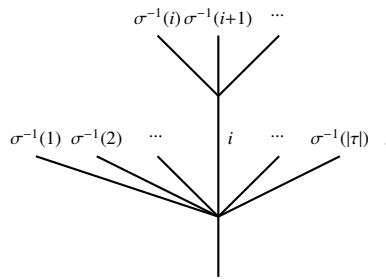
{\hbox{
\begin{tikzpicture}[yscale=0.8,xscale=0.8]

\draw[thick] (0,-1)--(0,0) -- (0,1) -- (0, 3);
\draw[thick] (1,1) -- (0,0) -- (-1,1);
\draw[thick] (-2,1) -- (0,0) ;
\draw[thick] (2,1) -- (0,0) ;
\draw[thick] (-3,1) -- (0,0) ;
\draw[thick] (1,3) -- (0,2) -- (-1,3);

\node at (-3,1.3) {$\scriptstyle\sigma^{-1}(1)$};
\node at (-2,1.3) {$\scriptstyle\sigma^{-1}(2)$};
\node at (-1,1.3) {$\scriptstyle\cdots$};
\node at (-1,3.3) {$\scriptstyle\sigma^{-1}(i)$};
\node at (0,3.3) {$\scriptstyle\sigma^{-1}(i+1)$};
\node at (1,3.3) {$\scriptstyle\cdots$};
\node at (1,1.3) {$\scriptstyle\cdots$};
\node at (2,1.3) {$\scriptstyle\sigma^{-1}(|\tau|)$};
\node at (0.2,1.3) {$\scriptstyle i$};
\end{tikzpicture}}}\ ,\]
where the bottom corolla has arity $|\tau/\upsilon|$ and where the top corolla has arity $|\upsilon|$, and  is a shuffle tree \cite[Section~8.2.2]{LodayVallette12}. 

\begin{multline*}
\text{For }\ %\quad
\tau\coloneqq\vcenter{\hbox{\scalebox{1}{
\xymatrix@M=5pt@R=10pt@C=0pt{
*+[o][F-]{1}\ar@{-}[dr] & & *+[o][F-]{4}\ar@{-}[dl]\\
 &  *+[o][F-]{2}\ar@{-}[d]& \\
& *+[o][F-]{3} & 
}}}}%\quad 
\text{ and } 
%\quad
\upsilon\coloneqq
\vcenter{\hbox{
\scalebox{1}{\xymatrix@M=5pt@R=10pt@C=0pt{
*+[o][F-]{1}\ar@{-}[dr] & & *+[o][F-]{3}\ar@{-}[dl]\\
 &  *+[o][F-]{2}
}}}}\ ,\\
  \text{ we get }
\tau/\upsilon=
\vcenter{\hbox{
\scalebox{1}{\xymatrix@M=5pt@R=10pt@C=10pt{
  *+[o][F-]{1}\ar@{-}[d] \\
*+[o][F-]{2}  
}}}}\ ,\  i=1, \text{ and }\sigma=(34).
\end{multline*}
 
 \begin{lemma}\index{shifted homotopy permutative algebra}
 A shifted homotopy permutative algebra is a graded module $A$ equipped with a collection of degree $-1$ operations 
\[ \left\{m_\tau : A^{\otimes |\tau|}\to A\right\}_{\tau \in \mathrm{RT}}\ ,
\]
where $|\tau|$ stands for the number of vertices of a rooted tree. These generating operations are required to satisfy the relations 
\[
\sum_{\upsilon \subset \tau} \left(m_{\tau/\upsilon} \circ_i m_\upsilon \right)^\sigma=0\ ,
\]
for any tree $\tau\in \mathrm{RT}$.
 \end{lemma}
 
 \begin{proof}
 This is a direct corollary of the definition of the convolution algebra $\a_{\PreLie^*, A}$ and the cooperad structure on 
 $\PreLie^*$ obtained by linear dualisation of the operad structure on $\PreLie$ described above. The choice of decompositions $\left(\tau/\upsilon \circ_i \upsilon\right)^\sigma=\tau$ along shuffle trees does not alter the result: any choice of decomposition would produce the same result in the end, by the properties of the endomorphism operad $\End_A$. In other words, the composite $\left(m_{\tau/\upsilon} \circ_i m_\upsilon \right)^\sigma$ does not depend of the choice of decomposition.
 \end{proof}

Here are the first relations, where we use the  notation $d\coloneqq m_{\vcenter{
\xymatrix@M=3pt@R=8pt@C=8pt{
  *+[o][F-]{\scriptstyle 1}
}}}$\ .
\begin{align*}
\tau& =
\vcenter{
	\xymatrix@M=5pt@R=10pt@C=10pt{
		*+[o][F-]{1}
}}
&\quad & d^2=0 \phantom{\vcenter{
		\xymatrix@M=5pt@R=10pt@C=10pt{
			*+[o][F-]{2}\ar@{-}[d] \\
			*+[o][F-]{1} 
}}}
\\
\tau& =
\vcenter{
	\xymatrix@M=5pt@R=10pt@C=10pt{
		*+[o][F-]{2}\ar@{-}[d] \\
		*+[o][F-]{1} 
}}
& \ &
\partial m_{\vcenter{
		\xymatrix@M=3pt@R=5pt@C=5pt{
			*+[o][F-]{\scriptstyle 2}\ar@{-}[d] \\
			*+[o][F-]{\scriptstyle 1} 
}}} \coloneqq
d m_{\vcenter{
		\xymatrix@M=3pt@R=5pt@C=5pt{
			*+[o][F-]{\scriptstyle 2}\ar@{-}[d] \\
			*+[o][F-]{\scriptstyle 1} 
}}}(-,-)+m_{\vcenter{
		\xymatrix@M=3pt@R=5pt@C=5pt{
			*+[o][F-]{\scriptstyle 2}\ar@{-}[d] \\
			*+[o][F-]{\scriptstyle 1} 
}}}(d(-), -)+m_{\vcenter{
		\xymatrix@M=3pt@R=5pt@C=5pt{
			*+[o][F-]{\scriptstyle 2}\ar@{-}[d] \\
			*+[o][F-]{\scriptstyle 1} 
}}}(-, d(-))=0;
\\
\tau& =
\vcenter{
	\xymatrix@M=5pt@R=10pt@C=10pt{
		*+[o][F-]{3}\ar@{-}[d] \\
		*+[o][F-]{2}\ar@{-}[d] \\
		*+[o][F-]{1} 
}}
& \ & \partial m_{\vcenter{
		\xymatrix@M=3pt@R=5pt@C=5pt{
			*+[o][F-]{\scriptstyle 3}\ar@{-}[d] \\
			*+[o][F-]{\scriptstyle 2}\ar@{-}[d] \\
			*+[o][F-]{\scriptstyle 1}
}}}  =-
m_{\vcenter{
		\xymatrix@M=3pt@R=5pt@C=5pt{
			*+[o][F-]{\scriptstyle 2}\ar@{-}[d] \\
			*+[o][F-]{\scriptstyle 1} 
}}}\circ_1 
m_{\vcenter{
		\xymatrix@M=3pt@R=5pt@C=5pt{
			*+[o][F-]{\scriptstyle 2}\ar@{-}[d] \\
			*+[o][F-]{\scriptstyle 1} 
}}}
-
m_{\vcenter{
		\xymatrix@M=3pt@R=5pt@C=5pt{
			*+[o][F-]{\scriptstyle 2}\ar@{-}[d] \\
			*+[o][F-]{\scriptstyle 1} 
}}}\circ_2
m_{\vcenter{
		\xymatrix@M=3pt@R=5pt@C=5pt{
			*+[o][F-]{\scriptstyle 2}\ar@{-}[d] \\
			*+[o][F-]{\scriptstyle 1} 
}}}\ ;
\\
\tau& =
\vcenter{
\xymatrix@M=5pt@R=10pt@C=0pt{
*+[o][F-]{2}\ar@{-}[dr] &&  *+[o][F-]{3}\ar@{-}[dl] \\
& *+[o][F-]{1} & 
}}
& \ & \partial m_{\vcenter{
\xymatrix@M=3pt@R=5pt@C=0pt{
*+[o][F-]{\scriptstyle 2}\ar@{-}[dr] &&  *+[o][F-]{\scriptstyle 3}\ar@{-}[dl] \\
& *+[o][F-]{\scriptstyle 1} & 
}}}=-
m_{\vcenter{
\xymatrix@M=3pt@R=5pt@C=5pt{
  *+[o][F-]{\scriptstyle 2}\ar@{-}[d] \\
*+[o][F-]{\scriptstyle 1} 
}}}\circ_1 
m_{\vcenter{
\xymatrix@M=3pt@R=5pt@C=5pt{
  *+[o][F-]{\scriptstyle 2}\ar@{-}[d] \\
*+[o][F-]{\scriptstyle 1} 
}}}
-
\big(m_{\vcenter{
\xymatrix@M=3pt@R=5pt@C=5pt{
  *+[o][F-]{\scriptstyle 2}\ar@{-}[d] \\
*+[o][F-]{\scriptstyle 1} 
}}}\circ_1 
m_{\vcenter{
\xymatrix@M=3pt@R=5pt@C=5pt{
  *+[o][F-]{\scriptstyle 2}\ar@{-}[d] \\
*+[o][F-]{\scriptstyle 1} 
}}}\big)^{(23)}\ . 
\end{align*}

%\smallskip

This shows that $m_{\vcenter{
\xymatrix@M=3pt@R=5pt@C=5pt{
 *+[o][F-]{\scriptstyle 3}\ar@{-}[d] \\
  *+[o][F-]{\scriptstyle 2}\ar@{-}[d] \\
*+[o][F-]{\scriptstyle 1}
}}}$\ , respectively $m_{\vcenter{
\xymatrix@M=3pt@R=5pt@C=0pt{
*+[o][F-]{\scriptstyle 2}\ar@{-}[dr] &&  *+[o][F-]{\scriptstyle 3}\ar@{-}[dl] \\
& *+[o][F-]{\scriptstyle 1} & 
}}}$\ ,  is a homotopy for the first relation (anti-associativity), respectively for the second relation (partial skew-symmetry),  defining a shifted permutative algebra.\\

We consider the unital extension of operad $\PreLie$ introduced in the proof of Proposition~\ref{prop:Extendable} that 
we simply denote here by $\uPreLie$. This operad encodes pre-Lie algebras $(A, \star)$ equipped with a degree $0$ element $1$ satisfying 
\[1\star x=x=x\star 1\ ,\]
for any $x\in A$. 

\begin{lemma}\label{lem:uCompPreLie}
In the operad $\uPreLie$, the composite of the arity $0$ element $u$ at a vertex of a rooted tree $\tau$ 
depends on the  position of this latter one as follows:
\begin{description}
\item[\sc At a leaf] if the number of inputs of the vertex supporting the leaf is equal to $n$, then the resulting rooted tree is obtained by removing the leaf, relabelling the other vertices, and multiplying by $2-n$,

\item[\sc At the root or at an internal vertex] when the vertex has just one input, then it is deleted and the remaining vertices relabelled accordingly, otherwise when the vertex has at least two inputs, the upshot is equal to~$0$. 
\end{description}
\end{lemma}

\begin{proof}
To demonstrate that, we shall use symmetric braces introduced in Definition \ref{def:SymBrace}. Using the recursive formula for braces, one can show that
 \[
\{1;x_1,\ldots,x_n\}=0 \text{  for  } n\ge 2~,
 \] 
and that
 \[
\{x_0;x_1,\ldots,x_{i-1},1,x_{i+1},\ldots,x_n\}=(2-n)\{x_0;x_1,\ldots,x_{i-1},x_{i+1},\ldots,x_n\}~. 
 \] 
Since it is known \cite{GuinOudom08} that symmetric braces corresponds to the ``naive'' way of building a rooted tree from corollas, the claim follows. 
\end{proof}

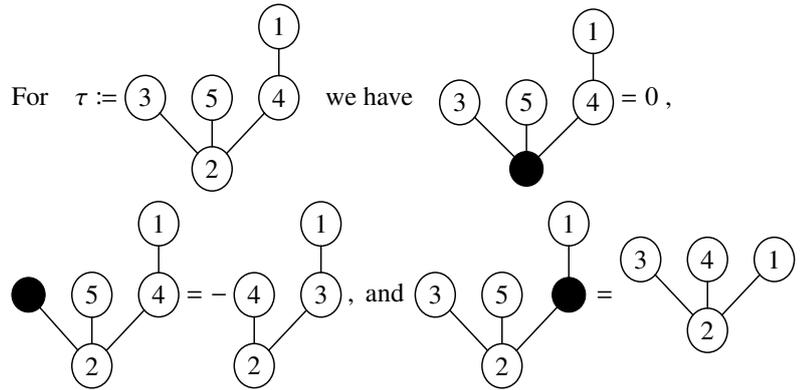
\begin{eqnarray*}
&&\text{For}\quad \tau\coloneqq
\vcenter{\xymatrix@M=5pt@R=10pt@C=10pt{
 & & *+[o][F-]{1}\ar@{-}[d]\\
*+[o][F-]{3}\ar@{-}[dr] &*+[o][F-]{5}\ar@{-}[d]&  *+[o][F-]{4}\ar@{-}[dl] \\
& *+[o][F-]{2} & }}\quad \text{we have} \quad 
\vcenter{\xymatrix@M=5pt@R=10pt@C=10pt{
 & & *+[o][F-]{1}\ar@{-}[d]\\
*+[o][F-]{3}\ar@{-}[dr] &*+[o][F-]{5}\ar@{-}[d]&  *+[o][F-]{4}\ar@{-}[dl] \\
& *+[o][F**]{\cdot}  & 
}}=0\ , \\
&&  
\vcenter{\xymatrix@M=5pt@R=10pt@C=10pt{
 & & *+[o][F-]{1}\ar@{-}[d]\\
*+[o][F**]{\cdot}\ar@{-}[dr] &*+[o][F-]{5}\ar@{-}[d]&  *+[o][F-]{4}\ar@{-}[dl] \\
& *+[o][F-]{2} & }}=
-\ \vcenter{\xymatrix@M=5pt@R=10pt@C=10pt{
 & *+[o][F-]{1}\ar@{-}[d]\\
*+[o][F-]{4}\ar@{-}[d]&  *+[o][F-]{3}\ar@{-}[dl] \\
 *+[o][F-]{2} & }}\ , \  \text{and} \ \, 
 \vcenter{\xymatrix@M=5pt@R=10pt@C=10pt{
 & & *+[o][F-]{1}\ar@{-}[d]\\
*+[o][F-]{3}\ar@{-}[dr] &*+[o][F-]{5}\ar@{-}[d]&  *+[o][F**]{\cdot}\ar@{-}[dl] \\
& *+[o][F-]{2} & }}=
 \vcenter{\xymatrix@M=5pt@R=10pt@C=10pt{
*+[o][F-]{3}\ar@{-}[dr] &*+[o][F-]{4}\ar@{-}[d]&  *+[o][F-]{1}\ar@{-}[dl] \\
& *+[o][F-]{2} &}}\ .
\end{eqnarray*}
The black vertex represents where the element $u$ is plugged. \\

Given a rooted tree $\tau$, we consider its \emph{unital expansions} which are rooted trees $\widetilde{\tau}$ with two types of vertices: the ``white'' ones which are bijectively labeled by $\{1, \ldots, |\tau|\}$ and the ``black'' ones which receive no label. Every such black and white rooted trees $\widetilde{\tau}$ are moreover required to give $\tau$ under the rule given in Lemma~\ref{lem:uCompPreLie}.
We denote by $c_{\widetilde{\tau}}$ the coefficient of $\tau$ in the operad $\uPreLie$ obtained by replacing each black vertex of $\widetilde{\tau}$ by the element~$u$. For example, in the case 
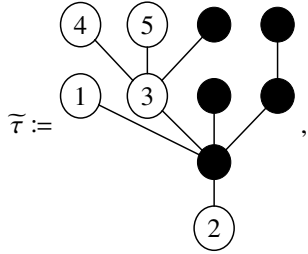
\[
\widetilde{\tau}\coloneqq
\vcenter{\xymatrix@M=5pt@R=10pt@C=10pt{
 *+[o][F-]{4}\ar@{-}[dr]&*+[o][F-]{5}\ar@{-}[d]& *+[o][F**]{\cdot}\ar@{-}[dl]& *+[o][F**]{\cdot}\ar@{-}[d]\\
*+[o][F-]{1}\ar@{-}[drr] &*+[o][F-]{3}\ar@{-}[dr] &*+[o][F**]{\cdot}\ar@{-}[d]&  *+[o][F**]{\cdot}\ar@{-}[dl] \\
&& *+[o][F**]{\cdot} \ar@{-}[d]&\\
&& *+[o][F-]{2} & }}\ ,
\] 
we have $c_{\widetilde{\tau}}=-2$ and
 \[ 
\tau=\vcenter
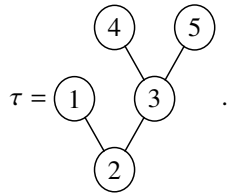
{\xymatrix@M=5pt@R=10pt@C=0pt{
 &*+[o][F-]{4}\ar@{-}[dr]&&*+[o][F-]{5}\ar@{-}[dl] \\
*+[o][F-]{1}\ar@{-}[dr] &&*+[o][F-]{3}\ar@{-}[dl]&   \\
& *+[o][F-]{2}& }}\ .
\]

\begin{proposition}
Let $\left(A, \F, \{m_\tau\}_{\tau \in \mathrm{RT}}\right)$ be a complete shifted $\Perm_\infty$-algebra and let $a\in \F_1A_0$.
The twisted operations 
\[ m_\tau^a\coloneqq \sum_{\widetilde{\tau}} {\tfrac{1}{c_{\widetilde{\tau}}}} m_{\widetilde{\tau}}(a, \ldots, a, - ,  \cdots, -)\ ,
\]
where the sum runs over unital expansions of the rooted tree $\tau$ and where the elements $a$ are inserted at the black vertices of $\widetilde{\tau}$, 
define a complete shifted $\Perm_\infty$-algebra if and only if the element $a$ satisfies the following Maurer--Cartan equation 
\[
\sum_{n\ge 1} m_{{\vcenter{
\xymatrix@M=3pt@R=5pt@C=5pt{
  *+[o][F-]{\scriptstyle n}\ar@{..}[d] \\
*+[o][F-]{\scriptstyle 2} \ar@{-}[d]\\  
*+[o][F-]{\scriptstyle 1} 
}}}}(a, \ldots, a)=0
\ .\]
\end{proposition}

\begin{proof}
We apply here Theorem~\ref{thm:TwCatAlg}: let us denote by $\alpha$ the Maurer--Cartan element in the convolution 
algebra $\a_{\PreLie^*, A}$ corresponding to the complete shifted $\Perm_\infty$-algebra structure, that is $\alpha(\tau^*)\coloneqq m_\tau$~. 

First, the  the Maurer--Cartan equation  
 is equal to 
\[(a.\alpha)(u^*)=\sum_{\tau \in \mathrm{RT}} m_\tau (a, \ldots, a)=0\ .\]
But Lemma~\ref{lem:uCompPreLie} shows that, the composite of  elements $u$ at all the vertices of a rooted tree $\tau$ in the operad $\uPreLie$ is equal to $u$ for ladders and vanishes otherwise. 

Then, the twisted operation $m_\tau^a$ is equal to $\big(\alpha\circledcirc (1+a)\big)(\tau^*)$ in the convolution algebra 
$\a_{\uPreLie^*, A}$, which is thus equal to $m_\tau^a\coloneqq \sum_{\widetilde{\tau}} {\textstyle \frac{1}{c_{\widetilde{\tau}}}} m_{\widetilde{\tau}}(a, \ldots, a, - ,  \cdots, -)$ again by Lemma~\ref{lem:uCompPreLie}. 
\end{proof}

\begin{remark}
Under the notation $m_n\coloneqq m_{{\vcenter{
\xymatrix@M=3pt@R=5pt@C=5pt{
  *+[o][F-]{\scriptstyle n}\ar@{..}[d] \\
*+[o][F-]{\scriptstyle 2} \ar@{-}[d]\\  
*+[o][F-]{\scriptstyle 1} 
}}}}$\ , one can see that $\big(A, \{m_n\}_{n\ge 1}\big)$ forms a shifted  $\Ai$-algebra. The twisting procedure for shifted $\Perm_\infty$-algebras produces the exact same twisting procedure seen in Section~\ref{subsec:TwistGrp} for this underlying $\Ai$-algebra, ie 
$m_n^a= m^a_{{\vcenter{
\xymatrix@M=3pt@R=5pt@C=5pt{
  *+[o][F-]{\scriptstyle n}\ar@{..}[d] \\
*+[o][F-]{\scriptstyle 2} \ar@{-}[d]\\  
*+[o][F-]{\scriptstyle 1} 
}}}}$ and the same Maurer--Cartan equation: 
\[\sum_{n\ge 1} m_n(a,\allowbreak \ldots, a)=0\ .\] This is not a surprise, since the forgetful functor from permutative algebras to associative algebras is actually produced by pulling back along the epimorphism $\Ass \twoheadrightarrow \Perm$ of operads obtained by sending the usual generator of $\Ass$ to the usual generator of $\Perm$.
Since this latter one comes from a morphism of operadic quadratic data,  it induces a monomorphism of cooperads between the Koszul dual cooperads 
$\Ass^{\ac} \hookrightarrow \Perm^{\ac}$ and thus a monomorphism between the Koszul resolution 
$\Ass_\infty \hookrightarrow \Perm_\infty$, which sends precisely $\mu_n$ to the rooted tree 
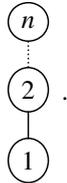
\[\vcenter{
\xymatrix@M=5pt@R=10pt@C=10pt{
  *+[o][F-]{ n}\ar@{..}[d] \\
*+[o][F-]{ 2} \ar@{-}[d]\\  
*+[o][F-]{ 1} 
}}\ \ . \] 
The twisting procedure for shifted $\Perm_\infty$-algebras extends the twisting procedure for shifted $\Ai$-algebras since the morphism of cooperads 
$\Ass^{\ac} \hookrightarrow \Perm^{\ac}$ extends to the morphism of cooperads 
$\mathrm{uAss}^* \hookrightarrow \uPreLie^*$.
\end{remark}

\chapter{The twisting procedure for operads}\footnotetext{\hrule\smallskip\noindent This material will be published by Cambridge University Press \& Assessment as ‘Maurer-Cartan Methods in Deformation Theory: the twisting procedure’ by Vladimir Dotsenko, Sergey Shadrin and Bruno Vallette. This version is free to view and download for personal use only. Not for re-distribution, re-sale or use in derivative works. \copyright Cambridge University Press \& Assessment}\label{sec:TwNsOp}

In this chapter, we lift the twisting procedure to the level of complete differential graded operads and we apply it in detail to the example of  multiplicative nonsymmetric operads. In the symmetric case, this is the seminal theory of T. Willwacher \cite{Willwacher15} which was developed further by V. Dolgushev, C. Rogers, and T. Willwacher \cite{DolgushevRogers12, DolgushevWillwacher15} but treated in a different way, first suggested by J.~Chuang and A.~Lazarev in \cite{MR3004818}. \\

For pedagogical reasons, we give a detailed presentation in the case of ns operads, and only touch the symmetric case briefly, since our approach works \textit{mutatis mutandis} for complete symmetric dg operads; we hope that this will make the theory of operadic twisting more accessible. 
We use the categorical notion of the \emph{coproduct of operads}, 
insisting that 
all the properties of the operadic twisting follow in a straightforward way from the universal property of this algebraic notion. 

\section{Twisting of nonsymmetric operads}
One can try to twist  any dg pre-Lie algebra $(A, d, \star)$ by Maurer--Cartan elements $d\mu+\mu\star\mu=0$ using the usual formulas $(A, d^\mu, \star)$, where the twisted differential is given by 
$$d^\mu(\nu)\coloneqq d\nu+\ad_\mu(\nu)\coloneqq d\nu+\mu\star \nu - (-1)^{|\nu|} \nu \star \mu\ .$$
From Proposition~\ref{prop:TwistableAlg}, we know that the general twisting procedure does not work for general pre-Lie algebras since the Koszul dual operad $\PreLie^!\cong \Perm$ for permutative algebras does not admit an extension by a unit. Let us illustrate that by exhibiting an explicit constraint for the pre-Lie identity of the twisted algebra. Remarkably, that constraint arises from an algebraic notion used, for example, in the classification of pre-Lie algebra structures on $\mathfrak{gl}_n$ in \cite{MR1608273}.

\begin{definition}[Left nucleus of a pre-Lie algebra]\index{pre-Lie algebra!left nucleus}
Let $\a=(A,\star)$ be a pre-Lie algebra. The  \emph{left nucleus} of $\a$ is the set of  elements $x$ such that  
\begin{equation}\label{Eq:MCspecial}
x\star(a_1\star a_2)=(x\star a_1)\star a_2\ , \  \text{ for all }\ \  a_1, a_2\in A\ . 
\end{equation}
\end{definition}

\begin{proposition}\label{prop:Twpre-Lie}
The twisted data $(A, d^\mu, \star)$ forms a dg pre-Lie algebra if and only if the Maurer--Cartan element $\mu$ belongs to the nucleus of $\a$. 
\end{proposition}

\begin{proof}
Clearly, $(A, d^\mu, \star)$ forms a dg pre-Lie algebra if and only if the twisted operator $d^\mu$ is a derivation which squares to zero. The latter property is always true since it can be formulated using only the dg Lie algebra associated to our dg pre-Lie algebra, and so follows from the properties of twisting of dg Lie algebras. To examine the former property, one performs the following computation:
\begin{eqnarray*}
d^\mu(\nu\star \omega)-d^\mu(\nu)\star \omega-(-1)^{|\nu|}\nu\star d^\mu(\omega)&=&
\mu\star(\nu\star \omega)-(\mu\star \nu)\star \omega\ ,
\end{eqnarray*}
which concludes the proof. 
\end{proof}

Recall that for any ns operad $\calP$, one defines a Lie bracket by anti-sym\-me\-tri\-zing the pre-Lie product 
$$\mu\star \nu \coloneqq \sum_{i=1}^n \mu \circ_i \nu\  , $$
where $\mu$ lives in $\calP(n)$.

\begin{definition}[Operadic Maurer--Cartan element]\index{Maurer--Cartan element!operadic}
Let $\left(\calP, \dd\right)$ be a dg ns operad. An \emph{operadic Maurer--Cartan} in $\calP$ is a 
 degree $-1$ and arity $1$ element 
$\mu\in \calP(1)_{-1}$ satisfying 
$$ \dd\mu+\mu\star \mu=\dd\mu+\mu\circ_1 \mu=0\ .$$
\end{definition}
The purpose of the arity $1$ constraint in this definition is to ensure that operadic Maurer--Cartan elements satisfy Equation~(\ref{Eq:MCspecial}). Therefore, one can twist the pre-Lie algebra associated to an ns operad by such elements. This result actually lifts to the level of the dg ns operad itself. 

\begin{example}
For the endomorphism operad~$\End_{(A,d)}$~, the Maurer--Cartan elements are the degree $-1$ maps $ m : A \to A$ satisfying the equation $dm+md+m^2=0$, i.e. perturbations of the differential. 
\end{example}

\begin{proposition}[Twisted operad]\index{twisted operad}
Let $\left(\calP, \dd, \left\{ \circ_i \right\}\right)$ be a dg ns operad and let $\mu$ be one of its Maurer--Cartan elements. 
The degree $-1$ operator  
$$\dd^\mu(\nu)\coloneqq\dd\nu+\ad_\mu(\nu)\coloneqq\dd\nu+\mu\star \nu - (-1)^{|\nu|} \nu \star \mu$$
is a  square-zero derivation.
Therefore, the data $\calP^\mu\coloneqq\left(\calP,   \dd^\mu,  \left\{ \circ_i \right\}\right)$ defines a dg ns operad, called 
the \emph{ns operad twisted  by the Maurer-Cartan element $\mu$}. 
\end{proposition}

\begin{proof}
The computations are similar to that of Proposition~\ref{prop:Twpre-Lie}. The fact that the twisted operator 
$\ad_\mu$ is a derivation with respect to  the partial composition products $\circ_i$ is equivalent to the equations 
\begin{eqnarray*}
\mu\star(\nu\circ_i \omega)=(\mu\star \nu)\circ_i \omega\ , 
\end{eqnarray*}
which holds true since the Maurer--Cartan element $\mu$ is supported in arity $1$. 
\end{proof}

\begin{example}
The endomorphism operad $\End_{(A,d)}$ of a chain complex $(A,d)$ twisted by a Maurer--Cartan element $m$ is actually the endomorphism operad 
$$\End_{(A,d)}^m=\End_{(A,d+m)} $$
of the twisted chain complex $(A,d+m)$. 
\end{example}

\begin{example}
One can also twist the ns operad $\mathrm{ncBV}$ of non-commutative Batalin--Vilkovisky algebras \cite{DotsenkoShadrinVallette19} by its square-zero element $\Delta$, under cohomological degree convention. 
\end{example}

The operadic twisting operation satifises the following usual functorial property of the twisting procedure. 

\begin{proposition}\label{prop:MorphMCOperadic}
Let $\rho \colon \calP \to \calQ$ be a morphism of dg ns operads. The image $\rho(\mu)$ of a Maurer--Cartan $\mu$ of the dg ns operad $\calP$ is a Maurer--Cartan element in the dg ns operad $\calQ$. The morphism $\rho$ of ns operads commutes with the twisted differential, that is yields a morphism of dg ns operads 
$$ 
\widetilde\rho\  :\ \calP^\mu \to \calQ^{\rho(\mu)}
\ . $$
\end{proposition}

\begin{proof}
This is proved by straightforward  computations. 
\end{proof}

This proposition applied to the endomorphism operad $\calQ=\End_A$ will play a  key role in the following sections.

\begin{proposition}\label{prop:-MC}
Let $\mu$ be a Maurer--Cartan element of a dg ns operad $\calP$. 
\begin{enumerate}
\item  An element $\alpha$ is a Maurer--Cartan element of the twisted dg ns operad $\calP^\mu$ if and only if 
$\alpha-\mu$ is  a Maurer--Cartan element of the original  dg ns operad $\calP$. 

\item 
The element $-\mu$ is a Maurer--Cartan element in the twisted dg ns operad $\calP^\mu$ and twisting this latter one again by this Maurer--Cartan element produces the original operad:
$$ 
\left(\calP^\mu\right)^{-\mu}=\calP\ .
$$
\end{enumerate}
\end{proposition}

\begin{proof}
The first statement is proved by straightforward  computations. The second one is a special case for $\alpha=0$.
\end{proof}

All these results hold as well for complete dg ns operads; we will treat a particular case in the next section. 

\section{Twisted \texorpdfstring{$\calA_\infty$}{Ai}-operad}\label{subsec:TwAInftyOp}

\begin{proposition}\label{prop:MCAi}
The data of a complete $\Ai$-algebra structure together with a Maurer--Cartan element is encoded by the complete dg ns operad 
\index{operad!$\MC\Ai$}
$$\MC\Ai\coloneqq\left(
\widehat\calT\big(\alpha, \mu_2, \mu_3, \ldots\big), \dd
\right) \ , $$
where $\alpha$ has arity $0$ and degree $-1$ and $\mu_n$ has arity $n$ and degree $n-2$, for $n\ge 2$, where the filtration on the space $M=\big(\k \alpha, 0, \k \mu_2, \k \mu_3, \dots\big)$ of generators is given by 
\[
\alpha \in \F_1 M(0), \ \F_2 M(0)=\{0\} \quad \text{and} \quad \mu_n\in \F_0 M(n),\  \F_1 M(n)=\{0\}\ , \ \text{for} \ n\ge 2
\ , \]
 and where the differential is defined by 
\begin{eqnarray*} 
&&\dd\mu_n\coloneqq\sum_{\substack{p+q+r=n\\ p+1+r, q\geq 2}}(-1)^{pq+r+1} \mu_{p+1+r}\circ_{p+1} \mu_q\ ,
\\
&&\dd\alpha\coloneqq-\sum_{n\ge 2}\mu_n(\alpha, \ldots, \alpha) \ .
\end{eqnarray*}
\end{proposition}

\begin{proof}
The fact that the map $\dd$ extends to a unique square-zero  derivation is direct corollary of Proposition~\ref{prop:DGAction} applied to the identity map $\As^{\ac}\to \As^{\ac}$. 
So we get a well-defined complete dg ns operad. 
A complete $\MC\Ai$-algebra structure on a complete dg module $(A, \F , d)$ amounts to a morphism of filtered dg ns operads $\MC\Ai \to \eend_A$ according to Theorem~\ref{thm:CompleteAlg}. Such an assignment $\alpha \mapsto a$ and $\mu_n\mapsto m_n$ is equivalent to the datum of an element $a\in F_1 A_{-1}$ and a sequence of maps $m_n\colon  A^{\otimes n}\to A$ degree $n-2$ satisfying the relations of a Maurer--Cartan elements in an $\Ai$-algebra by the form of the differential~$\dd$. 
\end{proof}

\begin{remark}
Under the convention of Section~\ref{subsec:TwistGrp}, we have $\mu_n=s^{-1}\nu_n$~. 
\end{remark}

Inside the dg ns operad $\MC \Ai$, we consider the following elements 
$$\mu_n^\alpha\coloneqq\sum_{r_0, \ldots, r_n\ge 0} (-1)^{\sum_{k=0}^n kr_k} \mu_{n+r_0+\cdots+r_n}\big(\alpha^{r_0}, -, \alpha^{r_1}, -, \ldots,  - , \alpha^{r_{n-1}}, -,\alpha^{r_n}  \big)
\ , $$
for $n\ge 0$, that is for example: 
$$\mu_0^\alpha \coloneqq\sum_{n \ge 2} \mu_n(\alpha, \ldots, \alpha)\quad \text{and} \quad \mu_1^\alpha\coloneqq
\sum_{\substack{n\ge 2 \\ 1\leq i \leq n}}(-1)^{n-i} \, \mu_n\big(\alpha^{i-1}, -, \alpha^{n-i}\big)\ .
$$

\begin{lemma}\label{Lem:Diffmu}
The elements $\mu_n^\alpha$ satisfy the relations 
 \[
\dd \mu_n^\alpha=\sum_{\substack{p+q+r=n\\ p+1+r, q\geq 1}}(-1)^{pq+r+1} \mu^\alpha_{p+1+r}\circ_{p+1} \mu^\alpha_q\  \ \text{for}\  n\ge 2\ .
 \]
\end{lemma}

\begin{proof}
In plain words, the relations we should check  that the elements $\mu_k^\alpha$~, for $k\ge 2$~, satisfy the identities of operations of an $\Ai$-algebra for the differential $d+\mu_1^\alpha$~. This is an immediate consequence of Theorem \ref{thm:TwProcGp}. That result is valid for every complete algebra, in particular in any free algebra, proving that the required formula for the twisted operations holds on the operad level. 
\end{proof}

In particular, for $n=1$~,  we obtain the relation
 \[
\dd\mu_1^\alpha=-\mu_1^\alpha\circ_1\mu_1^\alpha\ ,
 \]
so the element $\mu_1^\alpha$ is a Maurer--Cartan element of the complete dg ns operad $\MC \Ai$. So, the operadic twisting procedure produces the following new complete dg ns operad. 

\begin{definition}[Twisted $\Ai$-operad]\index{twisted $\Ai$-operad}
The complete dg ns operad obtained by twisting the complete operad $\MC \Ai$ by the Maurer--Cartan $\mu_1^\alpha$ is called the 
 \emph{twisted $\Ai$-operad} and denoted by 
 \index{operad!$\Tw\Ai$}
$$\Tw\Ai\coloneqq \left(\MC \Ai\right)^{\mu_1^\alpha}
=
\left(
\widehat\calT(\alpha, \mu_2, \mu_3, \ldots), \dd^{\mu_1^\alpha}
\right) \ . $$
\end{definition}

The twisted differential is actually equal to
\begin{eqnarray*}
\mathrm{d}^{\mu_1^\alpha}(\alpha)&=&\dd(\alpha)+\mu_1^\alpha(\alpha)= -\sum_{n\ge 2}\mu_n(\alpha, \ldots, \alpha)
+\sum_{n\ge 2}n\mu_n(\alpha, \ldots, \alpha)\\&=&\sum_{n\ge 2}(n-1)\mu_n(\alpha, \ldots, \alpha)\ ,\\
\mathrm{d}^{\mu_1^\alpha}(\mu_n)&=&\dd(\mu_n) 
+\mu_1^\alpha \star \mu_n - (-1)^{n}
\mu_n\star \mu_1^\alpha\\
&=&\sum_{\substack{p+q+r=n\\ p+1+r, q\geq 2}}(-1)^{pq+r+1} \mu_{p+1+r}\circ_{p+1} \mu_q+\sum_{\substack{k\ge 2\\ 1\leq i \leq k}} 
(-1)^{(k-i)(n+1)}
\mu_k\big(\alpha^{i-1}, \mu_n, \alpha^{k-i}\big)\\
&&
-
\sum_{j=1}^n\sum_{\substack{k\ge 2 \\ 1\leq i \leq k}} (-1)^{n+k-i}\, \mu_n \circ_j \mu_k\big(\alpha^{i-1}, -, \alpha^{k-i}\big)
\ .
\end{eqnarray*}

Note that there is a morphism of complete dg ns operads 
$\Tw \Ai \to \Ai$ which sends $\alpha$ to $0$ and $\mu_n$ to $\mu_n$. On the algebra level, this corresponds to twisting an $\Ai$-algebra with the trivial Maurer--Cartan element. It turns out that there is a deeper morphism in the opposite direction.

\begin{proposition}\label{prop:MorhpAiMCAi}
The assignment $\mu_n\mapsto \mu_n^\alpha$ defines a morphism of complete dg ns operads 
$$\Ai \to \Tw \Ai \ . $$
\end{proposition}

\begin{remark}
By contrast, the morphism of complete ns operads from $\Ai$ to $\MC \Ai$ which sends the generators $\mu_n$ to $\mu_n^\alpha$ \emph{does not} commute with the differentials. There is a conceptual  reason for this: the twisted operations form an $\Ai$-algebra only with the twisted differential and not with the underlying differential. 
\end{remark}

\begin{proof}
The only point to check is the commutativity with the differentials on the generators $\mu_n$ of the quasi-free dg ns operad $\Ai$, that is:
\begin{align*}
\dd^{\mu_1^\alpha}\left(\mu_n^\alpha \right)&=
\dd\left(\mu_n^\alpha \right)+\mu_1^\alpha \star \mu^\alpha_n - (-1)^{n}
\mu^\alpha_n\star \mu_1^\alpha\\ &=
\sum_{\substack{p+q+r=n\\ p+1+r, q\geq 1}}(-1)^{pq+r+1} \mu_{p+1+r}^\alpha\circ_{p+1} \mu_q^\alpha
+\sum_{\substack{p=r=0\\ q=n}} \mu_{p+1+r}^\alpha\circ_{p+1} \mu_q^\alpha\\&\ \ \ \ \ 
-(-1)^n\sum_{\substack{p+1+r=n \\ q= 1}} \mu_{p+1+r}^\alpha\circ_{p+1} \mu_q^\alpha\\&=
\sum_{\substack{p+q+r=n\\ p+1+r, q\geq 2}}(-1)^{pq+r+1} \mu_{p+1+r}^\alpha\circ_{p+1} \mu_q^\alpha \ ,
\end{align*}
thanks to Lemma~\ref{Lem:Diffmu}. 
\end{proof}

Alternatively, if one checks the identities of Lemma \ref{Lem:Diffmu} on the operadic level directly, it leads to an alternative proof of Proposition~\ref{prop:TwCurvedGauge} as follows. The data of a complete $\Ai$-algebra structure  together with a Maurer--Cartan element $a$ on a complete dg module  $(A, \F, d)$ is equivalent to a morphism of complete dg ns operads $\rho\colon \MC\Ai \to \End_{(A,d)}$, where $\rho(\mu_n)=m_n$ and $\rho(\alpha)=a$ under the previous notations. By Proposition~\ref{prop:MorphMCOperadic}, the image of the Maurer--Cartan element $\mu_1^\alpha$ of $\MC\Ai$ under the morphism~$\rho$ gives a Maurer--Cartan element on the complete endomorphism ns operad $\eend_A$.
As emphasised above, the complete endomorphism operad twisted by this Maurer--Cartan element is equal to the complete endomorphism operad of the twisted chain complex $(A, d^a)$. 
 Therefore, the second point of Proposition~\ref{prop:MorphMCOperadic} shows that 
$$\widetilde\rho\colon \Tw\Ai \to \eend_{(A, d^a)}$$
is a morphism of complete dg ns operads. Pulling back that morphism along the morphism of complete dg ns operads $\Ai \to \Tw \Ai$, one gets  that the twisted operations $m_n^a$ do form an $\Ai$-algebra structure. 

\section{Twisting of multiplicative nonsymmetric operads}\label{subsec=TwMultiOp}
The following notion arose from the study of the Deligne conjecture on the Hochschild cochain complex, see \cite{McClureSmith02} as well as the discussion of twisting in the context of the Deligne conjecture in Section~\ref{sec:Deligne}. 

\begin{definition}[Multiplicative ns operad]\index{multiplicative ns operad}
A \emph{multiplicative ns operad} is a complete dg ns operad $\calP$ equip\-ped with a morphism of complete dg ns operads $\calA_\infty \to \calP$. We still  denote by $\mu_2$, $\mu_3$, etc. the images in $\calP$ of the generating operations of $\calA_\infty$. 
\end{definition}

Therefore any complete dg algebra over a multiplicative ns operad acquires a natural complete $\calA_\infty$-algebra structure. \\

The categorical \emph{coproduct of two ns operads} \index{coproduct of operads} is denoted by $\calP\vee \calQ$ (respectively by $\calP\hat{\vee} \calQ$ in the complete case). It is given by the free ns operad on the underlying $\mathbb{N}$-modules of $\calP$ and $\calQ$ modulo 
the relations that equate planar trees with two vertices labelled both by elements of $\calP$ or  by elements of $\calQ$ to their operadic composition. Thus planar trees with vertices labelled alternatively by elements from $\calP$ and $\calQ$ are representatives for this coproduct. Its operadic composition is given by the grafting of planar trees, followed possibly by the composite of two adjacent vertices labelled by elements from the same ns operad. \\

By a slight abuse of notation, we simply denote by $\alpha$ the (trivial) complete ns operad $(\k \alpha, 0,\ldots)$ spanned by the arity $0$ element $\alpha$. Thus  $\calP\hat{\vee} \alpha$ denotes  the complete coproduct of a complete ns operad $\calP$ with it. As a consequence of the above description, the elements of this coproduct are series, indexed by~$n\in \mathbb{N}$, of linear combinations of operations from $\calP$ with $n$ copies of $\alpha$  plugged at their inputs. 

\begin{proposition}\label{prop:MCP}
Let $\calA_\infty \to \calP$ be a multiplicative ns operad. 
The data of a complete $\calP$-algebra structure together with a Maurer--Cartan element is encoded by the complete dg ns operad 
$$\MC\calP\coloneqq\left(
\calP\hat{\vee} \alpha, \dd
\right) \ , $$
where $\alpha$ is a degree $-1$ element of arity $0$ placed in $F_1$ and where $\hat{\vee}$ stands for the coproduct of complete ns operads, and where the differential $\mathrm{d}$ is characterized by 
\begin{eqnarray*} 
&&\dd\alpha\coloneqq-\sum_{n\ge 2}\mu_n(\alpha, \ldots, \alpha) \ ,\\
&&\dd\nu\coloneqq d_\calP(\nu)\ , \ \text{for} \ \nu\in \calP\ .
\end{eqnarray*}
\end{proposition}

\begin{proof}The fact that the map $\dd$ extends to a unique square-zero  derivation is direct corollary of Proposition~\ref{prop:DGAction} applied to the map $\As^{\ac}\to \calP$ induced by the multiplicative ns operad structure.
The rest of the proof is straightforward from the definition of the coproduct of complete ns operads. 
\end{proof}

\begin{lemma}
The element $\mu_1^\alpha$ is a Maurer--Cartan element of the complete dg ns operad $\MC\calP$. 
\end{lemma}

\begin{proof}
The morphism of complete dg ns operads $\calA_\infty \to \calP$ induces a morphism of complete dg ns operads $\MC\calA_\infty \to \MC\calP$ and we know, by Proposition~\ref{prop:MorphMCOperadic}, that the image of an operadic Maurer--Cartan element is again an operadic Maurer--Cartan element. 
\end{proof}

\begin{definition}[Twisted multiplicative ns operad]\index{twisted operad!multiplicative ns operad}
Let $\calA_\infty \to \calP$ be a multiplicative ns operad. 
The complete dg ns operad obtained by twisting the operad~$\MC \calP$ by the Maurer--Cartan element $\mu_1^\alpha$ is called the 
 \emph{twisted complete ns operad} and denoted by 
$$\Tw\calP\coloneqq \left(\MC \calP\right)^{\mu_1^\alpha}=
\left(
\calP\hat{\vee} \alpha, \dd^{\mu_1^\alpha}
\right) \ . $$
\end{definition}

The twisted differential is actually equal to
\begin{eqnarray*}
\mathrm{d}^{\mu_1^\alpha}(\alpha)&=&\dd(\alpha)+\mu_1^\alpha(\alpha)= -\sum_{n\ge 2}\mu_n(\alpha, \ldots, \alpha)
+\sum_{n\ge 2}n\mu_n(\alpha, \ldots, \alpha)\\&=&\sum_{n\ge 2}(n-1)\mu_n(\alpha, \ldots, \alpha)\ ,\\ \mathrm{d}^{\mu_1^\alpha}(\nu)&=&d_\calP(\nu) 
+\mu_1^\alpha \star \nu - (-1)^{|\nu|}
 \nu\star \mu_1^\alpha\\
 &=&d_\calP(\nu) 
+
\sum_{\substack{n\ge 2\\ 1\leq i \leq n}} 
(-1)^{(n-i)(|\nu|+1)}
\mu_n\big(\alpha^{i-1}, \nu, \alpha^{n-i}\big)\\
&&
-
\sum_{j=1}^k\sum_{\substack{n\ge 2 \\ 1\leq i \leq n}} (-1)^{|\nu|+n-i}\, \nu \circ_j \mu_n\big(\alpha^{i-1}, -, \alpha^{n-i}\big)
\ ,
\end{eqnarray*}
for $\nu \in \calP(k)$. 

\begin{proposition}
Any complete dg $\calP$-algebra $(A, d)$ with a given Maurer--Cartan element ``$a$'' gives a complete $\Tw \calP$-algebra with underlying twisted differential $(A, d^a)$~. This assignment 
defines a functor $\MC \calP\emph{\textsf{-alg}} \to \Tw \calP\emph{\textsf{-alg}}$, 
which is an isomorphism of categories.
\end{proposition}

\begin{proof}
Let us reproduce here the argument given above in the $\Ai$ case. The data of a complete $\calP$-algebra structure on  $(A,\F, d)$ with a Maurer--Cartan element $a$ is equivalent to a morphism of complete dg ns operads~$\MC \calP\to \eend_{(A,d)}$. Since the element $\mu_1^\alpha$ is a Maurer--Cartan element in the complete ns operad $\calP$, we get a morphism between the twisted complete dg ns operads $$\Tw \calP \to \End_{(A,d+m_1^a)}.$$ The other way round, one uses the exact same arguments but starting with the Maurer--Cartan element $-\mu_1^\alpha$ of the complete dg ns operad $\Tw \calP$, see Proposition~\ref{prop:-MC}. This shows that a complete algebra structure over the complete dg ns operad $\Tw \calP \to \End_{(A,d)}$
induces a complete algebra structure over the complete dg ns operad 
$$\MC \calP=\left(\Tw\calP\right)^{-\mu_1^\alpha} \to \End_{(A,d-m_1^a )}\ , $$
by Proposition~\ref{prop:-MC}. These two functors $\MC \calP\textsf{-alg} \to \Tw \calP\textsf{-alg}$ and 
$\Tw \calP\textsf{-alg} \to \MC \calP\textsf{-alg}$ are  inverse to each other.
\end{proof}

\begin{proposition}\label{prop:MorhpAiMCP}
The assignment $\mu_n\mapsto \mu_n^\alpha$ defines a morphism of complete dg ns operads 
$$\Ai \to \Tw \calP \ . $$
\end{proposition}

\begin{proof}
The morphism $\MC\calA_\infty \to \MC\calP$ of complete dg ns operads induces a morphism between the twisted complete dg ns operads $\Tw\calA_\infty \allowbreak \to \Tw\calP$ by Proposition~\ref{prop:MorphMCOperadic}. It just remains to pull back by the map 
$\Ai \to \Tw\Ai$ of Proposition~\ref{prop:MorhpAiMCAi}.
\end{proof}

This proposition shows that the upshot of the operadic twisting procedure gives again a multiplicative ns operad. From now on, we will call the category of complete
dg ns operads under $\Ai$ the \emph{category of multiplicative ns operads}.

\begin{lemma}
The twisting procedure defines an endofunctor on the category of multiplicative ns operads.
\end{lemma}

\begin{proof}
Proposition~\ref{prop:MorhpAiMCP} shows that the result of the operadic twisting lives in the category of multiplicative ns operads. 
Given a morphism$$\xymatrix{&\Ai\ar[dl]\ar[dr]&\\
\calP\ar[rr]^f &&\calQ}$$
of multiplicative ns operads, we define a morphism of complete dg ns operads $\Tw f :  \Tw\calP \to \Tw\calQ$ by $
\alpha \mapsto \alpha$ and 
$\nu \mapsto f(\nu)$, for $\nu \in \calP$.
Then, the compatibility relations $\Tw \id_\calP=  \id_{\Tw\calP}$ and  $\Tw (f\circ g)=\Tw f \circ \Tw g$ are automatic. 
\end{proof}

We recall an important result on homotopy invariance of the endofunctor~$\Tw$. 

\begin{proposition}[{\cite[Theorem~$5.1$]{DolgushevWillwacher15}}]\label{prop:TwQI}
The endofunctor $\Tw$ preserves quasi-isomorphisms. 
\end{proposition}

The following result is an operadic version  of the fact that $\alpha+\beta$ is a Maurer--Cartan in an algebra if and only if $\beta$ is a Maurer--Cartan element in the algebra twisted by $\alpha$. Then the algebra twisted first by $\alpha$ and then by $\beta$ is equal to the (dg Lie) algebra twisted by $\alpha+\beta$, see Proposition~\ref{lem:subGaugeGroup}. 

\begin{lemma}\label{Lemma=TwTw}
Let $\calP$ be a multiplicative ns operad. The complete dg ns operad $\Tw (\Tw \calP)$ is isomorphic to 
$$\Tw \big(\Tw \calP\big)\cong  \big(\calP\hat{\vee} \, \alpha\, \hat{\vee}\,  \beta,  \dd+\ad_{\mu_1^{\alpha+\beta}},  \left\{ \circ_i \right\}\big)\ .$$
\end{lemma}

\begin{proof}
The only point to check is the  twisted differential.
First, one shows that the 
$\mu_1^{\alpha+\beta}$ 
is equal to  $\mu_1^\alpha+\widetilde{\mu}_1^{\, \beta}$, where we denote here the twisted operations by $\widetilde{\mu}_n\coloneqq\mu_n^\alpha$:
\begin{eqnarray*}
\mu_1^\alpha+\widetilde{\mu}_1^{\, \beta}&=&
\sum_{\substack{k\ge 2 \\ 1\leq j \leq k}}(-1)^{k-j} \, \mu_k\big(\alpha^{j-1}, -, \alpha^{k-j}\big)+
\sum_{\substack{n\ge 2 \\ 1\leq i \leq n}}(-1)^{n-i} \, \mu_n^\alpha\big(\beta^{i-1}, -, \beta^{n-i}\big)\\
&=&
\sum_{\substack{k\ge 2 \\ 1\leq j \leq k}}(-1)^{k-j} \, \mu_k\big(\alpha^{j-1}, -, \alpha^{k-j}\big)\\
&& +\sum_{\substack{n\ge 2 \\ 1\leq i \leq n}}(-1)^{n-i} \,\sum_{r_0, \ldots, r_n\ge 0} (-1)^{r_i+\cdots +r_n} \mu_{n+r_0+\cdots+r_n}\big(\alpha^{r_0}, \beta, \alpha^{r_1}, \beta, \ldots, \alpha^{r_{i-1}}, - , \\
&&\qquad\qquad\qquad\qquad\qquad\qquad\qquad\qquad\quad  \alpha^{r_{i_1}}, \beta, \ldots, \alpha^{r_{n-1}}, \beta,\alpha^{r_n}  \big)\\
&=&
\sum_{\substack{k\ge 2 \\ 1\leq j \leq k}}(-1)^{k-j} \, \mu_k\big(\alpha^{j-1}, -, \alpha^{k-j}\big)\\
&&+\sum_{\substack{n\ge 2 \\ 1\leq i \leq n}}\,\sum_{r_0, \ldots, r_n\ge 0} (-1)^{n-i+r_i+\cdots +r_n} \mu_{n+r_0+\cdots+r_n}\big(\alpha^{r_0}, \beta, \alpha^{r_1}, \beta, \ldots, \alpha^{r_{i-1}}, - , \\
&&\qquad\qquad\qquad\qquad\qquad\qquad\qquad\qquad\quad 
\alpha^{r_{i_1}}, \beta, \ldots, \alpha^{r_{n-1}}, \beta,\alpha^{r_n}  \big)\\
&=& \sum_{\substack{k\ge 2 \\ 1\leq j \leq k}}(-1)^{k-j} \, \mu_k\big((\alpha+\beta)^{j-1}, -, (\alpha+\beta)^{k-j}\big)\\
&=&\mu_1^{\alpha+\beta}\ .
\end{eqnarray*}
Then, one concludes the proof with 
$$\ad_{\mu_1^\alpha}+\ad_{\widetilde{\mu}_1^{\, \beta}}=\ad_{\mu_1^\alpha+\widetilde{\mu}_1^{\, \beta}}=\ad_{\mu_1^{\alpha+\beta}} \ .$$
\end{proof}

\begin{corollary}\label{cor:ComStr}
Given any multiplicative ns operad $\calP$, the assignment 
\begin{eqnarray*}
\begin{array}{rrcl}
\Delta(\calP)\ \ \  : &\Tw \calP\cong \calP\hat{\vee} \alpha&\to &\Tw\, \big(\Tw \calP\big)\cong \calP\hat{\vee} \alpha \hat{\vee} \beta\\
&\alpha &\mapsto & \alpha+\beta\\
&\nu&\mapsto&\nu\ ,
\end{array}
\end{eqnarray*}
for $\nu\in\calP$, 
defines a morphism of multiplicative ns operads. 
\end{corollary}

\begin{proof}
Again, the only point to check is the compatibility with the differential. Let us denote the above given morphism by $f : \Tw \calP \to \Tw \big(\Tw \calP\big)$. For any element $\nu \in \calP$, we have
\begin{eqnarray*}
f\left(
d_\calP(\nu)+\ad_{\mu_1^\alpha}(\nu)
\right)
=d_\calP(f(\nu))+\ad_{\mu_1^{\alpha+\beta}}(f(\nu))\ ,
\end{eqnarray*}
and
\begin{eqnarray*}
f\circ\big(
\dd+\ad_{\mu_1^\alpha}
\big)(\alpha)&=&
f\left(\sum_{n\ge 2}(n-1)\mu_n(\alpha, \ldots, \alpha)\right)\\
&=&
\sum_{n\ge 2}(n-1)\mu_n(\alpha+\beta, \ldots, \alpha+\beta)\\
&=&\left(\dd +\ad_{\mu_1^{\alpha+\beta}}\right)\circ (\alpha+\beta)
= \left(\dd +\ad_{\mu_1^{\alpha+\beta}}\right)\circ f(\alpha)\ . 
\end{eqnarray*}
We conclude with Lemma~\ref{Lemma=TwTw}.
\end{proof}

Similarly to the case of the twisted $\Ai$ operad, there is a morphism 
of multiplicative ns operads 
$\varepsilon(\calP)\colon \Tw \calP \to \calP$ defined by sending $\alpha$ to $0$ and $\nu\in \calP$ to $\nu$. 

\begin{theorem}
The two morphisms $$\Delta(\calP) \colon  \Tw \calP \to \Tw(\Tw\calP)$$ and $$\varepsilon(\calP)\colon \Tw \calP \to \calP$$ of multiplicative ns operads provide the endofunctor $\Tw$ with a comonad structure. 
\end{theorem}

\begin{proof}
Let  $\Ai\to\calP$ be a  multiplicative ns operad, we have to check the counit relations 
\begin{eqnarray*}
\big(\varepsilon(\calP)\circ(\id_{\Tw \calP})\big)
(\Delta(\calP))=\id_{\Tw \calP}= 
\big((\id_{\Tw \calP})\circ \varepsilon(\calP)\big)
\end{eqnarray*}
and the coassociativity relation 
\begin{eqnarray*}
\Delta_{\Tw \calP}(\Delta(\calP))
=\Tw(\Delta_{\calP})(\Delta(\calP))\ .
\end{eqnarray*}
In each cases, the image of any element $\nu \in \calP$ is send to itself. The left counit relation is given by 
$$\alpha \mapsto \alpha +\beta \mapsto \alpha$$
since the second morphism sends $\alpha$ to $\alpha$ and $\beta$ to $0$. The right counit relation is proved similarly since the second morphism sends $\alpha$ to $0$ and $\beta$ to $\alpha$.
Both sides for the coassociativity relation give 
$$\alpha \mapsto \alpha +\beta \mapsto \alpha+\beta+\gamma \ , $$
which concludes the proof. 
\end{proof}

\begin{definition}[$\Tw$-stable operad]\index{$\Tw$-stable operad!non-symmetric}
A multiplicative ns operad is called \emph{$\Tw$-stable operad}, pronounced ``twist-stable'',  if it admits a $\Tw$-coalgebra structure. 
\end{definition}

This definition means that, given a multiplicative ns operad $\Ai\to \calP$, there exists a morphism $\Delta_\calP$ of multiplicative ns operads 
$$\xymatrix@C=20pt@R=20pt{&\Ai\ar[dl]\ar[dr]&\\
\calP\ar[rr]_{\Delta_\calP} &&\Tw \, \calP}$$
 satisfying 
\begin{eqnarray}
&&\xymatrix@C=30pt{\calP \ar[r]^(0.42){\Delta_\calP} \ar@/_1pc/[rr]_{\id_\calP}& \Tw\, \calP \ar[r]^(0.56){\varepsilon(\calP)}& \calP \ ,} \label{Rel1}\\
&&\xymatrix@C=30pt@R=30pt{\calP \ar[r]^{\Delta_\calP} \ar[d]^{\Delta_\calP} &   \Tw\, \calP \ar[d]^{\Tw(\Delta_\calP)} \\ 
\Tw\, \calP \ar[r]^(0.45){\Delta(\calP)}& \Tw(\Tw\, \calP)\ .}\label{Rel2}
\end{eqnarray}

In terms of type of algebras, the map $\Delta_\calP$ gives a concrete way to produce functorial complete $\calP$-algebra structures on any complete $\calP$-algebra endowed with a Maurer--Cartan element coming from the internal $\Ai$-algebra structure but with twisted differential (hence the terminology ``$\Tw$-stable'': $\calP$-algebra structures are stable under twisting). Indeed, as explained above, the data of a complete $\calP$-algebra structure with a Maurer--Cartan element $a$ is faithfully encoded in 
a morphism of dg ns operads $\MC\calP \to \eend_{(A,d)}$, which gives rise to a morphism of twisted complete dg ns operads 
$\Tw\calP \to \End_{(A,d^a)}$ by Proposition~\ref{prop:MorphMCOperadic}. Pulling back with the morphism of complete dg ns operads $\Delta(\calP) : \calP \to \Tw\calP$ produces the twisted $\calP$-algebra structure. \\

The fact that the structure map $\Delta_\calP$ is a morphism of multiplicative ns operads says that one has to twist the $\Ai$-operations in $\calP$ as usual, that is according to the formulas given in Proposition~\ref{prop:TwCurvedGauge}.
Relation~(\ref{Rel1}) expresses the fact that the twisted operation associated to any  $\nu\in \calP$ is the sum of two terms: the first one being equal to $\nu$ itself and the second one begin the sum of perturbation terms which all contain at least one Maurer--Cartan element. Relation~(\ref{Rel2}) amounts to say that the operations twisted twice under the same formulas, first by a Maurer--Cartan element $a$ and then by a second Maurer--Cartan element $b$, are equal to the operations twisted once by the Maurer--Cartan element $a+b$, thanks to Lemma~\ref{Lemma=TwTw} and Corollary~\ref{cor:ComStr}. These are the constrains of $\Tw$-stable multiplicative ns operad. 

\begin{example}
The ns operad $\Ai$ is the prototypical example of a $\Tw$-stable ns operad. Its $\Tw$-coalgebra structure map $\Ai \to \Tw\, \Ai$ is given by Proposition~\ref{prop:MorhpAiMCAi}.  
\end{example}

\begin{proposition}\label{prop:Twistable}
Let $\calP$ be a multiplicative ns operad with zero differential. 
If $\calP$ is $\Tw$-stable, then any element $\nu\in \calP$ satisfies 
 \[
\ad_{\mu_2(\alpha, -)-\mu_2(-, \alpha)}(\nu)=0
 \]
 in $\Tw\calP$.  
When the multiplicative structure of $\calP$ factors through the canonical resolution 
$\Ai \twoheadrightarrow \As \to \calP$, the  reverse statement  holds true.
\end{proposition}

\begin{proof}
If the operad $\calP$ is $\Tw$-stable, it  admits a morphism of multiplicative ns operads 
$\Delta_\calP: \calP \to \Tw\calP$ satisfying the commutative diagrams \eqref{Rel1} and \eqref{Rel2}. 
So the image of any element $\nu\in \calP(n)$ has the form 
\[\Delta_\calP(\nu)=\nu+\sum_{k\ge 1} \omega_k\ ,\]
where $\omega_k$ is a finite sum of elements of $\calP(n+k)$ composed with $k$ elements $\alpha$. The compatibility with respect to the differentials shows that 
\[
\dd^{\mu_1^\alpha}\big(\Delta_\calP(\nu)\big)=
\ad_{\mu_2(\alpha, -)-\mu_2(-, \alpha)}(\nu)+\sum_{k\ge 2} \widetilde{\omega}_k=0\ , 
\]
where $\widetilde{\omega}_k$ is a finite sum of elements of $\calP(n+k)$ composed with $k$ elements $\alpha$. Therefore, $\ad_{\mu_2(\alpha, -)-\mu_2(-, \alpha)}(\nu)$ vanishes since it involves only one element $\alpha$.

In the other way round, the multiplicative structure of $\calP$ factors through the canonical resolution 
$\Ai \twoheadrightarrow \As \to \calP$ if and only if the elements $\mu_3, \mu_4, \ldots$ vanish in $\calP$. In this case, the twisted differential is equal to $\dd^{\mu_1^\alpha}\allowbreak=\allowbreak\ad_{\mu_2(\alpha, -)-\mu_2(-, \alpha)}$.
The condition $\ad_{\mu_2(\alpha, -)-\mu_2(-, \alpha)}(\nu)=0$,  for all $\nu\in \calP$, is equivalent to the fact  that the canonical morphism of ns operads $\calP \hookrightarrow \Tw\calP$ is a chain map. The commutative diagrams \eqref{Rel1} and \eqref{Rel2} are then straightforward to check. 
\end{proof}

In Chapter \ref{sec:Computations}, we shall consider both examples of operads that are $\Tw$-stable and examples of operads that are not.

\section{Action of the deformation complex}
Let $\calP$ be a complete dg ns operad. We consider the  total space 
$$\hom\left(\As^{\ac}, \calP\right)\coloneqq\prod_{n\ge 1} \hom\left(\As^{\ac}(n), \calP(n)\right)\cong \prod_{n\ge 1} s^{1-n}\calP(n)\ , $$ 
where we identify any map $\rho(n) : \As^{\ac}(n)\cong \End_{\k s}(n)^* \to \calP(n)$ with its images $\rho_n \coloneqq \rho\left(\nu_n\right)$. 
Let us recall from Section~\ref{subsec:CompConvAlg} that this space forms a  complete left-unital dg pre-Lie algebra and thus a complete dg Lie algebra by anti-symmetrization. \\

Any  element $\rho=(\rho_1, \rho_2, \ldots)\in \hom\left(\As^{\ac}, \calP\right)$ induces a derivation $\D_\rho$ of the complete ns operad $\calP\hat{\vee} \{\alpha\}$ by the following action on generators: 
\[
\D_\rho\colon
\begin{cases}
\alpha \mapsto -\sum_{n\ge 1 } \rho_n\big(\alpha^n\big)\ ,\\
\nu \mapsto 0, \quad \text{for} \ \nu \in \calP \ .
\end{cases}
\]
We denote by $\Der\big(\calP\hat{\vee} \alpha \big)$ the set of operadic derivations and,  
by a slight abuse of notation, we still denote by $d_\calP$ the differential on $\calP\hat{\vee} \alpha$ induced by that of~$\calP$.

\begin{lemma}\label{lem:MorphdgLie}
The assignment 
\begin{eqnarray*}
\left(\hom\left(\As^{\ac}, \calP\right), \partial, [\; , \,] \right)&\to& 
\left(\Der\big(\calP\hat{\vee} \alpha\big), [d_\calP, -], [\; ,\,]\right)\\
\rho &\mapsto& \D_\rho
\end{eqnarray*}
is a morphism of dg Lie algebras.
\end{lemma}

\begin{proof}
Notice first that the Lie bracket on the right-hand side is given by the skew-symmetrization of the following binary product $\D\circ^{\mathrm{op}}\D'\coloneqq -(-1)^{|\D||\D'|} \allowbreak \D'\circ \D$ (which individually does not produce a derivation). Since the Lie bracket on the left-hand side is given by the skew-symmetrization of the pre-Lie product $\star$, we prove that the assignment $\rho\mapsto\D_\rho$ preserves these two products. Let us consider two elements $\rho, \xi \in \hom\left(\As^{\ac}, \calP\right)$. It is enough to check the relation 
$\D_{\rho\star\xi} = \D_\rho\circ^{\mathrm{op}} \D_\xi$ 
on the generators of $\calP\hat{\vee} \alpha$: this is trivial for $\nu \in \calP$ and for $\alpha$ this is given by 
\begin{align*}
\D_{\rho\star \xi}(\alpha)&=-
\sum_{n\ge 1 } (\rho\star \xi)_n\big(\alpha^n\big)=-\sum_{\substack{n\ge 1 \\ p+q+r=n}}  
(-1)^{p(q+1)+|\xi|(p+r)}
\rho_{p+1+r}\circ_{p+1} \xi_{q} \big(\alpha^n\big)=
\\&=
-(-1)^{|\rho||\xi|}\D_\xi\circ \D_\rho(\alpha) 
=
\D_\rho\circ^{\mathrm{op}}\D_\xi(\alpha)\ .
\end{align*}
We also check the commutativity of the differentials $\D_{\partial(\rho)}=[d_\calP, \D_\rho]$ on the generators of $\calP\hat{\vee} \alpha$: this is again trivial for $\nu \in \calP$ and for $\alpha$ this is given by 
\begin{align*}
\D_{\partial(\rho)}(\alpha)&=\D_{d_\calP \circ \rho}(\alpha)=-\sum_{n\ge 1} d_\calP(\rho_n)(\alpha^n)= d_\calP\left(
\D_\rho(\alpha)\right)-(-1)^{|\rho|}\D_\rho(d_\calP(\alpha))\\&=[d_\calP, \D_\rho](\alpha)\ .
\end{align*}
\end{proof}

A morphism $\Ai \to \calP$  of  complete dg ns operads is equivalent to a degree~$-1$ element 
$\mu
 \coloneqq(0, \mu_2, \mu_3, \ldots)$, notation which agrees with that of Section~\ref{subsec=TwMultiOp}, 
satisfying the Maurer--Cartan equation 
$$ \partial \mu+\mu \star \mu=0\ ,$$ 
 in this dg pre-Lie algebra. Therefore one can twist the associated dg Lie algebra  with this Maurer--Cartan element, that is consider the twisted differential 
 $$\partial^\mu\coloneqq\partial + \ad_\mu\ . $$
 (One cannot twist the dg pre-Lie algebra, unless $\mu$ satisfies Equation~\eqref{Eq:MCspecial}, which imposes unrealistic constraints on an algebra.)
 
\begin{definition}[Deformation complex of morphisms of complete dg ns operads \cite{MerkulovVallette09I, MerkulovVallette09II}]\index{deformation complex!morphisms of complete dg ns operads}
The \emph{deformation complex} of the morphism $ \Ai\to \calP$ of complete dg ns operads is the complete  twisted dg Lie algebra 
$$\Def\big(\Ai\to \calP\big)\coloneqq\left(\hom\left(\As^{\ac}, \calP\right), \partial^\mu, [\; , \,] \right)
\ . $$
\end{definition}

\begin{proposition}\label{prop:DGAction}
The assignment 
\begin{eqnarray*}
\left(\hom\left(\As^{\ac}, \calP\right), \partial^\mu, [\; , \,] \right)&\to&
\left(\Der\big(\calP\hat{\vee} \alpha\big), [d_\calP+\D_\mu, -], [\; ,\,]
\right)\\
\rho &\mapsto& \D_\rho
\end{eqnarray*}
is a morphism of dg Lie algebras.
 In plain words, this 
defines a dg Lie action by derivation of the deformation complex $\Def\big(\Ai\to \calP\big)$ on the Maurer--Cartan operad 
$\MC\calP$. 
\end{proposition}

\begin{proof}
This is a direct corollary of the morphism of dg Lie algebras established in Lemma~\ref{lem:MorphdgLie}: the Maurer--Cartan element $\mu$ on the left-hand side is sent to the Maurer--Cartan $\D_\mu$ on the right-hand side. This proves that $d_\calP+\D_\mu$ is a square-zero derivation on the complete ns operad $\calP\hat{\vee} \alpha$. In the end, we get a morphism between the respectively twisted dg Lie algebras. 
\end{proof}

This result provides us with an alternative proof of Proposition~\ref{prop:MCP} defining the complete dg ns operad 
\[
\MC\calP\coloneqq\left(
\calP\hat{\vee} \alpha, \dd\coloneqq d_\calP+\D_\mu\right)\ .
\]

\begin{remark}
We remark that, due to the twisting, this does  not define a pre-Lie action, but just a Lie action. A conceptual reason for that is revealed in Section~\ref{sec:natural}.
\end{remark}

In order to reach the same kind of results for the complete dg ns operad $\Tw \calP$, whose differential contains one more term then that of  $\MC \calP$, 
we need to consider the following extension of  the deformation complex. 
Notice first that the dg Lie algebra action of Proposition~\ref{prop:DGAction} on the complete dg ns operad $\calP \hat{\vee} \alpha$ reduces to a dg Lie algebra action on the dg Lie algebra consisting of the arity $1$ elements 
$\big(\calP \hat{\vee} \alpha\big)(1)$. This gives rise to the following semi-direct product of dg Lie algebras: 
\begin{eqnarray*}
\left(\hom\left(\calA s^{\ac}, \calP\right), \partial, [\; , \,] \right) \ltimes 
\left(\big(\calP\hat{\vee} \alpha\big)(1), d_\calP, [\; ,\,]\right)\ .
\end{eqnarray*}

\begin{lemma}\label{lem:semidirectPreLie}
The semi-direct product dg Lie algebra above comes from the skew-sym\-me\-tri\-za\-tion of the 
semi-direct product of dg pre-Lie algebras
 \[
\left(\hom\left(\As^{\ac}, \calP\right), \partial, \star \right) \ltimes 
\left(\big(\calP\hat{\vee} \alpha\big)(1), d_\calP, \circ_1\right)
 \]
defined by the formula 
 \[
\left(\hom\left(\As^{\ac}, \calP \right)\oplus \left(\calP\hat{\vee} \alpha\right)(1), 
\partial+d_\calP, \bigstar
\right)\ ,
 \]
where 
\[
(\rho, \nu)\bigstar (\xi, \omega)\coloneqq 
\left(\rho\star \xi, \nu \circ_1 \omega - (-1)^{|\xi||\nu|}\D_\xi(\nu)
\right)\ .\]
\end{lemma}

\begin{proof}
The proof of Lemma~\ref{lem:MorphdgLie} shows that the assignement $\rho\mapsto \D_\rho$ defines a right dg pre-Lie action of $\left(\hom\left(\As^{\ac}, \calP\right), \partial, \star \right)$ on $\left(\big(\calP\hat{\vee} \alpha\big)(1), d_\calP, \circ_1\right)$. It is however not always true that dg pre-Lie actions give rise to semi-direct product dg pre-Lie algebras under  formulas like that of $\bigstar$. It is the case here since the action is by derivation, see  \cite{ManchonSaidi08} for another occurrence of this construction. If we denote the associator of a binary product $\star$ by $\mathrm{Assoc_\star}$, we have  
\begin{multline*}
\mathrm{Assoc_\bigstar\big( (\rho, \nu), (\xi, \omega), (\theta, \lambda) \big)}=
\left(
\mathrm{Assoc}_\star( \rho, \xi, \theta), \mathrm{Assoc}_{\circ_1}( \nu, \omega, \lambda)\right.\\ \left.-(-1)^{|\xi||\nu|}\D_\xi(\nu)\circ_1 \lambda -
(-1)^{|\theta|(|\nu|+|\omega|)} \D_\theta(\nu)\circ_1 \omega
\right)\ , 
\end{multline*}
which is right symmetric, see also \cite{ManchonSaidi08}. The compatibility of the differentials follows from Lemma~\ref{lem:MorphdgLie}.

Finally, the skew-symmetrization of this semi-direct product pre-Lie algebra gives
\[
\left[(\rho, \nu),  (\xi, \omega)\right]=
\left(\left[\rho, \xi\right], \left[\nu,  \omega\right] +\D_\rho(\omega)- (-1)^{|\xi||\nu|}\D_\xi(\nu)
\right)\ ,\]
which is  the formula for the Lie bracket of the semi-direct product Lie algebra. 
\end{proof}

\begin{lemma}\label{lem:MorphdgLieBIS}
 The assignment 
\begin{eqnarray*}
 \left(\hom\left(\As^{\ac}, \calP\right), \partial, [\; , \,] \right) &\to& \left(\hom\left(\As^{\ac}, \calP\right), \partial, [\; , \,] \right)
\ltimes 
\left(\big(\calP\hat{\vee} \alpha\big)(1), d_\calP, [\; ,\,]\right)
\\
\rho &\mapsto& \left(\rho, \rho_1^\alpha\right)\ ,
\end{eqnarray*}
with $\rho_1^\alpha\coloneqq
\displaystyle  \sum_{\substack{n\ge 1 \\ 1\leq i \leq n}}(-1)^{n-i} \, \rho_n(\alpha^{i-1}, -, \alpha^{n-i})$, defines a morphism of dg Lie algebras.
\end{lemma}

\begin{proof}
By Lemma~\ref{lem:semidirectPreLie}, it is enough to prove that such an assignment defines a morphism of dg pre-Lie algebras. 
To this extend, we  first prove 
\begin{equation}\label{eqn:PreLieAction}
(\rho \star \xi)^\alpha_1=\rho_1^\alpha \circ_1 \xi_1^\alpha -(-1)^{|\rho||\xi|}\D_\xi \left(\rho_1^\alpha\right)\ ,
\end{equation}
for any $\rho, \xi\in \hom\left(\calA s^{\ac}, \calP\right)$. The left-hand side is equal to 
\[
(\rho \star \xi)^\alpha_1=\sum_{\substack{ p+q+r=n\geq 1 \\ 1\leq i \leq n}}
(-1)^{p(q-1)+|\xi|(p+r)+n-i}\rho_{p+1+r}\circ_{p+1}\xi_q\left(\alpha^{i-1}, -, \alpha^{n-i}\right)\ ,
\]
which splits into three components according to the value of $i$: 

(i) when $1\leq i\leq p$, 

(ii) when $p+1\leq i\leq p+q$, and 

(iii) when $p+q+1\leq i\leq n=p+q+r$. 

\noindent
The first term on the right-hand side of \eqref{eqn:PreLieAction} corresponds to the component (ii) and the second term on the right-hand side of \eqref{eqn:PreLieAction} corresponds to the sum of the  two components (i) and (iii). Explicitly, we first have 
\begin{align*}
\rho_1^\alpha \circ_1\xi_1^\alpha&=
\sum_{\substack{p+q+r=n\geq 1\\ 1\leq j \leq q}} (-1)^{r+q-j} \rho_{p+1+r}(\alpha^p, -, \alpha^r)\circ_1\xi_q\left(\alpha^{j-1}, -, \alpha^{q-j}\right)\\
&=\sum_{\substack{p+q+r=n\geq 1\\ 1\leq j \leq q}} 
(-1)^{p(q-1)+|\xi|(p+r)+r+q-j} 
\rho_{p+1+r}\circ_{p+1} \xi_q \left(\alpha^{p+j-1}, -, \alpha^{r+q-j}\right).
\end{align*}
This gives (ii) with $i=j+p$, since then $n-i=r+q-j$. Regarding the second term on the right-hand side of Equation~\eqref{eqn:PreLieAction}, since 
$$\rho_1^\alpha=\sum_{\substack{n\ge 2 \\ 1\leq i \leq n}}(-1)^{n-i} \, \rho_n(\alpha^{i-1}, -, \alpha^{n-i})$$ and since $\D_\xi$ vanishes on $\rho_n$, for $n\geq 1$, we get two terms: the first one when $\D_\xi$ applies to the $\alpha$'s on the left-hand side of the input slot and the second one when $\D_\xi$ applies to the $\alpha$'s on the right-hand side of the input slot. The former term gives component (iii) and the latter term gives component (i). The last point of the proof  amounts to check the various signs.  Under the notation 
$\rho_1^\alpha=
\sum_{\substack{p+1+r=n\geq 1\\ 1\leq j\leq r}} 
(-1)^{r-j}
\rho_{p+1+r}\allowbreak\left(\alpha^{p+1+j-1}, \allowbreak -, \allowbreak \alpha^{r-j}\right)$, the former term becomes 
\begin{align*}
&\sum_{\substack{p+1+r=n\geq 1\\ 1\leq j\leq r}} 
(-1)^{r-j+|\xi|r}
\rho_{p+1+r}\left(\alpha^{p},\xi_q\left(\alpha^q\right),\alpha^{j-1}, -, \alpha^{r-j}\right)\\
&=
\sum_{\substack{p+1+r=n\geq 1\\ 1\leq j\leq r}} 
(-1)^{p(q-1)+|\xi|(p+r)+r-j}
\rho_{p+1+r}\circ_{p+1}\xi_q\left(\alpha^{p+q+j-1}, -, \alpha^{r-j}\right)\ , 
\end{align*}
which  is equal to (iii) with $i=j+p+q$.
Under the notation 
$$\rho_1^\alpha \allowbreak = \allowbreak 
\sum_{\substack{p+1+r=n\geq 1\\ 1\leq j\leq p}}\allowbreak 
(-1)^{p+1+r-j}\allowbreak
\rho_{p+1+r}\allowbreak \left(\alpha^{j-1},-, \alpha^{p+1+r-j}\right),$$ 
the latter term becomes 
\begin{align*}
&\sum_{\substack{p+1+r=n\geq 1\\ 1\leq j\leq r}} 
(-1)^{p+1+r-j+|\xi|(r-1)}
\rho_{p+1+r} \left(\alpha^{j-1},-, \alpha^{p}, \xi_q\left(\alpha^q\right),\alpha^{r-j}\right)\\
&=
\sum_{\substack{p+1+r=n\geq 1\\ 1\leq j\leq r}} 
(-1)^{p(q-1)+|\xi|(p+r)+p+q+r-j}
\rho_{p+1+r}\circ_{p+1}\xi_q\left(\alpha^{j-1},-, \alpha^{p+q+r-j}\right)\ , 
\end{align*}
which  is equal to (i) with $i=j$.

The commutativity of the differentials comes from the relation 
$$\left(d_\calP (\rho)\right)_1^\alpha=d_\calP\left(\rho_1^\alpha\right),$$ which is straightforward. 
\end{proof}

\begin{lemma}\label{lem:MorphdgLieTER}
The assignment 
\begin{eqnarray*}
\left(\hom\left(\As^{\ac}, \calP\right), \partial, [\; , \,] \right)
\ltimes 
\left(\big(\calP\hat{\vee} \alpha\big)(1), d_\calP, [\; ,\,]\right)
&\to& 
\left(\Der\big(\calP\hat{\vee} \alpha\big), [d_\calP, -], [\; ,\,]\right)\\
(\rho, \nu) &\mapsto& \D_\rho+\ad_\nu
\end{eqnarray*}
defines a morphism of dg Lie algebras.
\end{lemma}

\begin{proof}
The compatibility with respect to the Lie brackets amounts to proving that 
\begin{align*}
\D_{[\rho,\xi]}+\ad_{[\nu, \omega]+\D_\rho(\omega)-(-1)^{|\xi||\nu|}\D_\xi(\nu)}
=
[\D_\rho, \D_\xi]+[\ad_\nu, \ad_\omega]+[\D_\rho, \ad_\omega]+[\ad_\nu, \D_\xi]\ .
\end{align*}
The first two terms are equal by Lemma~\ref{lem:MorphdgLie}. The second two terms are equal since the adjoint action is always a morphism of Lie algebras. The relation $[\D_\rho, \ad_\omega]=\ad_{\D_\rho(\omega)}$ is also
 a general
 property of semidirect products. 

After Lemma~\ref{lem:MorphdgLie}, in order to prove the compatibility with respect to the differentials, it remains to show that 
$\ad_{d_\calP(\nu)}=[d_\calP, \ad_\nu]$, which come from the fact that $d_\calP$ is an operadic derivation. 
\end{proof}

\begin{theorem}\label{thm:DGActionOnTw}
The assignment 
\begin{eqnarray*}
\left(\hom\left(\As^{\ac}, \calP\right), \partial^\mu, [\; , \,] \right)&\to&
\left(\Der\big(\calP\hat{\vee} \alpha\big), [d_\calP+\D_\mu+\ad_{\mu_1^\alpha}, -], [\; ,\,]
\right)\\
\rho &\mapsto& \D_\rho+\ad_{\rho_1^\alpha}
\end{eqnarray*}
is a morphism of dg Lie algebras.
 In plain words, this 
defines a dg Lie action by derivation of the deformation complex $\Def\big(\Ai\to \calP\big)$ on the twisted operad 
$\Tw\calP$. 
\end{theorem}

\begin{proof}
The arguments are the same as in the proof of Proposition~\ref{prop:DGAction}, using 
the composite of dg Lie algebra morphisms given respectively  in Lemma~\ref{lem:MorphdgLieBIS} and in Lemma~\ref{lem:MorphdgLieTER}, and twisting in the end by the Maurer--Cartan element~$\mu$.
\end{proof}

\begin{remark}\label{rem:}
Theorem~\ref{thm:DGActionOnTw} gives another way, actually the original one from \cite[Appendix~I]{Willwacher15}, to define the twisted complete dg ns operad 
\[
\Tw\calP\coloneqq\left(
\calP\hat{\vee} \alpha, \dd^{\mu_1^\alpha}\coloneqq d_\calP+\D_\mu+\ad_{\mu_1^\alpha}\right)\ .
\]
\end{remark}

The relationship between the deformation complex and the twisted operad is actually a bit more rich. 

\begin{proposition}\label{prop:DefTwP0}
Let $\Ai\to \calP$ be a multiplicative ns operad satisfying  $\calP(0)=0$. Up to a degree shift, the deformation complex is isomorphic to the chain complex equal to the arity $0$ component of the twisted operad: 
$$ 
\left(\left(\Tw\, \calP\right) (0), \dd^{\mu_1^\alpha}\right)\cong \left(s^{-1}\hom\left(\As^{\ac}, \calP\right), \partial^\mu\right)
\ .$$
\end{proposition}

\begin{proof}
On the level of the  underlying spaces, these two chain complexes satisfy 
\begin{eqnarray*}
\left(\Tw\, \calP\right) (0)=
\prod_{n\ge 0} \calP(n)\otimes \alpha^{\otimes n}\cong 
\prod_{n\ge 0} s^{-n}\calP(n)
\cong
 s^{-1}\hom\left(\As^{\ac}, \calP\right)\bigoplus \calP(0)\ .
\end{eqnarray*}
When $\calP(0)=0$, one way to realise the  isomorphism $\hom\left(\As^{\ac}, \calP\right) \stackrel{\cong}{\to} s\left(\Tw\, \calP\right) (0)$ is given   by 
$$\rho \mapsto (-1)^{|\rho|}s \D_\rho(\alpha)=-\sum_{n\ge 1} (-1)^{|\rho|}\rho_n\big(\alpha^{\otimes n}\big)\ .$$ 
It remains to show that it commutes with the respective differentials, that is to prove the following relation:
\[\D_{\partial(\rho)}(\alpha)+\D_{[\mu, \rho]}(\alpha)=d_\calP\left( \D_\rho(\alpha) \right)+\D_\mu\left(\D_\rho(\alpha)\right)+\ad_{\mu_1^\alpha}\left( \D_\rho(\alpha)\right)\ .\]
We have already seen in Lemma~\ref{lem:MorphdgLie} that $\D_{\partial(\rho)}(\alpha)=d_\calP\left( \D_\rho(\alpha) \right)$ and that 
\[\D_{[\mu, \rho]}(\alpha)=\D_\mu\left(\D_\rho(\alpha)\right)-(-1)^{|\rho|}\D_\rho\left(\D_\mu(\alpha)\right)~.\] So it remains to show that $\ad_{\mu_1^\alpha}\left( \D_\rho(\alpha)\right)=\allowbreak-(-1)^{|\rho|}\allowbreak\D_\rho\left(\D_\mu(\alpha)\right)$, which comes from the fact that both are explicitly equal to 
\[-\sum_{\substack{p, r\geq 0\\ q\leq 1}} (-1)^{|\rho|r}\mu_{p+1+r}\left(\alpha^p, \rho_q\left(\alpha^q\right), \alpha^r\right)\ .\]
\end{proof}

\section{Twisting of symmetric operads}\label{sec:SymTw}

Section~\ref{subsec=TwMultiOp} deals with the twisting procedure for ns operads where we used in a crucial way the dg ns operad 
$\Ai$. The  entire same theory holds as well with the dg ns operad $\sAi\coloneqq {\End}_{\k s}\otimes \Ai$ encoding shifted $\Ai$-algebras; in this case, the signs are nearly all  trivial. In order to get the twisting procedure for (symmetric) operads, one would similarly start with the dg operad $\Li$ encoding homotopy Lie algebras or the dg operad $\sLi\coloneqq {\End}_{\k s}\otimes\Li$ encoding shifted homotopy Lie algebras. The various proofs are performed with similar computations, and thus are left to the reader. In this way, one gets the theory developed by T. Willwacher but with a presentation different from \cite{Willwacher15, DolgushevRogers12, DolgushevWillwacher15}. The present approach was first proposed by J. Chuang and A. Lazarev in \cite{MR3004818}.  \\

Let us now give a summary of key results in this case. We shall use the operad $\sLi$ of shifted homotopy Lie algebras, since it has the ``nicest'' homological degrees of generators: all generators of degree $-1$. The proofs are completely analogous to the corresponding proofs given in the previous three sections and are omitted. For the reader who wishes to apply the formulas of this section, we note, however, that shifted Lie brackets arising in applications usually have homological degree $1$, as they correspond to the circle action on the level of homology. This means that our formulas have all the correct parities of homological degrees (and therefore the correct signs), but the degrees themselves have to be adjusted: the relevant operad for many applications is the operad $\calS^{-1}\Li$ whose generator $\lambda_n$ is of degree $2n-3$ for each $n\ge 2$~. This will be used in most of examples treated in Chapters \ref{sec:Computations}  and \ref{sec:Applications}. 

\begin{proposition}
The complete dg operad encoding the data of a shifted homotopy Lie algebra together with a Maurer--Cartan element is 
\index{operad!$\MC\sLi$}
$$\MC\sLi\coloneqq\left(
\widehat\calT\big(\alpha, \lambda_2, \lambda_3, \ldots\big), \dd
\right) \ , $$
where the generator $\alpha$ has arity $0$ and degree $0$ and 
where the generator $\lambda_n$ has arity $n$, degree $-1$, and trivial $\Sy_n$-action, for $n\ge 2$, 
where the filtration on the space $M=\big(\k \alpha, 0, \k \lambda_2, \k \lambda_3, \dots\big)$ of generators is given by 
\[
\alpha \in \F_1 M(0), \ \F_2 M(0)=\{0\} \quad \text{and} \quad \lambda_n\in \F_0 M(n),\  \F_1 M(n)=\{0\}\ , \ \text{for} \ n\ge 2
\ , \]
 and where the differential is defined by 
\begin{eqnarray*} 
&&\dd \lambda_n\coloneqq
-\sum_{\substack{p+q=n+1\\ 2\leq p,q \leq n}}
\sum_{\sigma\in \mathrm{Sh}_{p,q}^{-1}}
 (\lambda_{p+1}\circ_{1} \lambda_q)^{\sigma}\ ,
\\
&&\dd \alpha \coloneqq-\sum_{n\ge 2}\frac{1}{n!}\lambda_n(\alpha, \ldots, \alpha) \ .
\end{eqnarray*}

\end{proposition}

The operad $\MC\sLi$ admits the following Maurer--Cartan element 
\[
\lambda_1^\alpha\coloneqq\sum_{n\ge 2} {\textstyle  \frac{1}{(n-1)!}}\lambda_n\left(\alpha^{n-1}, -\right)\ .
\]

\begin{definition}[Twisted $\sLi$-operad]
The twisted $\sLi$-operad is 
\index{operad!$\Tw\sLi$}
\[
\Tw\sLi\coloneqq \left(\MC\sLi\right)^{\lambda_1^\alpha}\cong 
\left(
\widehat\calT\big(\alpha, \lambda_2, \lambda_3, \ldots\big), \dd^{\lambda_1^\alpha}
\right)\ ,
\]
with 
$\dd^{\lambda_1^\alpha}(\alpha)=\sum_{n\ge 2} {\textstyle  \frac{n-1}{n!}}\lambda_n\left(\alpha^{n}\right)$.
\end{definition}

\begin{proposition}
The assignment $\lambda_n\mapsto \lambda_n^\alpha$, where 
\[
\lambda_n^\alpha\coloneqq\sum_{r\ge 0} {\textstyle  \frac{1}{r!}}\lambda_{n+r}\left(\alpha^{r}, -, \ldots, -\right)\ .
\]
 defines a morphism of complete dg  operads 
$$\sLi \to \Tw \sLi \ . $$
\end{proposition}

The following proposition, as well as the definition that follows, goes back to the work of J. Chuang and A. Lazarev~\cite{MR3004818} who first noticed that the theory of operadic twisting can be presented in this way.

\begin{proposition}\label{prop:MCPBIS}
Let $\sLi \to \calP$ be a morphism of complete dg operads. 
The data of a complete $\calP$-algebra structure together with a Maurer--Cartan element is encoded by the complete dg operad 
$$\MC\calP\coloneqq\left(
\calP\hat{\vee} \alpha, \dd
\right) \ , $$
where $\alpha$ is a degree $0$ element of arity $0$ placed in $F_1$ and where $\hat{\vee}$ stands for the coproduct of complete operads, and where the differential $\mathrm{d}$ is characterized by 
\begin{eqnarray*} 
&&\dd\alpha\coloneqq-\sum_{n\ge 2}{\textstyle  \frac{1}{n!}}\lambda_n(\alpha, \ldots, \alpha) \ ,\\
&&\dd\nu\coloneqq d_\calP(\nu)\ , \ \text{for} \ \nu\in \calP\ .
\end{eqnarray*}
\end{proposition}

\begin{definition}[Twisted operads under $\sLi$]\index{twisted operad!under $\sLi$}\label{def:TwistOp}
Let $\sLi \to \calP$ be a morphism of complete dg operads. 
The complete dg  operad obtained by twisting the operad $\MC \calP$ by the Maurer--Cartan $\lambda_1^\alpha$ is called the 
 \emph{twisted complete  operad} and denoted by 
$$\Tw\calP\coloneqq \left(\MC \calP\right)^{\lambda_1^\alpha}=
\left(
\calP\hat{\vee} \alpha, \dd^{\lambda_1^\alpha}
\right) \ . $$
\end{definition}

The differential of the operad $\Tw\calP$ is actually equal to
\begin{eqnarray*}
\mathrm{d}^{\lambda_1^\alpha}(\alpha)&=&\sum_{n\ge 2}{\textstyle  \frac{n-1}{n!}}\lambda_n(\alpha, \ldots, \alpha)\ ,\\ \mathrm{d}^{\lambda_1^\alpha}(\nu)&=&d_\calP(\nu) 
+
\sum_{n\ge 2} {\textstyle  \frac{1}{(n-1)!}}\lambda_n\left(\alpha^{n-1}, \nu\right)
-
(-1)^{|\nu|}\sum_{j=1}^k {\textstyle  \frac{1}{(n-1)!}}\nu \circ_j \lambda_n(\alpha^{n-1}, -)\ ,
\end{eqnarray*}
for $\nu \in \calP(k)$. 

\begin{proposition}\label{prop:MorhpAiMCPBIS}
The assignment $\lambda_n\mapsto \lambda_n^\alpha$ defines a morphism of complete dg operads 
$$\sLi \to \Tw \calP \ . $$
\end{proposition}

We consider the following morphisms of complete dg operads
\begin{eqnarray*}
\begin{array}{rrcl}
\Delta(\calP)\ \ \  : &\Tw \calP\cong \calP\hat{\vee} \alpha&\to &\Tw\, \big(\Tw \calP\big)\cong \calP\hat{\vee} \alpha \hat{\vee} \beta\\
&\alpha &\mapsto & \alpha+\beta\ ,\\
&\nu&\mapsto&\nu\ , \ \text{for}\ \nu\in\calP\ ,
\end{array}
\end{eqnarray*}
and 
\begin{eqnarray*}
\begin{array}{rrcl}
\varepsilon(\calP)\ \ \  : &\Tw \calP&\to & \calP \\
&\alpha &\mapsto & 0\ ,\\
&\nu&\mapsto&\nu\ , \ \text{for}\ \nu\in\calP\ .
\end{array}
\end{eqnarray*}

\begin{theorem}
The two morphisms 
 \[
\Delta(\calP) \colon  \Tw \calP \to \Tw(\Tw\calP) \text{  and } \varepsilon(\calP)\colon \Tw \calP \to \calP
 \]
 of operads under $\sLi$ provide the endofunctor $\Tw$ with a comonad structure. 
\end{theorem}

\begin{definition}[$\Tw$-stable operad]\index{$\Tw$-stable operad!symmetric}
A complete symmetric operad under $\sLi$ is called \emph{$\Tw$-stable} if it admits a $\Tw$-coalgebra structure. 
\end{definition}

\begin{example}
Let us recall from \cite{DolgushevWillwacher15} that the operads for Lie algebras and Gerstenhaber algebras, as well as their shifted versions and their minimal models, are $\Tw$-stable, see Proposition~\ref{prop:GerstTwistable}. One example of an operad that is not $\Tw$-stable is the operad of pre-Lie algebras. 
\end{example}

\begin{proposition}\label{prop:TwQIBIS}
The endofunctor $\Tw$ preserves quasi-isomorphisms. 
\end{proposition}

Let $\calP$ be a complete dg operad. We consider the  total space 
$$\hom_\Sy\left(\big(\calS\Lie\big)^{\ac}, \calP\right)\coloneqq\prod_{n\ge 1} \hom_\Sy\left(\Com(n)^*, \calP(n)\right)\cong \prod_{n\ge 1} \calP(n)^{\Sy_n}\ .$$ 

\begin{definition}[Deformation complex of morphisms of complete dg operads \cite{MerkulovVallette09I, MerkulovVallette09II}]\index{deformation complex!morphisms of complete dg operads}
The \emph{deformation complex} of the morphism $ \sLi\to \calP$ of complete dg  operads is the complete  twisted dg Lie algebra 
$$\Def\big(\sLi\to \calP\big)\coloneqq\left(\hom_\Sy\left(\Com^*, \calP\right), \partial^\lambda, [\; , \,] \right)
\ . $$
\end{definition}

For any  element $\rho=(\rho_1, \rho_2, \ldots)\in \hom_\Sy\left(\Com^*, \calP\right)$, we define a derivation $\D_\rho$ of the complete operad $\calP\hat{\vee} \{\alpha\}$ by the following action on generators:
\[
\D_\rho\colon
\begin{cases}
\alpha \mapsto -\sum_{n\ge 1 } \rho_n\big(\alpha^n\big)\ ,\\
\nu \mapsto 0, \quad \text{for} \ \nu \in \calP \ .
\end{cases}
\]

\begin{theorem}\label{thm:DGActionOnTwBIS}
The assignment 
\begin{eqnarray*}
\left(\hom_\Sy\left(\Com^*, \calP\right), \partial^\lambda, [\; , \,] \right)&\to&
\left(\Der\big(\calP\hat{\vee} \alpha\big), [d_\calP+\D_\lambda+\ad_{\lambda_1^\alpha}, -], [\; ,\,]
\right)\\
\rho &\mapsto& \D_\rho+\ad_{\rho_1^\alpha}
\end{eqnarray*}
is a morphism of dg Lie algebras, that is it  
defines a dg Lie action by derivation of the deformation complex $\Def\big(\sLi\to \calP\big)$ on the twisted operad 
$\Tw\calP$. 
\end{theorem}

\begin{proposition}\label{prop:DefTwP0BIS}
Let $\sLi\to \calP$ be a morphism of complete dg operads with  $\calP(0)=0$. The deformation complex is isomorphic to the chain complex equal to the arity $0$ component of the twisted operad: 
$$ 
\left(\left(\Tw\, \calP\right) (0), \dd^{\lambda_1^\alpha}\right)\cong \left(\hom_\Sy\left(\Com^*, \calP\right), \partial^\lambda\right)
\ .$$
\end{proposition}

\section{Generalisations}\label{Sec:Gen}

Let us mention briefly several possible generalisations and variations of the theory presented in this chapter. \\

Perhaps the easiest generalisation of the formalism developed here concerns replacing (ns) operads by coloured (ns) operads. For example, if one considers a ns coloured operad $\calP$ into which the operad $\Ai$ (with all inputs and the output of the same colour) maps, the theory of operadic twisting effortlessly adapts in this case, allowing one to recover some classical constructions. For example, if one considers the cofibrant replacement of the coloured operad encoding the pairs $(A,M)$ where $A$ is an associative algebra and $M$ is a left $A$-module, the arising twisted differentials have, for example, been studied by T. Kadeishvili~\cite{Kadeishvili80} in the context of $\infty$-twisted tensor products.  \\

Another possible direction in which one can generalise the twisting procedure is the case of a quadratic Koszul operad $\calP=\calP(E,R)$ whose associated category of $\calP_\infty$-algebras is twistable, according to Definition~\ref{def:TwHoAlg}, that happens when the Koszul dual operad $\calP^{\ac}$ is extendable. It turns out that this more general situation is much more subtle. The various proofs given above, like the one of Lemma~\ref{lem:MorphdgLie}, rely on the crucial fact that the ns cooperad $\As^{\ac}$ is one-dimensional in any arity and that its partial coproduct is the sum of all the possible way to compose operations in an ns operad. In other words, this amounts to the universal property satisfied by the ns operad $\As$ (respectively, the operad $\Lie$): it is the unit for the Manin's black product of (finitely generated) binary quadratic ns operads (respectively, binary quadratic operads) \cite{GinzburgKapranov94, GinzburgKapranov95, Vallette08}. The analogous universal property satisfied by the dg ns operad $\Ai$ (respectively, the dg operad~$\Li$) can be found in \cite{Wierstra16, Robert-Nicoud17}. This gives a hint on how to extend the twisting procedure to other kinds of algebraic structures like cyclic operads, modular operads or properads. Notice that the latter case was treated by S. Merkulov  in the very recent preprint \cite{Merkulov22} (which appeared a few days before we completed the final version of this monograph).  \\

There is also a way to ``extend'' the formalism of this chapter with the following ``mise en abyme'' of the operadic twisting theory. Given a $\Tw$-stable complete operad $\calG$ with the $\Tw$-comonad structure $\Delta_\calG : \calG\to\Tw\, \calG$ and a complete operad $f : \calG\to \calP$ under $\calG$, the twisted complete operad $\Tw\, \calP$ is naturally an operad under $\calG$ by 
$$\xymatrix@C=30pt{\calG \ar[r]^(0.43){\Delta_\calG}& \Tw\,\calG  \ar[r]^(0.47){\Tw(f)} &\Tw\, \calP\ .}$$
Therefore, the twisting construction induces a comonad in the category of operads under the operad $\calG$. 

\begin{definition}[$\calG$-$\Tw$-stable operad]\index{$\Tw$-stable operad!under an operad $\calG$}
Let $\calG$ be a $\Tw$-stable operad. A complete symmetric operad under $\calG$ is called \emph{$\calG$-$\Tw$-stable} if it admits a $\Tw$-co\-al\-ge\-bra structure. 
\end{definition}

The interpretation in terms of types of algebras is the same as above except that one should twist the operation of $\calP$ coming from $\calG$ as they are twisted in~$\calG$. 

\begin{example}
One obvious example is given by the $\Tw$-stable operad $\calG=\Gerst$ encoding Gerstenhaber algebras and the operad $\calP=\BV$ under it which encodes Batalin--Vilkovisky algebras. One can also consider their Koszul resolution $\Gerst_\infty\to \BV_\infty$, where the latter one is given in \cite{GCTV12}. 
\end{example}

\chapter{Operadic twisting and graph homology}\footnotetext{\hrule\smallskip\noindent This material will be published by Cambridge University Press \& Assessment as ‘Maurer-Cartan Methods in Deformation Theory: the twisting procedure’ by Vladimir Dotsenko, Sergey Shadrin and Bruno Vallette. This version is free to view and download for personal use only. Not for re-distribution, re-sale or use in derivative works. \copyright Cambridge University Press \& Assessment}\label{sec:Computations} 

In this chapter, we discuss instances of the operadic twisting leading to various graph complexes, and we outline the corresponding homology computations and their applications. The purpose here of this chapter is two-fold. First, it illustrates important examples of the operadic twisting, including those which motivated the very invention of this procedure by T. Willwacher \cite{Willwacher15}. Next, we show that the homology of twisted operads behaves quite unpredictably, and every new calculation has potential to surprise its author; we hope that this will motivate further research of this exciting topic. \\

The examples given here involve both symmetric and nonsymmetric operads, starting with the latter. The ns operads we chose as illustration of the operadic twisting procedure are those of noncommutative Gerstenhaber algebras and noncommutative Batalin--Vilkovisky algebras. Both of those operads were introduced recently in \cite[Section~3]{DotsenkoShadrinVallette19}. The two principal examples in the symmetric case are those of classical Gerstenhaber algebras and classical Batalin--Vilkovisky algebras. Those results are not new, but we derive them using a different method which is \emph{ad hoc} and thus shorter. We shall see that the classical operads and their ns analogues behave in a different way with respect to the twisting procedure.\\

 We then discuss the way operadic twisting fits into the computation of the homology of Kontsevich's graph complexes where the celebrated Gro\-then\-dieck--Teichm\"uller Lie algebra emerges \cite{Willwacher15}, and into the research related to the so called Deligne conjecture (now theorem by many people, see, for example, \cite{MR3261598,BergerFresse04, MR1321701,KontsevichSoibelman00,McClureSmith02, Tamarkin07, MR1328534}). The latter work has a natural version involving the operadic twisting of the operad of pre-Lie algebras \cite{DotKhPreLie}, which we outline in the end of the chapter. In the view of existence of nonsymmetric counterparts of the operads $\Gerst$ and $\BV$, it is interesting to ask what is a meaningful nonsymmetric counterpart of the operad pre-Lie, and how it fits into the theory of operadic twisting.

\section{Twisting the nonsymmetric operad \texorpdfstring{$\ncGerst$}{ncGerst}} \label{subsec:ncGerst}

\index{operad!$\ncGerst$} Recall after \cite[Section~3.1]{DotsenkoShadrinVallette19} that the nonsymmetric operad $\ncGerst$ can be defined as follows. As a graded $\k$-module the space $\ncGerst(n)$, for $n\ge 1$, is freely spanned by possibly disconnected linear graphs, called \emph{bamboos}\index{bamboo}, where the vertices are ordered from left to right, for instance,
\begin{equation*}
\vcenter{\xymatrix@M=5pt@R=10pt@C=10pt{
		*+[o][F-]{1}\ar@{-}[r] &*+[o][F-]{2}\ar@{-}[r]&  *+[o][F-]{3} & *+[o][F-]{4} &*+[o][F-]{5}\ar@{-}[r]&  *+[o][F-]{6} 
}
}
\quad \in \ncGerst(6)
\end{equation*}
The arity $0$ space $\ncGerst(n)=0$ is trivial. Each edge carries homological degree $1$ and the total homological degree is thus equal to the number of edges. 
(These elements corresponds to right-combs of binary generators with the presentation given in  \cite[Section~3.1]{DotsenkoShadrinVallette19}.)
It is convenient to think that the edges are ordered and reordering generates the sign; for a given bamboo, we assume,  by default, the ordering of the edges from left to right. The operadic composition $\circ_i$ of two bamboos, $\Gamma_1\in \ncGerst(n)$ and $\Gamma_2\in \ncGerst(k)$ amounts to placing the bamboo $\Gamma_2$ at the place of the vertex $i$ of the bamboo $\Gamma_1$, globally relabelling the vertices from left to right by $1,\dots,n+k-1$. 
The edge $(i-1,i)$ (resp. $(i+k-1,i+k)$) belongs to the resulting graph if and only if the edge $(i-1,i)$ (resp., $(i,i+1)$) belongs to the graph $\Gamma_1$. A Koszul sign is generated by reordering the edges; it is given by the parity of the product between the number of edges of $\Gamma_1$ on the right of vertex $i$ and the
 number of edges of $\Gamma_2$. For instance,
\begin{align*}
& \vcenter{\xymatrix@M=5pt@R=10pt@C=10pt{
		*+[o][F-]{1}\ar@{-}[r] &*+[o][F-]{2}\ar@{-}[r]&  *+[o][F-]{3} & *+[o][F-]{4} &*+[o][F-]{5}\ar@{-}[r]&  *+[o][F-]{6} 
	}
}
\quad \circ_5 \quad
\vcenter{\xymatrix@M=5pt@R=10pt@C=10pt{
		*+[o][F-]{1}\ar@{-}[r] &*+[o][F-]{2}&  *+[o][F-]{3} 
	}
}
\\
& = (-1)\cdot\ \vcenter{\xymatrix@M=5pt@R=10pt@C=10pt{
		*+[o][F-]{1}\ar@{-}[r] &*+[o][F-]{2}\ar@{-}[r]&  *+[o][F-]{3} & *+[o][F-]{4} &*+[o][F-]{5}\ar@{-}[r] &*+[o][F-]{6}&  *+[o][F-]{7} \ar@{-}[r]&  *+[o][F-]{8} 
	}
}
\
.
\end{align*}\\

A natural ns multiplicative structure $\Ai\twoheadrightarrow \As \to \ncGerst$ is given by the assignment 
\[
\mu_2\mapsto \quad
\vcenter{\xymatrix@M=5pt@R=10pt@C=10pt{
		*+[o][F-]{1} &*+[o][F-]{2} 
	}
}
\
,
\] 
or, alternatively, a natural shifted ns multiplicative structure \[\calS^{-1} \Ai \twoheadrightarrow \calS^{-1} \As\to \ncGerst\] is given by the assignment  
\[
\mu_2\mapsto \quad
\vcenter{\xymatrix@M=5pt@R=10pt@C=10pt{
		*+[o][F-]{1}\ar@{-}[r] &*+[o][F-]{2} 
	}
}
\
.
\]
So, one can twist the ns operad $\ncGerst$ in the way described above or by using a  shifted version of the twisting procedure, where we work with the dg ns operad  $\calS^{-1} \Ai$ of shifted $\Ai$-algebras.
The subsequent computations are equivalent for these two structures, but the shifted one is slightly more convenient for the presentation. So, we perform all computations for the shifted one, and we denote by $\Tw\ncGerst$ the ns operad twisted with respect to the shifted ns multiplicative structure.\\

The elements of the underlying  $\k$-module of  $\Tw\ncGerst(n)_{-d}$ are series indexed by $k\geq 0$ of finite sums of 
the bamboos with $n$ white vertices (labelled by $1,\dots,n$ from left to right), $k$ black vertices, and $d$ edges. Each black vertex carries degree $-2$ (the degree of the Maurer--Cartan elements for $\calS^{-1} \Ai$-algebras) and each edge still carries degree $1$. 
The twisted differential $\dd^{\mu_1^\alpha}$ is equal to the sum of the following five types of summands:
\begin{enumerate}
\item we attach a black vertex from the left to the leftmost vertex of the bamboo: 
$\vcenter{\xymatrix@M=5pt@R=10pt@C=10pt{
		*+[o][F**]{} \ar@{-}[r]& *+[o][F.]{} \ar@{..}[r] &   }}$\ , 
\item we attach a black vertex from the right to the rightmost vertex of the bamboo: 
$\vcenter{\xymatrix@M=5pt@R=10pt@C=10pt{
		 \ar@{..}[r] & *+[o][F.]{} \ar@{-}[r] &*+[o][F**]{}  }}$\ , 

\item we replace a white vertex by a black vertex connected by an edge to the white vertex on the left: 
$\vcenter{\xymatrix@M=5pt@R=10pt@C=10pt{
	\ar@{..}[r]&	*+[o][F**]{} \ar@{-}[r]& *+[o][F-]{\, i\, } \ar@{..}[r] &   }}$\ , 

\item we replace a white vertex by a black vertex connected by an edge to the white vertex on the right: 
$\vcenter{\xymatrix@M=5pt@R=10pt@C=10pt{
	\ar@{..}[r]&*+[o][F-]{\, i\, }  \ar@{-}[r]&	*+[o][F**]{}  \ar@{..}[r] &   }}$\ ,

\item we replace a black vertex by two black vertices connected by an edge: 
$\vcenter{\xymatrix@M=5pt@R=10pt@C=10pt{
		 \ar@{..}[r] & *+[o][F**]{} \ar@{-}[r] &*+[o][F**]{}  \ar@{..}[r]& }}$\ . 
\end{enumerate}
The Koszul type sign is given by counting how many edges from left to right that the new edge has to jump over, with an extra $-1$ sign for the terms $(3)$, $(4)$, and $(5)$. 

For instance,
\begin{align*}
& \mathrm{d}^{\mu_1^\alpha}\left( 
\vcenter{\xymatrix@M=5pt@R=10pt@C=10pt{
		*+[o][F-]{1}\ar@{-}[r] &*+[o][F-]{2}&  *+[o][F**]{} 
	}
}\right)
\ =
\vcenter{\xymatrix@M=5pt@R=10pt@C=10pt{
		*+[o][F-]{1}\ar@{-}[r] &*+[o][F-]{2}&*+[o][F**]{}\ar@{-}[l]&  *+[o][F**]{}}}
\ 
,
\end{align*}
and the terms 
\begin{align*}
&
 \vcenter{\xymatrix@M=5pt@R=10pt@C=10pt{
		*+[o][F**]{}\ar@{-}[r]&*+[o][F-]{1}\ar@{-}[r] &*+[o][F-]{2}&  *+[o][F**]{} 
	}}\ , \quad 
	\vcenter{\xymatrix@M=5pt@R=10pt@C=10pt{
		*+[o][F-]{1}\ar@{-}[r] &  *+[o][F-]{2}&  *+[o][F**]{}  &*+[o][F**]{}\ar@{-}[l] }}
		\ 
\quad \text{and}\quad
\vcenter{\xymatrix@M=5pt@R=10pt@C=10pt{
		*+[o][F-]{1}\ar@{-}[r] & *+[o][F**]{}\ar@{-}[r]& *+[o][F-]{2}&  *+[o][F**]{} }}
\end{align*}
appear in the expression for the differential twice, but with the opposite signs. 

\begin{proposition} \label{prop:ncGerstnotTwistable}
The ns operad $\ncGerst$ is not $\Tw$-stable. 
\end{proposition}

\begin{proof}
We apply Proposition~\ref{prop:Twistable} with the following computation: 
\begin{align*}
\ad_{\mu_2(\alpha, -)+\mu_2(-, \alpha)}\left(
\vcenter{\xymatrix@M=5pt@R=10pt@C=10pt{
		*+[o][F-]{1} &*+[o][F-]{2}}}
		\right)& =
(-1)\cdot\  \vcenter{\xymatrix@M=5pt@R=10pt@C=10pt{
		*+[o][F-]{1}\ar@{-}[r] & *+[o][F**]{} & 
		 *+[o][F-]{2}}}
\\
& \phantom{ =\ } + (-1)\cdot\  \vcenter{\xymatrix@M=5pt@R=10pt@C=10pt{
		*+[o][F-]{1} & *+[o][F**]{}\ar@{-}[r] & 
		 *+[o][F-]{2}}}
\\
& \neq 0\ .
\end{align*}
\end{proof}

The ns operad $\calS^{-1}\As$ of shifted associative algebras is isomorphic to the ns suboperad of $\ncGerst$ generated by the element $b\coloneqq 
\vcenter{\xymatrix@M=5pt@R=10pt@C=10pt{
		*+[o][F-]{1}\ar@{-}[r] &*+[o][F-]{2} }}$\ .
We consider the complete ns operad 
\[\calS^{-1}\As^+\coloneqq \frac{\calS^{-1} \As\hat{\vee} \, \gamma}{\big( b(\gamma, -), \ b(-,\gamma) \big)}\ ,
\]
where $\gamma$ is an arity $0$ degree $-4$ element placed in $\F_1$. 

\begin{theorem}  \label{thm:HTwncGerst} 
The map of complete dg ns operads 
\[\calS^{-1}\As\hat{\vee}\, \gamma \to \Tw\ncGerst\]  
defined by 
\[ b\mapsto \vcenter{\xymatrix@M=5pt@R=10pt@C=10pt{
		*+[o][F-]{1}\ar@{-}[r] &*+[o][F-]{2} }}\ , \ \ 
\gamma \mapsto \vcenter{\xymatrix@M=5pt@R=10pt@C=10pt{
			 *+[o][F**]{}  &*+[o][F**]{} }}\]
induces the isomorphism of complete ns operads 
\[  {H}_\bullet\left(\Tw\ncGerst\right)\cong \calS^{-1}\As^+\ .\]
\end{theorem}

\begin{proof} Since the elements of $\Tw \ncGerst$ are series indexed by $k\geq 0$ of finite sums of 
the bamboos with  $k$ black vertices and since the  differential $\dd^{\mu_1^\alpha}$ increases the number of black vertices by one, it is enough to consider the case of finite series, i.e. sums. 

Note that the differential preserves the number of the connected components~$K\geq 1$ and the number of the white vertices $N_1,\dots,N_K\geq 0$ on these components. 
The subgraph that consists of all black vertices and edges connecting them has $K+\sum_{i=1}^K N_i$
 disjoint connected components (some of them can be empty in a particular graph). The chain complex of all graphs with fixed~$K$ and fixed $N_1,\dots,N_K$ is isomorphic to the tensor product of the $K+\sum_{i=1}^K N_i$ chain complexes disjoint black components described below. 

Consider a connected black component of length $n$, for $n\ge 0$. Under the  action of the differential, it is replaced with a connected black component of length~$n+1$, with a coefficient $c_n$ that depends on its position within the ambient graph. Here is a full list of the possible cases:
\begin{enumerate}
	\item The black component is connected to white vertices both on the left and on the right: $c_n=0$ for even $n$ and $c_n=\pm 1$ for odd $n$. 
	\item The black component is connected to a white vertex only on the left and it is not the rightmost component of the ambient bamboo: $c_n=\pm 1$, for even $n$, and $c_n=0$, for odd $n$. The same for the interchanged left and right. 
	\item The black component is connected to a white vertex only on the left and it is the rightmost component of the ambient bamboo: $c_n=0$, for even $n$, and $c_n=\pm 1$, for odd $n$. The same for the interchanged left and right. 
	\item The black component is not connected to white vertices and is neither the leftmost nor the rightmost component of the ambient bamboo: $c_n=0$, for even $n$, and $c_n=\pm 1$ for odd $n$. Note that in this case $n\geq 1$. 
	\item The black component is not connected to white vertices and it is the rightmost but not the leftmost component of the ambient bamboo: $c_n=\pm 1$, for even $n$, and $c_n=0$ for odd $n$. The same for the interchanged left and right. Note that in this case $n\geq 1$. 
	\item The black component is not connected to white vertices and it is simultaneously the rightmost and the leftmost component of the ambient bamboo, that is this black component is the whole ambient bamboo graph satisfying $K=1$, $N_1=0$: $c_n=0$, for even $n$, and $c_n=\pm 1$, for odd $n$. Note that in this case $n\geq 1$. 
\end{enumerate}
In the cases (2), (4), and (6), the corresponding chain complex is acyclic, thus the homology is  equal to zero. In the cases (1) and (3), the homology is one-dimensional represented by a black component of length $0$. In the case (5), the homology is one-dimensional represented by a black component of length $1$. 

Thus, for the ambient graph, we either have the one-dimensional homology group for $K=1$, $N_1\geq 1$, represented by the connected bamboos  with only white vertices or 
we have the one-dimensional homology group for $K=2$, $N_1=N_2=0$, represented by two disjoint black vertices. These graphs generate the complete ns operad $\calS^{-1} \As \hat{\vee} \, \gamma
$, and it is clear that the substitution of the latter graph into a connected bamboo of white vertices of length at least two gives a boundary of the differential, cf. the acyclic case (2) above.
\end{proof}

 Theorem~\ref{thm:HTwncGerst} shows  that the homology of the deformation complex 
  \[
 \Def\big(\allowbreak \calS^{-1}\Ai\to \allowbreak\ncGerst\big)
  \]
 is the one-di\-men\-sion\-al Lie algebra concentrated in degree $-2$ with the trivial Lie bracket, since it is isomorphic to the double suspension of the homology of the arity $0$ component of $\Tw\ncGerst$ by Proposition~\ref{prop:DefTwP0} (we have to apply the double suspension here since the degree of the Maurer--Cartan elements of $\calS^{-1}\Ai$ is $-2$). 
 As usual in deformation theory, see Theorem~\ref{thm:InfDefGeneral}, this result can be interpreted as a strong rigidity statement about the class of ns operad morphisms $\calS^{-1}\Ai\to \ncGerst$: there is no non-trivial  infinitesimal deformation of the map considered here.
 The above computations show that the action 
of the  homology of the deformation Lie algebra $\Def\big(\allowbreak \calS^{-1}\Ai\to \allowbreak\ncGerst\big)$
 on ${H}_\bullet\left(\Tw\ncGerst\right)$ is trivial.

\section{Twisting the nonsymmetric operad \texorpdfstring{$\mathrm{ncBV}$}{ncBV}} \label{subsec:ncBV}

\index{operad!$\mathrm{ncBV}$}
The nonsymmetric operad $\ncBV$ from \cite[Section~3.1]{DotsenkoShadrinVallette19} can be defined as follows. As a graded $\k$-module the space $\ncBV(n)$, for $n\ge 1$, is freely spanned by possibly disconnected linear graphs (bamboos) with at most one tadpole edge at each vertex, where the vertices are ordered from left to right. For instance, we have 
\[
\ncBV(1) = \left\langle\ \
\vcenter{
	\xymatrix @M=5pt@R=5pt@C=10pt{
		*+[o][F-]{1}
	}
}\ , \  \vcenter{
	\xymatrix @M=5pt@R=5pt@C=10pt{
		*+[o][F-]{1} \ar@{-}@(ru,u) 
	}
}\ \ \
\right\rangle \ ,\]
and
\[ 
\vcenter{
	\xymatrix @M=5pt@R=5pt@C=10pt{
		*+[o][F-]{1}\ar@{-}[r] &*+[o][F-]{2}\ar@{-}[r] \ar@{-}@(ru,u) &  *+[o][F-]{3} & *+[o][F-]{4} \ar@{-}@(ru,u)&*+[o][F-]{5}\ar@{-}[r] \ar@{-}@(ru,u)&  *+[o][F-]{6} 
	}
}
\quad \in \ncBV(6)\ .
\]
Each edge, including the tadpoles, has homological degree equal to $1$ and the total homological degree is equal to the number of edges. 
(To match with the presentation given in \cite[Section~3.2]{DotsenkoShadrinVallette19}, these elements are in one-to-one correspondence with right-combs of binary generators labelled, at the very top of them, with nothing or one copy of the generator $\Delta$ at each leaf.)
 It is convenient to think that the edges are ordered and reordering generates the sign, like in the above $\ncGerst$ case. For a given bamboo with tadpoles, by default, we order first the edges from left to right first and then the tadpoles from left to right. \\

 The operadic composition $\circ_i$ of two  bamboos with tadpoles, $\Gamma_1\in \ncBV(n)$ and $\Gamma_2\in \ncBV(k)$, amounts to placing the graph $\Gamma_2$ at the place of the vertex $i$ of the graph $\Gamma_1$, globally relabelling the vertices from left to right by the integers $1,\dots,n+k-1$. The edge $(i-1,i)$ (resp. $(i+k-1,i+k)$) belongs to the resulting graph if and only if the edge $(i-1,i)$ (resp., $(i,i+1)$) belongs to the graph $\Gamma_1$. If there is a tadpole at the vertex $i$ of $\Gamma_1$, then it becomes either a new tadpole at one of the vertices $i,i+1,\dots,i+k-1$ of $\Gamma_1\circ_i\Gamma_2$, when no tadpole was yet present, or a new edge connecting two consecutive vertices from this set. A Koszul sign is generated by reordering the edges and tadpoles. For instance,
\begin{align*}
\rule{0pt}{22pt}
 \vcenter{\xymatrix@M=5pt@R=10pt@C=10pt{
		*+[o][F-]{1}\ar@{-}[r] &*+[o][F-]{2}\ar@{-}@(ru,u)
	}
}
\ \ \circ_2 \ \
\vcenter{\xymatrix@M=5pt@R=10pt@C=10pt{
		*+[o][F-]{1}\ar@{-}[r] &*+[o][F-]{2} \ar@{-}@(ru,u) &  *+[o][F-]{3} \ar@{-}@(ru,u) 
	}
}\ \
 = \ \ & (-1)\cdot\ \vcenter{\xymatrix@M=5pt@R=10pt@C=10pt{
 		*+[o][F-]{1}\ar@{-}[r] & *+[o][F-]{2}\ar@{-}[r] \ar@{-}@(ru,u) &*+[o][F-]{3} \ar@{-}@(ru,u) &  *+[o][F-]{4} \ar@{-}@(ru,u)
 	}
 }
\\ & +\rule{0pt}{30pt}
(-1)\cdot\ \vcenter{\xymatrix@M=5pt@R=10pt@C=10pt{
		*+[o][F-]{1}\ar@{-}[r] & *+[o][F-]{2}\ar@{-}[r]  &*+[o][F-]{3} \ar@{-}@(ru,u) \ar@{-}[r]&  *+[o][F-]{4} \ar@{-}@(ru,u)
	}
}
\
.
\end{align*}

As is the previous example, we consider a shifted ns multiplicative structure $\calS^{-1}\Ai \twoheadrightarrow \calS^{-1}\As\to \ncBV$ given by the assignment 
\[
\mu_2\mapsto \quad
\vcenter{\xymatrix@M=5pt@R=10pt@C=10pt{
		*+[o][F-]{1}\ar@{-}[r] &*+[o][F-]{2} 
	}
}
\
.
\]
The elements of the underlying  $\k$-module of  $\Tw\ncBV(n)_{-d}$ are series indexed by $k\geq 0$ of finite sums of 
the bamboos with tadpoles, with $n$ white vertices (labelled by $1,\dots,n$ from left to right), $k$ black vertices, and a total of $d$ edges and tadpoles. Each black vertex carries degree $-2$ and each edge carries degree $1$. 
The components of the twisted differential $\mathrm{d}^{\mu_1^\alpha}$ are the same as in the previous case of the dg ns operad $\Tw \ncGerst$. The only new case comes with vertices (white or black) having tadpoles: they are again replaced by two vertices connected by an edge (black-white plus white-black, or black-black) where the tadpole distributes over the two vertices. The signs remain the same as in the $\Tw \ncGerst$ case: 
the Koszul type sign is given by counting how many edges from left to right that the new edge has to jump over. With the  order considered on edges and tadpoles, this means that tadpoles will never be taken into account when computing this sign. There is an extra $-1$ sign for the terms which replace a vertex by two vertices. 
 For instance,
\begin{align}\label{Eqn:RelationHTwncBV}
\rule{0pt}{22pt} 
\mathrm{d}^{\mu_1^\alpha}\left(\ 
\vcenter{\xymatrix@M=5pt@R=10pt@C=10pt{
		*+[o][F-]{1} \ar@{-}@(ru,u) 
	}
}\ \right)
\ \
=
\ \ (-1)\cdot \
\vcenter{\xymatrix@M=5pt@R=10pt@C=10pt{
		*+[o][F-]{1} \ar@{-}[r] &*+[o][F**]{} \ar@{-}@(ru,u) 
	}
}
\ \
+(-1)\cdot\ 
\vcenter{\xymatrix@M=5pt@R=10pt@C=10pt{
		*+[o][F**]{} \ar@{-}@(ru,u) & *+[o][F-]{1} \ar@{-}[l] 
	}
}
\ 
,
\end{align}
and the graphs 
\begin{align*}
\rule{0pt}{22pt}
\vcenter{\xymatrix@M=5pt@R=10pt@C=10pt{
		*+[o][F-]{1} \ar@{-}[r] \ar@{-}@(ru,u) &*+[o][F**]{}  
	}
}
\ \ \
\text{ and }
\ \ \
\vcenter{\xymatrix@M=5pt@R=10pt@C=10pt{
		*+[o][F**]{} & *+[o][F-]{1} \ar@{-}[l] \ar@{-}@(ru,u) 
	}
}
\end{align*}
appear in the expression for the differential twice, with the opposite signs. 

\begin{proposition} \label{prop:ncBVnotTwistable}
The ns operad $\ncBV$ is not $\Tw$-stable. 
\end{proposition}

\begin{proof}
The same argument and computation as in the case of the ns operad $\ncGerst$ hold here and prove the result, see Proposition~\ref{prop:ncGerstnotTwistable} and Proposition \ref{prop:Twistable}. 
\end{proof}

In order to describe the homology ns operad ${H}_\bullet(\Tw\ncBV)$, we consider the following extension of the complete ns operad introduced in the previous $\ncGerst$ case:
\[\calS^{-1}\As^{++}\coloneqq \frac{\calS^{-1}\As\hat{\vee}\, \gamma\, \hat{\vee}\, \zeta}
{\big( b(\gamma, -),\ b(-,\gamma), b(\zeta, -)+b(-,\zeta)\big)}\ ,
\]
where $\zeta$ is an arity $0$ degree $-1$ element placed in $\F_1$. 

\begin{theorem} \label{thm:HTwncBV} 
The map of complete dg ns operads 
\[
\calS^{-1}\As\hat{\vee}\, \gamma\, \hat{\vee}\, \zeta \to \Tw\ncBV
\] 
defined by 
\[ 
b\mapsto \vcenter{\xymatrix@M=5pt@R=10pt@C=10pt{
		*+[o][F-]{1}\ar@{-}[r] &*+[o][F-]{2} }}\ , \ \ 
\gamma \mapsto \vcenter{\xymatrix@M=5pt@R=10pt@C=10pt{
			 *+[o][F**]{}  &*+[o][F**]{} }}
			 \ , \ \ 
\zeta \mapsto 
\vcenter{\xymatrix@M=5pt@R=10pt@C=10pt{
		*+[o][F**]{} \ar@{-}@(ru,u)  }}
			 \]
induces the isomorphism of complete ns operads 
\[ {H}_\bullet\left(\Tw\ncBV\right)\cong \calS^{-1}\As^{++}\ .\]
\end{theorem}

\begin{proof} As in the previous case, since the elements of $\Tw \ncBV$ are series indexed by $k\geq 0$ of finite sums of 
the bamboos with tadpoles with  $k$ black vertices and since the  differential $\dd^{\mu_1^\alpha}$ increases the number of black vertices by one, it is enough to consider the case of finite series, i.e. sums. 

At fixed arity $n\ge 0$, we consider the increasing filtration $F^p \Tw\ncBV(n)$ span\-ned by bamboos containing  at least $-p$ black vertices with tadpoles. Note that this filtration is exhaustive and bounded below since a black vertex with a tadpole carries degree $-1$, and
the twisted differential $\dd^{\mu_1^\alpha}$ preserves this filtration. So the associated spectral sequence converges to the homology of $\Tw\ncBV(n)$. The differential $d_0$ of the first page involves the components of 
 $d^{\mu_1^\alpha}$ which do not increase the number of black vertices with tadpoles, that is the ones which increase only the number of black vertices without tadpoles. In order to compute its homology groups, we apply the same arguments as in the proof of Theorem~\ref{thm:HTwncGerst}: here the black and white vertices with tadpoles play the same role as the white vertices without tadpoles. Therefore, the second page $E^1$ is spanned by the homology class $\gamma$ represented by $\vcenter{\xymatrix@M=5pt@R=10pt@C=10pt{
		*+[o][F**]{}  &*+[o][F**]{}}}$ and by the connected bamboos of white vertices with or without tadpoles and black vertices with tadpoles. The differential $d_1$ creates a black vertex with tadpole and an edge on the left and on the right of any white vertex with tadpole, that is 
\[
\rule{0pt}{22pt} d_1\left(\ 
\vcenter{\xymatrix@M=5pt@R=10pt@C=10pt{
		*+[o][F-]{\phantom{1}} \ar@{-}@(ru,u) }}
\ \right)
=
\ 
(-1)\cdot \
\vcenter{\xymatrix@M=5pt@R=10pt@C=10pt{
		*+[o][F**]{} \ar@{-}@(ru,u) & *+[o][F-]{\phantom{1}} \ar@{-}[l] 
	}
}
\ + \ 
(-1)\cdot \
\vcenter{\xymatrix@M=5pt@R=10pt@C=10pt{
		*+[o][F-]{\phantom{1}} \ar@{-}[r] &*+[o][F**]{} \ar@{-}@(ru,u) 
	}}
	\]
and	
\[
d_1\left( \ 
\vcenter{\xymatrix@M=5pt@R=10pt@C=10pt{
		*+[o][F**]{} \ar@{-}@(ru,u) 
	}
}
\ \right)
= 
d_1\left( \ 
\vcenter{\xymatrix@M=5pt@R=10pt@C=10pt{
		*+[o][F-]{\phantom{1}} 
	}
}
\ \right)
= 
d_1\left( \ 
\vcenter{\xymatrix@M=5pt@R=10pt@C=10pt{
		*+[o][F**]{} 
	}
}
\ \right)
= 0 \ .
\]
We consider a filtration on the chain complex $(E^1,d_1)$ where $\F^p E^1$ is spanned by the above mentioned bamboos with at least $-p$ sub-bamboos of the form 
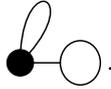
\begin{align*}
\rule{0pt}{22pt} \vcenter{\xymatrix@M=5pt@R=10pt@C=10pt{
		*+[o][F**]{} \ar@{-}@(ru,u) & *+[o][F-]{\phantom{1}} \ar@{-}[l] }}\ .
\end{align*}
This filtration is again exhaustive and bounded below, so 
its associated spectral sequence converges to the homology $E^2={H}_\bullet(E^1, d_1)$. 
The differential $\dd_0$ of the first page ${E}^0$ is the second component of $d_1$ mentioned above, that is the one which produces a black vertex with a tadpole on the right-hand side of a white vertex with tadpole. The subcomplexes of ${E}^0$ consisting of bamboos containing $k$ sub-bamboos of the form 
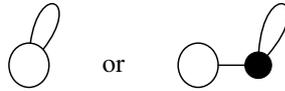
\begin{align*}
\rule{0pt}{22pt} \vcenter{\xymatrix@M=5pt@R=10pt@C=10pt{
		*+[o][F-]{\phantom{1}} \ar@{-}@(ru,u) }}
\qquad \text{or}\qquad 
\vcenter{\xymatrix@M=5pt@R=10pt@C=10pt{
		*+[o][F-]{\phantom{1}} \ar@{-}[r] &*+[o][F**]{} \ar@{-}@(ru,u) 
	}
}
\end{align*}
are acyclic, for $k\ge 1$, since they are isomorphic to the tensor product of $k$ acyclic chain complexes. 
The homology groups ${E}^1$ is thus spanned by the following bamboos: 
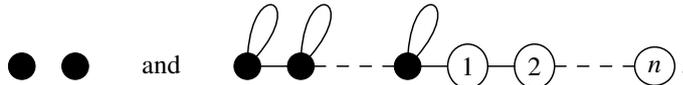
\begin{align*}
\rule{0pt}{22pt}
\vcenter{\xymatrix@M=5pt@R=10pt@C=10pt{
		*+[o][F**]{}  &*+[o][F**]{}}} 
		\qquad \text{and} \qquad	
		\vcenter{\xymatrix@M=5pt@R=10pt@C=10pt{
		*+[o][F**]{} \ar@{-}@(ru,u)  \ar@{-}[r] &
		*+[o][F**]{} \ar@{-}@(ru,u)  \ar@{--}[rr] & &
		*+[o][F**]{} \ar@{-}@(ru,u)  \ar@{-}[r] & *+[o][F-]{1} \ar@{-}[r] & *+[o][F-]{2} \ar@{--}[rr] & & *+[o][F-]{n}
	}
}
\ ,
\end{align*} 
with $j\ge 0$ black vertices with tadpoles and $n\ge 0$ white vertices without tadpoles. On such bamboos, the differential $\dd_1$ vanishes, since it is induced by the 
first component of $d_1$ as presented above, the one which produces a black vertex with a tadpole on the left-hand side from a white vertex with tadpole. So the second spectral sequence collapses at ${E}^1$ and the first spectral sequence collapses at $E^2$ with basis given by these latter bamboos. 

The three elements 
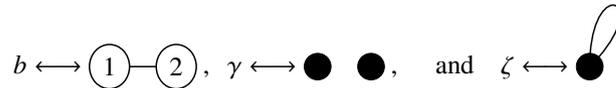
\[
\rule{0pt}{22pt}
b\longleftrightarrow \vcenter{\xymatrix@M=5pt@R=10pt@C=10pt{
		*+[o][F-]{1}\ar@{-}[r] &*+[o][F-]{2} }}\ , \ \ 
\gamma \longleftrightarrow \vcenter{\xymatrix@M=5pt@R=10pt@C=10pt{
			 *+[o][F**]{}  &*+[o][F**]{} }}
			 \ , \ \quad \text{and} \quad 
\zeta \longleftrightarrow
\vcenter{\xymatrix@M=5pt@R=10pt@C=10pt{
		*+[o][F**]{} \ar@{-}@(ru,u)  }}\]
		are clearly generators of the homology ns operad ${H}_\bullet(\Tw\ncBV)$. Formula~\ref{Eqn:RelationHTwncBV} gives the relation between $b$ and $\zeta$ introduced in the definition of 
$\calS^{-1} \As^{++}$. So the above assignement induces a morphism 
$\calS^{-1} \As^{++} \to {H}_\bullet\left(\Tw\ncBV\right)$
of complete ns operads, which turns out to be an isomorphism since the dimensions of the underlying graded $\mathbb{N}$-modules of $\calS^{-1} \As^{++}$ coincide with the number of bamboos spanning ${H}_\bullet(\Tw\ncBV)$. 
\end{proof}

Theorem~\ref{thm:HTwncBV} with Proposition~\ref{prop:DefTwP0} show that the homology of 
 \[
 \Def\big(\calS^{-1}\Ai\to \ncBV\big)
 \] 
 is isomorphic to 
$\k \bar\gamma \oplus \k[\hbar]\zeta$, where the degrees are given by $|\bar\gamma|=-2$, $|\zeta|=1$, $|\hbar|=0$, with the trivial Lie bracket. 
Theorem~\ref{thm:InfDefGeneral} implies that the morphism of dg ns operads $\calS^{-1}\Ai\to \ncBV$ considered here admits gauge  independent infinitesimal deformations parametrized by $\NN$. 
The  previous computations show that the action of the homology of the deformation Lie algebra $\Def\big(\calS^{-1}\Ai\to \ncBV\big)$  on ${H}_\bullet\left(\Tw\ncBV\right)$ is trivial. 

\section{Twisting the operad \texorpdfstring{$\Gerst$}{Gerst}} \label{subsec:TwGerst}

\index{operad!$\Gerst$}
The structure operations of a Gerstenhaber algebra are a commutative  product and a degree $1$ Lie bracket  satisfying together the Leibniz relation, see \cite[Section~13.3.10]{LodayVallette12} for more details. The associated operad $\Gerst$ is thus generated by an arity $2$ degree $0$ element $\mu$ with trivial $\Sy_2$ action and by an arity $2$ degree $1$ element $\lambda$ also with trivial $\Sy_2$ action. This latter one induces a  multiplicative structure 
$\dsLi \twoheadrightarrow \calS^{-1}\Lie \to \Gerst$. So, one can twist the operad $\Gerst$.
The various properties of the corresponding twisted operad $\Tw\Gerst\coloneqq \left(
 \Gerst\hat{\vee} \alpha , \dd^{\lambda_1^\alpha}
 \right)$ follow from the general statements developed in~\cite[Section 5]{DolgushevWillwacher15} for operads under distributive law.  We shall present shorter proofs of these properties. 
 
\begin{proposition}[{\cite[Corollary~5.12]{DolgushevWillwacher15}}] \label{prop:GerstTwistable}
The canonical morphism of operads $\Gerst \hookrightarrow \Tw\Gerst$ defines a $\Tw$-coalgebra structure. 
\end{proposition}

\begin{proof}
This a direct corollary of the symmetric operad analog of Proposition~\ref{prop:Twistable}. 
The Jacobi relation implies $\dd^{\lambda_1^\alpha}(\lambda)=\ad_{\lambda(\alpha, -)}(\lambda)=0$ and  
the Leibniz relation implies $\dd^{\lambda_1^\alpha}(\mu)=\ad_{\lambda(\alpha, -)}(\mu)=0$~.
\end{proof}

\begin{remark}
Such a result says, in an operadic way, that the commutative product and the degree $1$ Lie bracket of any dg Gerstenhaber algebra form again a dg  Gerstenhaber algebra structure with the twisted differential produced by any Maurer--Cartan element. 
\end{remark}

\begin{theorem}[{\cite[Corollary~5.12 and Corollary~5.13]{DolgushevWillwacher15}}] \label{thm:TwGerst} 
The canonical morphisms of complete dg operads 
\[\Tw\Gerst \xrightarrow{\sim} \Gerst\qquad \text{and} \qquad 
\Tw\Gerst_\infty \xrightarrow{\sim} \Gerst_\infty\]
are quasi-isomorphisms. 
\end{theorem}

\begin{proof} 
Let us begin with the first statement about the operad $\Gerst$. The arguments given in the above proof of Proposition~\ref{prop:GerstTwistable} show that 
 the only non-trivial part of the twisted differential $\dd^{\lambda_1^\alpha}$ is on $\alpha$, where it is equal to 
$\dd^{\lambda_1^\alpha}(\alpha)=\frac12 \lambda(\alpha, \alpha)$~.

First, we consider  arity $0$ part of $\Tw\Gerst$. We recall that the operad $\Gerst$ is obtained from the operads $\Com$ and $\calS^{-1}\Lie$ by a distributive law. The  Jacobi relation ensures that 
$\widehat{\calS^{-1}\Lie}(\alpha)\cong \k \alpha \oplus \k \lambda(\alpha, \alpha)$. Since this latter term has degree $-3$, it can  appear only once in $\widehat{\Gerst}(\alpha)$; therefore we get $\Tw\Gerst(0)\cong \widehat{\Com}(\alpha) \oplus \widehat{\Com}\big(\alpha\big) \lambda(\alpha, \alpha)$. We denote by $\mu^{k}$ any composite of $k$ times $\mu$ in the operad $\Gerst$. 
Since 
$\dd^{\lambda_1^\alpha}\left(\mu^{k-1}\left(\alpha^k\right)\right)=\frac{k}{2}\mu^{k-1}\left(\alpha^{k-1},\lambda(\alpha, \alpha) \right)$ and since
$\dd^{\lambda_1^\alpha}\big(\lambda(\alpha, \alpha)\big)=0$, this chain complex is acyclic. 

In higher arity, we use the notation $\lambda^k\coloneqq (\cdots( \lambda\circ_1 \lambda)\cdots )\circ_1 \lambda$ for the composite of $k$ operations $\lambda$ at the first input. Recall that a basis of the operad $\calS^{-1}\Lie$ is given by the elements $\left(\lambda^{n-1}\right)^\sigma$, for $n\ge 1$, where $\sigma$ runs over the permuations of $\Sy_n$ which fix $1$. 
For any finite set $J$, we denote by 
\[\lambda^{\left(\sigma, \bar{k}\right)}\coloneqq \lambda^{|\bar{k}|+|J|-1}
\left(-, \alpha^{k_1},-, \alpha^{k_2}, -, \cdots, -, \alpha^{k_{|J|}}\right)^\sigma\ ,
\]
where $\sigma$ is a permutation of $J$ fixing its least element and where 
$\bar{k}=\left( 
k_1, \ldots, k_{|J|}
\right)$ is  a $|J|$-tuple of non-negative integers; such elements form a $\k$-linear basis of the operad $\calS^{-1}\Lie\vee \alpha$\ . 
As a consequence, the $\k$-module $\Tw\Gerst(n)_{d}$ is generated by the linearly independent elements of the form 
\[
\mu^{m+p+q-1}\left(
\alpha^m, \lambda(\alpha, \alpha)^p, \lambda^{\left(\sigma_1, \overline{k_1}\right)}, \ldots, \lambda^{\left(\sigma_q, \overline{k_q}\right)}
\right)\ ,
\]
where $m\ge 0$, $p=0$ or $p=1$, where the permutations $\sigma_1, \ldots, \sigma_q$ are associated to a partition $\sqcup_{i=1}^q J_i=\{1,\dots,n\}$\ , and where the total number of  $\lambda$'s minus twice the total number of $\alpha$'s is equal to $d$.  
%(By degree reason, the only series that can appear are indexed by $m$; the rest are finite terms).
Since $\frac12 \lambda(-, \lambda(\alpha, \alpha)) =-\lambda(\lambda(-,\alpha), \allowbreak \alpha)$, the differential $\mathrm{d}^{\lambda_1^\alpha}$ preserves such basis elements. 

So the chain complex $\Tw\Gerst(n)$ splits into the direct sum of chain complexes indexed by the decompositions $\sqcup_{i=1}^q J_i=\{1,\dots,n\}$ and the permutations $\sigma_1, \ldots, \sigma_q$. The form of the differential $\mathrm{d}^{\lambda_1^\alpha}$ implies that  each of these direct summand is isomorphic to the tensor product of $1+q$ chain complexes, where the first one is isomorphic to $\k \oplus \Tw\Gerst(0)$ and where 
the $q$ other ones are spanned by the elements $\lambda^{\left(\sigma_i, \overline{k_i}\right)}$, for any possible $|J_i|$-tuples $\overline{k_i}$.
The above result about the  arity zero case implies that the homology of  the first factor is  one-dimensional.
It is straightforward to see that $$\mathrm{d}^{\lambda_1^\alpha} \left(
\lambda^k\left(-, \alpha^k\right)
\right) = - \lambda^{k+1}\left(-, \alpha^{k+1}\right)$$ for odd $k$, and $\mathrm{d}^{\lambda_1^\alpha} \left(
\lambda^k\left(-, \alpha^k\right)
\right)  = 0 $ for even $k$. 
Thus each of the other $q$ tensor factors spanned by the elements $\lambda^{\left(\sigma_i, \overline{k_i}\right)}$ has one-dimensional homology represented by  $\lambda^{\left(\sigma_i, \bar{0}\right)}$ 
In the end, the only non-trivial class for each of these direct summand is represented by the basis element 
\[
\mu^{q}\left(\lambda^{\left(\sigma_1, \bar{0}\right)}, \ldots, \lambda^{\left(\sigma_q, \bar{0}\right)}
\right)
\]
that has no $\alpha$'s at all. These representatives form a natural basis for the operad $\Gerst$, which concludes the proof. 

The second statement about the operadic resolution  $\Gerst_\infty \xrightarrow{\sim} \Gerst$ follows directly from the result about the operad $\Gerst$ and Proposition~\ref{prop:TwQIBIS}. 
\end{proof}

Theorem~\ref{thm:TwGerst} and Proposition~\ref{prop:DefTwP0} show  that the homology of the deformation complex  $\Def\big(\allowbreak \dsLi\allowbreak \to \allowbreak \Gerst \allowbreak \big)$ is trivial in this case.
This result can be 
 interpreted as a strong rigidity statement after Theorem~\ref{thm:InfDefGeneral}: there is no non-trivial infinitesimal of the morphism of dg operads $\dsLi\allowbreak \to \allowbreak \Gerst$ considered here.
The action of the deformation Lie algebra on the operad $\Tw \Gerst$ is homologically trivial. 
 
\section{Twisting the operad \texorpdfstring{$\BV$}{BV}} \label{subsec:TwBV}

\index{operad!$\BV$}
Recall that a Batalin--Vilkovisky (BV) algebra is a Gerstenhaber algebra endowed with a compatible degree $1$ square-zero linear operator, see \cite[Section~13.7]{LodayVallette12} for more details. 
The associated operad, denoted $\BV$, is thus generated by the same kind of elements  $\mu$ and $\lambda$ as above plus an arity $1$ and degree $1$ element~$\Delta$. 
By the same argument, it acquires a multiplicative structure 
$\dsLi \twoheadrightarrow \calS^{-1}\Lie \to \BV$. 
 We consider the corresponding twisted operad $\Tw\BV$.

\begin{proposition}\label{prop:BVnotTwistable}
The operad $\BV$ is not $\Tw$-stable. 
\end{proposition}

\begin{proof}
This a direct corollary of the symmetric operad analog of the first statement of Proposition~\ref{prop:Twistable}. 
In the operad $\Tw\BV$, the differential is given by 
\[\dd^{\lambda_1^\alpha}(\Delta)=\ad_{\lambda(\alpha, -)}(\Delta)=-\lambda(\Delta(\alpha),-)\neq 0\ ,\] which concludes the proof. 
\end{proof}

We consider the complete operad 
\[\Gerst^+\coloneqq \frac{\Gerst\hat{\vee}\, \eta}{\big( \lambda(\eta, -)\big)}\ ,
\]
where $\eta$ is an arity $0$ degree $-1$ element placed in $\F_1$. For  degree reasons 
$\mu(\eta, \eta)=\lambda(\eta, \eta)=0$, so $\Gerst^+(0)\cong \k \eta$ is one-dimensional and $$\Gerst^+(1)\cong \k \id \oplus\, \k \mu(\eta, -)$$ is two-dimensional. 

\begin{theorem} \label{thm:HTwBV}
The map of complete dg operads 
\[\Gerst\hat{\vee}\, \eta \to \Tw\BV\]
defined by 
\[
\mu\mapsto \mu\ , \ \lambda\mapsto \lambda\ , \ 
\eta \mapsto \Delta(\alpha)\]
induces the isomorphism of complete operads 
\[  {H}_\bullet\left(\Tw\BV\right)\cong \Gerst^+\ .\]
\end{theorem}

\begin{proof}We use the same notations and arguments as in the proof of Theorem~\ref{thm:TwGerst}. For instance, the distributive law method gives an isomorphism of $\mathbb{S}$-modules $$\BV\cong \Com\circ \calS^{-1}\Lie \circ \k[\Delta].$$ The differential in the twisted operad $\Tw\BV$ is given by  
\[\dd^{\lambda_1^\alpha}(\lambda)=\allowbreak \dd^{\lambda_1^\alpha}(\mu)=\allowbreak\dd^{\lambda_1^\alpha}\left(\Delta(\alpha)\right)=\allowbreak0\] and by $\dd^{\lambda_1^\alpha}(\Delta)=-\lambda(\Delta(\alpha),-)$\ .

We first consider the part of arity zero of the operad $\Tw\BV$. Let us denote by $\Delta\Lie$ the sub-operad generated by $\Delta$ and $\lambda$ in the operad $\BV$; its underlying graded $\Sy$-module is isomorphic to 
$\Delta\Lie\cong \calS^{-1}\Lie \circ \k[\Delta]$. The complete sub-operad $\Delta\Lie\hat{\vee} \alpha$ of $\Tw \BV$ is stable under the differential $\dd^{\lambda_1^\alpha}$. 
We denote by $M\coloneqq \big(\Delta\Lie\hat{\vee} \alpha(0), 0,  \ldots \big)$ the dg $\Sy$-module concentrated in arity $0$; it satifies the following isomorphism of dg $\Sy$-modules 
\[
\big( 
\Tw\BV(0), 0, \ldots\big)
\cong 
\Com \circ M\ .
\]
The operadic K\"unneth formula \cite[Proposition~6.2.3]{LodayVallette12} implies that the homology of $\Tw\BV(0)$
is isomorphic to $\big(\Com\circ \mathrm{H}(M)\big)(0)$. Then, we consider the isomorphism of dg modules 
\[\big(\Delta\Lie\hat{\vee}\alpha\big)(0) \cong 
\big(\calS^{-1}\Lie \hat{\vee}\alpha\circ \k[\Delta(\alpha)]\big)(0)~.\]
The arguments and computations given in the proof of Theorem~\ref{thm:TwGerst} show that 
${H}_\bullet\left(\calS^{-1}\Lie \hat{\vee}\alpha, \dd^{\lambda_1^\alpha}
\right)\cong \calS^{-1}\Lie$. Applying again the operadic K\"unneth formula, we get 
${H}_\bullet\left(\big(\calS^{-1}\Lie \hat{\vee}\alpha\circ \k[\Delta(\alpha)]\big)(0)\right)\cong 
\big(\calS^{-1}\Lie \circ \k[\Delta(\alpha)]\big)(0)$. So the homology of $\Tw\BV(0)$ is isomorphic to 
$\left(\Com\circ\big(\calS^{-1}\Lie \circ \k[\Delta(\alpha)]\big)\right)(0)$. By degree reasons, since $|\Delta(\alpha)|=-1$, we get in the end 
${H}_\bullet(\Tw\BV(0))\cong \k\Delta(\alpha)$\ .

To treat the case of arity $n\geq 1$, we consider, for any finite set $J$, the elements of the form 
\[\lambda^{\left(\sigma, \bar{k}, \bar{*}\right)}\coloneqq \lambda^{|\bar{k}|+|J|+d-1}
\left(*^1, \alpha^{k^1},*^2, \alpha^{k^2}, \cdots, *^{|J|+d}, \alpha^{k^{|J|+d}}\right)^\sigma\ ,
\]
where $\sigma$ and $\bar{k}$ are as in the proof of Theorem~\ref{thm:TwGerst} and where $\bar{*}=\left(*^1, \ldots, *^{|J|+d}\right)$ is  a $|J|+d$-tuple with $*^1$ equals to $\Delta$ or $-$ and with $*^i$ equals to 
$\Delta$, $-$, or $\Delta(\alpha)$, such that the number of $\Delta(\alpha)$'s is equal to $d$. The $\k$-module $\Tw\BV(n)$ is generated as above by the elements of the form 
\[
\mu^{m+r-1}\left(\rho, 
\lambda^{\left(\sigma_1, \overline{l_1}, \overline{*_1}\right)}, \ldots, \lambda^{\left(\sigma_r, \overline{l_r}, \overline{*_r} \right)}
\right)\ ,
\]
where $\rho$ is an element of the symmetric $m$-tensor power of $\big(\Delta\Lie\hat{\vee}\alpha\big)(0)$ and the permutations $\sigma_1, \ldots, \sigma_r$ are associated to a partition 
$$\{1,\dots,n\}=\bigsqcup_{i=1}^r J_i.$$

Let us compute the homology of the chain complex $\Lambda(J)$ generated by the elements of the form 
$\lambda^{\left(\sigma, \bar{k}, \bar{*}\right)}$ for a fixed finite set $J$. We consider the increasing filtration $\F^p \Lambda(J)$ consisting of the elements $\lambda^{\left(\sigma, \bar{k}, \bar{*}\right)}$ containing  at least $-p$ elements $\Delta(\alpha)$. The differential $\dd^{\lambda_1^\alpha}$ preserves this filtration. It is exhaustive and bounded below, so its associated spectral sequence converges to the homology of $\Lambda(J)$. The differential of the first page of the associated spectral sequence is equal to the sole term $\dd^0(\alpha)=\frac12 \lambda(\alpha, \alpha)$. 

The same argument as in the proof of Theorem~\ref{thm:TwGerst} shows that the second page $\mathrm{E}^1$ is generated by the elements of the form 
$\lambda^{\left(\sigma, \bar{0}, \bar{*}\right)}$. The differential of this second page is given by the sole term   
$\dd^1(\Delta)=-\lambda(\Delta(\alpha),-)$, thus 
$
\dd^1\big(\lambda(\nu, \Delta)\big)=-\lambda(\lambda(\nu, \Delta(\alpha)),-)-\lambda(\lambda(\nu, -),\Delta(\alpha))-
\lambda\big(\dd^1(\nu), \Delta\big)$, for $\nu$ involving at least one $\lambda$,  and 
$\dd^1\big(\lambda(\Delta, *)\big)= \lambda(\lambda(-, \Delta(\alpha),*)+\lambda\left(\Delta, \dd^1(*)\right) $ \ .
We consider the filtration of $\mathrm{E}^1$ defined by counting the numbers of $\Delta$'s at the first input: for $p\leq 0$, $F^p \mathrm{E}^1$ is generated by the elements of the form $\lambda^{\left(\sigma, \bar{0}, \bar{*}\right)}$ with $*^1=-$, and $F^p\mathrm{E}^1=\mathrm{E}^1$, for $p\ge 1$. This filtration is exhaustive and bounded below, so it converges to the homology of $\mathrm{E}^1$. The differential $d^0$ of the first page $E^0$ of the associated spectral sequence is given by $\dd^1$, except when $\Delta$ labels the first input: in this case $d^0$ vanishes on it. Therefore, the chain complex $(E^0, d^0)$ is isomorphic to two copies, labeled respectively by the input $\Delta$ or $-$  of the first leaf, of the same chain complex. 
We consider the coaugmented coalgebra 
$C\coloneqq \k 1\oplus \k x \oplus \k y \oplus \k z$, with $|x|=0$, $|y|=1$, $|z|=-1$, where $x$ and $z$ are primitive elements, and where the (reduced) coproduct of $y$ is equal to $x\otimes z -z \otimes x$.
Under the correspondence $x\leftrightarrow -$, $y\leftrightarrow \Delta$, and $z\leftrightarrow \Delta(\alpha)$, the chain complex $(E^0, d^0)$ is isomorphic to the cobar construction of $C$ (with a bit unusual homological degree convention, cf.~\cite[Section~2.2.5]{LodayVallette12}). Since the Koszul dual algebra of $C$ is the Koszul algebra 
$C^{\ac}\cong T(X,Z)/(X\otimes Z-Z\otimes X)$, with $|X|=1$ and $|Z|=0$, the second page $E^1$ is isomorphic to $\left(\k x \oplus \k y \right)\otimes C^{\ac}$, which admits for basis $xZ^kX^l$ and $yZ^kX^l$, for $k,l\ge 0$. The differential $d^1$ is given by $d^1\left(xZ^kX^l\right)=0$ and $d^1\left(yZ^kX^l\right)=(-1)^{l+1}xZ^{k+1}X^l$. So this  spectral sequence collapses at $E^2$, where is it spanned by the elements $xX^l$, for $l\ge 0$. In other words, the first spectral sequence collapses at ${E}^2\cong {H}_\bullet\left(\Lambda(J)\right)$, which is spanned by the elements of the form $\lambda^{\left(\sigma, \bar{0}, \bar{*}\right)}$ with $\bar{*}=(-, \ldots, -)$. 

Combining this computation with the earlier computation of the homology of $\big(\Delta\Lie\hat{\vee}\alpha\big)(0)$, we see that the only non-trivial homology classes of $\big(\Tw\BV(n), \dd^{\lambda_1^\alpha}\big)$  
are represented by the basis elements
\[
\mu^{r+p-1}\left(\Delta(\alpha)^p, \lambda^{\left(\sigma_1, \bar{0}\right)}, \ldots, \lambda^{\left(\sigma_r, \bar{0}\right)}
\right)\ ,
\]
where $p=0,1$ and where the permutations $\sigma_1, \ldots, \sigma_r$ are associated to a partition $\sqcup_{i=1}^r J_i=\{1,\dots,n\}$. These elements corresponds to a basis of the operad $\Gerst^+$ under the correspondence $\eta\leftrightarrow \Delta(\alpha)$, which concludes the proof. 
\end{proof}

Theorem~\ref{thm:HTwBV} and Proposition~\ref{prop:DefTwP0} show that the homology of the deformation complex $\Def\big(\calS^{-1}\Li\to \BV\big)$ is a one-di\-men\-sio\-nal Lie algebra concentrated in degree $1$ and with the abelian Lie bracket. 
This result provides us with the following 
 rigidity statement after Theorem~\ref{thm:InfDefGeneral}: there exists no non-trivial infinitesimal deformations of the morphism 
 of dg  operads   $\calS^{-1}\Lie_\infty\to \BV$ considered here. 
The computations performed in the above proof show that 
the action of the homology of the deformation Lie algebra $\Def\big(\calS^{-1}\Li\to \BV\big)$
 on the homology ${H}_\bullet\left(\Tw\BV\right)\allowbreak\cong \Gerst^+$ is equal to the degree $1$ differential of $\Gerst^+$ which assigns $\mu \mapsto \lambda$~. 

\section{Grothendieck--Teichm\"uller Lie algebra}\label{sec:GTGra}

The concept of operadic twisting was introduced by T.~Willwacher in \cite{Willwacher15} and the goal of this section is to recall the main concepts and results of \emph{op.~cit.} in order to highlight its original motivation. Note that since we use the homological grading convention, some of the gradings in our presentation differ from \emph{op.~cit.}\\

	A convenient formalism to work with a graph $\Gamma$ is to consider its set of half-edges $\mathrm{H}$ and its set of vertices $V$, with an involution $e\colon \mathrm{H}\to \mathrm{H}$ that acts without fixed points, and a map $v\colon \mathrm{H}\to \mathrm{V}$ that indicates for each half-edge the vertex to which it is attached. The set of edges is the orbit set $\mathrm{H}/e$, and we   choose an orientation $\omega$ on $\RR^{\mathrm{H}/e}$, with the convention that for the opposite orientation $\omega^{\mathrm{op}}$, we have $(\mathrm{H},\mathrm{V},e,v,\omega^{\mathrm{op}}) = - (\mathrm{H}, \mathrm{V},e,v,\omega)$~.  

\begin{definition}[The graph operad $\Gra$] \label{def:DefGra}\index{operad!$\Gra$}\index{graph operad}
The \emph{graph operad}  $\Gra$ is defined as follows. As a graded $\k$-module, the space $\Gra(n)$, for $n\ge 1$, is freely spanned by possibly disconnected graphs with $n$ vertices labelled from $1$ to $n$, with edges of degree $1$. The symmetric group $\Sy_n$ acts by permutations of the labels on the vertices. 
	
	The operadic composition 
	\[\circ_i\ \colon\ \Gra(n)\otimes \Gra(m)\to \Gra(n+m-1)~,\] for $n,m \geq 1$ and $1 \leq i \leq n$, applied to two graphs, $\Gamma_1 = (\mathrm{H}_1,\mathrm{V}_1,e_1,v_1,\omega_1)\in \Gra(n)$ and $\Gamma_2=(\mathrm{H}_2,\mathrm{V}_2,e_2,v_2,\omega_2)\in \Gra(m)$, is equal to the sum over all possible maps $f\colon e_1^{-1}(i)\to \mathrm{V}_2$ of all graphs obtained by replacing the $i$-th vertex of $\Gamma_1$ by $\Gamma_2$ and attaching the half-edges previously attached to the $i$-th vertex of $\Gamma_1$ to the vertices of $\Gamma_2$ according to the map $f$. Thus, the partial composite $\Gamma_1\circ_i\Gamma_2$ is the sum of graphs $\Gamma = (\mathrm{H}, \mathrm{V}, e, v, \omega)$, where $\mathrm{H}\coloneqq\mathrm{H}_1\sqcup \mathrm{H}_2$, $\mathrm{V}\coloneqq \mathrm{V}_1\setminus\{i\}\sqcup \mathrm{V}_2$, $e|_{\mathrm{H}_1} = e_1$,  $e|_{\mathrm{H}_2} = e_2$, $v|_{\mathrm{H}_1\setminus e_1^{-1}(i)} = v_1$, $v|_{e_1^{-1}(i)} = f$, $v|_{\mathrm{H}_2}=v_2$, and the orientation $\omega$ is the standard one for the direct sum $\RR^{(\mathrm{H}_1\sqcup \mathrm{H}_2)/e} = \RR^{\mathrm{H}_1/e_1}\oplus \RR^{\mathrm{H}_2/e_2}$.
\end{definition}

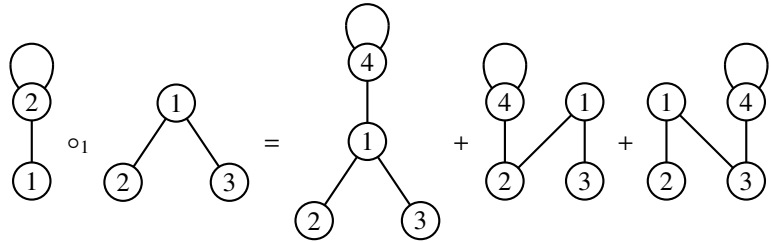
\[
\vcenter{\hbox{\begin{tikzpicture}[scale=0.7]
	 \draw[thick] (0,0.5)--(0,2);

	\draw[thick]  (-0.4,2.7) arc [radius=0.4, start angle=180, end angle= 0];
	\draw[thick] (0, 2) to [out=135,in=270] (-0.4,2.7);
	\draw[thick] (0, 2) to [out=45,in=270] (0.4,2.7);

	 \draw[fill=white, thick] (0,0.5) circle [radius=10pt];
 	 \draw[fill=white, thick] (0,2) circle [radius=10pt];

 	\node at (0,0.5) {\scalebox{1}{$1$}};
 	\node at (0,2) {\scalebox{1}{$2$}};

	\node (n) at (0,-0.5) {};
	\end{tikzpicture}}}
\ \circ_1 \
	\vcenter{\hbox{\begin{tikzpicture}[scale=0.7]
	 \draw[thick] (-1,0.5)--(0,2)--(1, 0.5);

	 \draw[fill=white, thick] (-1,0.5) circle [radius=10pt];
	 \draw[fill=white, thick] (1,0.5) circle [radius=10pt];
 	 \draw[fill=white, thick] (0,2) circle [radius=10pt];

 	\node at (-1,0.5) {\scalebox{1}{$2$}};
 	\node at (1,0.5) {\scalebox{1}{$3$}};
 	\node at (0,2) {\scalebox{1}{$1$}};
	\end{tikzpicture}}}
\ = \
\vcenter{\hbox{\begin{tikzpicture}[scale=0.7]
	 \draw[thick] (-1,-1)--(0,0.5)--(1, -1);
 	 \draw[thick] (0,0.5)--(0,2);

	\draw[thick]  (-0.4,2.7) arc [radius=0.4, start angle=180, end angle= 0];
	\draw[thick] (0, 2) to [out=135,in=270] (-0.4,2.7);
	\draw[thick] (0, 2) to [out=45,in=270] (0.4,2.7);

	 \draw[fill=white, thick] (-1,-1) circle [radius=10pt];
	 \draw[fill=white, thick] (1,-1) circle [radius=10pt];
 	 \draw[fill=white, thick] (0,2) circle [radius=10pt];
  	 \draw[fill=white, thick] (0,0.5) circle [radius=10pt];

 	\node at (-1,-1) {\scalebox{1}{$2$}};
 	\node at (1,-1) {\scalebox{1}{$3$}};
 	\node at (0,2) {\scalebox{1}{$4$}};
 	\node at (0,0.5) {\scalebox{1}{$1$}};

	\node (n) at (0,-2) {};
	\end{tikzpicture}}}
\ + \
\vcenter{\hbox{\begin{tikzpicture}[scale=0.7]
	 \draw[thick] (0,2)--(0,0.5)--(1.5,2)--(1.5,0.5);

	\draw[thick]  (-0.4,2.7) arc [radius=0.4, start angle=180, end angle= 0];
	\draw[thick] (0, 2) to [out=135,in=270] (-0.4,2.7);
	\draw[thick] (0, 2) to [out=45,in=270] (0.4,2.7);

	 \draw[fill=white, thick] (0,0.5) circle [radius=10pt];
 	 \draw[fill=white, thick] (0,2) circle [radius=10pt];
	 \draw[fill=white, thick] (1.5,0.5) circle [radius=10pt];
  	 \draw[fill=white, thick] (1.5,2) circle [radius=10pt];

 	\node at (0,0.5) {\scalebox{1}{$2$}};
 	\node at (0,2) {\scalebox{1}{$4$}};
 	\node at (1.5,0.5) {\scalebox{1}{$3$}};
 	\node at (1.5,2) {\scalebox{1}{$1$}};

	\node (n) at (0,-0.5) {};
	\end{tikzpicture}}}
\ + \
\vcenter{\hbox{\begin{tikzpicture}[scale=0.7]
	 \draw[thick] (0,2)--(0,0.5)--(-1.5,2)--(-1.5,0.5);

	\draw[thick]  (-0.4,2.7) arc [radius=0.4, start angle=180, end angle= 0];
	\draw[thick] (0, 2) to [out=135,in=270] (-0.4,2.7);
	\draw[thick] (0, 2) to [out=45,in=270] (0.4,2.7);

	 \draw[fill=white, thick] (0,0.5) circle [radius=10pt];
 	 \draw[fill=white, thick] (0,2) circle [radius=10pt];
	 \draw[fill=white, thick] (-1.5,0.5) circle [radius=10pt];
  	 \draw[fill=white, thick] (-1.5,2) circle [radius=10pt];

 	\node at (0,0.5) {\scalebox{1}{$3$}};
 	\node at (0,2) {\scalebox{1}{$4$}};
 	\node at (-1.5,0.5) {\scalebox{1}{$2$}};
 	\node at (-1.5,2) {\scalebox{1}{$1$}};

	\node (n) at (0,-0.5) {};
\end{tikzpicture}}} \]

\begin{remark}
	Note that this definition implies that any graph with multiple edges connecting the same vertices %and, more generally, any graph with an automorphism that preserves the vertices but induces an odd permutation of edges 
	vanishes. 
\end{remark}

M.~Kontsevich considered the graph operad $\Gra$ in \cite{Kontsevich97} as an operad of natural operations acting on the sheaf of polyvector fields. We consider a graded commutative associative algebra $A = \k[x^1,\dots,x^n,\theta_1,\dots,\theta_n]$, where $\deg x^i=0$, $\deg \theta_i=-1$, for $i=1,\dots,n$, and we think of $\theta_i$'s as partial derivatives with respect to the variables $x^i$, that is, $[\theta_i,x^j]=\delta_i^j$, for $i,j=1,\dots,n$~. 
In order to define the action $A^{\otimes n}\to A$ of a graph $\Gamma\in\Gra(n)$, we associate to each edge the operator $A\otimes A\to \k$ given by 
\[
\sum_{i=1}^n \Bigg( \frac{\partial}{\partial x^i} \otimes \frac{\partial}{\partial \theta_i}  +
\frac{\partial}{\partial \theta_i} \otimes \frac{\partial}{\partial x^i}\Bigg)~.
\]
Then we associate the $i$-th factor in $A^{\otimes n}$ to the vertex labelled by $i$ in $\Gamma$ and apply the operators on the edges to the corresponding factors in $A^{\otimes n}$. Finally, we take the product over all factors in $A^{\otimes n}$ of the resulting expressions according to the algebra structure of $A$. For instance, the graph 
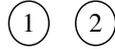
\[
\vcenter{
	\xymatrix@M=5pt@R=10pt@C=10pt{
		*+[o][F-]{1} &*+[o][F-]{2} 
	}
}
\] 
gives the algebra structure map $A\otimes A\to A$, and the graph 
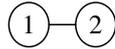
\[
\vcenter{\xymatrix@M=5pt@R=10pt@C=10pt{
		*+[o][F-]{1}\ar@{-}[r] &*+[o][F-]{2} 
	}
}
\]
acts as the Schouten bracket. \\

The graph operad $\Gra$ includes the operad $\Gerst$ of Gerstenhaber algebras, see Section~\ref{subsec:TwGerst}. The inclusion map is given on the arity $2$ generators $\mu$ and $\lambda$ of degrees $0$ and $1$, respectively, by 
\[
\mu\mapsto\vcenter{
	\xymatrix@M=5pt@R=10pt@C=10pt{
		*+[o][F-]{1} &*+[o][F-]{2} 
	}
}\qquad \text{and}
\qquad
\lambda\mapsto
\vcenter{\xymatrix@M=5pt@R=10pt@C=10pt{
		*+[o][F-]{1}\ar@{-}[r] &*+[o][F-]{2} 
	}
}~.
\]
The map $\dsLi \twoheadrightarrow \calS^{-1}\Lie \to \Gerst$ discussed in Section~\ref{subsec:TwGerst},  induces a multiplicative structure on the graph operad $\Gra$, which allows one to twist it.  The theory of operadic twisting ensures that there is a natural action of the deformation dg Lie algebra $\Def\big(\dsLi \allowbreak \to \Gra\big)$ on the twisted operad $\Tw\Gra$ by derivation, see Theorem~\ref{thm:DGActionOnTw}. Let us now discuss the structure of $\Def\big(\dsLi \allowbreak \to \Gra\big)$ and $\Tw\Gra$\index{operad!$\Tw\Gra$} as well as  the main statements about them in more detail. \\

Unfolding the definitions given in Section~\ref{Sec:Gen}, we see that the elements of $\Tw\Gra(n)$ are series indexed by $k\geq 0$ of linear combinations of  graphs with $n$ white vertices labelled by $1,\dots,n$ and $k$ unlabelled black vertices of degree $-2$, with edges of degree $1$. We identify the Maurer--Cartan element 
\[\alpha = \vcenter{\xymatrix@M=5pt@R=10pt@C=10pt{
		*+[o][F**]{}}} \]
 with the graph with no edges and one black vertex. In  general, we think of graphs with $n$ white vertices labelled by $1,\dots,n$ and $k$ unlabelled black vertices as coming from an element of $\Gra(n+k)$ with the last $k$ inputs filled in with $\alpha$'s.
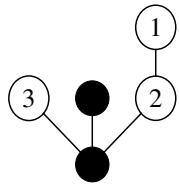
 \[
\vcenter{\xymatrix@M=5pt@R=10pt@C=10pt{
 & & *+[o][F-]{1}\ar@{-}[d]\\
*+[o][F-]{3}\ar@{-}[dr] &*+[o][F**]{\cdot}\ar@{-}[d]&  *+[o][F-]{2}\ar@{-}[dl] \\
& *+[o][F**]{\cdot}  & 
}}\] 
 
  The operadic composition for these series of graphs is defined in exactly the same way as in the case of the graph operad $\Gra$, see Definition~\ref{def:DefGra} above. 

\begin{remark}
	Note that now these graphs  might have automorphisms permuting the black vertices. Any graph with an automorphism that induces an odd permutation of edges 
	vanishes. 
\end{remark}

The twisted differential $\dd^{\lambda_1^\alpha}$ can be computed on each graph $\Gamma$ as the sum of the following three types of terms.
\begin{enumerate}
	\item Replace a black vertex with the graph $\vcenter{\xymatrix@M=5pt@R=10pt@C=10pt{
			*+[o][F**]{}\ar@{-}[r] &*+[o][F**]{} 
		}
	}$~, connect all half-edged previously attached to the original black vertex to the two new ones in all possible ways (at that point we temporarily distinguish the two black vertices of the inserted graph), and multiply the resulting sum of  graphs by $-\frac 12$. We consider the sum over all black vertices of $\Gamma$.
	
	\item Replace the white vertex labelled by $i$ with the graph $\vcenter{\xymatrix@M=4pt@R=10pt@C=10pt{
			*+[o][F-]{\, i\, }\ar@{-}[r] &*+[o][F**]{} 
		}
	}$~, connect all half-edged previously attached to the white vertex to the two new vertices in all possible ways, and multiply the resulting sum of  graphs by $-1$.  We take the sum over all white vertices of $\Gamma$.
	\item Add one extra black vertex to $\Gamma$ with an edge that connects it to an existing vertex of $\Gamma$. We consider the sum over all white and black vertices of $\Gamma$. 
\end{enumerate}
In all these cases, the orientation of the space of edges of the new graphs obtained from $\Gamma$ descends from $\Gamma$, with the new edge added at the first place. 
Notice that many terms cancel due to the sign issue. For trees with at least two vertices (either white or black), it will remain in the image under the twisted differential only the trees from
\begin{enumerate}
\item with at least three half-edges attached to each of the two black vertices, 
\item with at least two half-edges attached to the black vertex, 
\end{enumerate}
since the other terms of (i) and (ii) get cancelled by themselves and by the terms of (iii). 
\\

Let us now consider the deformation complex $\Def\big(\dsLi \to \Gra\big)$. It is a dg Lie algebra identified with $s^2\Tw\Gra(0)$, that is with graphs that have only black vertices. In particular, the differential is the special case of the one described above applied to graphs with no white vertices. It has a natural dg Lie subalgebra $\mathfrak{gc}\subset \Def\big(\dsLi \to \Gra\big)$ that consists of finite linear combinations of connected graphs all of whose vertices are at lease trivalent. 
The latter dg Lie algebra coincides with Kontsevich's graph complex~\cite{Kontsevich93,Kontsevich-FeynDiagLowDim,Kontsevich97}.  

\begin{theorem}[\cite{Willwacher15}] The $0^{\text{th}}$ homology group of the dg Lie algebra $\mathfrak{gc}$ is isomorphic to the Grothendieck--Tei\-chm\"uller Lie algebra $\mathfrak{grt}_1$. 
\end{theorem}\index{Grothendieck--Teichm\"uller Lie algebra}

Let us now go back to the dg operad $\Tw\Gra$. It has a natural dg suboperad $\Tw\Gra^{\mathrm{c}}$ where we retain only those graphs whose all connected components contain at lease one white vertex. The natural map $\Gerst\to \Gra$ descends to a natural map $\Gerst\to \Tw\Gra^{\mathrm{c}}$~.

\begin{theorem}[\cite{Kontsevich99}, \cite{LambrechtsVolic14}, see also an exposition in \cite{Willwacher15}] The natural map  $\Gerst\to \Tw\Gra^{\mathrm{c}}$ is a quasi-isomorphism of dg operads.
\end{theorem}

The theory of operadic twisting ensures that there is a natural action of the deformation dg Lie algebra $\Def\big(\dsLi \allowbreak \to \Gra\big)$ on the twisted operad $\Tw \Gra$ by derivation (\cref{thm:DGActionOnTw}). This action naturally descends to an action of the dg Lie algebra $\mathfrak{gc}$ on $\Tw\Gra^{\mathrm{c}}$~. 
Since this latter one is quasi-isomorphic to the operad $\Gerst$, one gets a natural map of the Grothendieck--Teichm\"uller Lie algebra $\mathfrak{grt}_1$ to ${H}_0\big(\Der(\Gerst_\infty)\big)$~.  \\

Actually T. Willwacher proved that the Grothendieck--Teichm\"uller Lie algebra encaptures all the homotopy derivations. 

\begin{theorem}[\cite{Willwacher15}] There is an isormophism of Lie algebras:
	$$\mathfrak{grt}\coloneqq \mathfrak{grt}_1 \rtimes \k \cong {H}_0\left(\Der\left(\Gerst_\infty\right)\right)\ ,$$
where $\k$ is considered as an abelian Lie algebra that acts on the elements of $\mathfrak{grt}_1$ of degree $k$ by multiplication by $k$.  
\end{theorem}

This gives, among other things, a proof of the fact that the group of homotopy automorphisms of the rational completion of the little disks operad is isomorphic to the pro-unipotent Grothendieck--Teichm\"uller group\index{Grothendieck--Teichm\"uller group}, see also the unstable approach of B. Fresse \cite{Fresse17II} using the rational homotopy theory for operads.

\section{Deligne conjecture}\label{sec:Deligne}

The celebrated Deligne conjecture states that there should exist a dg operad, quasi-isomorphic to the singular chain operad of the little disks, that acts on the Hochschild cochain complex of an $\Ai$-algebra, lifting the Gerstenhaber algebra structure on cohomology, see \cite{MR3261598,BergerFresse04, MR1321701,KontsevichSoibelman00,McClureSmith02, Tamarkin07, MR1328534}. 
In this section, we follow the exposition given in~\cite{DolgushevWillwacher15} and we present a solution to this conjecture based on two  operads, the \emph{brace trees operad} and the \emph{braces operad},  constructed via the operadic twisting procedure. 
The definitions are similar to the ones given in the previous section dealing with the graph operad: we consider first an operad spanned by a certain kind of graphs (planar rooted trees) with the composition map given by the insertion at a vertex and a sum over all the possible ways to graft the incoming sub-graphs, then we twist it, and finally we consider a sub-operad of this latter one. \\

Let $E=E(\Gamma)$ be the set of edges of a planar rooted tree  $\Gamma$. 
We will actually work with the set linear spans of pairs $(\Gamma, \omega)$, where $\omega$ is an orientation of $\mathbb{R}^E$, under the identification that the opposite choice of orientation produces the opposite pair:  $(\Gamma,\omega) = - (\Gamma, \omega^{\mathrm{op}})$. 
For the $i$-th vertex of $\Gamma$, we denote by $L_i(\Gamma)$ the ordered set, from left to right with respect to the planar structure, of all incoming edges attached above the $i$-th vertex.

\begin{definition}[The brace trees operad $\mathrm{BT}$]\label{def:BT} \index{brace trees operad}\index{operad!$\mathrm{BT}$}
The \emph{brace trees operad} $\mathrm{BT}$ is defined as follows. As a graded $\k$-module, the space $\mathrm{BT}(n)$, for $n\ge 1$, has a basis of planar rooted trees with $n$ vertices labelled from $1$ to $n$, where one is designated as the root, with $n-1$ edges of degree $1$, and one particular vertex for the root. 
The symmetric group $\Sy_n$ acts by permutations of the labels on the vertices. 
 \[
\vcenter
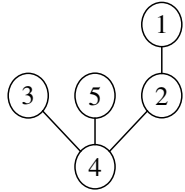
{\xymatrix@M=5pt@R=10pt@C=10pt{
 & & *+[o][F-]{1}\ar@{-}[d]\\
*+[o][F-]{3}\ar@{-}[dr] &*+[o][F-]{5}\ar@{-}[d]&  *+[o][F-]{2}\ar@{-}[dl] \\
& *+[o][F]{4}  & 
}}\] 	
The operadic composition 
\[\circ_i\ \colon\  \mathrm{BT}(n)\otimes \mathrm{BT}(m)\to \mathrm{BT}(n+m-1)\ ,\] 
for $m,n\geq 1$, and$1\leq i \leq n$, applied to two planar rooted trees $\Gamma_1$ and $\Gamma_2$ is equal to the sum over all possible planar rooted trees obtained by the following procedure.
	\begin{enumerate}
		\item Replace the vertex $i$ in $\Gamma_1$ with the planar rooted tree $\Gamma_2$. 
		\item When the vertex $i$ of $\Gamma_1$  is not its root vertex, then attach its output edge to the root vertex of $\Gamma_2$. 
		\item Attach all edges in $L_i(\Gamma_1)$ to the vertices of $\Gamma_2$ preserving their order imposed by the planar structure of $\Gamma_1$.
		\item Relabel accordingly the new overall set of vertices. 
		\item Consider the orientation on the resulting planar rooted tree induced by the standard one for the direct sum $\RR^{E(\Gamma_1)}\oplus \RR^{E(\Gamma_2)}$.
	\end{enumerate}
\end{definition}

Here is an example of a composition in the brace trees operad $\mathrm{BT}$:
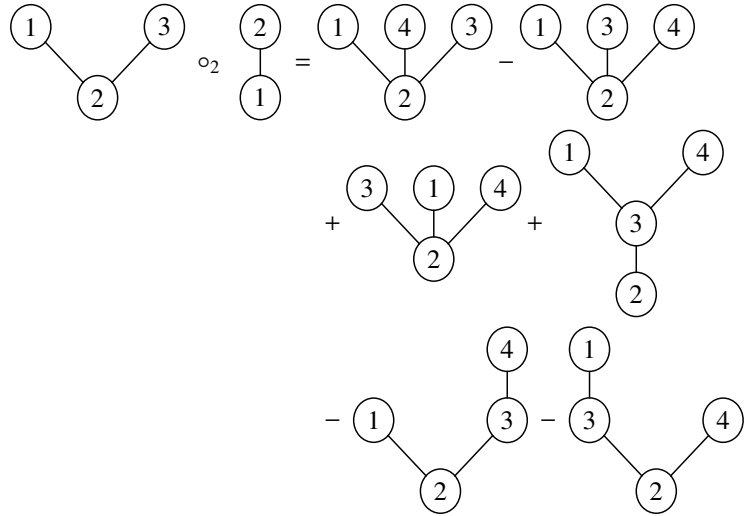
\begin{align*}
\vcenter{\hbox{
	\xymatrix@M=5pt@C=10pt@R=10pt{
		*+[o][F-]{1}\ar@{-}[dr] &&  *+[o][F-]{3}\ar@{-}[dl] \\
		& *+[o][F-]{2} & 
}}}
\ \circ_2\ 
\vcenter{\hbox{
	\xymatrix@M=5pt@R=10pt@C=10pt{
		*+[o][F-]{2}\ar@{-}[d]  \\
		*+[o][F-]{1}
}}}\ = \ &
\vcenter{\hbox{\xymatrix@M=5pt@C=10pt@R=10pt{
	*+[o][F-]{1}\ar@{-}[dr] & *+[o][F-]{4}\ar@{-}[d] &  *+[o][F-]{3}\ar@{-}[dl] \\
	& *+[o][F-]{2}
}}}
-\vcenter{\hbox{\xymatrix@M=5pt@C=10pt@R=10pt{
	*+[o][F-]{1}\ar@{-}[dr] & *+[o][F-]{3}\ar@{-}[d] &  *+[o][F-]{4}\ar@{-}[dl] \\
	& *+[o][F-]{2}
}}}\\ & 
+\vcenter{\hbox{\xymatrix@M=5pt@C=10pt@R=10pt{
	*+[o][F-]{3}\ar@{-}[dr] & *+[o][F-]{1}\ar@{-}[d] &  *+[o][F-]{4}\ar@{-}[dl] \\
	& *+[o][F-]{2}
}}}
+\vcenter{\hbox{\xymatrix@M=5pt@C=10pt@R=10pt{
	*+[o][F-]{1}\ar@{-}[dr] &  &  *+[o][F-]{4}\ar@{-}[dl] \\
	& *+[o][F-]{3}\ar@{-}[d] & \\
	&*+[o][F-]{2}
}}}\\
&
-\ \vcenter{\hbox{\xymatrix@M=5pt@C=10pt@R=10pt{
	 &  &  *+[o][F-]{4}\ar@{-}[d] \\
	*+[o][F-]{1}\ar@{-}[dr] && *+[o][F-]{3}\ar@{-}[dl]  \\
	&*+[o][F-]{2}
}}}
\ -\ \vcenter{\hbox{\xymatrix@M=5pt@C=10pt@R=10pt{
	 *+[o][F-]{1}\ar@{-}[d]&  &  \\
	 *+[o][F-]{3}\ar@{-}[dr]&&  *+[o][F-]{4}\ar@{-}[dl]  \\
	&*+[o][F-]{2}
}}}\ .
\end{align*}
The signs here assume a natural order on edges of the rooted planar trees listed from left to right and from bottom to top.

\begin{remark}
The brace trees operad is a planar and shifted version of the rooted trees operad $\mathrm{RT}$ introduced in Section~\ref{subsec:TwistableAlg}. 
That latter operad plays a similar part in Section~\ref{sec:natural} below. 
\end{remark}

One can check that there is a map of operads $\dsLi \twoheadrightarrow \calS^{-1}\Lie \to \mathrm{BT}$ that sends the binary generator of $\dsLi$ to
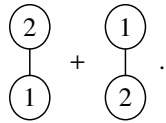
\[
\vcenter{\hbox{
	\xymatrix@M=5pt@R=10pt@C=10pt{
		*+[o][F-]{2}\ar@{-}[d]  \\
		*+[o][F-]{1}
}}}\ \ + \ 
\vcenter{\hbox{
	\xymatrix@M=5pt@R=10pt@C=10pt{
		*+[o][F-]{1}\ar@{-}[d]  \\
		*+[o][F-]{2}
}}}\ \  .
\]
Thus the operad $\mathrm{BT}$ is an operad under $\dsLi$, and we can twist it. Unfolding the definitions given in Section~\ref{Sec:Gen}, one can see that the elements of $\Tw\, \mathrm{BT}(n)$ are series, \index{operad!$\Tw\, \mathrm{BT}$} indexed by $k\geq 0$, of linear combinations of planar rooted trees with $n$ white vertices labelled by $1,\dots,n$ and $k$ unlabelled black vertices of degree $-2$, with edges of degree $1$. We identify the Maurer--Cartan element $\alpha$ with the planar rooted tree with no edges and one black vertex:
\[
\alpha = \vcenter{\xymatrix@M=5pt@R=10pt@C=10pt{
		*+[o][F**]{}}}\  . 
\] 
In general, we think of a planar rooted tree with $n$ white vertices labelled by $1,\dots,n$ and $k$ unlabelled black vertices as an element of $\mathrm{BT}(n+k)$ with the last $k$ inputs filled in with $\alpha$'s. The operadic composition for these series of planar rooted trees is defined in exactly the same way as in the case of the operad $\mathrm{BT}$, see Definition~\ref{def:BT} above. \\

The twisted differential $\dd^{\lambda_1^\alpha}$ can be computed on each planar rooted tree $\Gamma$ as the sum of four types of terms.
\begin{enumerate}
	\item Replace a black vertex $v$ by the planar rooted tree 
	\[\vcenter{\xymatrix@M=5pt@R=10pt@C=10pt{
			*+[o][F**]{}\ar@{-}[d] \\ *+[o][F**]{}
		}
	}\] with two black vertices and connect all  edges from $L_v(\Gamma)$ to these two new black vertices in all possible ways preserving their order imposed by the planar structure. 
	When the black vertex $\nu$ is not the root vertex, then attach its output edge to the root vertex of the above two-vertices tree. 
	Multiply the resulting sum of graphs by $-1$. We consider the sum over all black vertices~$v$ of $\Gamma$.
	
	\item Replace the white vertex labelled by $i$ with the sum of the two planar rooted trees 
	\[
	\vcenter
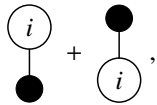
{\xymatrix@M=5pt@R=10pt@C=10pt{
			*+[o][F-]{\, i\,}\ar@{-}[d] \\ *+[o][F**]{}
		} 
	} \,\,+\,\,
	\vcenter{\xymatrix@M=5pt@R=10pt@C=10pt{
			*+[o][F**]{}\ar@{-}[d] \\*+[o][F-]{\, i\,}
		}
	}\,\,,
	\]  and connect all edges from $L_i(\Gamma)$ to these two new vertices in all possible ways preserving their order imposed by the planar structure. 
		When the white vertex $i$ is not the root vertex, then attach its output edge to the root vertex of the above two-vertices trees. 
	Multiply the resulting sum of planar rooted trees by $-1$.  We consider the sum over all white vertices of $\Gamma$.
	
	\item Add one extra black vertex to $\Gamma$ with an edge that connects it to an existing vertex of $\Gamma$. We consider the sum over all white and black vertices of $\Gamma$ and over all different ways to connect a new edge to these vertices with respect to the planar structure.
	 
	\item Attach a new root black vertex
	\[
	\vcenter{\xymatrix@M=5pt@R=10pt@C=10pt{
			{} \\ *+[o][F**]{}\ar@{-}[u]}}
	\]
	 below the root.
\end{enumerate}

In all these cases, the orientation of the space of edges of the new planar rooted trees obtained from $\Gamma$ descends from $\Gamma$, with the new edge added at the first place. 
Notice that many terms cancel due to the sign issue. 
For planar rooted trees with at least two vertices (either white or black), only some planar rooted trees from (i) and (ii) will remain in the image of the twisted differential since all the terms coming from (iii) and  (iv) will get cancelled by corresponding terms coming from (i) or (ii). 

\begin{proposition}[{\cite[Section~5.4]{KontsevichSoibelman00}}]\label{prop:TwBTactsonCH}
There is a natural action of the twisted brace trees operad $\Tw\,\mathrm{BT}$ on the Hochschild cochain complex of an $\Ai$-al\-ge\-bras such that the actions of 
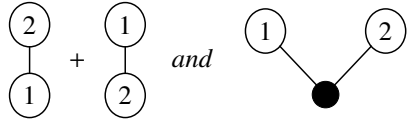
\[ \vcenter{\hbox{
	\xymatrix@M=5pt@R=10pt@C=10pt{
		*+[o][F-]{2}\ar@{-}[d]  \\
		*+[o][F-]{1}\
}}}\ \ + \ 
\vcenter{\hbox{
	\xymatrix@M=5pt@R=10pt@C=10pt{
		*+[o][F-]{1}\ar@{-}[d]  \\
		*+[o][F-]{2}
}}} \quad \text{and} \quad  \vcenter{\hbox{
	\xymatrix@M=5pt@C=10pt@R=10pt{
		*+[o][F-]{1}\ar@{-}[dr] &&  *+[o][F-]{2}\ar@{-}[dl] \\
		& *+[o][F**]{}
}}}
\]
give respectively the Gerstenhaber bracket and the cup product on cohomology.
\end{proposition}

\begin{proof}
Let us recall that the Hochschild cochain complex of an $\Ai$-algebra  $(A, d, m_2, m_3, \ldots)$ is a graded module with components of the form  $\Hom\big((sA)^{\otimes l}, \allowbreak A\big)$ for $l\geq 1$, see Example~\ref{ex:HochschildCochainCx} or Section~\ref{sec:CurvedAiAlg}.  Here we view the structural operations $m_l$ as degree $-2$ elements of them. The action 
\[
F_\Gamma\colon \Hom\big((sA)^{\otimes l_1},A\big)\otimes \cdots \otimes \Hom\big((sA)^{\otimes l_n},A\big) \mapsto \Hom\big((sA)^{\otimes p},A\big) 
\]
of a planar rooted tree $\Gamma$ with $n$ labelled white vertices and $k$ black vertices is defined as follows. 
Let $f_i\in \Hom\big((sA)^{\otimes l_i},A\big)$, for $1\leq i\leq n$, and let $a_j\in A$, for $1\leq j \leq p$.
We attach $m$ additional vertices decorated respectively by $a_j$, for $1\leq j \leq m$, to the planar rooted tree $\Gamma$, each connected by a single edge to one of the already existing vertices (black or white) in $\Gamma$, so that their order agrees with the one imposed by the planar structure of $\Gamma$. We denote the new planar rooted tree by $\Gamma'$. 
If $|L_i(\Gamma')|\not= m_i$, for at least one index $1\leq i \leq n$, we set the upshot value to be trivial, that is 
$F_\Gamma(f_1, \ldots, f_n)(sa_1, \ldots, sa_p)\coloneqq 0$. 
Otherwise, we decorate each edge by the suspension $s$, each white vertex with the corresponding map $f_i$, and each black vertex $v$ with $m_{|L_v(\Gamma')|}$. 
The sum over all the possible evaluations of such operations on the respective elements according the composition scheme provided by the planar rooted trees 
$\Gamma'$ gives an element of $A$ which defines the requested action 
$F_\Gamma(f_1, \ldots, f_n)(sa_1, \ldots, sa_p)$~. 
We leave the interested reader to check that this definition is compatible with the respective differentials, see \cite[Appendix~B]{DolgushevWillwacher15} if needed. 
\end{proof}

We are now going to see that the dg operad $\Tw\, \mathrm{BT}$ provides us with a solution to the Deligne conjecture. However, this is better seen with the following homotopy equivalent sub-operad, introduced originally by M. Kontsevich and Y. Soibelman in \cite{KontsevichSoibelman00}. 

\begin{definition}[The braces operad $\mathrm{Br}$]\index{braces operad}\index{operad!$\mathrm{Br}$}
The \emph{brace operad} $\mathrm{Br}$ is the dg sub-operad of $\Tw\mathrm{BT}$ spanned by planar rooted trees whose black vertices have at least two incoming edges. 
\end{definition}

The following statement shows that there is no loss of operadic homotopical data. 

\begin{theorem}[{\cite[Theorem~9.3]{DolgushevWillwacher15}}] \label{thm:Br-TwBT}
	The natural inclusion $\mathrm{Br}\to \Tw\,\mathrm{BT}$ is a quasi-iso\-mor\-phi\-sm. 
\end{theorem}

In the end, Proposition~\ref{prop:TwBTactsonCH} and the following result give a proof to the Deligne conjecture.  

\begin{theorem}[{\cite[Theorem~3]{KontsevichSoibelman00}}]
\label{thm:DeligneConjecture} 
The dg operad $\mathrm{Br}$ is quasi-isomorphic to the dg operad of the singular chains of the little disks operad. 
\end{theorem}

One can actually go a bit further and show that this solution of the Deligne conjecture shares some universal properties. First, one can see, by a direct computation, that the canonical $\Tw$-coalgebra structure on the twisted operad $\Tw\, \mathrm{BT}$ restricts to the brace suboperad $\mathrm{Br}$. 
Recall that the dg operad of the singular chains of the little disks is formal \cite{Kontsevich99, Tamarkin03, LambrechtsVolic14} and that its homology is isomorphic to the Gerstenhaber operad \cite{Cohen76}. 
These above-mentioned results with 
Proposition~\ref{prop:TwQIBIS} 
and
Theorem~\ref{thm:TwGerst}, 
show that the respective counits of these two $\Tw$-stable dg operads $\Tw\, \mathrm{BT}$ and $\mathrm{Br}$  are quasi-isomorphisms: 
\[
\Tw\, \big(\Tw\,  \mathrm{BT}\big) \xrightarrow{\sim} \Tw\, \mathrm{BT}
\qquad \text{and} \qquad 
\Tw\, \mathrm{Br} \xrightarrow{\sim} \mathrm{Br}
\ .\]
Using all these facts, Dolgushev--Willwacher  \cite[Theorem~1.1]{DolgushevWillwacher15} show that ``every solution to the Deligne conjecture is homotopic to one compatible with the operadic twisting''. 

\section{Lie version of the Deligne conjecture}\label{sec:natural}

One of the motivations for studying the twisted operad $\Tw\,\calP$ of an operad $\calP$, into which the operad $\Li$, maps comes from the search for operations that are naturally defined on twisted $\Li$-algebras coming from $\calP$-algebras. Since the deformation complexes 
$\a_{\, \calC, A}$
 of morphisms of operads carry canonical pre-Lie algebra structures, see Section~\ref{subsec:CompConvAlg}, the twisted dg Lie algebra structure on each of them is in fact a part of a $\Tw \,\PreLie$-algebra structure. It is thus natural to ask whether the dg Lie algebra structure is the only homotopy meaningful structure on those complexes. \\
 
A general framework to studying some universal operations on the deformation complexes $\a_{\, \calP^{\ac}, A}^\alpha$ of 
$\calP$-algebra structures $\alpha \colon \calP \to \End_A$, for a Koszul operad $\calP$, was proposed by M. Markl in \cite{MR2309979}. 
Inspired by the construction of the brace operad $\mathrm{Br}$ of Kontsevich--Soibelman \cite{KontsevichSoibelman00} in the case of the Koszul operad $\calP=\Ass$, see Proposition~\ref{prop:TwBTactsonCH}, he defined in \emph{op. cit.} a dg operad $\calB_\calP$ acting on such deformation complexes. 
(Recall the usual operadic convention for the category of associative algebras: they can be either encoded by an ns operad denoted $\As$ or by an operad denoted $\Ass$.) That operad is not very manageable in general, as it is spanned by operations indexed by certain decorated trees labelled additionally by maps 
 \[
\Phi_{l_1,\ldots,l_n,p}\ \colon \ \calP^{\ac}(p) \to \calP^{\ac}(l_1)\otimes\cdots\otimes\calP^{\ac}(l_n)
 \]
for appropriate values of arities $p, l_1,\ldots, l_n,$. \\

There are exactly two situations where the definition of the operad $\calB_\calP$ simplifies rather drastically: the cases $\calP=\Lie$ and $\calP=\Ass$. In each of these cases, the components of $\calP^{\ac}$ are, in a sense, one-dimensional: for the Lie operad, that is literally true, while for the associative operad, it is the Koszul dual ns cooperad which is one-dimensional in each arity. 
An exhaustive study of the operad $\calB_{\Ass}$ was undertaken by M. Batanin and M. Markl in \cite{MR3261598}: they established that this operad has the homotopy type of the operad of singular chains on the little disks, thus establishing a strong  form of the Deligne conjecture. The operad $\calB_{\Lie}$ has not been thoroughly investigated until recently: M. Markl used it in~\cite{MR2325698} to describe Lie elements in free pre-Lie algebras, but otherwise it only received due attention in a recent article of the first author and A. Khoroshkin in \cite{DotKhPreLie}. \\

From the previous section~\ref{sec:Deligne}, we already know that the operadic twisting procedure plays a key role in understanding the Deligne conjecture in a conceptual way. We shall now explain that the operad $\calB_{\Lie}$ also admits an interpretation in the context of the operadic twisting, that leads to a Lie version of the Deligne conjecture. 
The presentation is aligned with that of the two previous sections \ref{sec:GTGra} and \ref{sec:Deligne}. \\

Recall from Section~\ref{subsec:TwistableAlg}, that the definition of the shifted version $\calS^{-1}\mathrm{RT}$ of the rooted trees operad is similar to the graph operad $\Gra$ (Section~\ref{sec:GTGra}) and to the brace trees operad $\mathrm{BT}$ (Section~\ref{sec:Deligne}) but where one is using non-planar rooted trees this time.  
Recall also that the rooted trees operad has been proved by Chapoton--Livernet~\cite[Theorem~1.9]{ChapotonLivernet01} to be isomorphic to the pre-Lie operad:
\[\mathrm{RT}\cong \PreLie\ .\]

There is a map of operads $\dsLi \twoheadrightarrow \calS^{-1}\Lie \to \calS^{-1}\mathrm{RT}$ that sends the binary generator of $\dsLi$ to
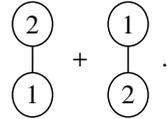
\[
\vcenter{\hbox{
	\xymatrix@M=5pt@R=10pt@C=10pt{
		*+[o][F-]{2}\ar@{-}[d]  \\
		*+[o][F-]{1}
}}}\ \ + \ 
\vcenter{\hbox{
	\xymatrix@M=5pt@R=10pt@C=10pt{
		*+[o][F-]{1}\ar@{-}[d]  \\
		*+[o][F-]{2}
}}}\ \  .
\]
Thus the operad $\calS^{-1}\mathrm{RT}$ is an operad under $\dsLi$, and we can twist it. Unfolding the definitions given in Section~\ref{Sec:Gen}, we see that the elements of $\Tw\, \allowbreak\calS^{-1}\mathrm{RT}(n)$\index{operad!$\Tw\, \calS^{-1}\mathrm{RT}(n)$} are series, indexed by $k\geq 0$, of linear combinations of rooted trees with $n$ white vertices labelled by $1,\dots,n$ and $k$ unlabelled black vertices of degree $-2$, with edges of degree $1$. We identify the Maurer--Cartan element $\alpha$ with the rooted tree with no edges 
\[
\alpha = \vcenter{\xymatrix@M=5pt@R=10pt@C=10pt{
		*+[o][F**]{}}}
\] 
and in general we think of a rooted tree with $n$ white vertices labelled by $1,\dots,n$ and $k$ unlabelled black vertices as an element of $\calS^{-1}\mathrm{RT}(n+k)$ with the last $k$ inputs filled in with $\alpha$'s. The operadic composition for these series of planar rooted trees is defined in exactly the same way as in the case of the rooted trees operad $\mathrm{RT}$, see Section~\ref{subsec:TwistableAlg}.\\

The twisted differential $\dd^{\lambda_1^\alpha}$ can be computed on each rooted tree $\Gamma$ as the sum of four types of terms.
\begin{enumerate}
	\item Replace a black vertex $v$ by the rooted tree 
	\[\vcenter{\xymatrix@M=5pt@R=10pt@C=10pt{
			*+[o][F**]{}\ar@{-}[d] \\ *+[o][F**]{}
		}
	}\] with two black vertices and connect all  edges from $L_v(\Gamma)$ to these two new black vertices in all possible ways. 
	When the black vertex $\nu$ is not the root vertex, then attach its output edge to the root vertex of the above two-vertices tree. 
	Multiply the resulting sum of graphs by $-1$. We consider the sum over all black vertices~$v$ of $\Gamma$.
	
		\item Replace the white vertex labelled by $i$ with the sum of the two rooted trees 
	\[
	\vcenter
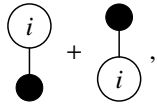
{\xymatrix@M=5pt@R=10pt@C=10pt{
			*+[o][F-]{\, i\,}\ar@{-}[d] \\ *+[o][F**]{}
		} 
	} \,\,+\,\,
	\vcenter{\xymatrix@M=5pt@R=10pt@C=10pt{
			*+[o][F**]{}\ar@{-}[d] \\*+[o][F-]{\, i\,}
		}
	}\,\,,
	\]  and connect all edges from $L_i(\Gamma)$ to these two new vertices in all possible ways. 
	When the white vertex $i$ is not the root vertex, then attach its output edge to the root vertex of the above two-vertices trees.
	Multiply the resulting sum of planar rooted trees by $-1$.  We consider the sum over all white vertices of $\Gamma$.

	\item Add one extra black vertex to $\Gamma$ with an edge that connects it to an existing vertex of $\Gamma$. We consider the sum over all white and black vertices of $\Gamma$ and over all different ways to connect a new edge to these vertices with respect to the planar structure.

	\item Attach a new root black vertex
	\[
	\vcenter{\xymatrix@M=5pt@R=10pt@C=10pt{
			{} \\ *+[o][F**]{}\ar@{-}[u]}}
	\]
	 below the root.
\end{enumerate}
In all these cases, the orientation of the space of edges of the new rooted trees obtained from $\Gamma$ descends from $\Gamma$, with the new edge added at the first place. Notice that many terms cancel due to the sign issue. 
For  rooted trees with at least two vertices (either white or black), only some planar rooted trees from (i) and (ii) will remain in the image of the twisted differential since all the terms coming from (iii) and  (iv) will get cancelled by corresponding terms coming from (i) or (ii). 
 \\

It turns out that the homology of this twisted dg operad is  simpler than that of the operad $\Tw\, \mathrm{BT}$, which is given by the Gerstenhaber operad. We outline the corresponding proof, and we refer the reader to~\cite{DotKhPreLie} for further details on the other results of this section.

\begin{theorem}[{\cite[Theorem~5.1]{DotKhPreLie}}]\label{th:TwPL}
The inclusion of dg operads $$\calS^{-1}\Lie\rightarrow\Tw \, \calS^{-1}\mathrm{RT}$$ induces a homology isomorphism.
\end{theorem}

\begin{proof}
We shall first establish that the arity zero part of $\Tw\, \calS^{-1}\mathrm{RT}$ is acyclic. For that, we note that for such rooted trees the differential is the ``usual'' graph complex differential of M. Kontsevich \cite{MR1247289}, and its acyclicity can be proved by the following version of an argument of T. Willwacher \cite[Proposition~3.4]{Willwacher15}. Let us call an \emph{antenna} of a rooted tree $\Gamma$ a maximal connected subtree consisting of vertices of valences $1$ and $2$ (in particular, each leaf of $\Gamma$ is included in its own antenna); a tree can be viewed as a ``core'' without vertices of valence $1$ with several antennas attached to it. We consider the filtration of the chain complex $\Tw\, \calS^{-1}\mathrm{RT}(0)$ by the size of the core: 
$F^p \Tw\, \calS^{-1}\mathrm{RT}(0)$ is spanned by rooted trees whose all vertices are black and whose core has at least $-p$ vertices.  The differential of the first page of the corresponding spectral sequence is the summand that preserves the size of the core and increases the length of the antennas; it can be easily seen acyclic as follows. We may first consider planar rooted trees, fixing a total order on the children of each vertex; this forces the graph automorphisms to disappear, and the first page of the spectral sequence is the tensor product of acyclic complexes for (the non-empty set of) individual antennas; dealing with trees without a fixed planar structure amounts to taking the subcomplex of invariants complex with respect to graph automorphisms, which is acyclic since it splits as a direct summand due to the Maschke theorem. Since this filtration is  exhaustive and bounded below, we conclude with the convergence of the spectral sequence. 

Let us now consider the case of positive arities $n>0$. We consider the decomposition 
 $
\Tw\, \calS^{-1}\mathrm{RT}\cong V(n)\oplus W(n),
 $ 
 where $V(n)$ is spanned by the rooted trees where the vertex with label~$1$ has no incident edge, and $W(n)$ is spanned by the trees where the white vertex with label~$1$ has at least one incident edge. 
 The above-mentioned analysis of the twisted differential $d^{\lambda_1^\alpha}$ shows that it has  components mapping $V(n)$ to $V(n)$, mapping $W(n)$ to $V(n)$, and mapping $W(n)$ to $W(n)$. We consider the increasing  filtration $F^p\Tw\, \calS^{-1}\mathrm{RT}(n)$ for which $F^pV(n)$ is spanned by rooted trees from $V(n)$ with at least $-p$ edges and $F^pW(n)$ is spanned by rooted trees from $W(n)$ with at least $-p-1$ edges.
Notice already that this filtration is exhaustive and bounded below, so its associated spectral sequence converges. 
 On the first page $E^0$ of the associated spectral sequence, the first differential $d_0$ is given by the part of the full differential that maps $W(n)$ to $V(n)$; it takes the white vertex labelled by $1$ in~$\Gamma$, turns it into a vertex black, and attach to it a new white vertex without inputs and labelled by~$1$:
 \[d_0 \ \  \colon \ 
\vcenter
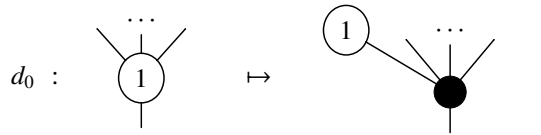
{\hbox{
\xymatrix@M=3pt@R=8pt@C=8pt{
\ar@{-}[dr]&\cdots\ar@{-}[d]&\ar@{-}[dl]\\
&*++[o][F-]{1}\ar@{-}[d]&&\\
&&
}}}
\mapsto 
\vcenter{\hbox{
\xymatrix@M=3pt@R=8pt@C=8pt{
&*++[o][F-]{1}\ar@{-}[drr]&\ar@{-}[dr]&\cdots\ar@{-}[d]&\ar@{-}[dl]&\\
&&&*++[o][F**]{}\ar@{-}[d]&&.&&\\
&&&
}}} 
 \]
Such an assignment is injective. For $n=1$, the cokernel of $d_0$ is spanned by the single one-vertex tree with its white vertex labelled by~$1$, proving that the map $\calS^{-1}\Lie(1)\to\Tw\, \calS^{-1}\mathrm{RT}(1)$ is a quasi-isomorphism.
For $n>1$, the cokernel of $d_0$ is spanned by the rooted trees $\Gamma$ for which the white vertex labelled by$1$ has no input edge and such that its output is connected to another white vertex.  
It splits into a direct sum of subcomplexes according to the number $i$ labelling that latter white vertex; the number of such subcomplexes in arity $n$ is equal to $n-1$. 
One can see that each of these subcomplexes are isomorphic to $\Tw\, \calS^{-1}\mathrm{RT}(n-1)$, by erasing the white vertex labelled by $1$. We may proceed and conclude by induction: each of these subcomplexes is assumed to have its homology isomorphic to $\calS^{-1}\Lie(n-1)$ of dimension $(n-2)!$\ .  So the total dimension of the homology of $\Tw\, \calS^{-1}\mathrm{RT}(n)$ is $(n-1)!$ which is the same as the dimension of $\calS^{-1}\Lie(n)$. It remains to note that the map of operads $\calS^{-1}\Lie\to\calS^{-1}\mathrm{RT}$ is injective, since even the composite 
 \[
\calS^{-1}\Lie\to\calS^{-1}\mathrm{RT}\cong\calS^{-1}\mathrm{PreLie}\to\calS^{-1}\Ass
 \] 
is injective. Since all the elements in the image of the twisted differential $d^{\lambda_1^\alpha}$ contain at least one occurrence of $\alpha$ (i.e. black vertex), this inclusion of dg operads $\calS^{-1}\Lie\rightarrow\Tw\, \calS^{-1}\mathrm{RT}$ yields an injective on the level of homology. The equality of dimensions implies that this inclusion is an isomorphism. 
\end{proof}

%At fixed arity $n\ge 0$, we consider the increasing filtration $F^p \Tw\ncBV(n)$ span\-ned by bamboos containing  at least $-p$ black vertices with tadpoles. Note that this filtration is exhaustive and bounded below since a black vertex with a tadpole carries degree $-1$, and
%the twisted differential $\dd^{\mu_1^\alpha}$ preserves this filtration.

Examining the operad $\calB_{\Lie}$ of universal operations of deformation complexes of $\Li$-algebra structures, one discovers that it is very close to the operad $\Tw\, \calS^{-1}\mathrm{RT}$, so very similar methods apply for computing its homology. This way one can prove the following result.

\begin{theorem}[Lie version of the Deligne conjecture, {\cite[Theorem~6.3]{DotKhPreLie}}]
The operad $\calB_{\Lie}$\index{operad!$\calB_{\Lie}$} of universal operations on deformation complexes of $\Li$-algebra structures has the homotopy type of the operad $\Lie$. 
\end{theorem}

We conclude this section with another guise of operadic twisting arising in this research area. In \cite{MR2325698}, M. Markl considers the pre-Lie algebra $\mathrm{rPL}(V)$, where ``r'' stands for ``reduced'',  defined by the formula
 \[
\mathrm{rPL}(V):=\frac{\PreLie(V\oplus\k\circ)}{(\circ\star\circ)},
 \]
where $\circ$ is an additional generator of degree $-1$. According to \cite[Proposition~3.2]{MR2325698}, there is a well defined map $d\colon\mathrm{rPL}(V)\to\mathrm{rPL}(V)$ of degree~$-1$ that annihilates all generators ($V$ and $\circ$) and satisfies $d^2=0$. But, 
it fails to make $\mathrm{rPL}(V)$ a dg pre-Lie algebra, since the ``derivation'' relation required here takes the following form: 
 \[
d(a\star b)=d(a)\star b+(-1)^{|a|} a\star d(b)+Q(a,b) \ ,
 \]
where $Q(a,b)\coloneqq(\circ\star a)\star b-\circ\star (a\star b)$. The bright point of this definition lies in the following result. 

\begin{theorem}[{\cite[Theorem~3.3]{MR2325698}}]
The subspace of Lie elements in $\PreLie(V)$ equals the kernel of $d$ on the space of degree~$0$ elements $\mathrm{rPL}(V)_0\cong\PreLie(V)$:
 \[
\Lie(V)\cong \ker\big(d\colon\mathrm{rPL}(V)_0\to\mathrm{rPL}(V)_{-1}\big)\ .
 \]
\end{theorem}

In fact, as explained in \emph{loc. cit.}, one can define a differential graded operad $(\mathrm{rPL}, \mathrm{d})$ and view the chain complex $(\mathrm{rPL}(V),d)$ as the result of evaluating the Schur functor corresponding to differential graded $\mathbb{S}$-module $\mathrm{rPL}$ on the vector space~$V$. This operad is given by the coproduct of the rooted tree operad with an arity $0$ operation 
\[\vcenter{
\xymatrix@M=5pt@R=10pt@C=10pt{
*+[o][F-]{}
}
}\]
placed in degree $-1$ modulo the operadic ideal generated by 
\[\vcenter{
\xymatrix@M=5pt@R=10pt@C=10pt{
*+[o][F-]{}\ar@{-}[d]  \\
*+[o][F-]{}
}}\ .\]
Its differential is defined by 
\begin{align*}
\mathrm{d}\left(
\vcenter{\hbox{
\xymatrix@M=5pt@R=10pt@C=10pt{
*+[o][F-]{}
}}}\;\right)=0 \quad \text{and} \quad 
\mathrm{d}
\left(\,\vcenter{\hbox{\xymatrix@M=5pt@R=10pt@C=10pt{
*+[o][F-]{1}\ar@{-}[d]  \\
*+[o][F-]{2}
}}}\,\right)=
\vcenter{\hbox{
\xymatrix@M=5pt@C=10pt@R=10pt{
*+[o][F-]{1}\ar@{-}[dr] &&  *+[o][F-]{2}\ar@{-}[dl] \\
& *+[o][F-]{} & 
}}}\ .
\end{align*}

This brings us very close to the key observation: examining the formulas for the differential in the operad $\Tw\, \calS^{-1}\mathrm{RT}$ and performing the operadic suspension, we find that 
 \[
d^{\lambda_1^\alpha}\  \colon\ 
\vcenter{
\xymatrix@M=5pt@R=10pt@C=10pt{
*+[o][F-]{1}\ar@{-}[d]  \\
*+[o][F-]{2}
}
}\ 
\mapsto\ 
\vcenter{
\xymatrix@M=5pt@C=10pt@R=10pt{
*+[o][F-]{1}\ar@{-}[dr] &&  *+[o][F-]{2}\ar@{-}[dl] \\
& *+[o][F**]{}
}} 
 \]
 in the operad $\Tw\, \PreLie\cong\calS\Tw\, \calS^{-1}\mathrm{RT}$.
So one must think of the element $\circ$ as a shadow of the element $\alpha\in\Tw\, \PreLie$. Incidentally, this calculation also shows why the brace tree corresponding to the Gerstenhaber product in the case of the operad $\Tw\, \mathrm{BT}$, see Proposition~\ref{prop:TwBTactsonCH}, does not survive in the homology of the operad $\Tw\, \calS^{-1}\mathrm{RT}$.  Thus, the operadic twisting emerges naturally in the algebraic context of Lie elements in pre-Lie algebras, from the study of universal operations on deformation complexes.\\

An argument similar to that of the proof of Theorem~\ref{th:TwPL} can now be used to establish a result conjectured by M. Markl about fifteen years ago. We note that the operadic ideal of $\Tw\, \PreLie$ generated by 
\[\vcenter{\hbox{\xymatrix@M=5pt@R=10pt@C=10pt{
*+[o][F**]{}\ar@{-}[d]  \\
*+[o][F**]{}
}}}\]
is closed under differential, and so one may consider the filtration by powers of that ideal and the associated graded chain complex. It turns out that the homology of the associated graded complex is already isomorphic to the operad $\Lie$. 

\begin{theorem}[{\cite[Theorem~6.2]{DotKhPreLie}}]
The quotient dg operad 
\[
\frac{\Tw \, \PreLie}{\left(\,\vcenter{\hbox{\xymatrix@M=5pt@R=10pt@C=10pt{
*+[o][F**]{}\ar@{-}[d]  \\
*+[o][F**]{}
}}}\,\right)}
\] is isomorphic to the dg operad $\mathrm{rPL}$, and the morphism
$\Lie\to \mathrm{rPL}$
induces a homology isomorphism.
\end{theorem}

\chapter{Applications}\label{sec:Applications}\footnotetext{\hrule\smallskip\noindent This material will be published by Cambridge University Press \& Assessment as ‘Maurer-Cartan Methods in Deformation Theory: the twisting procedure’ by Vladimir Dotsenko, Sergey Shadrin and Bruno Vallette. This version is free to view and download for personal use only. Not for re-distribution, re-sale or use in derivative works. \copyright Cambridge University Press \& Assessment}

In this final chapter, we provide details on seminal applications of the twisting procedure in several domains of mathematics mentioned throughout the book: deformation theory, higher Lie theory, rational homotopy theory, higher category theory, and symplectic topology. Other occurrences of the twisting procedure kept emerging as we were preparing this book, and it has been clear from the beginning that it would be impossible to give an exhaustive overview of applications. Yet we attempted to cover several crucial works in many different domains, both to confirm the omnipresence of the twisting procedure and to encourage the reader to make the list of related areas even longer. 

\section{Fundamental theorem of deformation theory}\label{sec:Lurie}

Let us go back to the fundamental theorem of deformation theory. In order to formulate it, one first needs to provide a proper definition of a ``deformation problem''. In this direction, considering simply functors from local Artinian rings to sets is too restrictive. \\

Given any complete dg Lie algebra $\g$, its gauge group $\Gamma$ is well-defined and it acts on the set of Maurer--Cartan elements, as explained in Section~\ref{sec:MCdgLie}. 
One can consider 
the moduli space of Maurer--Cartan elements (Definition~\ref{def:ModuliSpacedgLie})
\[ \mathscr{MC}(\g)\coloneqq \mathrm{MC}(\g)/\Gamma~, \]
given by the set of orbits. 
This way, one looses the data of the gauge group elements. 
To circumvent this, we introduced in Definition~\ref{def:DeligneGroupoid} the Deligne gr\-ou\-poid \index{Deligne!groupoid} whose objects are Maurer--Cartan elements and whose isomorphisms are given by the action of the gauge group elements. From there, one might ask whether this is a way to compare these actions among them, that is to coin a suitable notion of a \emph{Deligne $2$-groupoid}, and, why not, higher up to a 
\emph{Deligne $\infty$-groupoid}\index{Deligne!$\infty$-groupoid}. This can be achieved as follows. 

\begin{definition}[Sullivan simplicial algebra \cite{Sullivan77}]\index{Sullivan simplicial algebra}\label{def:SullivanAlgebra}
The \emph{Sullivan simplicial algebra} is the simplicial commutative dg algebra $\Omega_\bullet$ defined by the piecewise polynomial differential forms on the geometric $n$-simplicies: 
\[
\Omega_n\coloneqq \frac{\k[t_0, \ldots, t_n, d t_0, \ldots, d t_n]}{(t_0+\cdots+t_n-1, d t_0+\cdots+ d t_n)}~,
\]
equipped with their simplicial maps. 
\end{definition}

\begin{theorem}[Hinich~{\cite{Hinich97}}]\label{thm:HinichInt}
For any complete dg Lie algebra $\g$, the simplicial set 
\[\MC_\bullet(\g)\coloneqq \MC\left(\g \,\widehat{\otimes}\, \Omega_\bullet\right)\]
is a Kan complex, that is an $\infty$-groupoid, satisfying 
\[\pi_0(\MC_\bullet(\g))\cong \mathscr{MC}(\g)~. \]
\end{theorem}

This construction shows that one would rather consider functors 
\[\R\cong \k \oplus \m \mapsto \MC_\bullet(\g\otimes \m)\]
from the category of Artinian local rings to the category $\mathsf{sSet}$ of simplicial sets. The source of deformation functors has also to be upgraded to the derived world, i.e., to differential graded objects. 

\begin{definition}[Differential graded Artin algebra]\index{differential graded Artin algebra}
A \emph{differential graded (dg) Artin algebra} is an augmented commutative differential graded algebra $A$ satisfying: 
\begin{enumerate}
\item each $H_i(A)$ is finite dimensional, for any $i\in \NN$, 

\item the homology groups $H_i(A)\cong 0$ vanish, for $i<0$ and for $i$ big enough, 

\item the commutative algebra $H_0(A)$ is a local Artinian algebra. 
\end{enumerate}
We denote their category by $\mathsf{dgArt}$. 
\end{definition}

The idea behind this definition is that any dg Artin algebra can be obtained, up to a quasi-isomorphism, by successive square-zero extensions from $\k$, see \cite{Toen17}.

\begin{definition}[Formal moduli problem]\index{formal moduli problem}
	\label{def:FormalModuliProblem}
A \emph{formal moduli problem} is a functor 
\[
F \ : \ \mathsf{dgArt} \to \mathsf{sSet}
\] such that:
\begin{enumerate}
\item it sends quasi-isomorphisms of dg commutative algebras to weak equivalences of simplicial sets, 

\item the image $F(\k)\sim *$ of the initial algebra  is contractible, 

\item the image under $F$ of any homotopy cartesian square 
\[
\xymatrix@C=30pt@R=30pt{
D \ar[r] \ar[d] &   B  \ar[d] \\ 
C\ar[r] & A  }
\]
of augmented dg commutative algebras, such that $H_0(B) \twoheadrightarrow H_0(A)$ and 
$H_0(C) \twoheadrightarrow H_0(A)$, 
 is a homotopy cartesian square of simplicial sets. 
\end{enumerate}
\end{definition}

\begin{theorem}[Fundamental theorem of deformation theory {\cite{Lurie10, Pridham10}}]\index{fundamental theorem of deformation theory}
Over a field $\k$ of characteristic $0$, the $\infty$-categories of dg Lie algebras and of formal moduli problems are equivalent. 
\end{theorem}

The proof of such a conceptual result goes far beyond the scope of the present manuscript. However, one can notice that this equivalence heuristically comes from the following  two functors. From formal moduli problems to dg Lie algebras, one  can consider the tangent space functor 
\[\mathrm{T} \ : F \mapsto F\left(\k[t]/\big(t^2\big)\right)~.\]
In the other way round, one can consider the functor given above
\[\mathrm{Def}_\g \ :  A \mapsto \MC_\bullet\big(\g\otimes \bar{A}\big)~,\]
where $\bar{A}$ stands for the augmentation ideal. 
The fact that $\mathrm{Def}_\g$ forms a formal moduli problem is not completely trivial: Point (i) of Definition~\ref{def:FormalModuliProblem} comes from  the homotopy invariance property of Hinich's construction \cite{Hinich97}, which ultimately goes back to Goldman--Millson \cite{GoldmanMillson88}. 

\begin{example}
There is a general case where one can make  the deformation dg Lie algebra explicit: the deformation problem of $\calP_\infty$-algebras when $\calP$ is a Koszul operad or properad. The operadic calculus shows that the deformation dg Lie algebra is given by the convolution dg Lie algebra 
\[\g\coloneqq\left(
\Hom_\Sy\left(
\calP^{\ac}, \End_V
\right),  \partial, [\, , ]
\right)~,\]
see \cite[Section~12.2]{LodayVallette12} and \cite{MerkulovVallette09I, HLV20} for more details. 
\end{example}

\begin{remark}
The above treatment actually deals with \emph{pointed} deformation problems:  formal moduli problems are defined by augmented dg commutative algebras and they are modelled by dg Lie algebras which admit $0$ for canonical Maurer--Cartan element. It is natural to expect that \emph{non-necessarily pointed} deformation problems should be encoded by curved Lie algebras. 
\end{remark}

\section{Higher Lie theory}\label{sec:HigherLieTh}

Hinich's construction of a Deligne $\infty$-groupoid associated to a complete dg Lie algebra introduced in the previous section prompts the following two questions. 
\begin{enumerate}

\item Does the set $\MC_1(\g)$ of 1-simplicies  coincide with the set of gauges? 

\item Does Hinich's construction $\MC_\bullet$ extend to complete $\Li$-algebras?
\end{enumerate}

The answer to the second question is actually straightforward as the same formula 
\[\MC_\bullet(\g)\coloneqq \MC\big(\g \,\widehat{\otimes}\, \Omega_\bullet\big)\]
applies \emph{mutatis mutandis} to complete $\Li$-algebras $\g$~.\\ 

The answer to the first question is negative: 
the set $\MC_1(\g)$ of 1-simplicies is \emph{slightly bigger} than the set of gauges. As a consequence, the last point of \cref{thm:HinichInt} is not trivial: the set of 1-simplicies of $\MC_\bullet(\g)$ is the set of Maurer--Cartan elements of the complete dg Lie algebra $\g \,\widehat{\otimes}\, \k[t,dt]$, as such it contains strictly the set of gauges but they define equivalent equivalence relation on the Maurer--Cartan elements of $\g$~. The issue raised by Point (i) was solved by E. Getzler  \cite{Getzler09} who coined 
a homotopy equivalent  $\infty$-subgroupoid of $\MC_\bullet(\g)$, whose set of 1-simplicies coincides with the set of gauges. \\

Actually, Getzler's construction solves both questions. In this section, we give a presentation of it along the lines of \cite{Robert-NicoudVallette20}. 
Let us first recall that the notion of a gauge between Maurer--Cartan elements in complete $\Li$-algebras is defined in a way similar to dg Lie algebras (Section~\ref{sec:MCdgLie}): for any degree $0$ element $\lambda$, the twisted differential defines a vector field $\Upsilon_\lambda\in \Gamma(T(\mathrm{MC}(\g)))$ by the formula
 \[
\Upsilon_\lambda(\omega) := d^\omega(\lambda)= 
\sum_{k\ge 0} {\textstyle \frac{1}{k!}} \ell_{k+1}\big(\omega^k, \lambda \big)~.
 \]
We refer the reader to \cite[Proposition~1.10]{Robert-NicoudVallette20} for a closed formula for the integration of a Maurer--Cartan element along the associated flow. \\

The key idea of the integration theory of complete $\Li$-algebras amounts to using the linear dual of the cellular chain complexes of the geometric $n$-simplices 
\[\mathrm{C}_\bullet\coloneqq C_{\mathrm{cell}}(\Delta^{\bullet})^\vee\ ,\] instead of the Sullivan simplicial algebra,  
under Dupont's simplicial contraction \cite{Dupont76}: 
\[
\hbox{
	\begin{tikzpicture}
	\def\upshift{0.075}
	\def\downshift{0.075}
	\pgfmathsetmacro{\midshift}{0.005}
	
	\node[left] (x) at (0, 0) {$\Omega_\bullet$};
	\node[right=1.5 cm of x] (y) {$\mathrm{C}_\bullet$\quad .};
	
	\draw[->] ($(x.east) + (0.1, \upshift)$) -- node[above]{\mbox{\small{$p_\bullet$}}} ($(y.west) + (-0.1, \upshift)$);
	\draw[->] ($(y.west) + (-0.1, -\downshift)$) -- node[below]{\mbox{\small{$i_\bullet$}}} ($(x.east) + (0.1, -\downshift)$);
	
	\draw[->] ($(x.south west) + (0, 0.1)$) to [out=-160,in=160,looseness=5] node[left]{\mbox{\small{$h_\bullet$}}} ($(x.north west) - (0, 0.1)$);
	
	\end{tikzpicture}}
\]
Pulling back along the bar-cobar resolution $\Omega \mathrm{B} \mathrm{Com}\to  \mathrm{Com}$, one endows $\Omega_\bullet$ with a simplicial $\Omega \mathrm{B} \mathrm{Com}$-algebra structure, which is then transferred to 
$\mathrm{C}_\bullet$ under the homotopy transfer theorem \cite[Section~10.3]{LodayVallette12}. Since this latter one is degree-wise finite dimensional, its linear dual 
$\mathrm{C}^\bullet\coloneqq C_{\mathrm{cell}}(\Delta^{\bullet})$ carries a cosimplicial 
$ \Omega \mathrm{B}  \mathrm{Com}$-coalgebra structure. Finally, we consider its image under the complete cobar construction $\widehat{\Omega}_\pi$ associated to the 
 operadic twisting morphism 
$\pi \colon \mathrm{B} \Omega \mathrm{Com}^\vee \to \Omega \mathrm{Com}^\vee\cong \sLi$~.

\begin{proposition}
The collection 
	\[
	\mathfrak{mc}^\bullet\coloneqq \widehat{\Omega}_\pi\, \mathrm{C}^\bullet
	\]
	 forms a cosimplicial complete shifted $\Li$-algebra, 
	 called the \emph{universal Maurer--Cartan algebra}\index{universal Maurer--Cartan algebra}.
\end{proposition}

For instance, the complete shifted $\Li$-algebra $\mathfrak{mc}^0$ is quasi-free on one Mau\-rer--Cartan element
\[
\mathfrak{mc}^0\cong\widehat{\sLi}(a)
\]
with differential given by
\[
\mathrm{d}a=-\sum_{m\geqslant 2} {\textstyle \frac{1}{m!}}\ell_m(a, \ldots, a)\ .
\]

\begin{definition}[Integration functor]\index{integration functor}
The  \emph{integration functor} of complete shifted $\Li$-algebras is defined by 
\begin{align*}
\mathrm{R}\ :\ \mathsf{complete}\ \sLi\textsf{-}\mathsf{alg}&\ \longrightarrow\ \mathsf{sSet}\\
\g{}&\ \longmapsto\ \Hom_{\mathsf{complete}\ \sLi\mathsf{-alg}}\left(\mathfrak{mc}^\bullet, \g\right)~.
\end{align*}
\end{definition}

The category of complete shifted $\Li$-algebras is considered here with respect to strict morphisms. The functoriality of $\mathrm{R}$ with respect to $\infty$-morphisms is more subtle, so  we refer the reader to \cite[Section~3]{Robert-NicoudVallette20} for more details. Applied to chain complexes viewed as trivial complete shifted $\Li$-algebras, the integration functor gives the Dold--Kan functor. The form of $\mathfrak{mc}^0$ mentioned above shows that the 0-simplicies of the integration functor are the Maurer--Cartan elements: 
\[\mathrm{R}(\g)_0\cong \MC(\g)~.\]

\begin{proposition}[{\cite[Proposition~4.3]{Robert-Nicoud19}}]
The set $\mathrm{R}(\g)_1$ of 1-simplicies is in canonical bijection with the set of triples consisting of a pair of Maurer--Cartan elements related by a gauge. 
\end{proposition}

This provides us with a direct computation of the connected components:
	\[\pi_0(\mathrm{R}(\g))\cong \mathscr{MC}(\g)\ .\]

\begin{remark}
Getzler's original construction was given in terms of the  following ``gauge condition'': 
\[\gamma_\bullet(\g)\coloneqq \MC_\bullet(\g)\,\cap\, \ker \left(\id\widehat{\otimes} \, h_\bullet\right)\ .\]
It turns out that both higher integration functors are canonically isomorphic: 
\[\mathrm{R}_\bullet(\g)\cong \gamma_\bullet(\g)~.\]
\end{remark}

The definition given here in term of the universal Maurer--Cartan algebra allows one to introduce the following  left adjoint functor. 

\begin{proposition}[{\cite[Theorem~2.12]{Robert-NicoudVallette20}}]\label{prop:AdjHigherLie}
The functor
\[\mathfrak{L}\coloneqq \mathrm{Lan}_\mathrm{Y} \mathfrak{mc}_\bullet\]
 of \emph{$\sLi$-algebra of a simplicial set}\index{shifted $\Li$-algebra of a simplicial set},  
defined as the left Kan extension along the 
 Yoneda embedding $\mathrm{Y}$, is left adjoint to the integration functor
\[\vcenter{\hbox{
\begin{tikzcd}[column sep=1.2cm]
\mathfrak{L}  \ \colon \ 
\mathsf{sSet} 
\arrow[r, harpoon, shift left=1ex, "\perp"']
&
\arrow[l, harpoon,  shift left=1ex]
\mathsf{complete}\ \sLi\textsf{-}\mathsf{alg}
\ : \ \mathrm{R}\quad .
\end{tikzcd}
}}\]
\end{proposition}

As usual, such left Kan extensions can be computed as the colimit 
\[\mathfrak{L}(X_\bullet)
 \cong 
\mathop{\mathrm{colim}}_{\mathsf{E}(X_\bullet)} \mathfrak{mc}^\bullet\]
over the category of elements of a simplicial set, that is by gluing, in a suitable way, 
quasi-free complete shifted $\Li$-algebras generated by the cells of standard simplicies. 

\begin{theorem}[\cite{Getzler09}]\label{thm:Getzler}
The simplicial set $\mathrm{R}(\g)$ integrating any complete shifted $\Li$-algebra $\g$ is a Kan complex with a canonical algebraic $\infty$-groupoid structure. It is a homotopy equivalent sub-Kan complex of $\MC_\bullet(\g)$~.
\end{theorem}

\begin{proof}
For $n\geqslant 2$, one can show \cite[Lemma~5.1]{Robert-NicoudVallette20} that there is an isomorphism of complete shifted $\Li$-algebras 
\[\mathfrak{mc}^n=\mathfrak{L}\big(\Delta^{n}\big)\cong \mathfrak{L}\big(\Lambda_{k}^{n}\big)\,\widehat{\vee}\, \widehat{\sLi}(u, du)\ ,\]
	where $u$ is a generator of degree $n$ and where the differential on the right-hand side is given by $d(u)=du$~.
Using the adjunction of Proposition~\ref{prop:AdjHigherLie}, this shows that the set of $\Lambda^n_k$-horn fillers is in one-to-one correspondence with the component $\g_n$ of degree $n$ of the complete shifted $\Li$-algebra. Therefore any horn can be canonically filled by $0\in \g_n$~. 
\end{proof}

The first horn filler 
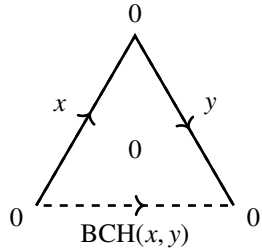
\[
\begin{tikzpicture}

\coordinate (v0) at (210:1.5);
\coordinate (v1) at (90:1.5);
\coordinate (v2) at (-30:1.5);

% draw triangle
\draw[line width=1] (v0)--(v1)--(v2);
\draw[dashed, line width=1]  (v0)--(v1)--(v2)--cycle;

% decorate triangle sides with arrows
\begin{scope}[decoration={
	markings,
	mark=at position 0.55 with {\arrow{>}}},
line width=1
]

\path[postaction={decorate}] (v0) -- (v1);
\path[postaction={decorate}] (v1) -- (v2);
\path[postaction={decorate}] (v0) -- (v2);

\end{scope}

% labels

% vertices
\node at ($(v0) + (30:-0.3)$) {$0$};
\node at ($(v1) + (0,0.3)$) {$0$};
\node at ($(v2) + (-30:0.3)$) {$0$};

% edges
\node at ($(v0)!0.5!(v1) + (-30:-0.4)$) {$x$};
\node at ($(v1)!0.5!(v2) + (30:0.4)$) {$y$};
\node at ($(v0)!0.5!(v2) + (0,-0.4)$) {$\BCH(x, y)$};

% face
\node at (0,0) {$0$};

\end{tikzpicture}
\]
produces the \emph{Baker--Campbell--Hausdorff formula}\index{Baker--Campbell--Hausdorff formula} \cite[Pro\-po\-si\-tion~5.2.36]{Bandiera14}, wh\-ich is the most conceptual formula underlying Lie theory. 
The higher horn fillers produce  \emph{higher Baker--Campbell--Hausdorff formulae}, whose promising extensive stu\-dy is yet to be achieved. \\

The algebraic $\infty$-groupoid $\R(\g)$, which integrates complete shifted $\Li$-al\-ge\-bras, behaves as follows with respect to the twisting procedure. 

\begin{lemma}[{\cite[Theorem~3.17]{Robert-NicoudVallette20}}]\label{lem:TwIntegr}
For any complete shifted $\Li$-algebra $\g$ and any Maurer--Cartan element $\alpha \in \MC(\g)$~, 
the translation by $\alpha$ induces an isomorphism of pointed Kan complexes 
\[
\left(\mathrm{R}(\g), \alpha\right) \cong \left(\mathrm{R}\left(\g^\alpha\right), 0\right)~.
\]
\end{lemma}

Since $\mathrm{R}(\g)$ is a Kan complex, one might try to compute its homotopy groups. This can be done using the following beautiful Hurewicz type theorem due to A. Berglund.

\begin{theorem}[{\cite{Berglund15}}]\label{thm:Berglund}
	For any complete shifted $\Li$-algebra $\g$ and any Maurer--Cartan element $\alpha \in \MC(\g)$~, there is a canonical group isomorphism
	\[
	\pi_n(\mathrm{R}(\g), \alpha)\cong {H}_n\left(\g^\alpha\right)\ , \quad \text{for}\  \ n\geqslant 1\ ,
	\]
	where the group structure on the right-hand side is given by the Baker--Campbell--Hausdorff formula  for $n=1$  and by the sum  for $n\geqslant 2$~. 
\end{theorem}

\begin{proof}
The idea of the proof given in \cite[Section~6.1]{Robert-NicoudVallette20} 
amounts first to using Lemma~\ref{lem:TwIntegr} to reduce the problem to the trivial Maurer--Cartan element
	\[
	\pi_n(\mathrm{R}(\g), \alpha)\cong\pi_n\left(\mathrm{R}\left(\g^\alpha\right), 0\right)
	\]
and then, since everything is based at $0$, to using  the properties of the higher Baker--Campbell--Hausdorff formulae: 
	\[
\pi_n\left(\mathrm{R}\left(\g^\alpha\right), 0\right)\cong  {H}_n\left(\g^\alpha\right)~.
	\]
\end{proof}

\begin{remark}
The present higher Lie theory extends to complete shifted \emph{curved} $\Li$-algebras; this is the subject of the  Ph.D. Thesis of Victor Roca i Lucio \cite{RocaLucio21}. 
This generalisation requires however to improve the operadic calculus to the curved level and to introduce new genuine model category structures. 
\end{remark}

\section{Rational homotopy theory}\label{sec:RHT}
The Maurer--Cartan methods alluded to in the title of this monograph are also behind the construction of models of spaces in rational homotopy theory. They are already omnipresent in the two foundational articles of D. Quillen \cite{Quillen69} and D. Sullivan \cite{Sullivan77}, see also the works of 
K.-T. Chen \cite{Chen73, Chen77} and 
 of M. Schlessinger and J. Stasheff \cite{SchlessingerStasheff}. We give here a short survey of recent progress made using the results presented in this monograph. \index{rational homotopy theory} \\

Let us quickly recall the basic framework of rational homotopy theory.

\begin{definition}[Nilpotent simplicial set]\index{simplicial set!nilpotent }
A connected pointed simplicial set $X_\bullet$ is \emph{nilpotent} when its fundamental group $\pi_1 X_\bullet$ is nilpotent and when its higher homotopy groups $\pi_n X_\bullet$ are nilpotent as $\pi_1 X_\bullet$-modules, for all $n\geqslant2$~.
\end{definition}

\begin{definition}[Rational simplicial set]\index{simplicial set!rational}
A nilpotent  simplicial set $X_\bullet$ is \emph{rational} when 
its homotopy groups $\pi_n X_\bullet$ are uniquely divisible, for any $n\geqslant1$, that is
the equation $x^m=g$ has a unique solution $x\in \pi_n X_\bullet$, for any $g\in \pi_n X_\bullet$ and any integer $m\geqslant1$~. 
\end{definition}

Since the homotopy groups $\pi_n X_\bullet$ are abelian for $n\geqslant 2$, this latter condition simply means that they are $\mathbb{Q}$-vector spaces. 

\begin{definition}[Rationalisation]\index{rationalisation}
A \emph{rationalisation} of a nilpotent simplicial set $X_\bullet$ is a rational simplicial set $X^\mathbb{Q}_\bullet$ equipped with a morphism $r : X_\bullet\to X^\mathbb{Q}_\bullet$ such that 
 \[\pi_n r : \pi_n X_\bullet \otimes \mathbb{Q} \xrightarrow{\cong} \pi_n X_\bullet^\mathbb{Q}~,\]
  for any $n\geqslant1$~, where the left-hand side stands for the Malcev completion \cite{Malcev} in the case $n=1$~.
\end{definition}

Any nilpotent simplicial set admits a rationalisation \cite[Theorem~5.3.2]{MayPonto12};  some classes of nilpotent simplicial sets admit functorial rationalisations, two of which are given below. \\

The first algebraic models for the rational homotopy type of topological spaces were given by D. Quillen \cite{Quillen69} in terms of dg Lie algebras. However, this original approach relies on the composite of many intermediate functors. 
The adjunction described in Proposition~\ref{prop:AdjHigherLie}, which arose in the higher Lie theory, provides us with direct functors between simplicial sets and complete shifted $\Li$-algebras. (We refer the reader to \cite{Chen73, Chen77, Neisendorfer78, Hain84} with emphasis on \cite{BFMT20} for the similar theory developed on the level of shifted Lie algebras.) We will show below that they faithfully encode the rational homotopy types.  
Let us first extend their definitions to pointed spaces. 
The integration functor extends naturally to a functor $\widetilde{\mathrm{R}}$ landing in pointed simplicial sets 
under 
\[
*\cong\mathrm{R}(0)\xrightarrow{\mathrm{R}(0)}\mathrm{R}(\g)~.
\]

\begin{definition}[$\Li$-algebra model]\index{model!$\Li$-algebra}
A complete shifted $\Li$-algebra $\g$ is a \emph{model} of a pointed simplicial set $X_\bullet$ when it is equipped with a 
homotopy equivalence 
\[
\widetilde{\mathrm{R}}(\g)\simeq X_\bullet^\mathbb{Q}~. 
\]
\end{definition}

\begin{proposition}
When $\g$ is a complete shifted $\Li$-algebra model of a pointed simplicial set $X_\bullet$, 
 there are group isomorphisms 
	\[
	\pi_n(X_\bullet)\otimes \mathbb{Q} \cong {H}_n(\g)\ , \quad \text{for}\  \ n\geqslant 1\ ,
	\]
where the right-hand side is equipped with the Baker--Campbell--Hausdorff product, for $n=1$, and with the sum, for $n\geqslant 2$~.
\end{proposition}

\begin{proof}
This is a direct corollary of Berglund's Hurewicz type theorem~\ref{thm:Berglund}. 
\end{proof}

In the pointed setting, the 
complete shifted $\Li$-algebra functor produces a new artificial base point provided by the trivial Maurer--Cartan element $0$, which thus needs to be killed.

\begin{lemma}
The \emph{reduced shifted $\Li$-algebra} of pointed simplicial sets given by the coequalizer
\begin{center}
	\begin{tikzpicture}
	\node (a) at (0,0) {$\mathfrak{L}(*) \cong \mathfrak{mc}^0$};
	\node (b) at (3.5,0) {$\mathfrak{L}(X_\bullet)$};
	\node (c) at (6.5,0) {$\widetilde{\mathfrak{L}}(X_\bullet)$};
	
	\path[->] ($(a.east)+(0,.07)$) edge node[above]{\small $\mathfrak{L}\left(*\to X_\bullet\right)$} ($(b.west)+(0,.07)$);
	\path[->] ($(a.east)-(0,.07)$) edge node[below]{\small $0$} ($(b.west)-(0,.07)$);
	
	\path[->] (b) edge (c);
	\end{tikzpicture}
\end{center}
defines a functor which is left adjoint to the pointed integration functor 
\[\vcenter{\hbox{
\begin{tikzcd}[column sep=1.2cm]
\widetilde{\mathfrak{L}}  \ \colon \ 
\mathsf{sSet}_* 
\arrow[r, harpoon, shift left=1ex, "\perp"']
&
\arrow[l, harpoon,  shift left=1ex]
\mathsf{complete}\ \sLi\textsf{-}\mathsf{alg}
\ : \ 
\widetilde{\mathrm{R}}\quad .
\end{tikzcd}
}}\]
\end{lemma}

\begin{theorem}[{\cite[Theorem~7.19]{Robert-NicoudVallette20}}]
For any  pointed connected finite type simplicial set $X_\bullet$, the unit 
	\[
	X_\bullet \longrightarrow \widetilde{\mathrm{R}}\, \widetilde{\mathfrak{L}}(X_\bullet)
	\]
	of this adjunction  is homotopy equivalent to the $\mathbb{Q}$-completion  of Bous\-field--Kan. 
	So it gives a functorial rationalisation for pointed connected finite simplicial sets that are nilpotent. 
\end{theorem}

Otherwise stated, the reduced shifted $\Li$-algebra $\widetilde{\mathfrak{L}}(X_\bullet)$ is a functorial model 
for pointed connected finite nilpotent simplicial sets. 

\begin{corollary}\label{Cor:LastbutnotLeast}
For any finite type nilpotent  simplicial set  $X_\bullet$, there are group isomorphisms 
	\[
	\pi_n(X_\bullet)\otimes \mathbb{Q} \cong {H}_n\left(\widetilde{\mathfrak{L}}(X_\bullet)\right)\ , \quad \text{for}\  \ n\geqslant 1\ ,
	\]
where the right-hand side is equipped with the Baker--Campbell--Hausdorff product, for $n=1$, and with the sum, for $n\geqslant 2$~.
\end{corollary}

\begin{remark}
The proof given in \emph{loc. cit.} relies on the Sullivan approach, given below,  under linear dualisation. 
So we conjecture that the present finiteness hypothesis can be significantly weakened by a more intrinsic proof. 
\end{remark}

The second algebraic model for the rational homotopy type of topological spaces was given by D. Sullivan \cite{Sullivan77},  in terms of dg (augmented unital) commutative algebras. (Recall that the category of augmented unital commutative algebras is isomorphic to the category of non-necessarily unital commutative algebras).
The first step of his theory  amounts to defining a suitable generalisation of the de Rham algebra on the level of topological spaces, or equivalently simplicial sets. 

\begin{lemma}\label{lem:SullivanAdj}
The two functors
\[\vcenter{\hbox{
\begin{tikzcd}[column sep=1.2cm]
\mathrm{A}_\mathrm{PL}  \ \colon \ 
\mathsf{sSet}_* 
\arrow[r, harpoon, shift left=1ex, "\perp"']
&
\arrow[l, harpoon,  shift left=1ex]
\mathsf{dg}\ \mathrm{Com}\textsf{-}\mathsf{alg}^\mathsf{op}
\ : \ \langle-\rangle
\end{tikzcd}
}}\]
defined respectively by 
\[ \mathrm{A}_\mathrm{PL}(X_\bullet)\coloneqq \Hom_\mathsf{sSet}(X_\bullet, \Omega_\bullet)
\quad \text{and} \quad
\langle A \rangle \coloneqq \Hom_{\mathsf{dg}\ \mathrm{Com}\textsf{-}\mathsf{alg}}(A, \Omega_\bullet)
\]
are adjoint. 
\end{lemma}

The functor $\mathrm{A}_\mathrm{PL}(X_\bullet)$ produces the 
\emph{algebra of polynomial differential forms}\index{algebra of polynomial differential forms}
of a simplicial set
and the 
functor $\langle A \rangle$ is the \emph{spatial realisation}\index{spatial realisation} of dg commutative algebras. 

\begin{remark}
Notice that this adjunction is ``Koszul dual'' to the main adjunction of higher Lie theory given in Proposition~\ref{prop:AdjHigherLie}. For instance, the algebra of polynomial differential forms can be obtained as the right Kan extension along the Yoneda embedding: 
\[\mathrm{A}_\mathrm{PL}(X_\bullet)\cong \mathrm{Ran}_{\mathrm{Y}^{\mathrm{op}}} {\Omega_\bullet} (X_\bullet) \ .\]
\end{remark}

The original rationalisation of spaces was given by the unit of the derived version of this adjunction. 

\begin{theorem}[{\cite{Sullivan77}, \cite[Theorem~11.2]{BousfieldGugenheim}}]\label{thm:SuillvanRat}
For finite type nilpotent simplicial sets $X_\bullet$~, the unit 
	\[
	X_\bullet\longrightarrow \mathbb{R}\langle\mathrm{A}_{\mathrm{PL}}(X_\bullet)\rangle
	\]
of the derived  $\mathrm{A}_{\mathrm{PL}} \dashv \mathbb{R}\langle -\rangle$-adjunction is a functorial rationalisation.
\end{theorem}

As any pointed simplicial set is already cofibrant, one does not need to consider any cofibrant replacement to derive the functor $\mathrm{A}_{\mathrm{PL}}$. On the other hand, 
the right derived functor  $\RR \langle-\rangle$ is defined by applying $\langle-\rangle$ to a cofibrant replacement of a dg (augmented unital) commutative algebra. Such cofibrant replacement are given by (retracts of some) quasi-free dg commutative algebras. 

\begin{definition}[Quasi-free dg unital commutative algebra]\index{quasi-free dg unital commutative algebra}
A \emph{quasi-free dg unital commutative algebra} is a dg unital commutative algebra $(S(V), \mathrm{d})$, whose underlying graded algebra is free. 
\end{definition}

In this case, $V$ is a graded module and $\dd$ is a square-zero derivation of homological degree $-1$~. 
As such it is completely characterised by the image $\dd|_V \colon V \to S(V)$ of the generators, which splits according to the length of monomials as 
\begin{align*}
&\dd_0 \colon V \to \mathbb{Q}~,\\
&\dd_1 \colon V \to V~,\\
&\dd_2 \colon V \to V^{\odot 2}~,\\
&\dd_3 \colon V \to V^{\odot 3}~,\\
&\qquad  \vdots\quad \ , 
\end{align*}
where $V^{\odot n}\coloneqq \left(V^{\otimes n}\right)_{\Sy_n}$ denotes the symmetric tensor product. 
Under suitable finiteness assumptions, the degree-wise linear dual of $V$ is thus equipped with the following dual maps 
\begin{align*}
&\ell_0 \colon \mathbb{Q}\to V^\vee~,\\
&\ell_1 \colon V^\vee \to V^\vee~,\\
&\ell_2 \colon  \big(V^\vee\big)^{\odot 2}\to V^\vee~,\\
&\ell_3 \colon  \big(V^\vee\big)^{\odot 3}\to V^\vee~,\\
&\qquad  \vdots\quad 
\end{align*}
of degree $-1$, where $\ell_n\coloneqq \dd_n^\vee$~. The relation $\dd^2=0$ is equivalent to the shifted curved $\Li$-relations for these operations. 

\begin{proposition}
Under suitable finiteness assumptions, the data of a quasi-free dg unital commutative algebra
$(S(V), \mathrm{d})$
 is equivalent to the data of a shifted curved $\Li$-algebra $\left(V^\vee, \ell_0, \ell_1, \ell_2, \ell_3, \ldots\right)$~.
\end{proposition}

One can thus say that D. Sullivan was the first one to consider shifted curved $\Li$-algebras, though this notion is not clearly emphasised in \cite{Sullivan77}.

\begin{definition}[Sullivan model]\index{model!Sullivan }
A \emph{Sullivan model} of a pointed simplicial set $X_\bullet$ is a quasi-free dg unital commutative algebra 
$(S(V), \mathrm{d})$ 
generated by a graded module $V$ concentrated in  negative homological degrees and endowed with an exhaustive filtration 
\[
\{0\}=V^{(-1)} \subseteq
V^{(0)}\subseteq 
V^{(1)}\subseteq 
V^{(2)}\subseteq \cdots
\]
satisfying $\dd\left(V^{(k)}\right)\subset S\left(V^{(k-1)}\right)$, for any $k\geqslant 0$, and 
equipped with a quasi-isomorphism 
\[(S(V), \mathrm{d})   \xrightarrow{\sim} \mathrm{A}_\mathrm{PL}(X_\bullet)~.\]
\end{definition}

The degree condition on the space of generators implies that $\dd_0=0$, so Sullivan models are augmented dg unital commutative algebras, where $\dd_1$ squares to zero. This means that under the finite type condition the data of the differential of the space of generators of a Sullivan model is equivalent to the data of a shifted $\Li$-algebra structure. Recall that a chain complex has \emph{finite type} when each component is finite dimensional. 

\begin{proposition}[{\cite[Proposition~2.2]{Berglund15}}]
A finite type quasi-free dg unital commutative algebra
 $\big(S\big(\g^\vee\big), \dd\big)$  satisfies the underlying conditions of a Sullivan model if and only if $\g$ carries a finite type positively graded profinite shifted $\Li$-algebra structure. 
\end{proposition}

\begin{proof}
The chain complex $V\coloneqq \g^\vee$ is negatively graded and since it is of finite type, 
one can dualise the $\Li$-operations in order to form a differential $\dd$ on
the quasi-free dg unital commutative algebra 
$S\big(\g^\vee\big)$~.
The increasing exhaustive filtration $V^{(k)}$ corresponds to a decreasing complete filtration $\mathrm{F}_k \g$ under linear dualisation $V^{(k)}\coloneqq (\mathrm{F}_k \g)^\perp$~. For the definition of the profinite condition and the rest of the proof, we refer the reader to \emph{op. cit.}.
\end{proof}

One can go even further in the relationship between Sullivan's approach to rational homotopy theory and higher Lie theory. 

\begin{proposition}\label{prop:MCvsSullivan}
For any finite type non-negatively graded complete shifted $\Li$-algebra $\g$, there is a canonical isomorphism of simplicial sets 
\[
\MC_\bullet(\g)\cong \big\langle S\big(\g^\vee\big)\big\rangle
~.\]
\end{proposition}

\begin{proof}
The proof relies on the direct inspection 
\begin{align*}
\big\langle S\big(\g^\vee\big)\big\rangle
=\Hom_{\mathsf{dg}\ \mathrm{Com}\textsf{-}\mathsf{alg}}\big(S\big(\g^\vee\big), \Omega_\bullet\big)
\cong \MC\big(\g \,\widehat{\otimes}\, \Omega_\bullet\big)=\MC_\bullet(\g)
\end{align*}
that the Maurer--Cartan equation is equivalent to the commutativity with the respective differentials. 
\end{proof}

This result allows us to connect the two notions of algebraic models of the rational homotopy type. 

\begin{corollary}
Let $\g$ be a finite type positively graded profinite shifted $\Li$-algebra such that $\big(S\big(\g^\vee\big), \dd\big)$ 
is a Sullivan model of 
a  finite type nilpotent simplicial set $X_\bullet$~. Then $\g$  is an $\Li$-algebra model of $X_\bullet$~. 
\end{corollary}

\begin{proof}
Since a Sullivan model is a cofibrant dg commutative algebra, it can be used to compute the derived composite 
\[\mathbb{R}\langle\mathrm{A}_{\mathrm{PL}}(X_\bullet)\rangle\simeq 
\big\langle S\big(\g^\vee\big)\big\rangle~, \]
which is a rationalisation of $X_\bullet$ by \cref{thm:SuillvanRat}. 
Then Proposition~\ref{prop:MCvsSullivan} shows that $\MC_\bullet(\g)$ is also a rationalisation of $X_\bullet$~. 
We conclude by \cref{thm:Getzler}: Hinich's construction $\MC_\bullet(\g)$ is homotopy equivalent to the integration functor $\mathrm{R}(\g)$~.
\end{proof}

A deep application of the two types of algebraic models of the rational homotopy type lies in the following result dealing with the most ubiquitous but most complicated type of topological spaces:  mapping spaces. 

\begin{theorem}[{\cite{Lazarev13, Berglund15, BuijsMurillo11}}]
For any Sullivan model  $A$ of a connected pointed space $X_\bullet$ and 
any profinite shifted $\Li$-algebra model $\g$ of a nilpotent space $Y_\bullet$ of finite $\mathbb{Q}$-type, 
the complete tensor product $A\,\widehat{\otimes}\, \g$ is a model of the mapping space, that is 
\[ \mathrm{map}\left(X_\bullet, Y^\mathbb{Q}_\bullet\right)\simeq \mathrm{R}\big(A\,\widehat{\otimes}\, \g\big)~.  \]
\end{theorem}

In this case, the homotopy classes of maps $f \colon X_\bullet \to Y^\mathbb{Q}_\bullet$ are in one-to-one correspondence with the gauge classes of Maurer--Cartan elements $\alpha\in \MC\big(A\,\widehat{\otimes}\, \g\big)$~,  
\[\left[X_\bullet, Y_\bullet^\mathbb{Q}\right] \cong \mathscr{MC}\big(A\, \widehat{\otimes}\, \g\big)~. \] 

\begin{corollary}
Under the same hypotheses, the twisted shifted $\Li$-algebra 
$\big(A\,\widehat{\otimes}\, \g\big)^\alpha$ is a model for the connected component of $f$ of the mapping space: there  group isomorphisms 
	\[
	\pi_n\left( \mathrm{map}\left(X_\bullet, Y^\mathbb{Q}_\bullet\right), f\right)\otimes \mathbb{Q} \cong 
	{H}_n\left(\big(A\,\widehat{\otimes}\, \g\big)^\alpha\right)\ , \quad \text{for}\  \ n\geqslant 1\ ,
	\]
where the right-hand side is equipped with the Baker--Campbell--Hausdorff product, for $n=1$, and with the sum, for $n\geqslant 2$~.
\end{corollary}

\begin{remark}
The same results hold \emph{mutatis mutandis} when one works with dg (homotopy) \emph{cocommuative} coalgebra models and convolution algebras under 
\[
\Hom(C,\g)\cong C^\vee\,\widehat{\otimes}\, \g~.
\]
We refer the reader to \cite[Section~9]{RNW19} for more details. 
\end{remark}

\section{Simplicial theory of homotopy algebras}\label{sec:SimplicialTheoryHoAlg}

Dolgushev--Hoffnung--Rogers settled in \cite{DolgushevHoffnungRogers14} a meaningful simplicial enrichment of the category of homotopy algebras with their $\infty$-morphisms. In this section, we give it a short presentation using the properties of the twisting procedure established in \cref{sec:GaugeTwist}. \\ \index{simplicial enrichment}

Let  $\calC$ be a filtered dg cooperad 
and let $A$ and $B$ be two  complete graded modules. 
Recall that the  direct sum provides us with the product in the category of complete dg Lie algebras. 
We consider here the complete dg Lie algebra 
\[\left(\hom_\Sy\big(\calC, {\eend}_B\big) \oplus \hom_\Sy\big(\calC, {\eend}_A\big), \partial, [\;,\,]\right) \]
with bracket defined by 
\[\left[\left(f,f'\right), \left(g,g'\right)\right]\coloneqq (-1)^{|g||f'|}\left(\left[f,g\right], \left[f', g'\right]\right)~. \]
Its  Maurer--Cartan elements coincide with 
 pairs of complete $\Omega \calC$-algebra structures on $B$ and $A$ respectively. 
We will need  the  complete $\Sy$-module 
\[\mathrm{end}^A_B\coloneqq\left\{\hom(A^{\otimes n}, B)\right\}_{n\in \NN}\]
equipped with its canonical left ${\eend_A}$-module  and 
right ${\eend_B}$-module structures. 
We denote by 
\[\calC \xrightarrow{\Delta_{(n)}}  \calC\; \widehat{\circ}_{(n)}\, \calC\]
the suitable $(n-1)$-fold iterations of the partial decomposition maps which produce 
all the 2-levelled trees with  $n$ top vertices. 

\begin{lemma}
The complete dg module 
\[
s\hom_\Sy\big(\calC, {\eend}_B\big)\oplus 
\hom_\Sy\left(\calC, {\eend}^A_B\right)\oplus 
s\hom_\Sy\big(\calC, {\eend}_A\big)
\]
equipped with 
\begin{align*}
&\ell_2\left(f_1, s\alpha\right)\, \colon 
\calC \xrightarrow{\Delta_{(1)}}  \calC\; \widehat{\circ}_{(1)}\, \calC 
\xrightarrow{(-1)^{|f_1|}f_1\; \widehat{\circ}_{(1)}\; \alpha}  
{\eend}^A_B\; \widehat{\circ}_{(1)}\, {\eend}_A  
\to {\eend}^A_B \ ,
\\
&\ell_{n+1}\left(s\beta, f_1, \ldots, f_n\right)\, \colon
\calC \xrightarrow{\Delta_{(n)}}  \calC\; \widehat{\circ}_{(n)}\, \calC 
\xrightarrow{-\beta \; \widehat{\circ}_{(n)}\; f_1\odot \cdots \odot f_n}  
{\eend}^A_B\; \widehat{\circ}_{(n)}\, {\eend}_A  
\to  {\eend}^A_B \ ,
\end{align*}
for any $n\geqslant 1$ and for any $\alpha \in \hom_\Sy\big(\calC, {\eend}_A\big)$, 
$\beta \in \hom_\Sy\big(\calC, {\eend}_B\big)$, 
$f_1, \ldots, f_n \in \hom_\Sy\left(\calC, \allowbreak {\eend}^A_B\right)$, 
forms a complete shifted $\Li$-algebra extension of the shifted dg Lie algebra 
$s\hom_\Sy\big(\calC, {\eend}_B\big)\oplus s\hom_\Sy\big(\calC, {\eend}_A\big)$~.
\end{lemma}

\begin{proof}
The proof amounts to a straightforward computation.
\end{proof}

\begin{proposition}\label{prop:MCinfiniMorph}
Maurer--Cartan elements $(s\beta, f, s\alpha)$ of this complete shifted $\Li$-algebra are in one-to-one correspondence with the data of two complete $\Omega \calC$-algebra structures on $A$ and $B$ respectively related by an $\infty$-morphism. 
\end{proposition}

\begin{proof}
It remains to check that the part of the Maurer--Cartan equation satisfied by $f$ on $\hom_\Sy\left(\calC, {\eend}^A_B\right)$
is equivalent to the equation~\eqref{eqn=InftyMor} defining the notion of an $\infty$-morphism. 
Notice that one has to consider here Maurer--Cartan elements which live in the first part $\mathcal{F}_0$ of the filtration. 
The Maurer--Cartan equation is an infinite sum of terms involving $\ell_2, \ell_3, \ldots$; it actually makes sense here in fully generality since their evaluation on elements of $\mathcal{C}$ only involve a first number of non-trivial terms. 
\end{proof}

Given any $\Omega \calC$-algebra structures $(A, \alpha)$ and $(B,\beta)$, we consider 
the following complete shifted $\Li$-sub-algebra of the above algebra twisted by $(s\beta, 0, s\alpha)$:
\begin{multline*}
\mathfrak{hom}_{\alpha, \beta}\coloneqq \\\left(\hom_\Sy\left(\calC, {\eend}^A_B\right),
\ell_2(-, s\alpha)+\ell_2(s\beta, -) -\partial, \ell_{n+1}(s\beta, -, \ldots, -), \  n\geqslant 2
\right)~.
\end{multline*}
\cref{prop:MCinfiniMorph} shows that its Maurer--Cartan elements $\MC\left(\mathfrak{hom}_{\alpha, \beta}\right)$ correspond bijectively to $\infty$-morphisms from $\alpha$ to $\beta$. 

\begin{remark}
In order to make this last claim precise and to solve the  issue raised at the end of the proof of \cref{prop:MCinfiniMorph}, Dolgushev--Hoffnung--Rogers  restricted themselves in \cite{DolgushevHoffnungRogers14} to
filtered dg cooperads $\calC$ with trivial component $\calC(0)=0$ of arity $0$. Then, they considered the complete filtration on $\mathfrak{hom}_{\alpha, \beta}$ given by $f\in \F_k \mathfrak{hom}_{\alpha, \beta}$ 
if $f(n) : \calC(n) \to \hom\big(A^{{\otimes} n}, B\big)$ is trivial for $n<k$~. 
 Finally, Maurer--Cartan elements concentrated in $\F_1 \mathfrak{hom}_{\alpha, \beta}$ are indeed in one-to-one correspondence with $\infty$-morphisms from $\alpha$ to $\beta$. 
As we mentioned above, the Maurer--Cartan equation in this shifted $\Li$-algebra always makes sense, no matter the filtration considered. As such it forms an \emph{absolute shifted $\Li$-algebra}\index{$\Li$-algebra!absolute shifted}, notion introduced and studied by V. Roca i Lucio in \cite{RocaLucio21}, where infinite series of operations make sense. 
\end{remark}

The first step toward a simplicial enrichment of complete $\Omega \calC$-algebras with their $\infty$-morphisms am\-oun\-ts to 
enrich them over  complete shifted (curved) $\Li$-algebras using the above construction. In this direction, one first needs to extend the cartesian  monoidal structure on dg Lie algebras.

\begin{lemma}
The category 
\[
\left(\infty\textsf{-}\mathsf{complete}\ \mathsf{curved}\ \sLi\textsf{-}\mathsf{alg}, \oplus, 0\right)
\] 
of complete shifted curved $\Li$-algebras with  $\infty$-morphisms equipped with the direct sum forms a cartesian  monoidal category.
\end{lemma}

\begin{proof}
It is straightforward to check that the direct sum 
\begin{multline*}
(A,\F, \ell_0, \ell_1, \ell_2, \ell_3, \ldots)\oplus 
(B,\G, k_0, k_1, k_2, k_3, \ldots)
\coloneqq \\
(A\oplus B,\F\oplus \G, \ell_0+k_0, \ell_1+k_1, \ell_2+k_, \ell_3+k_3, \ldots)
\end{multline*}
of two complete shifted curved $\Li$-algebras is again a complete shifted curved $\Li$-algebra. The functoriality of this construction is given by 
\[(f\oplus g)_n\big(
(a_1, b_1), \ldots, (a_n, b_n)\big)\coloneqq 
f_n(a_1, \ldots, a_n)\oplus g_n(b_1, \ldots, b_n)\]
for any $\infty$-morphisms $f$ and $g$. 
One can check that this construction provides us with  the product in the category of 
complete shifted curved $\Li$-algebras. 
\end{proof}

\begin{proposition}[{\cite[Section~3]{DolgushevHoffnungRogers14}}]\label{prop:EnrichLinfin}
The composition 
\begin{align*}
\Phi^{\alpha, \beta, \gamma}\ \colon \ \mathfrak{hom}_{\beta, \gamma} \oplus \mathfrak{hom}_{\alpha, \beta} \rightsquigarrow
\mathfrak{hom}_{\alpha, \gamma}
\end{align*}
defined  by 
$\Phi^{\alpha, \beta, \gamma}_0\coloneqq 0$, 
$\Phi^{\alpha, \beta, \gamma}_1\coloneqq 0$, and, for $n\geqslant 2$, by 
\begin{align*}
\Phi^{\alpha, \beta, \gamma}_n(f_1, \ldots, f_n)\, : \, 
\calC \xrightarrow{\Delta_{(n-1)}}  \calC\; \widehat{\circ}_{(n-1)}\, \calC 
\to 
{\eend}^B_C\; \widehat{\circ}_{(n-1)}\, {\eend}^A_B
\to  {\eend}^A_C \ ,
\end{align*}
 where the second map is equal to 
 \[\sum_{i=1}^n f'_i\,  \widehat{\circ}_{(n-1)}\, f''_1 \odot \cdots \odot f''_{i-1} \odot f''_{i+1} \odot \cdots \odot f''_n~,\]
 under the notation $f_i=\left(f'_i, f''_i\right)\in\mathfrak{hom}_{\beta, \gamma} \oplus \mathfrak{hom}_{\alpha,\beta}$, 
and the  unit 
\begin{align*}
\Upsilon^{\alpha}\ \colon \  0 \rightsquigarrow
\mathfrak{hom}_{\alpha, \alpha}
\end{align*}
defined by 
\begin{align*}
\Upsilon^{\alpha}_0\coloneqq \id_A \quad \text{and}\quad \Upsilon^{\alpha}_n\coloneqq 0~, \  \ \text{for}\ 
n\geqslant 1~,
\end{align*}
endow $\Omega \calC$-algebras with a structure of a category enriched over the cartesian monoidal category of 
complete shifted curved $\Li$-algebras. 
\end{proposition}

\begin{proof}
All the axioms are  straightforward applications of the operadic calculus.
\end{proof}

It remains to integrate these complete shifted (curved) $\Li$-algebras in order to obtain a simplicial enrichment. 

\begin{lemma}\label{lem:HinichMonoidal}
Hinich's construction of a Deligne $\infty$-groupoid defines 
a symmetric monoidal functor
\[   
\MC_\bullet \ \colon \ 
\left(\infty\textsf{-}\mathsf{complete}\ \mathsf{curved}\ \sLi\textsf{-}\mathsf{alg}, \oplus, 0\right)
\to 
\left(\mathsf{sSet}, \times, *\right)~.
\]
\end{lemma}

\begin{proof}
Given any complete unital dg commutative algebra $(B, \G, \dd, \mu, 1)$, the assignment 
\[
(A,\F, \ell_0, \ell_1, \ell_2, \ell_3, \ldots) \mapsto 
\left(A\,\widehat{\otimes}\,B, \F\otimes \G, \ell_0\,\widehat{\otimes}\,1, \ell_1\,\widehat{\otimes}\,\dd, \ell_2\,\widehat{\otimes}\,\mu, \ell_3\otimes \mu^2, \ldots\right)
\]
generates a new complete shifted curved $\Li$-algebra out of another one. It is functorial with respect to the following assignment of $\infty$-morphisms: given any $\infty$-morphism $f : (A,\alpha) \rightsquigarrow (A', \alpha')$, one considers the $\infty$-morphism defined by $f_0\otimes 1$ and by 
\[\left(a_1\,\widehat{\otimes}\, b_1, \ldots, a_n\,\widehat{\otimes}\, b_n\right)\mapsto f_n(a_1, \ldots, a_n)\,\widehat{\otimes}\,\mu^{n-1}(b_1, \ldots, b_n)~, \]
for $n\geqslant 1$. So this defines defines a symmetric monoidal endofunctor in the category $\infty\textsf{-}\mathsf{complete}\ \allowbreak \mathsf{curved}\ \allowbreak \sLi\textsf{-}\mathsf{alg}$. 
Proposition~\ref{prop:PropertyCurvInfMorph} shows that the Maurer--Cartan functor $\MC$ is a symmetric monoidal functor from this category to that of sets. In the end, since 
 Hinich's construction
 \[\MC_\bullet(\g)\coloneqq \MC\big(\g \,\widehat{\otimes}\, \Omega_\bullet\big)\]
  is the composite of all these symmetric monoidal functors, it is also a symmetric monoidal functor. 
\end{proof}

\begin{theorem}[{\cite[Theorem~3.6]{DolgushevHoffnungRogers14}}]
The category with objects the complete $\Omega \calC$-algebras and with mapping spaces  the Kan complexes 
\[\mathrm{Map}(\alpha, \beta)\coloneqq \MC_\bullet\left(\mathfrak{hom}_{\alpha, \beta}\right)\]
forms a simplicial category whose image under the $0$-sim\-pli\-cies functor gives the category of 
complete $\Omega \calC$-algebras with their $\infty$-morphisms.
\end{theorem}

\begin{proof}
The first part is a direct corollary of Proposition~\ref{prop:EnrichLinfin} and Lemma~\ref{lem:HinichMonoidal}. 
Regarding the second part, it is enough to check that 
\[
\MC_0\left(\Phi^{\alpha, \beta, \gamma}\right)(g,f)=\MC\left(\Phi^{\alpha, \beta, \gamma}\right)(g,f)=g\circledcirc f~,
\]
for any pair $f\in \MC\left(\mathfrak{hom}_{\alpha, \beta}\right)$ and 
$g\in \MC\left(\mathfrak{hom}_{\beta, \gamma}\right)$
of $\infty$-morphisms. 
\end{proof}

\begin{remark}
One might want instead to integrate ``more efficiently'' these complete shifted $\Li$-algebras with Getzler's functor $\mathrm{R}$. This is not directly possible as its functoriality with respect to $\infty$-morphisms is still unclear, see \cite[Section~3]{Robert-NicoudVallette20} for more details. 
\end{remark}

From now on, we suppose that the filtered dg cooperad $\calC$ is coaugmented. 
This hypothesis implies that the underlying space of any $\Omega\calC$-algebra is a chain complex and 
that the notion of an $\infty$-quasi-isomorphism makes sense. 
In this case, the above simplicial category $\infty\textsf{-}\Omega\calC\textsf{-}\mathsf{alg}^\Delta$ provides us with a suitable localisation of the 
category 
$\infty\textsf{-}\Omega\calC\textsf{-}\mathsf{alg}$ of complete $\Omega\calC$-algebras  with respect to $\infty$-quasi-isomorphisms. 

\begin{theorem}[{\cite[Theorem~4.1]{DolgushevHoffnungRogers14}}]
When the filtered dg cooperad $\calC$ is coaugmented, the canonical functor 
\[\infty\textsf{-}\Omega\calC\textsf{-}\mathsf{alg}\to 
\pi_0\left(\infty\textsf{-}\Omega\calC\textsf{-}\mathsf{alg}^\Delta\right)\]
sends $\infty$-quasi-isomorphisms to isomorphisms and is universal with respect to this property. 
\end{theorem}

One example of an appealing homotopical property that this setting allows to prove lies in the following characterisation of $\infty$-quasi-isomorphisms. 

\begin{proposition}[{\cite[]{DolgushevHoffnungRogers14}}]
When the filtered dg cooperad $\calC$ is coaugmented, 
the following assertions are equivalent:
\begin{enumerate}
\item an $\infty$-morphism $f \colon \alpha \rightsquigarrow \beta$ is an $\infty$-quasi-isomorphism, 
\item the pullback maps $\MC_\bullet\left(f^*\right)$:
\[\mathrm{Map}(\beta, \alpha)  \xrightarrow{\sim}\mathrm{Map}(\alpha, \alpha) \quad \text{and}\quad 
\mathrm{Map}(\beta, \beta)  \xrightarrow{\sim}\mathrm{Map}(\alpha, \beta)\]
are homotopy equivalences, 
\item the pushforward maps $\MC_\bullet\left(f_*\right)$:
\[\mathrm{Map}(\beta, \alpha)  \xrightarrow{\sim}\mathrm{Map}(\beta, \beta) \quad \text{and}\quad 
\mathrm{Map}(\alpha, \alpha)  \xrightarrow{\sim}\mathrm{Map}(\alpha, \beta)\]
are homotopy equivalences.
\end{enumerate}
\end{proposition}

\section{Floer cohomology of Lagrangian submanifolds}

In order to introduce a well-defined notion a Floer cohomology for (pairs of) Lagrangians submanifods, Fukaya--Oh--Ohta--Ono introduced and used in \cite{FOOO09I} the twisting procedure for curved $\Ai$-algebras in a crucial way. 
This provides symplectic geometry/topology with a conceptual and powerful tool. 
It also defines properly the sets of morphisms of Fukaya $\Ai$-categories which appear on the A-side of the homological mirror symmetry conjecture. 
\\

One version of the Arnold's conjecture \cite{Arnold65} states that the number of non-degenerate fixed points of any Hamiltonian diffeomorphism is bounded below by the total rank of the homology of the ambient (compact) symplectic manifold. 
In order to prove it, Floer introduced in \cite{Floer88} a (co)homology theory for Lagrangian submanifolds where the boundary operator is given by counting the number of pseudo-holomorphic discs lying between two Lagrangian submanifolds. Since he focused on the case of pairs of Lagrangian submanifolds where one is the image of the other under an Hamiltonian diffeomorphism, he was able to prove that this boundary operator  squares to zero. 
In the general case, such a property does not hold anymore and the authors of \cite{FOOO09I} had to study the associated obstruction problem. The problem of counting the number of pseudo-holomorphic discs is actually not well-defined as it depends on various choices like the perturbations to make the associated moduli spaces transversal. 
Fukaya--Oh--Ohta--Ono introduced some curved $\Ai$-algebras whose constant structures are based on the numbers of pseudo-holomorphic discs and they show that their homotopy type does not depend on the various choices involved. In the end, twisting these curved $\Ai$-algebras by Maurer--Cartan elements produces cochain complexes which define Floer cohomology for (pairs of) Lagrangian submanifolds. Let us now describe  their  work in more details.  \\

These authors initiate their theory with the following  $\Ai$-algebra that they call \emph{classical}.  

\begin{proposition}[{\cite[Theorem~3.1.2]{FOOO09I}}]\label{prop:ClassicalAinfini}
For any   smooth oriented manifold $L$, 
there exists a countable set $\chi_L$ of smooth singular simplicies on $L$ such that $\mathbb{Z}\chi_L$ admits an $\Ai$-algebra structure $\left(\overline{m}_1, \ldots, \overline{m}_n,\ldots\right)$ 
constructed from a Kuranishi structure on the  moduli space of constant pseudo-holomorphic discs
such that $\overline{m}_1=(-1)^n \mathrm{d}_{\text{sing}}$, where $\dim L=n$, where $\mathrm{d}_{\text{sing}}$ is  the singular boundary operator, and such that $H^\bullet\left(\mathbb{Z}\chi_L, \overline{m}_1\right)\cong H^\bullet_{\text{sing}}\left(L,\mathbb{Z}\right)$~.

This $\Ai$-algebra is uniquely determined by $L$ up to homotopy equivalence and its real extension 
$\left(\mathbb{R}\chi_L, \overline{m}_1, \ldots, \overline{m}_n,\ldots\right)$ is homotopy equivalent to the de Rham dg associative algebra of $L$.
\end{proposition}

They  introduce then a \emph{quantum} deformation of it  as a curved $\Ai$-algebra defined over the following extension of $\mathbb{Q}$. 

\begin{definition}[Universal Novikov ring]\index{universal Novikov ring}

\[\Lambda_{0, \mathrm{Nov}}\coloneqq
\left\{
\sum_{i=0}^\infty a_i T^{\lambda_i}e^{n_i}\ | \ a_i\in \mathbb{Q}, \, \lambda_i\in \RR_{\geqslant 0}, \, n_i\in \mathbb{Z}
, \, \lim_{i\to \infty} \lambda_i =\infty
\right\}
\]
where $T$ and $e$ are elements of degree $0$ and $2$ respectively. 
\end{definition}

This ring is filtered  and complete with respect to the \emph{energy} filtration
\[\F^\lambda \Lambda_{0, \mathrm{Nov}} \coloneqq T^\lambda \Lambda_{0, \mathrm{Nov}}\]
indexed by the monoid $\RR_{\geqslant 0}$, and not $\NN$ as in Chapter~\ref{sec:OptheoyFilMod}.  Geometrically, the parameters $\lambda_i$ correspond to the symplectic area of a pseudo-holomorphic disc and $2n_i$ correspond to its Maslov index. 

\begin{theorem}[{\cite[Theorem~3.1.5 $\&$ Theorem~3.1.9]{FOOO09I}}]
For any relatively spin Lagrangian submanifold $L$ and for any $\chi_L$ given in Proposition~\ref{prop:ClassicalAinfini}, 
there exists a countable set $\chi_1(L)\supset \chi_l$ of smooth singular simplicies on $L$ such that 
$\mathbb{Q}\, \chi_1(L)\, \widehat{\otimes}_\mathbb{Q}\,  \Lambda_{0, \mathrm{Nov}}$ admits a complete curved $\Ai$-algebra structure $\left({m}_0, \allowbreak m_1, \allowbreak \ldots, \allowbreak  {m}_n, \allowbreak \ldots\right)$
whose reduction to $\mathbb{Q}$ and to $\chi_L$ gives the $\Ai$-algebra of 
Proposition~\ref{prop:ClassicalAinfini}.
\end{theorem}

The construction of this complete curved $\Ai$-algebra associated to a relatively spin Lagrangian submanifold  encodes the contributions of all pseudo-holomorphic discs attached to it. 
There are non-trivial issues of  transversality and orientation of this moduli space of pseudo-holomorphic discs with Lagrangian boundary conditions in the choices of $\chi_1(L)$ and in the $\Ai$-algebra structure, which are discussed in \cite[Chapters~7-8]{FOOO09II} and which required 
the \emph{relatively spin} assumption, see \cite[Definition~3.1.1]{FOOO09I}.\\

The existence of pseudo-holomorphic discs bubbling off at the boundary of a pseudo-holomorphic strip prevents this  
curved $\Ai$-algebra to be strict, i.e., $(m_1)^2\neq 0$ in general.  
(The appearance of $\Ai$-algebra structure in the present context is not a surprise as all the configurations of bubbling off holomorphic discs can be described by planar trees.) 
A Maurer--Cartan element in the abovementioned complete curved $\Ai$-algebra is  an element $a$ in 
\[\F^{>0}\left(\mathbb{Q}\, \chi_1(L)\, \widehat{\otimes}_\mathbb{Q}\,  \Lambda_{0, \mathrm{Nov}}\right)^1~, \]
since the filtration is indexed by the monoid of non-negative integers and since the authors are working with cohomological degree convention. It is required to satisfy the Maurer--Cartan equation~\eqref{eqn:MC}: 
\begin{eqnarray}
m_0(1)+m_1(a)+\sum_{n\geqslant 2}m_n(a, \ldots, a)=0\ .
\end{eqnarray}
In this context, the set of Maurer--Cartan elements is denoted by ${{\MC}}(L)$.

\begin{definition}[Unobstructed Lagrangian submanifold]\index{unobstructed Lagrangian submanifold}
A relatively spin Lagrangian submanifold $L$ is called \emph{unobstructed} when the set 
${{\MC}}(L)$ of Mau\-rer--Car\-tan elements is not empty. 
\end{definition}

Cohomological Floer obstruction classes to the vanishing of ${{\MC}}(L)$ are given in \cite[Theorem~3.1.11]{FOOO09I}.

\begin{remark}
In \emph{loc. cit.}, the notion of a curved $\Ai$-algebra is called a \emph{filtered $\Ai$-algebra} and the notion of a Maurer--Cartan element is called a \emph{bounding cochain}. 
\end{remark}

\begin{definition}[Floer cohomology of a Lagrangian submanifold]\index{Floer cohomology}
The \emph{Floer cohomology} deformed by a Maurer--Cartan element $a\in \MC(L)$ is the underlying cohomology of the twisted $\Ai$-algebra of Theorem~\ref{thm:TwProcGp}:
\[HF^\bullet\left(L, a; \Lambda_{0, \mathrm{Nov}}\right)\coloneqq H^\bullet\left( \mathbb{Q}\, \chi_1(L)\, \widehat{\otimes}_\mathbb{Q}\,  \Lambda_{0, \mathrm{Nov}}, m_1^a\right)~.\]
\end{definition}

The fundamental class of $L$ can actually be represented by some linear combinations of the smooth singular simplicies of $\chi_L$; it provides the classical $\Ai$-algebra and the quantum curved $\Ai$-algebra with a homotopy unit. Though this notion is not the subject of the present monograph, let us just mention that it gives rise to another Maurer--Cartan equation where the right-hand side is not $0$ but equal to this homotopy unit. This induces  a weaker obstruction theory. \\

In the case of pairs $\left(L^{(0)},L^{(1)}  \right)$ of Lagrangian submanifolds which are either transversal or which intersect in a clean way, Fukaya--Oh--Ohta--Ono coin a notion of a Floer cohomology theory following the same method, but considering a curved $\Ai$-bimodule over 
$\mathbb{Q}\, \chi_1\left(L^{(0)}\right)\, \widehat{\otimes}_\mathbb{Q}\,  \Lambda_{0, \mathrm{Nov}}$ and 
$\mathbb{Q}\, \chi_1\left(L^{(1)}\right)\, \widehat{\otimes}_\mathbb{Q}\,  \Lambda_{0, \mathrm{Nov}}$
and by twisting it with Maurer--Cartan elements, see \cite[Section~3.7]{FOOO09I}.

\begin{definition}[Gauge equivalence]
Two Maurer--Cartan elements of a complete curved $\Ai$-algebra are \emph{gauge equivalent} if they are gauge equivalent in the associated complete curved $\Li$-algebras obtained by symmetrization, see Proposition~\ref{prop:SymmAiLi}.
\end{definition}

We refer the reader to the beginning of Section~\ref{sec:HigherLieTh} for this latter notion. 

\begin{proposition}[Gauge independence]
Two gauge equivalent Maurer--Car\-tan elements $a\sim a'$ induce two isomorphic deformed Floer cohomology 
\[HF^\bullet\left(L, a; \Lambda_{0, \mathrm{Nov}}\right)\cong HF^\bullet\left(L, a'; \Lambda_{0, \mathrm{Nov}}\right)~.\]
\end{proposition}

Passing from one gauge class of the Maurer--Cartan elements to another one can lead to 
two different Floer cohomology groups. This is called the \emph{wall crossing phenomenon}.\index{wall crossing phenomenon} 

\begin{proposition}\label{prop:HoInvariance}
Given any symplectic diffeomorphism $\psi \colon M \to M'$ and two Lagrangian submanifolds $L\subset M$ and $L'\subset M'$ such that $L'=\psi(L)$, there exists an $\infty$-morphism of curved $\Ai$-algebras
\begin{align*}
\left(\mathbb{Q}\, \chi_1(L)\, \widehat{\otimes}_\mathbb{Q}\,  \Lambda_{0, \mathrm{Nov}}, {m}_0, m_1, \allowbreak \ldots, \right)
\to 
\left(\mathbb{Q}\, \chi_1(L')\, \widehat{\otimes}_\mathbb{Q}\,  \Lambda_{0, \mathrm{Nov}}, {m}'_0, m'_1,  \ldots\right)~, 
\end{align*}
which is a homotopy equivalence. 
\end{proposition}

A Goldman--Millson type theorem \cite{GoldmanMillson88} holds in the complete curved $\Ai$-algebra case: 
any homotopy equivalent complete curved $\Ai$-algebras have bijective moduli spaces of Maurer--Cartan elements \cite[Corollary~4.3.14]{FOOO09I}. As a direct corollary of this and Proposition~\ref{prop:HoInvariance}, one obtains that the Floer cohomology theory is independent of the various choices made.

% Bibliography

\bibliography{bib}
\bibliographystyle{cambridgeauthordate}

% Index

\printindex

\end{document}